\documentclass[10pt,reqno]{amsart}
\usepackage{times}
\usepackage[T1]{fontenc}
\usepackage{amssymb}
\usepackage{amsmath,amsthm}
\usepackage{amsfonts}
\usepackage{leftidx}
\usepackage{mathrsfs}
\usepackage{enumerate}	
\usepackage{abstract}
\usepackage{stmaryrd}
\usepackage{ulem}
\usepackage{graphicx}
 \usepackage{hyperref}

\usepackage[top=0.5in, bottom=0.5in, left=0.8in, right=0.7in]{geometry}

\def\R{\mathbb{R}}

\def\d{|\nabla|}

\def\p{\partial}

\def\vo{\vspace{1\baselineskip}}

\def\h{\frac{1}{2}}

\def\be{\begin{equation}}
\def\ee{\end{equation}}

\newtheorem{theorem}{Theorem}[section]

\newtheorem{lemma}{Lemma}[section]
\newtheorem{proposition}{Proposition}[section]
\theoremstyle{definition}
\newtheorem{definition}{Definition}[section]

\theoremstyle{remark}
\newtheorem{remark}{Remark}[section]

\numberwithin{equation}{section}
\begin{document}
 \title[3D Relativistic Vlasov-Maxwell]{Propagation of  regularity and long time behavior of   the $3D$ massive relativistic transport equation II: Vlasov-Maxwell system}
\author{Xuecheng Wang}
\address{YMSC, Tsinghua  University, Beijing, China,  10084
}
\email{xuecheng@tsinghua.edu.cn,\quad xuecheng.wang.work@gmail.com 
}

\thanks{}
 
\maketitle

\begin{abstract}
Given any smooth, suitably small initial data, which decays polynomially at infinity,   we prove global regularity for the $3D$ relativistic massive Vlasov-Maxwell system. In particular, the compact support assumption, which is widely used in the literature, is not  imposed on the initial data. Our proofs are based on a combination of the Klainerman vector fields method and the Fourier method, which allows us to exploit a crucial hidden null structure in the relativistic Vlasov-Maxwell system.

\end{abstract}
 \tableofcontents
\section{Introduction}\label{introduction}

 In plasma physics,  the evolution of  a sufficiently diluted ionized gas or solar wind under the effect of the electromagnetic forces created by  particles themselves can be  described by the relativistic  Vlasov-Maxwell system. In contrast to the non-relativistic case,  the speed of light in the  relativistic case is assumed to be finite.

  After normalizing the mass of particles and the speed of light to be one, the $3D$ relativistic  Vlasov-Maxwell system with given initial data $(f_0(x,v), E_0(x), B_0(x))$ reads as follows,
 \be\label{mainequation}
(\textup{RVM})\qquad \left\{\begin{array}{c}
\p_t f + \hat{v} \cdot \nabla_x f + (E+ \hat{v}\times B)\cdot \nabla_v f =0,\\

\nabla \cdot E = 4 \pi \displaystyle{\int_{\R^3}  f(t, x, v) d v}, \qquad \nabla \cdot B =0, \\

\p_t E = \nabla \times B - 4\pi \displaystyle{\int_{\R^3} f(t, x, v) \hat{v} d v }, \quad \p_t B  =- \nabla\times E,\\

f(0,x,v)=f_0(x,v), \quad E(0,x)=E_0(x), \quad B(0,x)=B_0(x),
\end{array}\right.
\ee
where  $f(t,x,v)$ denotes the density distribution function of particles, $(E, B)$ stands for the classic electromagnetic field, and $\hat{v}:=\displaystyle{v/\sqrt{1+|v|^2}}$ denotes the relativistic speed of particles, which is strictly less that the speed of light.   We mainly restrict ourself to the three dimensions case and refer readers to \cite{glasseys1,luk2,luk} for the corresponding results in other dimensions.

The is a large literature in the study of the Cauchy problem for  the Vlasov-Maxwell system. 
  A remarkable result   obtained by the Glassey-Strauss \cite{glassey3} says that the  classical solution can be globally  extended  as long as the particle density has compact support in $v$ for all the time. A new proof of this result based on Fourier analysis was   given by Klainerman-Staffilani \cite{Klainerman3}, which adds a new perspective to the study of $3D$ RVM system, see also \cite{alfonso,glasseys2, pallard2}.  An interesting line of research is the continuation criterion for the global existence of the Vlasov-Maxwell system. In \cite{glassey4}, Glassey-Strauss showed that the lifespan of the solution of the relativistic Vlasov-Maxwell system can be continued if the initial data decay at rate $|v|^{-7}$ as $|v|\rightarrow \infty$ and   $\| (1+|v|) f(t,x,v)\|_{L^\infty_x L^1_v}$ remains bounded for all time.
 An improvement of this result and a new continuation criterion was given by Luk-Strain \cite{luk}, which says that a regular solution can be extended as long as $\|(1+|v|^2)^{\theta/2} f(t,x,v)\|_{L^q_x L_v^1}$ remains bounded for $\theta > 2/q, 2< q \leq +\infty$, see also Kunze\cite{kunze}, Pallard\cite{pallard3}, and Patel\cite{patel} for the recent improvements  on the continuation criterion. 

 Although the assumptions in above mentioned results don't depend on the size of energy, the assumptions are indeed  imposed  all the time, which are strong.  One can also ask whether it is possible to obtain global solution by only imposing assumptions on the initial data. The first positive result was given by Glassey-Strauss \cite{glassey2}. It, roughly speaking,  says that  if   the initial particle density $f(0,x,v)$ has a compact supports in both ``$x$'' and ``$v$'' and also the electromagnetic field $(E(0), B(0))$ has compact support in ``$x$'', and moreover the initial data is suitably small, then there exists a unique classical solution. Later, an interesting improvement was given by Schaffer \cite{schaeffer1}, which shows   that a similar result as in \cite{glassey2} also holds without assuming the compact support assumption for the initial particle density in ``$v$'' but with the compact support assumption in ``$x$''   for both the initial particle density and the electromagnetic field.

 An   interesting   question  one can ask for the $3D$ RVM is that whether the regularity of solution can be unconditionally propagated for all the time for unrestricted data. Recently, a very interesting result by Bigorgne\cite{bigorgne} shows that this question can be answered  in dimension $n\geq 4$ for small initial data.

The main goal of this paper is devoted to answer the above question in $3D$ for small data. More precisely, we show global regularity and scattering properties of the $3D$ relativistic Vlasov-Maxwell system for suitably small initial data {without any compact support assumption}. We also refer readers to our first paper \cite{wang} for more detailed introduction, which also includes related discussion on other  Vlasov-wave type coupled systems.

\subsection{A review of the framework proposed in the first paper} 

Since  the Vlasov-Maxwell system is also a Vlasov-Wave type coupled system, we can   use the  framework proposed in \cite{wang} to set up the study of relativistic Vlasov-Maxwell system in this paper. We briefly explain some main ideas here.

In our first paper \cite{wang},  we introduced a framework to study the Vlasov-wave type coupled system and  proved small data global regularity without compact support assumption for  the  $3D$ massive relativistic  Vlasov-Nordstr\"om system, which  reads as follows if we normalize the mass of the particle to be one,

\be\label{vlasovnordstrom}
(\textup{RVN})\quad \left\{\begin{array}{l}
(\p^2_t - \Delta) \phi = \displaystyle{\int_{\R^3} \frac{f}{\sqrt{1+|v|^2}} d v } \\
\\
\displaystyle{\p_t f + \hat{v}\cdot \nabla_x f - \big((\p_t + \hat{v}\cdot \nabla_x)\phi(t,x)\big)(4 f + v\cdot \nabla_v f )- \frac{1}{\sqrt{1+|v|^2}}\nabla_x \phi \cdot \nabla_v f =0}.\\
\end{array}\right.
\ee

 We mention that the main difficulty of the small data global regularity problem for both RVM and RVN is caused by the bulk   derivative ``$\nabla_v f $''. Because $\nabla_v$ doesn't commute with the linear operator $\p_t +\hat{v}\cdot\nabla_x$, it is not \textit{a-priori} clear how the energy of ``$\nabla_v f$'' grows over time. Also the bulk term $\nabla_v f$ appears in the source term, which complicates the problem further.

To get around of this issue, one of the main ideas proposed in \cite{wang} is that, instead of studying $\nabla_v$ directly, we study the following vector field,
\be\label{march6eqn1}
\widetilde{K}_{v }:=\nabla_v + (t-  \sqrt{1+|v|^2}\omega(x-\hat{v}t,v)) \nabla_{v }\hat{v}\cdot \nabla_x,\quad  \Longrightarrow \nabla_v = \widetilde{K}_{v }- (t-  \sqrt{1+|v|^2}\omega(x-\hat{v}t,v)) \nabla_{v }\hat{v}\cdot \nabla_x, 	
\ee
where ``$\omega(x,v)$'' (see (\ref{inhomogeneousmodulation})) depends on the study of the distance to the light cone in $(x,v)$-space, which will be discussed in details in section \ref{vectorfieldssetup}. Moreover, the coefficient ``$ (t-  \sqrt{1+|v|^2}\omega(x-\hat{v}t,v))$'' almost vanishes on the light cone ``$|t|-|x|=0$''. More precisely, we defined an \textit{inhomogeneous modulation} with respect to the light cone $|t|^2-|x+\hat{v}t|^2=0$ as follows, 
\be\label{march6eqn2}
\tilde{d}(t,x,v)= \frac{t}{1+|v|^2} - \frac{\omega(x,v)}{\sqrt{1+|v|^2}}, \quad \Longrightarrow |\tilde{d}(t,x,v)| \lesssim 1+||t|-|x+\hat{v}t||.
\ee

Based on the above main idea, we constructed  a new set of vector fields in \cite{wang}, which not only commute with the linear operator ``$\p_t + \hat{v}\cdot \nabla_x$'' of the relativistic Vlasov equation but also help to understand the bulk derivative $\nabla_v f $ inside the nonlinearity of the Vlasov equation.  We will explain with more details about the new set of vector fields in section  \ref{vectorfieldssetup}.

 Note that $\tilde{K}_v$ commutes with the linear operator ``$\p_t + \hat{v}\cdot\nabla_x$''.  It is more promising to control the energy of $\tilde{K}_v f $ than $\nabla_v  f$ over time.     Since it is ``$\nabla_v f $'' that appears in the nonlinearity of the Vlasov equation, it remains to  control the difference between $\tilde{K}_v f $ and $\nabla_v f $. Recall (\ref{march6eqn1}). The question is reduced to control the additional quantity $(t-  \sqrt{1+|v|^2}\omega(x-\hat{v}t,v)) \nabla_{v }\hat{v}$ in the energy estimate.   We notice that ``$\hat{v}$'' decreases much faster in the radial direction. More precisely, 
 \be\label{march6eqn3}
\frac{v}{|v|}\cdot \nabla_v \hat{v} = \frac{1}{ (1+|v|^2)^{3/2}} \frac{v}{|v|} , \quad  (  e_i\times\frac{v}{|v|})\cdot\nabla_v\hat{v} = \frac{1}{ (1+|v|^2)^{1/2}} (  e_i\times\frac{v}{|v|}), \quad i\in\{1,2,3\},
 \ee
 where $e_i$, $i\in\{1,2,3\}$, denote the standard unit  vectors of the Caretisian coordinates system in $\R^3$, see (\ref{2020feb18eqn1}). From  (\ref{march6eqn2}) and (\ref{march6eqn3}), we have
 \be
(t-  \sqrt{1+|v|^2}\omega(x-\hat{v}t,v)) \frac{v}{|v|}\cdot \nabla_{v }\hat{v}=\frac{\tilde{d}(t,x-\hat{v}t, v)}{\sqrt{1+|v|^2}}\frac{v}{|v|}, \ee
\be\label{march6eqn6}
 (t-  \sqrt{1+|v|^2}\omega(x-\hat{v}t,v)) (  e_i\times\frac{v}{|v|}) \cdot \nabla_{v }\hat{v}={\sqrt{1+|v|^2}} {\tilde{d}(t,x-\hat{v}t, v)}  (  e_i\times\frac{v}{|v|}).
 \ee

 To control the inhomogeneous modulation, recall (\ref{march6eqn2}), 
   we used a Fourier based method to prove that the scalar field decay at rate $1/\big((1+|t|)(1+||t|-|x||) \big)$ over time.  Recall that the decay rate of the scalar field suggested by the Klainerman-Sobolev embedding is $1/\big((1+|t|)(1+||t|-|x||)^{1/2} \big)$. To prove a  stronger decay estimate for the scalar field, we   carefully study the scalar field at the low frequencies. Instead  of working in the $L^2$ type space, we worked in  a $|\xi|^{-1} L^\infty_{\xi}$ type space. See \cite{wang} for more detailed discussion.

Thanks to the good coefficient ``$1/\sqrt{1+|v|^2}$'' in the relativistic Vlasov-Nordstr\"om system (\ref{vlasovnordstrom}), the loss of   weight of size ``$1+|v|$'' in the rotational direction (\ref{march6eqn6}) is not an issue.

\subsection{The losing weight of size ``$ |v|$'' issue}\label{losingvissue}

   Unfortunately, unlike the RVN system (\ref{vlasovnordstrom}),  the benefit of the good coefficient is not available in the relativistic Vlasov-Maxwell system, see (\ref{mainequation}).

 Recall the   decomposition used in (\ref{march6eqn1}) and the equality (\ref{march6eqn6}). We restate this decomposition in the rotational in ``$v$'' direction as follows, 
\be\label{march6eqn20}
(  e_i\times\frac{v}{|v|})\cdot \nabla_v f = (  e_i\times\frac{v}{|v|})\cdot \tilde{K}_v f + \sqrt{1+|v|^2} \tilde{d}(t, x-\hat{v}t, v)  (  e_i\times\frac{v}{|v|})\cdot \nabla_x f.
\ee
From the above equality  (\ref{march6eqn20}) and the first equality in (\ref{march6eqn3}), we know that { the issue of losing a weight of size ``$|v|$ '', which is very problematic when $|v|$ is extremely large, only appears for the rotational in $v$ directional derivative}.

 One might argue that  the issue looks artificial because it is caused by using the above decomposition  (\ref{march6eqn20}). Alternatively, instead of using the vector field $\tilde{K}_v$, we can    use the rotational vector field $\tilde{\Omega}_i:=(e_i\times x )\cdot \nabla_x + (e_i\times v )\cdot \nabla_v$, which also commutes with the linear operator of the relativistic Vlasov equation ``$\p_t +\hat{v}\cdot \nabla_x$''. More precisely,
\be\label{march6eqn21}
(  e_i\times\frac{v}{|v|})\cdot \nabla_v f= \frac{1}{|v| }\tilde{\Omega}_i f - (  e_i\times\frac{x}{|v|})\cdot \nabla_x f.
\ee
Note that the coefficients are small when $|v|$ is extremely large.

Unfortunately, there exists a regime that it doesn't make the essential difference by choosing either the decomposition (\ref{march6eqn20}) or the decomposition (\ref{march6eqn21}) to control the rotational in $v$ directional derivative of $f(t,x,v)$. For example,    the decomposition (\ref{march6eqn20}) and the decomposition (\ref{march6eqn21})  doesn't make the essential difference  in the case $|x|=|t|$, $|\tilde{d}(t,x-\hat{v}t, v)|\sim 1$, and $|v|\sim \sqrt{1+|t|}$. Because the  coefficients in   the decomposition (\ref{march6eqn20}) and the decomposition (\ref{march6eqn21}) are all of size $\sqrt{1+|t|}$. 

\subsection{The hidden null structure}\label{hiddennullstructure}

To get around the aforementioned issue of the losing weight of size ``$|v|$'', we reveal a hidden null structure inside the Vlasov equation of the RVM system (\ref{mainequation}).    
 
 The hidden  \textit{null structure} in the RVM is very subtle and also different from the classic sense. For the nonlinear wave equation, generally speaking, one can check the existence of  {null structure} by checking whether the corresponding symbol vanishes if the frequencies of two inputs are parallel to each other.  Usually, this structure is inherent with the nonlinearity and is  independent of how one takes derivatives for the solution.  

 We will show that there exists a \textit{null structure} inside the equation satisfied by the rotational in $v$ directional derivatives of the distribution function, i.e., the equation satisfied by $(e_i\times v/|v|)\cdot \nabla_x f(t,x,v)$. We understand this \textit{null structure} in the following sense: the symbol of the rotational derivative   $(e_i\times v/|v|)\cdot \nabla_x$ is small near the time resonance set of the nonlinearity of the Vlasov part, which allows us to take the advantage of oscillation with respect to time. 

 A more detailed discussion about this observation will be carried out in subsection \ref{reductionoflemmabulk}. For intuitive purpose, we give an example here. Before that, to better see the  oscillation of the electromagnetic field and to correct the linear effect, we study the profiles of the RVM system, which are defined as follows,
 \[
g(t,x,v):=f(t,x+\hat{v}t, v), \quad h_1(t):=e^{it\d}\d^{-1}\big(\p_t -i \d  \big)E(t), \quad  h_2(t):=e^{it\d}\d^{-1}\big(\p_t -i \d  \big)B(t). 
 \] 

 Intuitively speaking, we have the following bulk term when we study the evolution of $(e_i\times v/|v|)\cdot \nabla_x \nabla_x^{\alpha}f(t,x,v)$, where $|\alpha|$ is very large, 
 \[
\p_t (e_i\times v/|v|)\cdot \nabla_x \nabla_x^{\alpha}g(t,x,v) =a_i(x+\hat{v}t, v) e^{-it\d}  (e_i\times v/|v|)\cdot \nabla_x h_1(t,x+\hat{v}t) ,
 \]
\be\label{march7eqn1}
\times (e_i\times v/|v|)\cdot \nabla_x \nabla_x^{\alpha}g(t,x,v) + \textup{other terms},
\ee
where the coefficient $a_i(x+\hat{v}t, v)$ comes from the decomposition (\ref{march6eqn21}) of the rotation in ``$v$'' directional derivative of $f$. Moreover, the following estimate holds, 
\be\label{march7eqn28}
|a_i(x+\hat{v}t, v)| \lesssim {|x+\hat{v}t|}{|v|^{-1}}. 
\ee
We restrict ourself to the case when $|v|\approx \sqrt{1+|t|}$, where $|t|\gg 1. $ Define 
\be\label{2020feb14eqn1}
g^\alpha(t,x,v):= (e_i\times v/|v|)\cdot \nabla_x \nabla_x^{\alpha}g(t,x,v), \quad \tilde{g}^\alpha(t,x,v):= |v| a_i(x+\hat{v}t, v) g^\alpha(t,x,v).
\ee
Hence, we can rewrite the bulk term in (\ref{march7eqn1}) as follows, 
\be\label{march7eqn10}
\textup{bulk term}:= e^{-it\d}  (e_i\times v/|v|^2)\cdot \nabla_x h_1(t,x+\hat{v}t)\tilde{g}^\alpha(t,x,v)
\ee
A key observation for the above bulk term is that there exists an \textit{oscillation phase}, which only depends on   the profiles of the electromagnetic field. More precisely, after rewriting the bulk term in (\ref{march7eqn10}) on the Fourier side, we have
\be\label{march7eqn20}
\mathcal{F}(\textup{bulk term})(\xi):=\int_{\R^3} e^{-it |\eta| + i t \hat{v}\cdot\eta} i \big(\frac{e_i\times v}{|v|^2}\cdot \eta \big)\widehat{h_1}(t, \eta) \widehat{\tilde{g}^\alpha}(t, \xi-\eta, v) d \eta. 
\ee
Note that the oscillation phase ``$|\eta|-\hat{v}\cdot \eta$'' in  (\ref{march7eqn20})
satisfies the following estimate, 
\be\label{march20eqn23}
|\eta|-\hat{v}\cdot \eta\gtrsim \sum_{i=1,2,3} |\eta|\big(\frac{1}{1+|v|^2} + \big(\frac{e_i\times v}{|v|}\cdot \frac{\eta}{|\eta|}\big)^2 \big).
\ee
From the above estimate, we know that the symbol in (\ref{march7eqn20}) is very small near the time resonance set.  In other words, the rotation in $v$ directional derivative  $(e_i\times v/|v|)\cdot \nabla_x  $ plays the role of  {null structure}. Moreover, from (\ref{march20eqn23}), the following estimate holds for any fixed $v\in\R^3$, 
\[
\big(\frac{e_i\times v}{|v|^2}\cdot \eta \big) \lesssim \frac{1}{1+|t|}, \quad \textit{if}\,\,\,\, \eta\in \{\eta:\quad |\eta|-\hat{v}\cdot \eta \leq 1/(1+|t|)\}. 
\]
Therefore,  from the 	estimate (\ref{march7eqn28}),   we know that  the symbol of the bulk term in (\ref{march7eqn10}) exactly covers the loss caused by the coefficient ``$|v| a_i(x+\hat{v}t, v)$'' of $\tilde{g}^\alpha(t,x,v)$ (see (\ref{2020feb14eqn1})) near the time resonance set.

\subsection{The main result of this paper}

 Before stating the main theorem, we define 
the $X_n$-normed space    as follows, 
 \be\label{march6eqn30}
\|h\|_{X_n}:=  \sup_{k\in \mathbb{Z}}2^{(n+1)k }\| \nabla_\xi^n \widehat{h }(t, \xi)\psi_k(\xi)\|_{L^\infty_\xi},\quad  n\in\{0,1,2,3\}.
\ee
Moreover, we define the following classic set of vector fields,
\be
 \mathcal{P}_1:=\{S,  {\Omega}_i,  {L}_i, \p_{x_i}, i\in \{1,2,3\} \},
\ee
where $S:=t\p_t +x\cdot\nabla_x$ denotes the scaling vector field, $\Omega_i:=(e_i\times x)\cdot \nabla_x$ denotes the rotational vector field, and $L_i:=t\p_{x_i}+x_i\p_t$ denotes the Lorentz vector field, $i\in\{1,2,3\}$.

\begin{theorem}\label{precisetheorem}
Let $N_0=200, $ $\delta\in (0, 10^{-9}]$. Suppose that the given initial data $(f_0(x,v),  E_0(x),  B_0(x))$ of the $3D$ relativistic Vlasov-Maxwell system \textup{(\ref{mainequation})} satisfies the following smallness assumption,
\[
\sum_{|\alpha_1|+|\alpha_2|\leq N_0}  \| (1 +|x |^2 + (x\cdot v)^2 +|v|^{20} )^{30N_0} \nabla_v^{\alpha_1} \nabla_x^{\alpha_2 } f_0(x,v)\|_{L^2_x L^2_v} 
+\sum_{ |\alpha|\leq N_0}\sum_{\Gamma\in\mathcal{P}_1}\sum_{n\in\{0,1,2,3\}} \| \Gamma^\alpha E_0(x)\|_{L^2} 
\]
\be\label{march1steqn1}
+\| \Gamma^\alpha B_0(x)\|_{L^2} +\| \Gamma^\alpha E_0(x) \|_{X_n}+   \|  \Gamma^\alpha B_0(x)\|_{X_n}\leq \epsilon_0,
\ee
where    $\epsilon_0$ is some sufficiently small constant. Then the relativistic Vlasov-Maxwell system \textup{(\ref{mainequation})} admits a global solution and scatters to a linear solution in a lower regularity space  and the high order energy of the nonlinear solution grows at most at rate $(1+|t|)^{\delta}$ over time.  Moreover,  we have the following decay estimates for the derivatives of the average of the distribution function and the derivatives of the electromagnetic  field, 
\be\label{desiredecayaverage}
\sup_{t\in[0,\infty)}\sum_{|\alpha|\leq N_0-20}(1+|t|)^{ (3+|\alpha|)/p} \Big|\int_{\R^3} \nabla_x^\alpha \big|f(t,x,v)\big|^p d v  \Big|^{1/p} \lesssim \epsilon_0, \quad \textup{where}\,\, p\in[1, \infty)\cap \mathbb{Z},
\ee
\be\label{decayingeneral}
\sup_{t\in[0,\infty)}\sum_{|\alpha|\leq 10} (1+|t| )(1+||t|-|x||)^{|\alpha|+1}\big( \big|\nabla_x^{\alpha} E(t,x)\big| + \big|  \nabla_x^{\alpha} B(t,x)\big|\big)\lesssim \epsilon_0.
\ee
\end{theorem}
 
The rest of this paper is organized as follows.

\begin{enumerate}
	\item[$\bullet$]
    In section \ref{preliminary},  we introduce the notations used in this paper and then record a linear decay estimate for the half-wave equation and a decay estimate for the average of the distribution function. 
    \item[$\bullet$] In section \ref{vectorfieldssetup}, we introduce the vector fields method based framework developed in \cite{wang}, which includes two sets of vector fields, commutation rules and the process of trading regularities for the decay of with respect to the light cone.
     \item[$\bullet$] In section \ref{setupofenergyestimate}, we set up the energy estimate, classify the nonlinearities of the high order derivatives of the distribution function into the \textit{non-bulk term} and the \textit{bulk term},   define appropriate low order energy and high order energy for both the electromagnetic field and the distribution function, and  use the bootstrap argument to prove    Theorem \ref{precisetheorem}. 
      \item[$\bullet$] In section \ref{energyelectromagnetic}, we estimate the increment of both the low order energy and the high order energy over time for the electromagnetic field.
       \item[$\bullet$] In section \ref{estimateofnonbulkterm}, we estimate the low order energy and the high order energy of the \textit{non-bulk terms} for the distribution function  over time.
       \item[$\bullet$] In section \ref{corelemmaproof1}, we estimate  the high order energy  of the \textit{bulk terms} over time   under the assumption that a key Lemma, Lemma  \ref{trilinearestimate1}, holds. In section \ref{proofofmainlemmalinear}, we finish the proof of Lemma \ref{trilinearestimate1}. 
 \end{enumerate}

\noindent \textbf{Acknowledgment}\qquad This project was initiated when I was a semester postdoc at ICERM, Brown University. I thank   Yan Guo for suggesting this direction in kinetic theory and thank ICERM for the warm hospitality. Also, I thank Lingbing He and Pin Yu for several helpful conversations. The author is supported  by a startup grant at Tsinghua University. 

 \section{Preliminaries}\label{preliminary}
 For any two numbers $A$ and $B$, we use  $A\lesssim B$, $A\approx B$,  and $A\ll B$ to denote  $A\leq C B$, $|A-B|\leq c A$, and $A\leq c B$ respectively, where $C$ is an absolute constant and $c$ is a sufficiently small absolute constant. We use $A\sim B$ to denote the case when $A\lesssim B$ and $B\lesssim A$.   For an integer $k\in\mathbb{Z}$, we use ``$k_{+}$'' to denote $\max\{k,0\}$ and  use ``$k_{-}$'' to denote $\min\{k,0\}$. For any two vectors $\xi, \eta\in \R^3$, we use $\angle(\xi, \eta)$ to denote the angle between ``$\xi$'' and ``$\eta$''. Moreover, we use the convention that $\angle(\xi, \eta)\in [0, \pi]$. 

 For an integrable function $f(x)$, we use both $\widehat{f}(\xi)$ and $\mathcal{F}(f)(\xi)$ to denote the Fourier transform of $f$, which is defined as follows,
\[
\mathcal{F}(f)(\xi)= \int e^{-ix \cdot \xi} f(x) d x.
\]
We use $\mathcal{F}^{-1}(g)$ to denote the inverse Fourier transform of $g(\xi)$. Moreover, for a distribution function $f:\R_x^3 \times \R_v^3\rightarrow \mathbb{C}$, we use the following notation to denote the Fourier transform of $f(x,v)$ in ``$x$'', 
\[ 
\widehat{f}( \xi, v):=\int_{\R^3} e^{-i x\cdot \xi} f( x,v) d x. 
\] 
Basically, ``$v$'' is treated as a fixed parameter during the Fourier transform. 

We  fix an even smooth function $\tilde{\psi}:\R \rightarrow [0,1]$, which is supported in $[-3/2,3/2]$ and equals to one  in $[-5/4, 5/4]$. For any $k\in \mathbb{Z}$, we define
\[
\psi_{k}(x) := \tilde{\psi}(x/2^k) -\tilde{\psi}(x/2^{k-1}), \quad \psi_{\leq k}(x):= \tilde{\psi}(x/2^k)=\sum_{l\leq k}\psi_{l}(x), \quad \psi_{\geq k}(x):= 1-\psi_{\leq k-1}(x).
\]
Moreover, we use  $P_{k}$, $P_{\leq k}$ and $P_{\geq k}$ to denote the projection operators  by the Fourier multipliers $\psi_{k}(\cdot),$ $\psi_{\leq k}(\cdot)$ and $\psi_{\geq k }(\cdot)$ respectively. We   use  $f_{k}(x)$ to abbreviate $P_{k} f(x)$.

For any integrable function $f$, we define
\be\label{signnotation}
f^{+}:=f,\quad P_{+}[f]:=f , \quad f^{-}:= \bar{f}, \quad P_{-}[f]:=\bar{f}.
\ee
Define the cutoff function $\psi_{l;\bar{l}}(\cdot)$ with the threshold $\bar{l}$  as follows,
  \be\label{cutoffwiththreshold}
\psi_{l;\bar{l}}(x):=\left\{\begin{array}{ll}
\psi_{\leq \bar{l}}(x) & \textup{if\,\,} l=\bar{l}\\
\psi_l(x) & \textup{if\,\,} l>\bar{l}.\\
\end{array}
\right.
  \ee
In particular, if the threshold $\bar{l}=0$, we use the following notation, 
\be\label{cutoffwiththreshold100}
\varphi_k(x):=\psi_{k;0}(x), \quad k \in \mathbb{Z}_{+},\quad \varphi_{[k_1,k_2]}(x):=\left\{\begin{array}{ll}
\sum_{k_1\leq k\leq k_2} \psi_{k}(x) & \textup{if\,} k_1 >0\\
\\
 \psi_{\leq k_2}(x) & \textup{if\,} k_1 \leq 0.\\
\end{array}\right.
\ee

We define the unit vectors of the Cartesian coordinate  system in $\R^3$ as follows, 
\be\label{2020feb18eqn1}
e_1:=(1,0,0), \quad e_2:=(0,1,0), \quad e_3:=(0,0,1).
\ee
Moreover, for any $i\in \{1,2,3\}$,  we define the following   vectors, 
\be\label{eqn16}
 X_i =  e_i\times x,\quad   V_i=e_i\times v, \quad \hat{V}_i=e_i \times \hat{v},\quad \tilde{V}_i=e_i\times \tilde{v},   \quad \tilde{v}:=\frac{v}{|v|},\quad \tilde{v}_i:= \tilde{v}\cdot e_i,\quad \hat{v}_i:=\hat{v}\cdot e_i. 
\ee
As a result of direct computations,  we have
\be\label{octeqn456}
u= \tilde{v} \tilde{v}\cdot u     + \sum_{i=1,2,3}\tilde{V}_i \tilde{V}_i \cdot u, \quad u \in \mathbb{R}^3, \quad \tilde{v}\cdot\nabla_v \hat{v}= \frac{\tilde{v}}{(1+|v|^2)^{3/2}}, \quad \tilde{V}_i\cdot\nabla_v \hat{v}= \frac{\tilde{V}_i}{(1+|v|^2)^{1/2}}, \quad i\in\{1,2,3\}.
\ee

For any $k\in \mathbb{Z}$, we define  the $\mathcal{S}^\infty_k$-norm associated with symbols and  a class of symbol as follows, 
\be\label{symbolnorm}
\|m(\xi)\|_{\mathcal{S}^\infty_k}:= \sum_{l=0,1,\cdots, 10}2^{l k}\|\mathcal{F}^{-1}[\nabla_\xi^l m(\xi)\psi_k(\xi)]\|_{L^1}, \quad 
\mathcal{S}^\infty:=\{m(\xi): \quad\| m(\xi)\|_{\mathcal{S}^\infty}:=\sup_{k\in \mathbb{Z}} \| m(\xi)\|_{\mathcal{S}^\infty_k} < \infty \}. 
\ee

 \begin{definition}\label{definitioninversederivative}
We define a linear operator as follows, 
\be\label{sepeqn732}
Q_i:=-R_i \d^{-1}, Q:=(Q_1, Q_2, Q_3), \quad i\in \{1,2,3\},
\ee
where $R_i$, $i\in \{1,2,3\}$, denote the Riesz transforms.  Hence, 
\be\label{sepeqn734}
Id = \nabla\cdot Q. 
\ee
 \end{definition} 
  It is well known that the density of the distribution function decays over time. Now, there are several ways to prove this fact, e.g., performing change of variables, using the vector fields method. We refer readers to a recent result  by Wong \cite{wong} for more detailed discussion. In \cite{wang3}, we used a Fourier based method  to derive two  decay estimates as in the following Lemma, 
 \begin{lemma}\label{decayestimateofdensity}
For any fixed $a(v)\in\{v,\hat{v}\}$,  $x\in \R^3$, $a, t\in \R$, s.t, $|t|\geq 1$, $a> -3$, and any given symbol $m(\xi,v)\in L^\infty_v \mathcal{S}^\infty$, the following decay estimate holds, 
\[
 \big|\int_{\R^3}\int_{\R^3} e^{i x\cdot \xi + i t a(v)  \cdot \xi} m(\xi, v)|\xi|^{a}\widehat{g}(t, \xi, v) d v d\xi   \big|\lesssim \sum_{|\alpha|\leq 5+\lfloor a\rfloor} \big(\sum_{|\beta|\leq 5+\lfloor a\rfloor}   \|\nabla_v^\beta m(\xi,v)\|_{L^\infty_v\mathcal{S}^\infty}\big)
 \]
\be\label{densitydecay}
\times  \big[|t|^{-3-a}    \|(1+|v|)^{5+|a|} \nabla_v^\alpha \widehat{g}(t,0,v) \|_{L^1_v } +|t|^{-4-a} \| (1+|v|)^{5+|a|} (1+|x|)\nabla_v^\alpha g(t,x,v)\|_{L^1_x L^1_v}\big].
\ee
\end{lemma}
\begin{proof}
See \cite{wang3}[Lemma 3.1].
\end{proof}
In the later argument,   we will reduce  the Maxwell equation   to a nonlinear half wave equation, which is convenient to study on the Fourier side. 
For the linear half wave equation, we have  the following $L^\infty_x$-type decay estimate. 
\begin{lemma}[The linear decay estimate]\label{twistedlineardecay}
For any   $\mu \in\{+,-\}$, the following estimate holds, 
\[
\big|\int_{\R^3}  e^{ i x\cdot \xi-i \mu t |\xi|  }     m(\xi) \widehat{f}(\xi) \psi_k(\xi) d \xi \big|   \lesssim   \min\{2^{ k_{-}}, (1+|t|+|x|)^{-1} \} 2^{k}\|m(\xi)\|_{\mathcal{S}^\infty_k}
\]
\be\label{noveqn555}
\times \big(\sum_{|\alpha|\leq 1} 2^{k} \|\widehat{\nabla_x^\alpha f}(t, \xi)\psi_k(\xi)\|_{L^\infty_\xi} + 2^{2k} \|\nabla_\xi \widehat{  f}(t, \xi)\psi_k(\xi)\|_{L^\infty_\xi} \big).
\ee
\end{lemma}
\begin{proof}
See \cite{wang}[Lemma 2.2]. 
\end{proof}
\section{Two sets of vector fields for the relativistic Vlasov-Maxwell system}\label{vectorfieldssetup}
 
In this section, we review the framework we introduced in the first paper \cite{wang} for the study of $3D$ relativistic massive Vlasov-Nrodstr\"om system. This framework is very general and suitable for the study of Vlasov-wave type coupled system.

In \cite{wang}, we used two sets of vector fields for the relativistic Vlasov equation. The first set of vector fields for the distribution function $f(t,x,v)$ is given as follows, 
\be\label{firstsetfordist}
 \mathfrak{P}_1:=\{S,\tilde{\Omega}_i, \tilde{L}_i, \p_{x_i}, i\in \{1,2,3\} \}.
\ee
Correspondingly, we define the following set of vector fields for the electromagnetic field $(E(t,x), B(t,x))$, 
 \be\label{firstsetforeb}
 \mathcal{P}_1:=\{S,  {\Omega}_i,  {L}_i, \p_{x_i}, i\in \{1,2,3\} \},
 \ee
 where
 \be\label{dec18eqn5}
S:=t \p_t + x \cdot\nabla_x,\quad \Omega_i= X_i \cdot \nabla_x,  \quad \tilde{\Omega}_{i}:=  {V}_i \cdot \nabla_v +X_i \cdot \nabla_x,\quad i=1,2,3, 
\ee
\be\label{lorentz}
 \quad L_i:= t \p_{x_i} + x_i \p_t, \quad \tilde{L}_i: = t \p_{x_i} + x_i \p_t + \sqrt{1+|v|^2} \p_{v_i},\,\, L:=(L_1, L_2, L_3), \,\, \tilde{L}:=(\tilde{L}_1, \tilde{L}_2, \tilde{L}_3),
\ee
where ``$S$'', ``$\Omega_i$'', and ``$L_i$'' are the well-known scaling vector field, rotational vector fields, and the Lorentz vector fields, which all commutate with the linear operator of the nonlinear wave equation, see the classic works of Klainerman \cite{Klainerman1,Klainerman2} for the introduction of the original vector field method and the works of Fajman-Joudioux-Smulevici \cite{smulevic1,smulevic2,smulevic3} for more detailed introduction of  vector fields in $\mathfrak{P}_1.$

The second set of vector fields constructed in \cite{wang} aims  to better understand  $\nabla_v f$ in the nonlinearity of the Vlasov equation. To better see the structure of $\nabla_v f$,  we studied  the profile $g(t,x,v)$ of $f(t,x,v)$. More precisely, we define
 \[
g(t, x,v)= f(t, x+\hat{v}t, v), \Longrightarrow  \big(\p_t + \hat{v}\cdot\nabla_x) f(t,x,v)= \big(\p_t g \big)(t,x-\hat{v}t, v ),
 \]
 \[
\nabla_v f(t,x,v)= (D_v g)(t, x-\hat{v}t, v), \quad \textup{where\,}\, D_v:= \nabla_v - t\nabla_v \hat{v}\cdot\nabla_x.
 \]
 Therefore, for  any given vector field that commutes with $\p_t$, we can find a corresponding vector field that commutes with ``$\p_t +\hat{v}\cdot \nabla_x$''. With this intuition, we are looking for a good unknown $\omega(x,v)$	, which doesn't depend on $t$. Instead of decomposing $D_v$ into $\nabla_v$ and $-t \nabla_v \cdot \hat{v}\cdot\nabla_x $, we decompose it as  $\nabla_v- \omega(x,v)\nabla_x$ and $\omega(x,v)\nabla_x -t\nabla_v\hat{v}\cdot\nabla_x$. 

 The choice of good unknown $\omega(x,v)$ depends on an observation of the light cone $C_{t}:=\{(x,v): x,v \in \R^3, |t|-|x+t\hat{v}|=0\}$ in $(x,v)$ space. In \cite{wang}, we defined an inhomogeneous modulation for the light cone $|t|^2-|x+\hat{v}t|^2=0$ in $(x,v)$-space as follows, 
 \begin{definition}
We define the \textit{homogeneous modulation} $d(t,x,v)$ and the \textit{inhomogeneous modulation}  $\tilde{d}(t,x,v)$ as follows, 
\be\label{homogeneousmodulation}
 {d}(t,x,v):=  \frac{t}{1+|v|^2}-\frac{x\cdot v + \sqrt{(x\cdot v)^2+|x|^2} }{\sqrt{1+|v|^2}}, 
\ee
\be\label{inhomogeneousmodulation}
\tilde{d}(t,x,v):=  \frac{t}{1+|v|^2}-\frac{\omega (x,v)}{\sqrt{1+|v|^2}} ,\quad \textup{where}\, \omega(x, v)=  \psi_{\geq 0}(|x|^2 +(x\cdot v)^2)\big(x\cdot v +\sqrt{(x\cdot {v})^2  + |x|^2 }\big). 
\ee
\end{definition}
 The main intuition behind the above definition  is that $|t|^2-|x+\hat{v}t|^2=0$ if and only if $d(t,x,v)=0$. More precisely, the following identity holds, 
 \be\label{march18eqn54}
 |t^2|-|x+\hat{v}t|^2= \frac{t^2}{1+|v|^2}- \frac{2t x\cdot v}{\sqrt{1+|v|^2}} - |x|^2= d(t,x,v)\big(t- \sqrt{1+|v|^2}\big(x\cdot v -\sqrt{(x\cdot v)^2 +|x|^2}\big)\big).
 \ee
Moreover,  from (\ref{homogeneousmodulation})   (\ref{inhomogeneousmodulation}), and (\ref{march18eqn54}), it is easy to check that the following estimate holds, 
 \be\label{feb19eqn1}
|d(t,x,v)| + |\tilde{d}(t,x,v)|\lesssim 1+||t|-|x+\hat{v}t||.
 \ee

  	With the above motivation, we define, 
\be\label{eqq10}
K_v:= \nabla_v - \sqrt{1+|v|^2} \omega(x,v)\nabla_v  \hat{v}\cdot \nabla_x ,\quad S^v: = \tilde{v} \cdot \nabla_v, \quad S^x:= \tilde{v}  \cdot \nabla_x, \quad \Omega^v_i= \tilde{V}_i \cdot \nabla_v, \quad \Omega^x_i= \tilde{V}_i \cdot \nabla_x, 
\ee
where $i \in \{1,2,3\}$, $\tilde{v}$ and $\tilde{V}_i$  are defined in (\ref{eqn16}).  Moreover, we   define a set of vector fields  as follows, 
\be\label{sepeqq2}
   \widehat{S}^v:= \tilde{v}\cdot K_v= S^v - \frac{\omega(x,v)}{ {1+|v|^2}} S^x, \quad \widehat{\Omega}^{v}_i:=  \tilde{V}_i \cdot K_v = \Omega_i^v - \omega(x,v) \Omega_i^x,\quad  K_{v_i}:=K_v \cdot e_i. 
\ee

Note that the vector fields defined in (\ref{sepeqq2}) will be applied on the profile ``$g(t,x,v):=f(t,x+\hat{v}t,v)$'' instead of the original distribution ``$f(t,x,v)$''. Note that
\[
K_{v } g(t,x,v) = \big(\nabla_{v }- \sqrt{1+|v|^2}\omega(x,v) \nabla_{v }\hat{v}\cdot \nabla_x \big) \big(f(t,x+\hat{v}t, v)\big)
\]
\[
= (\nabla_{v }f)(t, x+\hat{v}t, v) +(t-  \sqrt{1+|v|^2}\omega(x,v)) \nabla_{v }\hat{v}\cdot \nabla_x f(t,x+\hat{v}t, v)=: (\widetilde{K}_{v } f)(t, x+\hat{v}t, v),
\]
where
\be\label{pullbackvector}
\widetilde{K}_{v }:= \nabla_{v }+ (t-  \sqrt{1+|v|^2}\omega(x-\hat{v} t ,v)) \nabla_{v }\hat{v}\cdot \nabla_x.
\ee
Since $K_v$ commutates with $\p_t$, we know that $\widetilde{K}_{v}$ commutates with the linear operator $\p_t + \hat{v} \cdot \nabla_x$ of the Vlasov equation.

To sum up,  we define the following  set of vector fields, which will be applied on \textit{the profile} $g(t,x,v)$ instead of the original distribution function $f(t,x,v)$, 
\be\label{dec28eqn1}
\mathfrak{P}_2:=\{\Gamma_i, \quad i\in\{1,\cdots, 17\}\},  
\ee
where 
\be\label{dec26eqn5}
\Gamma_1= \psi_{\geq 1}(|v|) \widehat{S}^v, \quad \Gamma_2:=\psi_{\geq 1}(|v|)S^x, \quad  \Gamma_{i+2}:=\psi_{\geq 1}(|v|)\widehat{\Omega}_i, \quad  \Gamma_{i+5}:=\psi_{\geq 1}(|v|) \Omega^x_i, 
\ee
\be\label{dec26eqn6}
\Gamma_{i+8}:=\psi_{\leq 0}(|v|) K_{v_i},\quad \Gamma_{i+11}:=\psi_{\leq 0}(|v|) \p_{x_i},\quad \Gamma_{i+14}:=\tilde{\Omega}_i, \quad i=1,2,3. 
\ee
Correspondingly, we can find the following associated set of  vector fields which will be applied on the original distribution function $f(t,x,v)$, 
\be\label{secondsetfordist}
\mathcal{P}_2:=\{\widehat{\Gamma}_i, \quad i\in\{1,\cdots, 17\}\}, 
\ee
where
\be\label{dec27eqn21}
\widehat{\Gamma}_1= \psi_{\geq 1}(|v|) \tilde{v}\cdot \widetilde{K}_v, \quad\widehat{\Gamma}_2:=\psi_{\geq 1}(|v|)S^x, \quad  \Gamma_{i+2}:=\psi_{\geq 1}(|v|)\tilde{V}_i \cdot \widetilde{K}_v, \quad  \widehat{\Gamma}_{i+5}:=\psi_{\geq 1}(|v|) \Omega^x_i, 
\ee
\be\label{dec27eqn22}
\widehat{\Gamma}_{i+8}:=\psi_{\leq 0}(|v|) K_{v_i},\quad \widehat{\Gamma}_{i+11}:=\psi_{\leq 0}(|v|) \p_{x_i},\quad \widehat{\Gamma}_{i+14}:=\tilde{\Omega}_i, \quad i=1,2,3. 
\ee

For convenience, we   define  notations to uniformly represents those vector fields.  The notations were introduced in \cite{wang}. For readers' convenience, we redefine them here. 
\begin{definition}
We define a set of vector fields as follows, 
\be\label{lowhighvdecomposition}
  X_1:= \psi_{\geq 1}(|v|)\tilde{v}\cdot D_v, \quad X_{i+1}=  \psi_{\geq 1}(|v|) \tilde{V}_i \cdot D_v,\quad  
X_{i+4}=\psi_{\leq 0}(|v|) D_{v_i}, \quad i=1,2,3,	
\ee
 From (\ref{lowhighvdecomposition}), we have 
 \be\label{noveq1}
 D_v = \tilde{v} X_1 + \tilde{V}_i X_{i+1} + e_i X_{i+4} :=\sum_{i=1,\cdots 7} \alpha_i(v) X_i,
 \ee
 where 
 \be\label{dec28eqn10}
\alpha_1(v):= \psi_{\geq 1}(|v|)\tilde{v}, \quad \alpha_{i+1}(v):= \psi_{\geq 1}(|v|)\tilde{V}_i, \quad \alpha_{i+4}(v):= \psi_{< 1}(|v|){e}_i,\quad i=1,2,3. 
\ee
 \end{definition}

For any vectors $ e=(e_1, \cdots, e_n)\in R^{n},$   $ f=(f_1,\cdots, f_m) \in \R^m$, where $e_1, \cdots, e_n, f_1, \cdots, f_m \in \R$,  we  define
\[
 e\circ f :=(e_1,\cdots, e_n,f_1, \cdots, f_m),\quad |e|:=\sum_{i=1,\cdots, n}|e_i|, \quad \Longrightarrow |e\circ f|=|e|+|f|.
\]
\begin{definition}
Let
\[
\mathcal{A}:=\{\vec{a}: \vec{a}\in\{0,1\}^{10}, |\vec{a}|=0,1\}, \quad \vec{0}:=(0,\cdots, 0),
\]
\[
  \vec{a}_i:=(0,\cdots,\underbrace{1}_{\text{$i$-th} },\cdots,0),\quad  \textit{if\,\,}   \vec{0}, \vec{a}_i\in \mathcal{A}, \quad \mathcal{B}:=\cup_{k\in\mathbb{N}_{+}}\mathcal{A}^k. 
\]
\be\label{dec28eqn28}
\Gamma^{\vec{0}} :=Id, \quad \Gamma^{\vec{a}_1}:=S, \quad \Gamma^{\vec{a}_{i+1}}:=\p_{x_i}, \quad \Gamma^{\vec{a}_{i+4}}:=\Omega_{i}, \quad \Gamma^{\vec{a}_{i+7}}:=L_i, \quad i=1,2,3,
\ee
\be\label{dec28eqn29}
\tilde{\Gamma}^{\vec{0}} :=Id, \quad \tilde{\Gamma}^{\vec{a}_1}:=  {S}, \quad \tilde{\Gamma}^{\vec{a}_{i+1}}:=\p_{x_i}, \quad \tilde{\Gamma}^{\vec{a}_{i+4}}:=\tilde{\Omega}_{i}, \quad \tilde{\Gamma}^{\vec{a}_{i+7}}:=\tilde{L}_i, \quad i=1,2,3.
\ee
Hence, we can represent the high order derivatives of the first set of vector field $\mathfrak{P}_1$ and $\mathcal{P}_1$ (see (\ref{firstsetfordist}) and (\ref{firstsetforeb})) as follows, 
\be\label{dec28eqn30}
\tilde{\Gamma}^{\alpha_1\circ\alpha_2}:=\tilde{\Gamma}^{\alpha_1}\tilde{\Gamma}^{\alpha_2}   ,\quad \Gamma^{\alpha_1\circ\alpha_2}:=\Gamma^{\alpha_1}\Gamma^{\alpha_2}, \quad \alpha_1, \alpha_2\in \mathcal{B}.
\ee
\end{definition}
\begin{definition}
We define
\[
\mathcal{K}:=\{\vec{e}:\,\,\vec{e}\in \{0,1\}^{17}, |\vec{e}|=0,1\},\quad \vec{0}:=(0,\cdots, 0), 
\vec{e}_i:=(0,\cdots,\underbrace{1}_{\text{$i$-th} },\cdots,0),\quad \textit{if\,\,} \vec{0}, \vec{e}_i\in \mathcal{K}, \]
\[
  \mathcal{S}:=\cup_{k\in\mathbb{N}_{+}}\mathcal{K}^{k},\quad 
\Lambda^{\vec{ {0}}}:=Id, \quad \Lambda^{\vec{e}_i}:=\Gamma_i,\quad \Gamma_i\in \mathfrak{P}_2, \quad  \vec{e}_i \in \mathcal{K},
\]
where $\mathfrak{P}_2$ is defined in (\ref{dec28eqn1}).
Hence, we can represent the high order derivatives of the second set of vector fields for the profile ``$g(t,x,v)$'' as follows, 
\[ 
\Lambda^{e\circ f}:= \Lambda^{e}\Lambda^{f},  \quad e, f\in \mathcal{S}.
\]
\end{definition}

\begin{definition}
For any $\kappa, \gamma \in \mathcal{S}$, we define the equivalence relation between ``$ \kappa$'' and ``$ \gamma$'' as follows,
\be\label{equivalencerelation}
 {\kappa}\thicksim  {\gamma}\, \textup{and}\,  \Lambda^{\kappa}\thicksim  \Lambda^{\gamma}, \,\, \textup{if and only if\,} \Lambda^{\kappa} h(x,v)= \Lambda^{\gamma} h(x,v)\,\,\textup{for any differentiable function\,} h(x,v),
\ee
\be\label{nontequivalencerelation}
  {\kappa}\nsim  {\gamma}\, \textup{and}\,  \Lambda^{\kappa}\nsim  \Lambda^{\gamma},\,\,   \textup{if and only if\,\,} \Lambda^{\kappa} h(x,v)\neq \Lambda^{\gamma} h(x,v)\,\,\textup{for all non-constant differentiable function\,} h(x,v).
\ee
Very similarly, we can define the corresponding equivalence relation for $\alpha_1, \alpha_2\in \mathcal{B}$. 
Note that, for any $\beta\in \mathcal{S}$ and $\alpha \in \mathcal{B}$, there exists a unique expansion such that 
\be\label{uniqueexpansion}
 {\beta} \thicksim {\iota_1\circ \cdots\iota_{|\beta|}},\quad \iota_i\in \mathcal{K}, |\iota_i|=1,\quad i\in\{1,\cdots, |\beta|\},
\ee
\be\label{uniqueexpansionBeta}
\alpha \thicksim \gamma_{1}\circ  \cdots \gamma_{|\alpha|}, \quad \gamma_i\in \mathcal{A}, |\gamma_i|=1, \quad i \in \{1,\cdots, |\alpha|\}.
\ee
\end{definition}

\begin{definition}\label{definitiongoodderivative}
For any $\iota\in \mathcal{K}/\{\vec{0}\}$ and $\beta\in \mathcal{S}$, we define two  indexes as follows, 
 \be\label{countingnumber}
c(\iota) =\left\{\begin{array}{ll}
1 & \textup{if\,} \Lambda^\iota \thicksim \widehat{S}^v, \textup{or}\, \Omega_i^x, i\in\{1,2,3\}  \\
\\
0 & \textup{otherwise} \\
\end{array}\right. , \quad i(\iota)=\left\{\begin{array}{ll}
1 & \textup{if}\, \Lambda^\iota \thicksim  {\Omega}_i^x, i\in\{1,2,3\}\\ 
\\ 
0 & \textup{otherwise}
\end{array}\right.,
\ee
\be\label{march26eqn1}
  c(\beta)= \sum_{i=1,\cdots, |\beta|} c(\iota_{i}),  \quad   i(\beta)=\sum_{i=1, \cdots, |\beta|} i(\iota_i), \quad  \textup{where}\, \beta \thicksim \iota_1\circ \cdots \iota_{|\beta|},  \iota_i\in \mathcal{K}/\{\vec{0}\}. 
\ee
\end{definition}
\begin{remark}
The indexes $c(\beta)$ and $i(\beta)$ defined above are same as the index $c_{\textup{vm}}(\beta)$ and $i(\beta)$ defined in \cite{wang}.  The index ``$c(\beta)$''  indicates the total number of ``\textit{good derivatives}'', which are  {$\widehat{S}^v$} and  {$\Omega_i^x$, $i\in\{1,2,3\}$}, inside the total derivative   ``$\Lambda^\beta$''.  We will explain with more details about in what sense derivatives $\widehat{S}^v$ and $\Omega_i^x$  are   ``\textit{good}'' in subsection \ref{reductionoflemmabulk}. The index $i(\beta)$  counts the total number of $\Omega_i^x,i\in\{1,2,3\}, $ derivatives inside $\Lambda^\beta$.
\end{remark}

 With the above defined notation and the vector fields defined in $\mathfrak{P}_1$ (\ref{firstsetfordist}) and $\mathfrak{P}_2$ (\ref{dec28eqn1}), we can understand the bulk derivative ``$D_v$'' as two linear combinations of the above defined vector fields with \textit{good} coefficients as in the following Lemma. As we mentioned in the subsection \ref{losingvissue},   two decompositions in (\ref{summaryoftwodecomposition}) correspond to the two decompositions in (\ref{march6eqn20}) and (\ref{march6eqn21}) 
 \begin{lemma}\label{twodecompositionlemma}
The following two decompositions for ``$D_v$'' holds, 
\be\label{summaryoftwodecomposition}
D_v=\sum_{\rho\in \mathcal{K},|\rho|=1}d_{\rho}(t,x,v) \Lambda^\rho= \sum_{\rho\in \mathcal{K},|\rho|=1}e_{\rho}(t,x,v) \Lambda^\rho,
\ee
where
\be\label{sepeq947}
 {d}_{\rho}(t,x,v)=\left\{\begin{array}{ll}
\tilde{v} \psi_{\geq -1}(|v|)& \textup{if\,\,} \Lambda^\rho \thicksim \psi_{\geq   1}(|v|) \widehat{S}^v \\
\tilde{v} \tilde{d}(t,x,v) (1+|v|^2)^{-1/2}  \psi_{\geq -1}(|v|) & \textup{if\,\,} \Lambda^\rho \thicksim \psi_{\geq 1}(|v|) {S}^x\\
\tilde{V}_i  \psi_{\geq -1}(|v|) & \textup{if\,\,} \Lambda^\rho \thicksim  \psi_{\geq 1}(|v|) \widehat{\Omega}_i^v , i=1,2,3\\
\tilde{V}_i \tilde{d}(t,x,v) (1+|v|^2)^{ 1/2}  \psi_{\geq -1}(|v|) & \textup{if\,\,} \Lambda^\rho \thicksim   \psi_{\geq 1}(|v|) {\Omega}^x_i , i=1,2,3\\
  \psi_{\leq 2 }(|v|) &  \textup{if\,\,} \Lambda^\rho \thicksim   \psi_{\leq 0}(|v|) K_{v_i} , i=1,2,3	\\
- \tilde{d}(t,x,v)(1+|v|^2)\nabla_v\hat{v}_i \psi_{\leq 2}(|v|) &  \textup{if\,\,} \Lambda^\rho \thicksim   \psi_{\leq 0}(|v|) \p_{x_i} , i=1,2,3 \\
0 & \textup{if\,\,}  \Lambda^\rho \thicksim   \tilde{\Omega}_i , i=1,2,3 \\
 \end{array}\right. ,
\ee
\be\label{sepeq932}
 {e}_{\rho}(t,x,v)=\left\{\begin{array}{ll}
\tilde{v} \psi_{\geq -1}(|v|) & \textup{if\,\,} \Lambda^\rho \thicksim \psi_{\geq 1}(|v|) \widehat{S}^v \\
\displaystyle{-\psi_{\geq -1}(|v|)\big(\frac{ \tilde{d}(t,x,v) \tilde{v}}{(1+|v|^2)^{1/2}}+\frac{\tilde{V}_i(X_i\cdot \tilde{v})}{|v|}}\big)  &   \textup{if\,\,} \Lambda^\rho \thicksim \psi_{\geq 1}(|v|) {S}^x\\
0 &  \textup{if\,\,} \Lambda^\rho \thicksim  \psi_{\geq  1}(|v|) \widehat{\Omega}_i^v , i=1,2,3\\
-\psi_{\geq -1 }(|v|) |v|^{-1}{\tilde{V}_j}(X_j+\hat{V}_j t )\cdot \tilde{V}_i  & \textup{if\,\,} \Lambda^\rho \thicksim   \psi_{\geq 1}(|v|) {\Omega}^x_i , i=1,2,3\\
\psi_{\leq  2}(|v|)&   \textup{if\,\,} \Lambda^\rho \thicksim   \psi_{\leq 0}(|v|) K_{v_i} , i=1,2,3	\\
- \psi_{\leq  2}(|v|) \tilde{d}(t,x,v)(1+|v|^2)\nabla_v\hat{v}_i &  \textup{if\,\,} \Lambda^\rho \thicksim   \psi_{\leq 0}(|v|) \p_{x_i} , i=1,2,3 \\
   \psi_{\geq -1}(|v|)|v|^{-1}{\tilde{V}_i} & \textup{if\,\,}  \Lambda^\rho \thicksim   \tilde{\Omega}_i , i=1,2,3 \\
\end{array}\right. .
\ee
From the detailed formula of $d_{\rho}(t,x,v)$ in \textup{(\ref{sepeq947})}, we have
\be\label{jan15eqn2}
\sum_{\rho\in \mathcal{K}, |\rho|=1} \big| (1+|v|)^{-c(\rho)} d_{\rho}(t,x,v)\big|\lesssim 1+|\tilde{d}(t,x,v)|. 
\ee
 \end{lemma}
 \begin{proof}
 See \cite{wang}[Lemma 3.4]. 
 \end{proof}
 As stated in the following Lemma,  a very interesting property of the inhomogeneous modulation ``$\tilde{d}(t,x,v)$'' is that its structure is stable when the vector fields in $\mathfrak{P}_2$ (see (\ref{dec28eqn1})) act  on it.  
 \begin{lemma}\label{derivativesofcoefficient}
For any $\rho \in \mathcal{S}, |\rho|=1$, the following equality holds, 
\be\label{dec26eqn1}
\Lambda^\rho(\tilde{d}(t,x,v)) = e^\rho_1(x,v) \tilde{d}(t,x,v) + e^\rho_2(x,v),  \quad 
D_v(\tilde{d}(t,x,v)) = \hat{e}_1(x,v) \tilde{d}(t,x,v) + \hat{e}_2(x,v),
\ee
where the coefficients  satisfy the following estimate, 
\be\label{jan26eqn101}
|e^\rho_1(x,v)|+ |e^\rho_2(x,v)| + |\hat{e}_1(x,v)|+ |\hat{e}_2(x,v)|   \lesssim 1, \quad |\hat{e}_2|(x,v)\psi_{\geq 2}(|x|)=0.
\ee
Moreover, the following rough estimate holds for any $\beta\in \mathcal{S}$, 
\be\label{jan23eqn11}
\sum_{i=1,2}|\Lambda^\beta e^\rho_i(x,v) | + |\Lambda^\beta \hat{e}_i (x,v) |\lesssim (1+|x|)^{|\beta|} (1+|v|)^{|\beta|}.
\ee
\end{lemma}
\begin{proof}
See \cite{wang}[Lemma A.1]. 
\end{proof}
 
Through using the vector fields in $\mathcal{P}_1$ in (\ref{firstsetforeb}), we can trade one spatial derivative for the decay of the distance with respect to the light cone  in the following sense,
\be\label{feb16eqn1}
(|t|-|x|)\p_i = \sum_{j=1,2,3} \frac{-x_j}{|t|+|x|} \Omega_{ij} + \frac{t}{|t|+|x|} L_i - \frac{x_i}{|t|+|x|}S, \quad i\in\{1,2,3\}, 
\ee  
where $ \Omega_{ij}=x_i\p_{x_j}- x_j\p_{x_i}\in \{\pm \Omega_i, i\in\{1,2,3\}\}$.

We will use this idea to  prove that the electromagnetic field $(E(t,x),B(t,x))$ decays at rate $1/\big((1+|t|)(1+||t|-|x||)$, which is slight stronger than the Klainerman-Sobolev embedding. To this end, instead of dealing with a perfect spatial derivative, we will deal with Fourier multiplier operators.  For any given symbol $m(\xi)\in \mathcal{S}^\infty$ and the associated  Fourier multiplier operator $T$, we   derive an analogue of (\ref{feb16eqn1}) in the following Lemma.

\begin{lemma}\label{tradethreetimes}
For any given Fourier multiplier operator $T$ with the Fourier symbol $m(\xi)$,  
the following equality holds for any $k\in \mathbb{Z},$ 
\be\label{noveqn171}
(|t|  -|x| )^3 T_k[f](t,x)= \sum_{
\begin{subarray}{c}
i=0,1,2,
\alpha\in \mathcal{B}, |\alpha|\leq 3\\
\end{subarray} } \tilde{c}_{\alpha}^i(t,x) \tilde{T}^i_{k,\alpha}(\p_t^i f^\alpha)+(|t| -|x| )e_{\alpha}(t, x) \tilde{T}_{k,\alpha}^3((\p_t^2-\Delta)f),
\ee
where  the coefficients $\tilde{c}_{\alpha}^i(t,x),  i=0,1,2 $, and $e_\alpha(t,x)$,   satisfy the following estimates 
\be\label{april3eqn1}
 |\tilde{c}_{\alpha}^i(t,x)|+ |t\p_t \tilde{c}_{\alpha}^i(t,x)|+  |e_{\alpha}(t, x)|+  |t\p_t e_{\alpha}(t,x)|\lesssim1,\quad 
  |\nabla_x \tilde{c}_{\alpha}^i(t,x)|+ |\nabla_x e_{\alpha}(t, x)|\lesssim (|t|+|x|)^{-1}.
\ee
Moreover, the symbols $\tilde{m}_{k,\alpha}^i(\xi) $   of the Fourier multiplier operators ``$\tilde{T}_{k,\alpha}^i(\cdot)$'', $i\in\{0,1,2,3\}$,   satisfy the following estimates
  \be\label{noveqn141}
\sum_{i=0,1,2 } 2^{ik }\|\tilde{m}_{k,\alpha}^i(\xi)  \|_{\mathcal{S}^\infty} \lesssim  2^{-3k}, \quad \| \tilde{m}_{k,\alpha}^3(\xi) \|_{\mathcal{S}^\infty}\lesssim 2^{-4k}.
\ee 
\end{lemma}
\begin{proof}
See \cite{wang}[Lemma 3.6]. 
\end{proof}

In the  energy estimate of the profile $g(t,x,v)$,   we will use      the  commutation rules between $D_v$ $\big($equivalently speaking, $X_i,  i\in\{1,\cdots, 7\}$, see (\ref{noveq1})$\big)$ and $\Lambda^\beta$, $\beta\in \mathcal{S}$, which are summarized in the following Lemma.   
 \begin{lemma}\label{summaryofhighordercommutation}
The following commutation rules hold for any $i\in \{1,\cdots,7\}$, and $\beta\in \mathcal{S},$  
\be\label{noveq521}
[X_i, \Lambda^\beta]= Y_i^\beta +   \sum_{\kappa\in \mathcal{S}, |\kappa|\leq |\beta|-1 } \big[ \tilde{d}(t,x,v)\tilde{e}_{\beta,i}^{\kappa, 1}(x,v) +\tilde{e}_{\beta,i}^{\kappa, 2}(x,v)\big]  \Lambda^\kappa,
\ee
where $Y_i^\beta$ denote  the top order commutators. More precisely, 
\be\label{sepeqn522}
Y_i^\beta= \sum_{\kappa\in \mathcal{S}, |\kappa|=|\beta|, |i(\kappa)-i(\beta)|\leq 1 } \big[  \tilde{d}(t,x,v)\tilde{e}_{\beta,i}^{\kappa, 1}(x,v) +\tilde{e}_{\beta,i}^{\kappa, 2}(x,v)\big]  \Lambda^\kappa. 
\ee
Moreover, for  any $  i\in \{1, \cdots, 7\}$, and $\kappa\in \mathcal{S}$,   the following   estimates hold for the coefficients   $\tilde{e}_{\beta,i}^{\kappa, 1}(x,v) $ and  $\tilde{e}_{\beta,i}^{\kappa,2}(x,v) $,
\be\label{jan15eqn41}
 |\Lambda^\rho \tilde{e}_{\beta,i}^{\kappa, 1}(x,v) |+|\Lambda^{\rho} \tilde{e}_{\beta,i}^{\kappa, 2}(x,v) |\lesssim (1+|x|)^{|\rho|+ |\beta|-|\kappa| +2} (1+|v|)^{|\rho|+|\beta|-|\kappa| +4},
\ee
\be\label{sepeqn524}
 |\tilde{e}_{\beta,i}^{\kappa, 1}(x,v) |+|\tilde{e}_{\beta,i}^{\kappa, 2}(x,v) |\lesssim (1+|x|)^{|\beta|-|\kappa| +2} (1+|v|)^{|\beta|-|\kappa| +4}, \quad \textup{when\,\,}  |\kappa|\leq |\beta|-1,
\ee
 \be\label{sepeqn904}
|\tilde{e}_{\beta,i}^{\kappa, 1}(x,v)| + |\tilde{e}_{\beta,i}^{\kappa, 2}(x,v) |\lesssim (1+|v|)^{ c(\kappa)-c(\beta)},\quad \textup{when\,\,} |\kappa|=|\beta|. 
\ee 
 Moreover, if $i(\kappa)-i(\beta)>0$ and $|\kappa|=|\beta|$, then   the following improved estimate holds for the coefficients $\tilde{e}_{\beta,i}^{\kappa, 2}(x,v)$ of the commutation rule in \textup{(\ref{noveq521})}, 
\be\label{march25eqn10}
 |\tilde{e}_{\beta,i}^{\kappa, 2}(x,v) |\lesssim (1+|v|)^{-1+c_{}(\kappa)-c_{}(\beta)}.
\ee
\end{lemma}
\begin{proof}
See \cite{wang}[Lemma 3.9]. 
\end{proof}

\section{Set-up of the energy estimate}\label{setupofenergyestimate}

Recall (\ref{mainequation}). We can reduce the equation satisfied by the electromagnetic field into   standard nonlinear wave equations as follows, 
 \be\label{standardwave}
(\textup{RVM})\qquad \left\{\begin{array}{c}
\p_t f +  \hat{v} \cdot \nabla_x f + (E+  \hat{v}\times B)\cdot \nabla_v f =0,\\
\nabla \cdot E = 4 \pi \displaystyle{\int  f(t, x, v) d v},\qquad \nabla \cdot B =0, \\
\p_t^2 E- \Delta E = - 4\pi \displaystyle{\int \p_t f(t, x, v)  \hat{v} d v }- 4\pi \int \nabla_x f(t, x , v) d v ,  \\
\p_t^2 B- \Delta B = \displaystyle{ 4\pi\int \hat v\times \nabla_x f(t, x, v) d v},\\

f(0,x,v)=f_0(x,v), \quad E(0,x)=E_0(x), \quad B(0,x)=B_0(x). 
\end{array}\right.
\ee
Let 
\be\label{eqn387}
 K(t, x,v)= E(t,x )+\hat{v}\times B(t,x ).
\ee
As a result of direct computations, we have 
\be\label{sepeqn61}
  \nabla_v \cdot K(t,x,v)=0.
 \ee
From (\ref{standardwave}) and (\ref{sepeqn61}), we can rewrite one of the nonlinearities  inside (\ref{standardwave}) as follows,
 \[
\int \p_t f(t,x,v)\hat{v} d v 
 = - \int \hat{v}\cdot \nabla_x f(t,x,v) \hat{v} d v + \int    f(t,x,v)  K(t,x,v)\cdot \nabla_v \hat{v} d v.
 \]
 Therefore, we reduce the equation satisfied by the electric field in (\ref{standardwave}) as follows,
 \be\label{eqn10}
 \p_t^2 E- \Delta E =   4\pi   \int \hat{v}\cdot \nabla_x f(t,x,v) \hat{v} d v - 4\pi \int \nabla_x f(t, x , v) d v+   4\pi \int    f(t,x,v)  K(t,x,v)\cdot \nabla_v \hat{v} d v. 
 \ee

\subsection{The equation satisfied by the high order derivatives of the profile $g(t,x,v)$}

In this subsection, our main goal is to compute the equation satisfied by the high order derivatives of the Vlasov-Maxwell system.  For the sake of readers, we show the final equation (\ref{sepeqn43}) step by step.

Note that the following commutation rules hold for any $i,j\in\{1,2,3\}$,
\be\label{commutationrules1}
[\nabla_v, S]=0, \quad [\p_{v_i}, \tilde{\Omega}_j]= \p_{v_i} V_j\cdot \nabla_v,\quad [\p_{v_i}, \tilde{L}_j]= [\p_{v_i}, t \p_{x_j} + x_j \p_t + \sqrt{1+|v|^2} \p_{v_j}]= \frac{v_i}{\sqrt{1+|v|^2}} \p_{v_j}. 
\ee
From the above commutation rules, we know that  the following equality holds for any $\alpha \in \mathcal{B},$
\be\label{dec27eqn4}
\tilde{\Gamma}^\alpha\big((\p_t +\hat{v}\cdot \nabla_v) f \big)= -\sum_{\beta, \gamma\in \mathcal{B}, \beta+\gamma=\alpha}(\Gamma^\beta E + \tilde{\Gamma}^\beta(\hat{v}\times B) )\cdot \tilde{\Gamma}^\gamma(\nabla_v f ).
\ee
For   simplicity in notation, we use the following abbreviation, 
\be\label{dec29eqn1}
f^{\alpha}(t,x,v):=\tilde{\Gamma}^\alpha f(t, x,v), \quad u^\beta(t, x):= \Gamma^\beta u(t,x), \quad u\in \{E, B\}.
\ee
  From the commutation rules in (\ref{commutationrules1}) and (\ref{dec27eqn4}),  the following equation satisfied by $f^\alpha(t,x,v)$ holds, 
\be\label{eqn830}
(\p_t + \hat{v}\cdot \nabla_x) f^\alpha= \sum_{\beta, \gamma\in \mathcal{B}, |\beta| +|\gamma|\leq |\alpha|} \big(\hat{a}_{\alpha;\beta, \gamma}(v) E^\beta+ \hat{b}_{\alpha;\beta, \gamma}(v) B^\beta\big)\cdot \nabla_v f^{\gamma}(t, x,v),
\ee
where $\hat{a}_{\alpha;\beta, \gamma}(v) $ and $\hat{b}_{\alpha;\beta, \gamma}(v)$ are some   coefficients, whose explicit formulas are not pursued here. Moreover, the following equalities and the rough estimate holds for any $\alpha, \beta, \gamma\in \mathcal{B},$
\be\label{dec28eqn41}
\hat{a}_{\alpha;0,\alpha}(v)=-1, \quad \hat{b}_{\alpha;0,\alpha}(v)=-\left[\begin{array}{ccc}
0  & -\hat{v}_3 & \hat{v}_2\\
\hat{v}_3 & 0 & -\hat{v}_1\\
-\hat{v}_2 & \hat{v}_1 & 0\\
\end{array}\right], \quad |\hat{a}_{\alpha;\beta, \gamma}(v) | + |\hat{b}_{\alpha;\beta, \gamma}(v)|\lesssim 1. 
\ee

Define the profile of    $f^{\alpha}(t,x,v)$ as follows, 
\[
g^\alpha(t,x,v):= f^\alpha(t,x+\hat{v}t,v), \quad \Longrightarrow f^\alpha(t,x,v)= g^\alpha(t,x-\hat{v}t,v). 
\]
From (\ref{eqn830}), we have
\be\label{eqn841}
\p_t g^\alpha(t, x,v)= \sum_{\beta, \gamma\in \mathcal{B}, |\beta| +|\gamma|\leq |\alpha|} \big(\hat{a}_{\alpha;\beta, \gamma}(v) E^\beta(t, x+\hat{v} t )+ \hat{b}_{\alpha;\beta, \gamma}(v) B^\beta(t, x+\hat{v}t )\big)\cdot D_v g^{\gamma}(t, x,v). 
\ee
Define 
\be\label{sepeqn334}
K_{\alpha;\beta,\gamma}(t,x+\hat{v} t,v)= \hat{a}_{\alpha;\beta, \gamma}(v) E^\beta(t, x+\hat{v} t )+ \hat{b}_{\alpha;\beta, \gamma}(v) B^\beta(t, x+\hat{v}t ). 
\ee
In particular, from (\ref{dec28eqn41}) and (\ref{eqn387}), we have
\be\label{dec28eqn42}
K_{\alpha;\vec{0},\alpha}(t,x,v)=-E(t,x) - \hat{v}\times B(t,x)= -K(t,x,v).
\ee
Recall the decomposition of $D_v$ in (\ref{noveq1}), we have
\be\label{sourceform}
K_{\alpha;\beta, \gamma}(t,x+\hat{v}t,v)\cdot D_v g^\gamma (t,x,v)= \sum_{i=  1, \cdots 7} K^i_{\alpha;\beta, \gamma}(t,x+\hat{v}t,v) X_i g^\gamma (t,x,v),
\ee
where
\be\label{sepeqn335}
K_{\alpha;\beta, \gamma}^i(t, x+\hat{v}t,v)=\alpha_i(v)\cdot K_{\alpha;\beta, \gamma} (t, x+\hat{v}t,v),\quad i=1,\cdots, 7.
\ee
 Therefore, we can rewrite (\ref{eqn841}) as follows, 
\be\label{sepeqn31}
\p_t g^\alpha(t, x,v)= \sum_{\beta, \gamma\in \mathcal{B}, |\beta| +|\gamma|\leq |\alpha|} \sum_{i=  1, \cdots 7} K_{\alpha;\beta, \gamma}^i(t,x+\hat{v}t,v)\cdot X_i g^{\gamma}(t, x,v).
\ee

Now, we apply the second set of vector fields on $g^\alpha(t,x,v)$.  
For any $\beta \in \mathcal{S} $ and any $\alpha\in \mathcal{B}$, we define 
\be\label{sepeqn40}
g^\alpha_\beta(t,x,v):= \Lambda^{\beta} g^\alpha(t,x,v),\quad \beta\thicksim \iota_1\circ\iota_2\circ\cdots\circ\iota_{|\beta|},\quad \iota_i\in \mathcal{K}, |\iota_i|=1,  \quad i=1,\cdots, |\beta|.
\ee
Note that $[\p_t, \Lambda^\beta]=0$. From (\ref{dec28eqn42}) and (\ref{sepeqn31}), based on the order of derivatives, we classify the nonlinearity of the equation satisfied by $g^\alpha_\beta(t, x,v)$ as follows, 
\be\label{sepeqn43}
\p_t g^\alpha_\beta(t,x,v)=- K(t,x+\hat{v}t,v)\cdot D_v g^\alpha_\beta(t,x,v)+ \textit{h.o.t}_{\beta}^\alpha(t,x,v) + \textit{l.o.t}_{\beta}^\alpha(t,x,v),
\ee
 where ``$\textit{h.o.t}_{\beta}^\alpha(t,x,v)$'' denotes all the terms in which the total number of derivatives acts on $g(t,x,v)$ is  `` $|\alpha|+|\beta|$ ''  and  ``$\textit{l.o.t}_{\beta}^\alpha(t,x,v)$'' denotes all the terms in which  the total number of derivatives  acts on $g(t,x,v)$  is strictly less than `` $|\alpha|+|\beta|$''.  We remind readers  that  the total number of derivatives act on the electromagnetic field is possible to be  $|\alpha| +|\beta|$ in ``$\textit{l.o.t}_{\beta}^\alpha(t,x,v)$''.

Based on the source of the top order terms, we classify  ``$\textit{h.o.t}_{\beta}^\alpha(t,x,v)$'' further into three parts as follow,
\be\label{sepeqn180}
\textit{h.o.t}_{\beta}^\alpha(t,x,v)=\sum_{i=1,2,3} \textit{h.o.t}_{\beta;i}^\alpha(t,x,v),
\ee
where $\textit{h.o.t}_{\beta;1}^\alpha(t,x,v)$ arises from the case when only one derivative of $\Lambda^\beta$ hits $K^i(t,x,v)$, $\textit{h.o.t}_{\beta;2}^\alpha(t,x,v)$ arises from the case when the entire derivative of $\Lambda^\beta$ hits $g^\gamma(t,x,v)$ where $|\gamma|=|\beta|-1$, and $\textit{h.o.t}_{\beta;3}^\alpha(t,x,v)$ arise from the top order commutator between $X_i$ and $\Lambda^\beta$, see (\ref{noveq521}) in Lemma \ref{summaryofhighordercommutation}. More precisely, 
\be\label{sepeqn312}
\textit{h.o.t}_{\beta;1}^\alpha(t,x,v)=\sum_{
\begin{subarray}{c}
\iota, \kappa\in \mathcal{S}, i=1,\cdots,7,\iota+\kappa=\beta, |\iota|=1 
\end{subarray}} \Lambda^\iota(K^i(t,x+\hat{v}t,v))  X_i g^\alpha_\kappa(t,x, v),
\ee
\be\label{sepeqn311}
\textit{h.o.t}_{\beta;2}^\alpha(t,x,v)=  \sum_{|\rho|  \leq 1,|\gamma|=|\alpha|-1}\sum_{i=1,\cdots, 7} K_{\alpha;\rho, \gamma}^i(t,x+\hat{v}t,v)   X_i g^{\gamma}_{\beta}(t, x,v),
 \ee
 \be\label{sepeqn310}
\textit{h.o.t}_{\beta;3}^\alpha(t,x,v)=\sum_{i=1,\cdots, 7}    K^i(t,x+\hat{v}t,v)    Y_i^\beta g^\alpha(t,x,v).
\ee

Next, we will identify the bulk term which appears in the high order terms $\textit{h.o.t}_{\beta}^\alpha(t,x,v)$. Based on the possible vector field of  ``$\Lambda^{\iota}$'' in (\ref{sepeqn312}), we   separate $\textit{h.o.t}_{\beta;1}^\alpha(t,x,v)$ further into two parts  as follows,
\be\label{sepeq90}
\textit{h.o.t}_{\beta;1}^\alpha(t,x,v)= \textit{h.o.t}_{\beta;1}^{\alpha;1}(t,x,v) + \textit{h.o.t}_{\beta;1}^{\alpha;2}(t,x,v),
\ee
where
 \be\label{jan15eqn60}
\textit{h.o.t}_{\beta;1}^{\alpha;1}(t,x,v)=\sum_{ 
j=1,2,3,
i=1,\cdots,7 
 } \sum_{
\iota+\kappa=\beta,\iota, \kappa\in \mathcal{S},
 |\iota|=1, \Lambda^\iota \nsim \psi_{\geq 1}(|v|) \widehat{\Omega}^v_j \textup{or}\,\psi_{\geq 1}(|v|) \Omega_j^x  }  \Lambda^\iota(K^i(t,x+\hat{v}t,v))  X_i g^\alpha_\kappa(t,x, v),
\ee
\be\label{sepeqn952}
\textit{h.o.t}_{\beta;1}^{\alpha;2}(t,x,v)=\sum_{\begin{subarray}{l}
j=1,2,3,
i=1,\cdots,7 
\end{subarray}} \sum_{
\begin{subarray}{c}
\iota+\kappa=\beta,\iota, \kappa\in \mathcal{S},
 |\iota|=1, \Lambda^\iota \sim \psi_{\geq 1}(|v|)\widehat{\Omega}^v_j \textup{or}\,\psi_{\geq 1}(|v|)\Omega_j^x  \\
\end{subarray}}  \Lambda^\iota(K^i(t,x+\hat{v}t,v))  X_i g^\alpha_\kappa(t,x, v).
\ee
  Note that the following equality holds if $\Lambda^{\iota}\sim \psi_{\geq 1}(|v|) \widehat{\Omega}^v_j$ or $ \psi_{\geq 1}(|v|) \Omega_j^x$, where $j\in\{1,2,3\},$
\be\label{nove105}
\Lambda^\iota(K^i(t,x+\hat{v}t,v))= K_{\iota;1}^i(t,x+\hat{v}t,v)+ K_{\iota;2}^i(t,x+\hat{v}t,v),
\ee
where
  \be\label{nove99}
K_{\iota;1}^i(t,x+\hat{v}t,v)= \big(\sqrt{1+|v|^2}\tilde{d}(t,x,v) \big)^{1-c(\iota)}\alpha_i(v)\cdot \Omega_j^x \big(E(t,x+\hat{v}) + \hat{v}\times B(t,x+\hat{v}t) \big)\psi_{\geq 1}(|v|),
 \ee
 \be\label{nove98}
 K_{\iota;2}^i(t,x+\hat{v}t,v)= (1-c(\iota))\psi_{\geq 1}(|v|) \big[\tilde{V}_j\cdot \nabla_v(\alpha_i(v))\big(E(t,x+\hat{v}t) + \hat{v}\times B(t,x+\hat{v}t) \big) + \alpha_i(v)	\cdot \big((\tilde{V}_j\cdot \nabla_v)\hat{v}\times B(t,x+\hat{v}t) \big)  \big]. 
 \ee
 Motivated from the above decomposition, we can separate ``$\textit{h.o.t}_{\beta;1}^{\alpha;2}(t,x,v)$'' further into two parts as follows, 
\be\label{nove100}
\textit{h.o.t}_{\beta;1}^{\alpha;2}(t,x,v)= \textit{bulk}_{\beta}^\alpha(t,x,v) + \textit{error}_{\beta}^\alpha(t,x,v),
\ee
where
\be\label{nove120}
\textit{bulk}_{\beta}^\alpha(t,x,v):=\sum_{\begin{subarray}{l}
j=1,2,3,
i=1,\cdots,7 
\end{subarray}} \sum_{
\begin{subarray}{c}
\iota+\kappa=\beta,\iota, \kappa\in \mathcal{S},
 |\iota|=1, \Lambda^\iota \sim \psi_{\geq 1}(|v|)\widehat{\Omega}^v_j \textup{or}\,\psi_{\geq 1}(|v|)\Omega_j^x 
\end{subarray}}  K_{\iota;1}^i(t,x+\hat{v}t,v)  X_i g^\alpha_\kappa(t,x, v),
\ee
\be\label{nove121}
\textit{error}_{\beta}^\alpha(t,x,v):=\sum_{\begin{subarray}{l}
j=1,2,3,
i=1,\cdots,7 
\end{subarray}} \sum_{
\begin{subarray}{c}
\iota+\kappa=\beta,\iota, \kappa\in \mathcal{S},
 |\iota|=1, \Lambda^\iota \sim  \psi_{\geq 1}(|v|) \widehat{\Omega}^v_j \textup{or}\,\psi_{\geq 1}(|v|) \Omega_j^x 
\end{subarray}}  K_{\iota;2}^i(t,x+\hat{v}t,v)  X_i g^\alpha_\kappa(t,x, v).	
\ee

  Recall (\ref{sepeqn43}). Lastly,  we   classify the 
 low order term  $\textit{l.o.t}_\beta^\alpha(t,x,v)$ and decompose it into four parts as follows,
\be\label{loworderterm}
 \textit{l.o.t}_{\beta}^\alpha(t,x,v)=\sum_{i=1,\cdots,4} \textit{l.o.t}_{\beta;i}^\alpha(t,x,v),
\ee
where
\be\label{sepeqn400}
\textit{l.o.t}_{\beta;1}^\alpha(t,x,v)= \sum_{i=1,\cdots, 7}  \sum_{\kappa\in \mathcal{S}, |\kappa|\leq |\beta|-1}   K^i(t,x+\hat{v}t,v)  \big[ \tilde{d}(t,x,v)\tilde{e}_{\beta,i}^{\kappa, 1}(x,v) +\tilde{e}_{\beta,i}^{\kappa, 2}(x,v)\big]  \Lambda^\kappa g^\alpha(t,x,v),
 \ee
 \[
\textit{l.o.t}_{\beta;2 }^\alpha(t,x,v)= \sum_{
\begin{subarray}{c}
\iota+\kappa=\beta, |\iota|=1\\
i=1,\cdots,7,\iota, \kappa\in \mathcal{S}
\end{subarray}} \big[   \Lambda^\iota(K^i(t,x+\hat{v}t,v)) [ \Lambda^\kappa,X_i] g^\alpha (t,x, v)  + \sum_{  |\gamma|\leq |\alpha|-1} \Lambda^\iota(K^i_{\alpha;\vec{0}, \gamma}(t,x+\hat{v}t,v))    \Lambda^\kappa  X_i g^\gamma(t,x,v)\big]
\]
\be\label{sepeq200}
 +\sum_{|\rho| = 1 }\big[\sum_{|\gamma|=|\alpha|-1} K_{\alpha;\rho, \gamma}^i(t,x+\hat{v}t,v)\big(  [\Lambda^\beta, X_i] g^{\gamma}_{ }(t, x,v)\big) + \sum_{  |\gamma|\leq |\alpha|-2}  K_{\alpha;\rho, \gamma}^i(t,x+\hat{v}t,v) \times \Lambda^\beta\big(X_i g^{\gamma} \big)(t,x,v)\big],
\ee
 
\be\label{sepeqn600}
\textit{l.o.t}_{\beta;3 }^\alpha(t,x,v)=  \sum_{\begin{subarray}{c}
\rho, \gamma\in \mathcal{B}, |\rho|+|\gamma|\leq |\alpha|\\
\iota+\kappa=\beta, \iota, \kappa\in \mathcal{S}\\
i=1,\cdots,7,|\iota|+|\rho|\geq 12\\
\end{subarray}}  \big(\Lambda^{\iota}K^i_{\alpha;\rho,\gamma}(t,x+\hat{v}t,v)\big)  \Lambda^{\kappa}\big(X_i g^{\gamma}(t, x,v)\big),
\ee
\be\label{sepeq201}
\textit{l.o.t}_{\beta;4 }^\alpha(t,x,v)=   \sum_{\begin{subarray}{c}
\rho, \gamma\in \mathcal{B}, |\rho|+|\gamma|\leq |\alpha|\\
\iota+\kappa=\beta,   \iota, \kappa\in \mathcal{S}\\
i=1,\cdots,7,1< |\iota|+|\rho|<  12\\
\end{subarray}} \big(\Lambda^{\iota}K^i_{\alpha;\rho,\gamma}(t,x+\hat{v}t,v)\big)\Lambda^{\kappa}\big(X_i g^{\gamma}(t, x,v)\big),
\ee
where $\textit{l.o.t}_{\beta;1}^\alpha(t,x,v)$ arises from the low order commutator between $X_i$ and $\Lambda^\beta$,  see (\ref{noveq521}) in Lemma \ref{summaryofhighordercommutation}; $\textit{l.o.t}_{\beta;2}^\alpha(t,x,v)$ arises from the commutator between $X_i$ and $\Lambda^{\kappa}$, $\kappa \in \mathcal{S}, |\kappa|=|\beta|-1$ or between $X_i$ and $\Lambda^\beta$ when there is only one derivative hits on $K^i(t,x,v)$, or all other low order terms which have at most one derivative on the electromagnetic field; $\textit{l.o.t}_{\beta;3}^\alpha(t,x,v)$ arises from the case when there are at least twelve  derivatives hit on the electromagnetic field, and $\textit{l.o.t}_{\beta;4}^\alpha(t,x,v)$ denotes all the other low order terms, in which  there are at most twelve derivatives and at least two derivatives hit on the electromagnetic field and the total number of derivatives hit on $g(t,x,v)$ is strictly less than $|\alpha|+|\beta|$.

Recall (\ref{sepeqn334}) and (\ref{sepeqn335}).  To analyze   ``$\Lambda^{\iota}\big(K^i_{\alpha;\rho,\gamma}(t,x+\hat{v} t,v)\big)$'' in  (\ref{sepeqn600}) and (\ref{sepeq201}),   the following Lemma is useful.

 \begin{lemma}\label{decompositionofderivatives}
The following identity holds for any $\rho \in \mathcal{S}, $  
\be\label{sepeqn610}
\Lambda^{\rho}\big( f  (t,x+\hat{v}t  )\big)=\sum_{\iota\in \mathcal{B}, |\iota|\leq |\rho|}  {c}_{ \rho}^{\iota} (x,v)f^{  \iota}(t,x+\hat{v} t) 
	\ee
where  the coefficients $ {c}_{ \rho}^{\iota} (x,v) $, $\iota\in \mathcal{B},$   satisfy the following estimate,
\be\label{sepeqn88}
| {c}_{ \rho}^{\iota} (x,v) |  \lesssim  (1+|x|)^{|\rho|-|\iota|}  (1+|v|)^{  |\rho|-|\iota|} (1+|v|)^{|\rho|-c_{}(\rho) },\quad  \textup{where}\,  |\iota|\leq |\rho|, 
\ee
\be\label{march13eqn10}
|c_{\rho}^{\iota}(x,v) | \lesssim (1+|v|)^{-c(\rho)},\quad \textup{if}\,\, |\rho|=1, \quad \Lambda^\rho \nsim \psi_{\geq 1}(|v|) \widehat{\Omega}_j^v,\textup{or}\, \psi_{\geq 1}(|v|) \Omega_j^x, \quad j\in\{1,2,3\}.
\ee
Moreover, the following rough estimate holds for any $\kappa\in \mathcal{S}, $
\be\label{noveq781}
| \Lambda^{\kappa}\big( {c}_{ \rho}^{\iota} (x,v)\big)| \lesssim (1+|x|)^{|\kappa|+|\rho|-|\iota| } (1+|v|)^{|\kappa|+ |\rho|-|\iota| }  (1+|v|)^{|\rho|-c_{}(\rho) } .
\ee
\end{lemma} 
\begin{proof}
See \cite{wang}[Lemma 4.1]. 
\end{proof}

 \subsection{The equation satisfied by the profiles of the electromagnetic field}
In this subsection, we mainly compute the equation satisfied by the high order derivatives of the electromagnetic fields. Recall (\ref{firstsetfordist}) and (\ref{firstsetforeb}). Note that the vector fields we will apply on the distribution function $f(t,x,v)$ and the electromagnetic field  are not exactly same. As a preliminary step before computing the equation satisfied by the high order derivatives of the electromagnetic fields, we compute the difference between the high order derivatives $\Gamma^\alpha$ and $\tilde{\Gamma}^\alpha, \alpha \in \mathcal{B}$, see (\ref{dec28eqn30}).

Recall (\ref{dec28eqn28}), (\ref{dec28eqn29}), and (\ref{dec28eqn30}). We have  
\be\label{dec28eqn71}
{\Gamma}^\alpha- \tilde{\Gamma}^\alpha = \sum_{\beta, \gamma\in \mathcal{B}, |\beta|+|\gamma|\leq |\alpha|, |\beta|\geq 1} a_{\alpha;\beta, \gamma}(v)\cdot\nabla_v^\beta \tilde{\Gamma}^\gamma,
\ee
where ``$a_{\alpha;\beta, \gamma}(v)$'', $\beta, \gamma\in\mathcal{B}, $ are some determined coefficients, whose explicit formulas are not pursued here and the vector ``$\nabla_v^\beta$'', $ \beta\in \mathcal{B}$, is defined as follows, 
\[
\nabla_v^\beta:= \nabla_v^{\gamma_1}\circ   \cdots \nabla_{v}^{\gamma_{|\beta|}}, \quad \beta \thicksim \gamma_1\circ\cdots \gamma_{|\beta|}, \quad \gamma_i\in \mathcal{A}, |\gamma_i|=1, i\in\{1,\cdots, |\beta|\},
\]
\be\label{feb16eqn10}
\nabla_v^\gamma =\left\{\begin{array}{ll}
 {V}_i\cdot \nabla_v & \textup{if}\,\, \gamma=\vec{a}_{i+4}, i=1,2,3 \\
 
\sqrt{1+|v|^2}\p_{v_i} & \textup{if}\,\, \gamma=\vec{a}_{i+7}, i=1,2,3 \\
 
Id & \textup{otherwise}  \\
\end{array}
\right. , \quad \textup{where}\,\, \gamma\in \mathcal{A}.
\ee
 Due to the fact that $\nabla_v$ is possible to hit the coefficients during the expansion, we have $|\beta|+ |\gamma|\leq |\alpha|$ instead of $|\beta|+ |\gamma|= |\alpha|$ in (\ref{dec28eqn71}). From (\ref{dec18eqn5}) and (\ref{lorentz}), we know that  the following rough estimate of  the coefficients  $a_{\alpha;\beta, \gamma}(v)$ holds, 	
\be\label{octeqn1701}
|a_{\alpha;\beta, \gamma}(v)|\lesssim 1, \quad \beta, \gamma\in \mathcal{B}, |\beta|+|\gamma|\leq |\alpha|, |\beta|\geq 1.
\ee
Let 
\be\label{dec28eqn81}
a_{\alpha:\vec{0},\alpha}(v):=1, \Longrightarrow 
\Gamma^\alpha = \sum_{\beta,\gamma\in \mathcal{B},|\beta|+|\gamma|\leq |\alpha| } a_{\alpha;\beta, \gamma}(v)\cdot\nabla_v^\beta \tilde{\Gamma}^\gamma.
\ee
  
Recall (\ref{standardwave}) and (\ref{eqn10}). From (\ref{dec28eqn81}),  we have
\[
 \Gamma^\alpha \big( \p_t^2 E- \Delta   E \big)=  \sum_{|\beta|+|\gamma|\leq |\alpha| }    \int  4\pi\hat{v} \cdot  \Big(a_{\alpha;\beta, \gamma}(v)\cdot\nabla_v^\beta  \tilde{\Gamma}^\gamma\nabla_x f\Big) \hat{v}  -   4\pi  a_{\alpha;\beta, \gamma}(v)\cdot\nabla_v^\beta \tilde{\Gamma}^\gamma\nabla_x f(t, x , v) 
 \]
\be\label{dec29eqn2}
 + \sum_{\beta+\gamma=\alpha}\sum_{|\iota|+ |\kappa|\leq |\gamma| }  4\pi  \nabla_v  \hat{v}  \cdot(E^\beta +\hat{v}\times B^\beta)a_{\gamma;\iota, \kappa}(v)\cdot\nabla_v^\iota \tilde{\Gamma}^\kappa f d v, 
\ee
\be\label{dec29eqn3}
\Gamma^\alpha \big(\p_t^2 B- \Delta     B) =  \sum_{|\beta|+ |\gamma|\leq |\alpha|}  \displaystyle{ \int 4\pi \hat{v}\times \big( a_{\alpha;\beta, \gamma}(v)\cdot \nabla_v^\beta  \tilde{\Gamma}^\gamma \nabla_x f(t, x, v) \big) d v}.
\ee
Note that
\[
[\p_t^2 - \Delta, S]= -S, \quad [\p_t^2 - \Delta, \Omega_{i} ]=0, \quad [ \p_t^2 - \Delta, L_i]= 0.
\]
From the above commutation rules,  after doing   the integration by parts in ``$v$'' for the integral on the right hand side of (\ref{dec29eqn2}) and (\ref{dec29eqn3}) to move around the  derivatives ``$\nabla_v^\beta$'', we know that the following  equations satisfied by $   E^\alpha(t,x)$ and $ B^\alpha(t,x)$ hold, 
 \be\label{maxwellvectorfield}
(\p_t^2 -\Delta) E^\alpha  =\mathcal{N}^{\alpha}_1, \quad (\p_t^2 -\Delta) B^\alpha  =\mathcal{N}^{\alpha}_2, 
\ee
where
 \be\label{feb16eqn16}
\mathcal{N}^{\alpha}_1=  \sum_{\begin{subarray}{c}
\beta, \gamma\in \mathcal{B},
|\beta| + |\gamma|\leq |\alpha| 
\end{subarray}} \int_{\R^3} \tilde{a}_{\alpha; \gamma}(v)  \nabla_x   f^\gamma (t,x,v)  + \big(   \tilde{b}^{\alpha}_{\beta, \gamma}(v)  E^\beta(t,x) +\tilde{c}^{\alpha}_{\beta, \gamma}(v)  B^\beta(t,x)\big)  f^\gamma(t,x,v)d v, 
  \ee
 \be\label{feb16eqn17}
   \mathcal{N}^\alpha_2=\sum_{\gamma\in \mathcal{B}, |\gamma| \leq |\alpha|}    \int_{\R^3}  \tilde{d}_{\alpha; \gamma}(v)  \nabla_x    f^\gamma (t, x, v)  d v,
\ee
where $\tilde{a}_{\alpha; \gamma}(v)$, $ \tilde{b}^{\alpha}_{\beta, \gamma}(v)$, $ \tilde{c}^{\alpha}_{\beta, \gamma}(v)$, and $ \tilde{d}_{\alpha;\gamma}(v)$ are some determined coefficients, 	  whose explicit formulas are not pursued here. Moreover, from the equality (\ref{feb16eqn10}) and the  rough estimate of coefficients in (\ref{octeqn1701}), we have the following rough estimate,
\be\label{feb16eqn15}
|\tilde{a}_{\alpha; \gamma}(v)| + |\tilde{b}^{\alpha}_{\beta, \gamma}(v) | + |\tilde{c}^{\alpha}_{\beta, \gamma}(v)| + |\tilde{d}_{\alpha; \gamma}(v)| \lesssim 1. 
\ee

Define the half wave part of the electromagnetic field as follows, 
\be\label{octeqn1054}
u_1^\alpha (t,x):= \d^{-1}(\p_t - i \d) E^\alpha (t,x),\qquad u_2^\alpha (t,x):=\d^{-1} (\p_t - i \d) B^\alpha (t,x).
\ee 
As a result of direct computations, we can recover the electromagnetic field from the above defined half wave as follows, 
\be\label{jan5eqn3}
\p_t E^\alpha= \frac{\d}{2}\big(u_1^\alpha+\overline{u_1^\alpha}\big), \quad \p_t B^\alpha= \frac{\d}{2}\big(u_2^\alpha+\overline{u_2^\alpha}\big),\quad E^\alpha=\frac{-u_1^\alpha+ \overline{u_1^\alpha}}{2i}, \quad B^\alpha=\frac{-u_2^\alpha+ \overline{u_2^\alpha}}{2i}.
\ee
From (\ref{maxwellvectorfield}) and (\ref{octeqn1054}), we have 
\be\label{octeqn1089}
\left\{\begin{array}{l}
(\p_t + i \d) u_1^\alpha (t, x)= \d^{-1} \mathcal{N}^{\alpha}_1 =: \widetilde{ \mathcal{N}_1^\alpha}  \\
(\p_t + i \d) u_2^\alpha (t, x) =   \d^{-1} \mathcal{N}^{\alpha}_2 =: \widetilde{ \mathcal{N}_2^\alpha}. \\
\end{array}\right.
\ee
Correspondingly, we define the profiles of $u^\alpha_i(t)$, $i\in\{1,2\},$ as follows,
\be\label{octeqn1055}
h_1^\alpha(t): =e^{ it \d} u_1^\alpha(t), \quad h_2^\alpha(t): =e^{ it \d} u_2^\alpha (t).
\ee
On the Fourier side, from (\ref{octeqn1089}),  the following equations satisfied by the profiles hold,
\be\label{octeqn1725}
\p_t \widehat{h_1^\alpha}(t,\xi)= \widehat{\mathcal{N}_{1;1}^\alpha}(t,\xi)+ \widehat{\mathcal{N}_{1;2}^\alpha}(t,\xi),  
\ee
\be\label{eqqq3}
\p_t \widehat{h_2^\alpha}(t,\xi)=\sum_{\gamma\in \mathcal{B}, |\gamma| \leq |\alpha|}    \int_{\R^3} e^{it|\xi|-it \hat{v}\cdot \xi }    \tilde{d}_{\alpha;\gamma}(v)  \frac{i \xi}{|\xi|}  \widehat{g^\gamma}(t,\xi, v) d v,
\ee
where
\be\label{octeqn1728}
 \widehat{\mathcal{N}_{1;1}^\alpha}(t, \xi) = \sum_{\gamma\in \mathcal{B},  |\gamma| \leq |\alpha|} \int_{\R^3} e^{it|\xi|-it \hat{v}\cdot \xi }   \tilde{a}_{\alpha; \gamma}(v)  \frac{ i \xi}{|\xi|}  \widehat{g^\gamma}(t,\xi, v) d v, 
 \ee
\be\label{octeqn798}
    \widehat{\mathcal{N}_{1;2}^\alpha}(t,\xi) = \sum_{\begin{subarray}{c}
\beta, \gamma\in \mathcal{B},\mu\in\{+,-\}\\
|\beta| + |\gamma|\leq |\alpha| 
\end{subarray}} \int_{\R^3} \int_{\R^3} e^{i t|\xi|-it\mu |\xi-\eta|- i t\hat{v}\cdot \eta} \frac{ c_{\mu }}{|\xi|} \big(\tilde{b}^{\alpha}_{\beta, \gamma}(v) \widehat{P_{\mu}[h^\beta_1]}(t, \xi-\eta)  +\tilde{c}^{\alpha}_{\beta, \gamma}(v) \widehat{P_{\mu}[h^\beta_2]}(t, \xi-\eta)\big) \widehat{g^\gamma}(t, \eta, v) d\eta d  v,
\ee
where the operator $P_{\mu}[\cdot], \mu\in\{+,-\}$, is defined in (\ref{signnotation}) and $c_{\mu}= i\mu/2$. 
\subsection{The modified profiles of the electromagnetic field}
Note that both the equations (\ref{octeqn1725}) and (\ref{eqqq3}) contain  a linear term, which is the density (i.e., the average of the  distribution  function). To adjust the growth effect comes from this part, we add the correction terms to the profile and define the modified profiles as follows,  
\be\label{octe631}
\widehat{\widetilde{h_1^\alpha}}(t, \xi)= {\widehat{h_1^\alpha}}(t, \xi)- \sum_{\gamma\in \mathcal{B}, |\gamma| \leq |\alpha|} \int_{\R^3} e^{it|\xi|-it \hat{v}\cdot \xi } \frac{   \tilde{a}_{\alpha; \gamma}(v) \xi}{|\xi|(|\xi|-\hat{v}\cdot \xi )}  \widehat{g^\gamma}(t,\xi, v) d v,
\ee
\be\label{octe633}
\widehat{\widetilde{h_2^\alpha}}(t, \xi)= {\widehat{h_2^\alpha}}(t, \xi) - \sum_{\gamma\in \mathcal{B}, |\gamma| \leq |\alpha|} \int_{\R^3} e^{it|\xi|-it \hat{v}\cdot \xi } \frac{   \tilde{d}_{\alpha; \gamma}(v) \xi}{|\xi|(|\xi|-\hat{v}\cdot \xi )}  \widehat{g^\gamma}(t,\xi, v) d v.
\ee
Correspondingly, for any $\alpha\in \mathcal{B}$, we can define the modified electromagnetic field as follows, 
\be\label{noveq800}
\widetilde{E^\alpha}(t)=    c_{+} e^{-it \d} \widetilde{h_1^\alpha}(t) + c_{-}\overline{e^{-it \d} \widetilde{h_1^\alpha}(t)} , \quad \widetilde{B^{\alpha}}(t)=  c_{+} e^{-it \d} \widetilde{h_2^\alpha}(t) + c_{-}\overline{e^{-it \d} \widetilde{h_2^\alpha}(t)}. 
\ee
Define 
\be\label{noveq510}
E^E_{\alpha;\gamma}(f)(t,x):=  \mathcal{F}^{-1}\big[\int_{\R^3} e^{-it \hat{v}\cdot \xi } \frac{     \tilde{a}_{\alpha; \gamma}(v) \xi}{|\xi|(|\xi|-\hat{v}\cdot \xi )}  \widehat{f}(t,\xi, v) d v\big](x), 
\ee
\be\label{noveq511}
E_{\alpha;\gamma}^B(f)(t,x):= \mathcal{F}^{-1}\big[\int_{\R^3} e^{-it \hat{v}\cdot \xi } \frac{  \tilde{d}_{\alpha; \gamma}(v) \xi}{|\xi|(|\xi|-\hat{v}\cdot \xi )}  \widehat{f}(t,\xi, v) d v\big](x).
\ee
Recall (\ref{jan5eqn3}), we have
\be\label{noveq241}
E^{\alpha}(t)=  c_{+}e^{-it \d} h_1^\alpha(t) +c_{-} \overline{e^{-it \d} h_1^\alpha(t)} , \quad B^{\alpha}(t)= c_{+}e^{-it \d} h_2^\alpha(t) +c_{-} \overline{e^{-it \d} h_2^\alpha(t)}.
\ee
From the  equalities (\ref{jan5eqn3}), (\ref{octeqn1055}), (\ref{octe631}), and (\ref{octe633}),  for any  $\alpha \in \mathcal{B}$, we  can link the relation between the electromagnetic field and the modified electromagnetic field via  the  equalities as follows, 
\be\label{noveq432}
E^{\alpha}(t)= \widetilde{E^\alpha}(t) -\sum_{ \gamma\in\mathcal{B}, |\gamma|\leq |\alpha|} \textup{Im}\big[ E_{\alpha;\gamma}^{E}(g^\gamma)(t)\big], \quad B^{\alpha}(t)= \widetilde{B^\alpha}(t) - \sum_{\gamma\in\mathcal{B},|\gamma|\leq |\alpha|} \textup{Im}\big[E_{\alpha;\gamma}^{B}(g^\gamma)(t)\big]. 
\ee

Recall the equations satisfied by the profiles $h^\alpha_1(t)$ and  $h^\alpha_2(t)$ in  (\ref{octeqn1725}) and (\ref{eqqq3}). From (\ref{octe631}) and (\ref{octe633}), it is easy to derive the equations satisfied by the modified profiles $\widetilde{h^\alpha_1}(t)$ and  $\widetilde{h^\alpha_2}(t)$ on the Fourier side as follows,   
\be\label{octe662}
\p_t \widehat{\widetilde{h_1^\alpha}}(t, \xi)= \sum_{ |\gamma| \leq |\alpha|}  \mathfrak{N}_{\alpha;  \gamma}^1(t,\xi)  + \widehat{\mathcal{N}_{1;2}^\alpha}(t,\xi), \quad 
\p_t \widehat{\widetilde{h_2^\alpha}}(t, \xi)= \sum_{ |\gamma| \leq |\alpha|}  \mathfrak{N}_{\alpha;\gamma}^2(t,\xi),
\ee
where
\be\label{dec29eqn41}
   \mathfrak{N}_{\alpha;  \gamma}^1(t,\xi):=- \int_{\R^3} e^{it|\xi|-it \hat{v}\cdot \xi } \frac{\tilde{a}_{\alpha; \gamma}(v)  \xi}{|\xi|(|\xi|-\hat{v}\cdot \xi)}   \widehat{\p_t g^\gamma}(t,\xi, v) d v, \quad \mathfrak{N}_{\alpha;  \gamma}^2(t,\xi):=- \int_{\R^3} e^{it|\xi|-it \hat{v}\cdot \xi } \frac{\tilde{d}_{\alpha;  \gamma}(v)  \xi}{|\xi|(|\xi|-\hat{v}\cdot \xi)}   \widehat{\p_t g^\gamma}(t,\xi, v) d v. 
\ee
 Recall  (\ref{eqn841}). After plugging in the equation satisfied by $\p_t g^\gamma$,   we have  the following equality for any fixed $\alpha, \beta, \gamma\in \mathcal{B}$, 
\[
\mathfrak{N}_{\alpha;\gamma}^1(t,\xi)=  \sum_{ \iota, \kappa\in\mathcal{B},
 |\iota| +|\kappa|\leq |\gamma|\\
  }\sum_{ \mu\in\{+,-\} } \int_{\R^3}\int_{\R^3} e^{it|\xi|- i \mu t |\xi-\eta| -it \hat{v}\cdot \eta } \frac{- c_{\mu} \tilde{a}_{\alpha;  \gamma}(v)  \xi}{|\xi|(|\xi|-\hat{v}\cdot \xi)} \]
\[
\times \big(\hat{a}_{\gamma;\iota, \kappa}(v) {\widehat{ P_{\mu}[h_1^\iota]}}(t,\xi-\eta)+ \hat{b}_{\gamma;\iota, \kappa}(v) \widehat{ P_{\mu}[h_2^\iota]}(t, \xi-\eta)\big)\cdot \big( \nabla_v  - i t \nabla_v(\hat{v}\cdot \eta)  )\widehat{g^\kappa}(t, \eta, v)  d \eta d v
\]
\[
= \sum_{ \iota, \kappa\in\mathcal{B},
 |\iota| +|\kappa|\leq |\gamma|\\
  } \sum_{\mu\in\{+,-\}}  \int_{\R^3}\int_{\R^3} e^{it|\xi|- i \mu t |\xi-\eta| -it \hat{v}\cdot \eta } \nabla_v\cdot\big[\frac{c_{\mu}\tilde{a}_{\alpha;  \gamma}(v) 	 \xi}{|\xi|(|\xi|-\hat{v}\cdot \xi)} \big(\hat{a}_{\gamma;\iota, \kappa}(v)\widehat{ P_{\mu}[h_1^\iota]}(t,\xi-\eta)
\]
\be\label{octe661}
 + \hat{b}_{\gamma;\iota, \kappa}(v) \widehat{ P_{\mu}[h_2^\iota]}(t, \xi-\eta)\big) \big]  \widehat{g^\kappa}(t, \eta, v)d \eta d v.
\ee
In the above equality, we did the integration by parts in ``$v$'' to move around the derivative ``$\nabla_v$'' in front of $\widehat{g^{\kappa}}(t, \eta, v)$. With minor modifications, we can reduce  ``$ \mathfrak{N}_{\alpha;  \gamma}^2(t,\xi)$ '' in (\ref{dec29eqn41}) as follows, 
\[
\mathfrak{N}_{\alpha;  \gamma}^2(t,\xi)= \sum_{ \iota, \kappa\in\mathcal{B},
 |\iota| +|\kappa|\leq |\gamma|\\
  }\sum_{\mu\in\{+,-\}}  \int_{\R^3}\int_{\R^3} e^{it|\xi|- i \mu t |\xi-\eta| -it \hat{v}\cdot \eta } \nabla_v\cdot\big[\frac{c_{\mu}\tilde{d}_{\alpha; \gamma}(v)  \xi}{|\xi|(|\xi|-\hat{v}\cdot \xi)}  \big(\hat{a}_{\gamma;\iota, \kappa}(v)\widehat{ P_{\mu}[h_1^\iota]}(t,\xi-\eta)
\]
\be\label{octe51}
+ \hat{b}_{\gamma;\iota, \kappa}(v) \widehat{ P_{\mu}[h_2^\iota]}(t, \xi-\eta)\big) \big]  \widehat{g^\kappa}(t, \eta, v)d \eta d v.
\ee

From the above equalities and the detailed formula of  $ \widehat{\mathcal{N}_{1;2}^\alpha}(t,\xi)$ in (\ref{octeqn798}), we know that the nonlinearities of the  equation satisfied by the modified profiles in (\ref{octe662}) are quadratic, which are more favorable in the energy estimate.  
 
 \subsection{The energy functionals   for the Vlasov-Maxwell system}\label{constructionofinitialenergy}
In this subsection, we  construct the energy functionals  for the Vlasov-Maxwell system.  The energy functionals we will use for the Vlasov-Maxwell system are   similar to  the energy functionals used in the Vlasov-Nordstr\"om system, see \cite{wang}. For the sake of readers, we still elaborate the ideas behind the construction of energy functionals.

We define the high order energy for the profile $g(t,x,v)$ of the distribution function as follows,
\be\label{highorderenergy1}
E_{\textup{high}}^{f}(t):=  E_{\textup{high}}^{f;1}(t)+ E_{\textup{high}}^{f;2}(t), \quad E_{\textup{high}}^{f:1}(t):= \sum_{\alpha\in \mathcal{B}, \beta\in \mathcal{S}, |\alpha|+|\beta|= N_0}   \| \omega_{\beta}^{\alpha}( x, v)  g^\alpha_\beta (t,x,v)\|_{L^2_{x,v}} ,
\ee
\be\label{highordersecondpro}
  E_{\textup{high}}^{f:2}(t):= \sum_{\alpha\in \mathcal{B}, \beta\in \mathcal{S}, |\alpha|+|\beta| < N_0}   \| \omega_{\beta}^{\alpha}( x, v)  g^\alpha_\beta (t,x,v)\|_{L^2_{x,v}} ,
\ee
where $g^\alpha_\beta(t,x,v)$ is defined in (\ref{sepeqn40}) and the weight function $\omega_\beta^\alpha(t,x,v)$ is defined as follows, 
 \be\label{highorderweight}
\omega^\alpha_\beta(t,x, v) = (1 +|x |^2 + (x\cdot v)^2 +|v|^{20} )^{20N_0-10(|\alpha|+|\beta|)}   (1+|v|)^{c(\beta)}(\phi(t,x,v))^{|\beta|-i(\beta)},
\ee
 where the indexes  $c(\beta)$ and $i(\beta)$ are  defined in (\ref{countingnumber}) and the time dependent  weight function $\phi(t,x,v)$ is defined as follows, 
 \be\label{april2eqn1}
  \phi(t,x,v):=1-\frac{ x\cdot v 	}{1+|x|} f((1+|t|)/(|x||v|))\eta(x\cdot\tilde{v})\psi_{\geq 1}(|v|), 
 \ee
where $\eta(x):\R\rightarrow \R$ is supported inside $(-\infty,-10]$ and  equals to one inside $(-\infty,-20]$ and the function $f(x):\R_{+}\rightarrow \R$ is a bump function defined as follows, 
\be\label{april2eqn151}
f(x):=\left\{\begin{array}{cc}
e^{-1/(1-2^5 x)} &  0\leq x < 2^{-5}\\
0 & x \geq 2^{-5}\\ 
\end{array}\right., \Longrightarrow \quad  f'(x)\leq 0. 
\ee

 The main ideas of making the choice of weight function as in (\ref{highorderweight}) can be summarized as follows: (i) For   different order of derivatives of the profile,  we set up a hierarchy for the order of the associated weight function. For the profile with  more  derivatives, we propagate less weight inside the energy. The choice of such hierarchy   makes the estimate of lower order terms easier in the estimate of the high order terms. (ii) Comparing with the ordinary derivatives of the profile, we expect that the \textit{  good derivatives} of the profile, which are $\widehat{S}^v$ and $\Omega_i^x$, are capable of propagating more weight in ``$|v|$''; (iii) We used an anisotropic weight in $x$ in (\ref{highorderweight}) instead of a radial weight $|x|$ to guarantee  that the first  estimate  (\ref{feb8eqn51}) in Lemma \ref{derivativeofweightfunction} holds, which plays an important role  for the case when all the derivatives hit on $D_v g(t,x,v)$,  see Proposition \ref{notbulkterm1}  for more details.  (iv) We used  $\phi(t,x,v)$ in the weight function $\omega^\alpha_\beta(t,x, v)$ to capture the fact that  the inhomogeneous modulation $\tilde{d}(t,x,v)$ is much smaller than the distance to the light cone $||t|-|x+\hat{v}t||$ 
 if $x\cdot \tilde{v} < 0, |v|\gtrsim 1 $, and $ |x|\geq (1+|t|)/|v|$, see the second part of the estimate (\ref{feb8eqn51}) in Lemma \ref{derivativeofweightfunction}. This observation is also crucial for the estimate of the worst scenario after exploiting the null structure by doing integration by parts in time, see the proof of   Lemma \textup{\ref{bulktermbilinearestimate1}} in subsection  \ref{proofofbulklemma}.

  \begin{lemma}\label{derivativeofweightfunction}
For any $\alpha\in \mathcal{B}, \beta\in \mathcal{S},$ s.t., $|\alpha|+|\beta|\leq N_0$, the following estimate holds for any $x, v\in \mathbb{R}^3$, 
\be\label{feb8eqn51}
  \Big| \frac{{D}_v \omega_{\beta}^{\alpha}(t, x, v)}{\omega_{\beta}^{\alpha}( t, x, v)} \Big| \frac{1}{1+||t|-|x+\hat{v} t||} \lesssim 1, \quad \Big|\frac{\tilde{d}(t,x,v)\phi(x,v)}{1+||t|-|x+\hat{v} t||}\Big|\lesssim 1. 
\ee
\end{lemma}
\begin{proof}
Recall (\ref{highorderweight}). Let 
\[
\hat{\omega}_{\beta}^\alpha(x,v):= (1 +|x |^2 + (x\cdot v)^2 +|v|^{20} )^{20N_0-10(|\alpha|+|\beta|)}(1+|v|)^{c(\beta)}.
\]
From \cite{wang}[Lemma 4.2], we know that the following estimate holds, 
\be\label{march25eqn42}
\Big|\frac{D_v\hat{\omega}_\beta^\alpha(x,v)}{\hat{\omega}_\beta^\alpha(x,v)}\Big|\frac{1}{1+||t|-|x+\hat{v}t| |}\lesssim 1. 
\ee
Therefore, to prove our desired first estimate in 
(\ref{feb8eqn51}), it would be sufficient to consider the case when  $D_v$ hits the weight function $\phi(t,x,v)$. Recall (\ref{april2eqn1}). Note that the following estimates hold for any fixed $x,v\in \textit{supp}(\phi(t,x,v)-1)$, 
\be\label{april2ndeqn21}
x\cdot \tilde{v}\leq -10,\quad |x|\geq 10,  |v|\gtrsim 1,  \quad \frac{1+|t|}{|x||v|} \leq 2^{-4}, \quad |x|\geq 2^{4}(1+|t|)/|v|. 	
\ee
Based on the possible destination of $D_v$, we separate into three cases as follows. 

\noindent $\bullet$\quad The case when $D_v$ hits the coefficient $(x\cdot v)/(1+|x|)$. Recall the equalities in (\ref{octeqn456}).  The following decomposition of $D_v$  holds, 
\be\label{april2eqn61}
D_v= \tilde{v} S^v - \frac{t}{(1+|v|^2)^{3/2}}S^x + \sum_{i=1,2,3} \tilde{V}_i \Omega_i^v - \frac{t}{(1+|v|^2)^{1/2}}\Omega_i^x.
\ee
As   results of direct computations, the following equalities hold for any $i\in\{1,2,3\}$,
\be\label{april2eqn51}
S^v\big( {x\cdot v}  \big)=  {x\cdot \tilde{v}},\,\,  \Omega_i^v\big( {x\cdot v}  \big)=  {x\cdot \tilde{V}_i},\,\,  S^x\big( {x\cdot v}  \big)=|v|, \,\,   \Omega_i^x\big( {x\cdot v}  \big)=0,  
\ee 
\be\label{april9eqn21}
S^x\big(\frac{1}{1+ |x|} \big)= - \frac{ x\cdot \tilde{v} }{ |x|(1+|x|)^2}, \,\, \Omega_i^x\big(\frac{1}{ 1+|x|} \big)=- \frac{  \tilde{V}_i \cdot x  }{ |x|(1+|x|)^2}.	
\ee
Therefore, from the above equalities, the following estimate holds for any fixed $x,v\in \textit{supp}(\phi(t,x,v)-1)$,
\[
\Big[\big|\frac{S^v\big((x\cdot v)/(1+|x|) \big)}{\phi(t,x,v)}\big| + \frac{t}{(1+|v|^2)^{3/2}}\big|\frac{S^x\big((x\cdot v)/(1+|x|) \big)}{\phi(t,x,v)}\big|+\sum_{i=1,2,3} \big|\frac{\Omega_i^v\big((x\cdot v)/(1+|x|) \big)}{\phi(t,x,v)}\big|\]
\[
+ \frac{t}{(1+|v|^2)^{1/2}} \big|\frac{\Omega_i^x\big((x\cdot v)/(1+|x|) \big)}{\phi(t,x,v)}\big|\Big] 
  \frac{1}{1+||t|-|x+\hat{v}t||}\]
\be\label{april2eqn31}
\lesssim  1+ \frac{|t|}{(1+|x|)(1+|v|)} +\big[\frac{|t|}{ (1+|v|)^3} \frac{1}{|x\cdot \tilde{v}|} + \frac{|x|}{|x\cdot v |}\big]\frac{1}{1+||t|-|x+\hat{v}t||}.
\ee
If $x\cdot \tilde{v}\leq -2^{-10}|x|^2/t$, from (\ref{april2ndeqn21}), we know that the following estimate holds for any fixed $x,v\in \textit{supp}(\phi(t,x,v)-1)$
\be\label{aprileqn33}
 \frac{t}{ (1+|v|)^2} \frac{1}{|x\cdot \tilde{v}|} + \frac{|x|}{|x\cdot v|} \lesssim \frac{ |t| }{(1+|v|) |x| }+ \frac{t^2}{(1+|v|)^2|x|^2}\lesssim 1. 
\ee
If $x\cdot \tilde{v}\geq  -2^{-10}|x|^2/t$ , then from (\ref{april2ndeqn21}),  the following estimate holds for the distance with respect to the light cone, 
\be\label{april2eqn26}
\frac{1}{1+||t|-|x+\hat{v}t||} = \frac{1+||t|+|x+\hat{v}t||}{1+||t|-|x+\hat{v}t|| + ||t|+|x+\hat{v}t|| + |\frac{t^2}{1+|v|^2}-2tx\cdot \tilde{v}-|x|^2|}  \lesssim \frac{|t|+|x|}{|x|^2}.
\ee
Therefore, from the above estimate and the estimate  (\ref{april2ndeqn21}), we have  
\be\label{april2eqn29}
\big(\frac{t}{(1+|v|^2)}+\frac{|x|}{|x\cdot v|}\big)\frac{1}{1+||t|-|x+\hat{v}t||} \lesssim 1+ \frac{ |t| }{(1+|v|) |x| }+ \frac{t^2}{(1+|v|)^2|x|^2}\lesssim 1. 
\ee
To sum up, from the estimates   (\ref{aprileqn33}) and (\ref{april2eqn29}), in whichever case,  the following estimate holds if $D_v$ hits the coefficient $(x\cdot v)/(1+|x|)$, 
\be\label{april2eqn81}
\textup{(\ref{april2eqn31})} \lesssim 1. 
\ee

\noindent $\bullet$\quad The case when $D_v$ hits the cutoff function $\eta(x\cdot\tilde{v})$.  For this case, we know that $x\cdot \tilde{v}\in(-20,-10)$ inside the support. Recall the equalities in (\ref{april2eqn51}) and (\ref{april9eqn21}). We know that the following estimate holds for any fixed $x,v\in \textit{supp}(\phi(t,x,v)-1)\cap \textit{supp}(\eta'(x\cdot \tilde{v}))$.  
\be\label{april2eqn55}
\big|D_v(\eta(x\cdot \tilde{v})) \big| \lesssim \frac{1}{|v|} + \frac{|x|}{|v|} + \frac{|t|}{1+|v|^3}.
\ee
We first consider the case when $|x|\leq 2^{10}|v|$.  Since $|x|\geq (1+|t|)/|v|$ (see(\ref{april2ndeqn21})), for this case,  we have $|v|\gtrsim (1+|t|)^{1/2}$. From the estimate (\ref{april2eqn55}), we have 
\be\label{april2eqn82}
\big|D_v(\eta(x\cdot \tilde{v})) \big| \lesssim 1.
\ee
It remains to consider the case when $|x|\geq 2^{10} |v|$. For this case we have $|x|^2 \geq 2^{10}|x||v|\geq  (1+|t|) $, which implies that $|x|\geq 2^{5}(1+|t|)^{1/2}$. Therefore, from the equality in (\ref{april2eqn26}), we have
\[
\frac{1}{1+||t|-|x+\hat{v}t||} \lesssim \frac{|t|+|x|}{|x|^2}.
\]
Therefore, from the above estimate and the estimate (\ref{april2eqn55}), the following estimate holds,  
\be\label{april2eqn83}
\big|D_v(\eta(x\cdot \tilde{v})) \big|  \frac{1}{1+||t|-|x+\hat{v}t||} \lesssim 1 + \frac{t}{|x||v|} + \frac{t^2}{|x|^2|v|^2}\lesssim 1. 
\ee

\noindent $\bullet$\quad The case when $D_v$ hits the cutoff function  $f((1+|t|)/( |x| |v| ))$. From the decomposition of ``$D_v$'' in (\ref{april2eqn61}), as a result of   direct computations, we know  that the following estimate holds for any $x,v\in\textit{supp}(\phi(t,x,v)-1)$,
\be\label{april2eqn84}
|D_v\big(f((1+|t|)/(|x||v|)\big) | \lesssim \frac{t}{|x||v|^2} + \frac{t^2}{|x|^2|v|^2} \lesssim 1. 
\ee
To sum up, from the estimate (\ref{april2eqn81}), (\ref{april2eqn82}), (\ref{april2eqn83}), and (\ref{april2eqn84}),  we have the following estimate, 
\be\label{april2eqn87}
\big|\frac{D_v\phi(t,x,v)}{\phi(t,x,v)} \big|\frac{1}{1+||t|-|x+\hat{v}t||} \lesssim 1. 
\ee
Recall the definition of $\omega_{\beta}^\alpha(t,x,v)$ in (\ref{highorderweight}),  our    desired first estimate in (\ref{feb8eqn51}) holds from the estimates (\ref{march25eqn42}) and (\ref{april2eqn87}).

 Now, we proceed to prove the second part of the desired estimate (\ref{feb8eqn51}). Recall the equality (\ref{march18eqn54}).   Note that the following estimate holds, 
\[
\big|\frac{\tilde{d}(t,x,v)\phi(x,v)}{1+||t|-|x+\hat{v}t||}\big| \lesssim 1+ \frac{1+ |t|+|x+\hat{v} t| }{ (t+(|1+|v|)(|x\cdot v | +|x|)} \frac{|x\cdot v |}{|x|} f((1+|t|)/(|x||v|))
\]
\be\label{april2eqn131}
\lesssim 1+\frac{|t|+|x|}{|x|(1+|v|)}f((1+|t|)/(|x||v|)) \lesssim 1. 
\ee
Hence finishing the proof of our desired second estimate in  (\ref{feb8eqn51}).
\end{proof}

 From  the estimate  (\ref{densitydecay})  in Lemma \ref{decayestimateofdensity}, we know that  the zero frequency  of the profile plays the leading role in the decay estimate of the density type function. With this intuition, 
similar to the study of the Vlasov-Poisson system in \cite{wang3} and the study of Vlasov-Nordstr\"om system in \cite{wang},  
 we define a lower order energy for the profile $g(t,x,v)$ as follows, 
\be\label{octeqn1896}
E_{\textup{low}}^{f}(t):=\sum_{\gamma\in \mathcal{B}, |\alpha|+|\gamma| \leq N_0 }   \| \widetilde{\omega_{ }^{\alpha}}(  v)\big(\nabla_v^\alpha \widehat{ g^\gamma }(t,0,v)-\nabla_v\cdot  \widetilde{g}_{\alpha, \gamma}(t, v) \big)\|_{  L^2_v},\,\,\,\, \widetilde{\omega_{\gamma }^{\alpha}}(  v):=  (1   +|v|^{ } )^{20N_0-10(|\alpha|+|\gamma|) } .
\ee 
where the correction term $ \widetilde{g}_{\alpha,\gamma}(t, v)$, which is introduced to avoid losing derivatives for the study of the time evolution of $\nabla_v^\alpha \widehat{ g^\gamma }(t,0,v)$, is defined as follows,
\be\label{correctionterm}
 \widetilde{g}_{\alpha,\gamma}(t, v):=\left\{\begin{array}{ll}
 \displaystyle{\int_0^t \int_{\R^3} -K(s,x+\hat{v}s,v) \nabla_v^\alpha g^{\gamma  }(s,x,v) } d x  d s & \textup{if}\, |\alpha| +|\gamma|=N_0\\ 
 &\\
 0 & \textup{if}\, |\alpha| +|\gamma| < N_0 ,\\ 
 \end{array}
 \right.
\ee
where $K(t,x,v)$ is defined in (\ref{eqn387}).

 We define a high order energy  of  the electromagnetic field as follows, 
\[
   E_{\textup{high}}^{eb}(t):=  \sum_{\alpha \in \mathcal{B}, |\alpha|\leq N_0}\sum_{i=1,2}  
  \big[\sup_{k\in \mathbb{Z}}   2^{k}  \| \widehat{h_i^\alpha }(t,\xi)\psi_k(\xi) \|_{L^\infty_\xi}  +  2^{k}  \| \widehat{\widetilde{h_i^\alpha} }(t,\xi)\psi_k(\xi) \|_{L^\infty_\xi} +  2^{k/2} \| \nabla_\xi	 \widehat{\widetilde{h_i^\alpha} }(t,\xi)\psi_k(\xi)\|_{L^2_\xi}   \big] 
\]
 \be\label{highorderenergy2}
 +     \| \widehat{h_i^\alpha }(t,\xi) \|_{L^2_\xi}  +  \|   \widehat{\widetilde{h_i^\alpha} }(t,\xi) \|_{L^2_\xi} .
\ee

The first part of energy $E_{\textup{high}}^{eb}(t)$, which is stronger than $L^2$ at low frequencies, controls the low frequency part of the profiles $h_i^\alpha(t)$, $i\in\{1,2\}$;  the second part of energy $E_{\textup{high}}^{eb}(t)$, which has the same scaling level as the first part of energy $E_{\textup{high}}^{eb}(t)$,  aims to control the first order weighted norm of the modified profiles $\widetilde{h_i^\alpha}(t)$, $i\in\{1,2\}$; the third part   of energy $E_{\textup{high}}^{eb}(t)$, controls the high frequency part of the profiles $h_i^\alpha(t)$, $i\in\{1,2\}$. 

Moreover, we define a low order energy for the profiles $h_i^\alpha(t)$, $i\in\{1,2\},$ of the electromagnetic field as follows, 
\be\label{secondorderloworder}
   E_{\textup{low}}^{eb}(t):= \big[\sum_{n=0,1,2,3 }\sum_{i=1,2}\sum_{\alpha\in \mathcal{B}, |\alpha|\leq 20-3n} \| h^\alpha_i(t)\|_{X_n} + (1+t)\| \p_t h_i^\alpha (t)\|_{X_n} + (1+t)^2\| \p_t \nabla_x(1+|\nabla_x|)^{-1} h_i^\alpha (t)\|_{X_n}\big],
\ee
where the $X_n$-normed space, $n\in\{0,1,2,3\}$, is defined as follows, 
 \be\label{definitionofXnorm}
\|h\|_{X_n}:=  \sup_{k\in \mathbb{Z}}2^{(n+1)k }\| \nabla_\xi^n \widehat{h }(t, \xi)\psi_k(\xi)\|_{L^\infty_\xi}.
\ee

To show that  the electromagnetic field decays at rate $1/\big((1+|t|)(1+||t|-|x||)\big)$ over time, from the   decay estimates in Lemma \ref{sharpdecaywithderivatives}, it would be sufficient to show that the low order energy $E_{\textup{low}}^{eb}(t)$ doesn't grow over time.

 \begin{lemma}\label{sharpdecaywithderivatives}
For any given Fourier multiplier operator $T$ with symbol $m(\xi)\in \mathcal{S}^\infty$,   the following estimate holds, 
\[
\sum_{\begin{subarray}{c}
 \alpha    \in \mathcal{B},  |\alpha|  \leq 10, 
 u\in \{E^\alpha, B^\alpha	\}\\
 \end{subarray}}     |   T(u) (t, x)| + (1+||t|-|x||) \big|   \nabla_x T(u) (t, x)\big| + \sum_{|\alpha|\leq 10, v\in\{E,B\}}(1+||t|-|x||)^{|\alpha|} \big|   \nabla_x^\alpha T( v) (t, x)\big| 
 \]
\be\label{noveqn78}
 \lesssim (1+|t|)^{-1} (1+||t|-|x||)^{-1 } \|m(\xi)\|_{\mathcal{S}^\infty} E_{\textup{low}}^{eb}(t),
\ee
\be\label{march18eqn30}
\sum_{ 
 \alpha    \in \mathcal{B},  |\alpha|  \leq 10,
 u\in \{E^\alpha, B^\alpha	\}
}   |   P_k\circ T(\p_t u) (t, x)| +  |   P_k\circ T(\d u) (t, x)|  \lesssim (1+|t|)^{-1} (1+||t|-|x||)^{-1 }2^{k-4k_{+}} \|m(\xi)\|_{\mathcal{S}^\infty_k} E_{\textup{low}}^{eb}(t). 
\ee
\end{lemma} 
\begin{proof}
 With minor modification in the proof of  \cite{wang}[Lemma 6.3], the desired estimates (\ref{noveqn78}) and (\ref{march18eqn30}) hold from the linear decay estimate  (\ref{noveqn555})   in Lemma \ref{twistedlineardecay}.  The main idea of the proof lies in the process of trading 
one spatial derivative for the decay of modulation of size ``$\big(1+|t|-|x|\big)^{-1}$'' in the sense of equality (\ref{feb16eqn1}) and the equality (\ref{noveqn171}) in Lemma \ref{tradethreetimes}. 
\end{proof}

 \subsection{Proof of the   Theorem \ref{precisetheorem}}\label{proofofthetheorem}

To prove our main  theorem, we use the standard  bootstrap argument. From the local existence theory, we know that the lifespan of the solution is at least of size $(1/\epsilon_0)^{1/2}$ if the given initial data is of size $\epsilon_0$, where $\epsilon_0\ll 1$. Moreover, the assumption imposed on the initial data in (\ref{march1steqn1}) is strong enough to guarantee that the  initial energy is of size $\epsilon_0$.  For convenience, the starting time of our bootstrap assumption is one. More precisely, the following estimate holds, 
\be\label{april8eqn51}
\sup_{t\in[0,1] } (1+t)^{-\delta}\big( E_{\textup{high}}^{f;1}(t )  + E_{\textup{high}}^{eb}(t ) \big)+(1+t)^{-\delta/2}  E_{\textup{high}}^{f;2}(t )  + E_{\textup{low}}^f(t )  + E_{\textup{low}}^{eb}(t ) \lesssim \epsilon_0.
\ee
   We expect that the high order energy grows sub-polynomially and the low order energy doesn't grow over time. Therefore, we make the following bootstrap assumption, 
\be\label{bootstrap}
\sup_{t\in[1,T] }  (1+t)^{-\delta}\big( E_{\textup{high}}^{f;1}(t )  + E_{\textup{high}}^{eb}(t ) \big)+(1+t)^{-\delta/2}  E_{\textup{high}}^{f;2}(t )  + E_{\textup{low}}^f(t )  + E_{\textup{low}}^{eb}(t )  \lesssim \epsilon_1:=\epsilon_0^{5/6},
\ee
where $T>1$. 

Recall the definition of the correction term $\widetilde{g}_{\alpha,\gamma}(t, v)$ in (\ref{correctionterm}). From the  $L^2_{x,v}-L^\infty_{x,v}$ type bilinear estimate, the equality (\ref{feb16eqn1}) and the  decay estimate (\ref{noveqn78}) in Lemma \ref{sharpdecaywithderivatives}, the following estimate holds for the correction term, 
\be\label{may22eqn100}
\sum_{|\alpha|+|\gamma|\leq N_0}\| \widetilde{\omega^\alpha_\gamma }(v) \widetilde{g}_{\alpha,\gamma}(t, v)\|_{L^2_v} \lesssim \int_{0}^t (1+s)^{-2} E_{\textup{low}}^{eb}(s)E_{\textup{high}}^{f}(s) ds \lesssim \int_{0}^t (1+s)^{-2+\delta} \epsilon_1^2 d s \lesssim \epsilon_0.
\ee

Recall the decompositions (\ref{eqq50}) and (\ref{nov104}). From the estimate  
(\ref{april2eqn171}), the estimate (\ref{sepeqn110}) in Proposition \ref{notbulkterm1}, the estimates (\ref{nov102}) and (\ref{may22eqn41})  in Proposition \ref{notbulkterm2},   the estimates (\ref{nov418}) in Lemma \ref{awayfromthecone}, and the estimate (\ref{nov412}) in Lemma \ref{dec30proposition1}, we have
\be\label{may22eqn1}
\sup_{t\in[1,T]} (1+t)^{-\delta} E_{\textup{high}}^{f;1}(t) +(1+t)^{-\delta/2} E_{\textup{high}}^{f;2}(t)  \lesssim \epsilon_0. 
\ee
From  the above estimate (\ref{may22eqn1}),  the estimate (\ref{octeq501}) in Proposition \ref{loworderenergyprop1}, the estimate (\ref{octe691}) in Proposition \ref{highorderenergyprop1},  the estimate (\ref{loworderenergyvlasov}) in Proposition \ref{loworderenergy},   the following estimate holds,  
\be\label{may22eqn3}
\sup_{t\in[0,T] } (1+t)^{-\delta}  E_{\textup{high}}^{eb}(t)   + E_{\textup{low}}^f(t )  + E_{\textup{low}}^{eb}(t ) \lesssim \epsilon_0.
\ee

From the estimates (\ref{may22eqn1}) and (\ref{may22eqn3}), we know  that our bootstrap assumption  (\ref{april8eqn51}) is improved. Hence,   we can keep extending  the length of the lifespan of the nonlinear solution, i.e.,  $T=+\infty$. Moreover,  the following estimate holds 
 \be\label{march30eqn110}
\sup_{t\in[0,\infty) }  (1+t)^{-\delta}\big( E_{\textup{high}}^{f;1}(t )  + E_{\textup{high}}^{eb}(t ) \big)+(1+t)^{-\delta/2}  E_{\textup{high}}^{f;2}(t )  + E_{\textup{low}}^f(t )  + E_{\textup{low}}^{eb}(t )  \lesssim \epsilon_0.
\ee

Since the low order energy doesn't grow over time, from  the definition of the low order energy of the electromagnetic field in (\ref{secondorderloworder})  and  the estimate (\ref{loworderenergyvlasov}) in Proposition \ref{loworderenergy}, we know  that the nonlinear solution scatters to a linear solution in a low regularity space.

Moreover, the desired decay estimates (\ref{desiredecayaverage}) and (\ref{decayingeneral}) holds directly from the decay estimate (\ref{densitydecay}) in Lemma \ref{decayestimateofdensity}, the decay estimate (\ref{noveqn78}) in Lemma \ref{sharpdecaywithderivatives}, and the fact that the low order energy $E_{\textup{low}}^f(t )$ and $E_{\textup{low}}^{eb}(t )$ do not grow over time, see the estimate (\ref{march30eqn110}). Hence finishing the proof of the main theorem.

\section{Energy estimates for the electromagnetic field}\label{energyelectromagnetic}
 
This section is devoted to control both the low order energy and the high order energy of the profiles of the electromagnetic field over time, i.e., controlling $E_{\textup{high}}^{eb}(t)$, which is defined in (\ref{highorderenergy2}), and $E_{\textup{low}}^{eb}(t)$, which is defined in (\ref{secondorderloworder}),
over time. 

Although there is little essential difference between the nonlinear wave part of the Vlasov-Maxwell system and the Vlasov-Nordstr\"om system, for the sake of readers, we still give a concise proof here. The main ingredients of the energy estimate of the electromagnetic field are a linear estimate and several bilinear estimates, which have been derived and proved in \cite{wang}. We first record these multilinear estimates here and then use these general multilinear estimates as black boxes to estimate the increment of the high order energy $E_{\textup{high}}^{eb}(t)$ and the low order energy $E_{\textup{low}}^{eb}(t)$ over time. 

\subsection{Some multilinear estimates}
Recall (\ref{octeqn1725}) and (\ref{eqqq3}). To estimate the $X_n$-norms, $n\in\{0,1,2,3\}$, of the linear terms inside the nonlinearities of $\p_t\widehat{h_i^\alpha}(t, \xi), i\in\{1,2\},\alpha\in \mathcal{B}$, we use the following Lemma. 
\begin{lemma}\label{Alinearestimate}
Given any given $n\in \mathbb{N}_{+}, n\leq 10$,   and any given symbol $m(\xi, v)$ that satisfies the following  estimate,
 \be\label{roughestimateofsymbolgenarl2}
 \sup_{k\in \mathbb{Z}}\sum_{i=0,1,\cdots,10,|a|\leq 15 }  2^{ik-(n-1)k} \|(1+|v|)^{-20-4i} \nabla_\xi^i\nabla_v^a m(\xi, v)\psi_k(\xi)\|_{L^\infty_\xi L^\infty_v}\lesssim 1,
\ee
then the following estimate holds for any $i\in\{0,1,2,3\},$  
\[
\| \int_{\R^3} e^{i t|\xi|- i\mu t\hat{v}\cdot \xi}    m(\xi, v) \widehat{g}(t, \xi, v) d v \|_{X_i} 
\]
\be\label{octe0303}
 \lesssim   \sum_{ |\alpha|\leq  i+n }  (1+|t|)^{-n} \|(1+|v|)^{30} \nabla_v^\alpha \widehat{g}(t,  0, v)  \|_{  L^1_v } +   \sum_{\beta\in \mathcal{S}, |\beta|\leq  i+n }  (1+|t|)^{-n-1} \| (1+|x|^2+|v|^2)^{20}\Lambda^\beta g(t,x,v)\|_{L^2_{x }L^2_v}.
 \ee
 Moreover, for any differentiable vector value function $\tilde{g}(t,v):\R_t\times \R_v^3\rightarrow \R^3$,  the following $L^\infty_\xi$-type estimate holds for any fixed $k\in \mathbb{Z},$ 
 \[
2^k \|\int_{\R^3} e^{i t|\xi|- i\mu t\hat{v}\cdot \xi}    m(\xi, v) \widehat{g}(t, \xi, v)\psi_{k}(\xi) d v\|_{L^\infty_\xi } \lesssim 2^{nk}\big(\|(1+|v|)^{20}\big(\widehat{g}(t,0,v)-\nabla_v\cdot \tilde{g}(t,v)\big)\|_{L^1_v} \]
\be\label{may17eqn12}
+ (1+|t|2^{k})\| (1+|v|)^{20} \tilde{g}(t,v)\|_{L^1_v} + 2^{k} \| (1+|x|+|v|)^{30} g(t,x,v)\|_{L^2_x L^2_v}\big).
 \ee
\end{lemma}
\begin{proof}
See \cite{wang}[Lemma 5.1]. 
\end{proof}
 
Recall     (\ref{octeqn1725}),  (\ref{octeqn798}), (\ref{octe662}), (\ref{octe661}), and  (\ref{octe51}). Motivated from the Vlasov-wave type interaction  structure of quadratic terms, we  study a bilinear form that will be suitable for the estimate of all quadratic terms in  $\p_t\widehat{h_i^\alpha}(t, \xi)$ and  $\p_t\widehat{\widetilde{h_i^\alpha}}(t, \xi)$, $i\in\{1,2\}.$ More precisely, for any $l\in \{0,1\}$ and any given symbol $m(\xi, v)$ that satisfies the following estimate, 
\be\label{generalestimateofsymbol}
 \sup_{k\in \mathbb{Z}}\sum_{n=0,1,2,3} \sum_{|\alpha|\leq 5}  2^{lk+nk} \|(1+|v|)^{-20}\nabla_\xi^n \nabla_{v}^\alpha m (\xi, v)\psi_k(\xi)\|_{  L^\infty_v \mathcal{S}^\infty_k}  \lesssim 1,
 \ee
 we define  a bilinear operator as follows, 
\be\label{octeqn1923}
T_{\mu}  (h, f)(t, \xi):= \int_{\R^3}\int_{\R^3} e^{it|\xi|-i\mu t |\xi-\eta| - i  t \hat{v}\cdot \eta} m(\xi, v)\widehat{h^\mu}(t, \xi-\eta) \widehat{f}(t,\eta, v) d \eta d v. 
\ee

For the above defined bilinear operator, we have several bilinear estimates in different function spaces, which will be used in the low order energy estimate and the high order energy estimate. 
 \begin{lemma}\label{bilinearinXnormed}
Given any $n \in\{0,1,2,3 \}$, any $l\in\{0,1\}$, and  any given symbol ``$m(\xi, v)$'' that satisfies the estimate \textup{(\ref{generalestimateofsymbol})}, the following estimate holds for the bilinear form  $T_{\mu} (h, f)(t, \xi) $ defined in \textup{(\ref{octeqn1923})}, 
\[
\sup_{k\in \mathbb{Z}} 2^{(n+1)k} \| \nabla_\xi^n \big(T_{\mu} (h, f)(t, \xi) \big)\psi_{k}(\xi)\|_{L^\infty_\xi}\lesssim \sum_{\beta\in\mathcal{S}, |\beta|\leq n+3} \sum_{|a|\leq n+3}   (1+|t|)^{-3+l} \big( \sum_{0\leq c \leq n } \sum_{0\leq b \leq n-c}\sum_{|\alpha|\leq c}\| h^{\alpha}\|_{X_b} \big)
\]
\be\label{octe364}
 \times    \| (1+|x|^2+|v|^2)^{20} \Lambda^{\beta} {f}(t, x,v) \|_{L^2_x  L^2_v } .
\ee
\end{lemma}
 \begin{proof}
 See \cite{wang}[Lemma 5.2]. 
 \end{proof}
    \begin{lemma}\label{highorderbilinearlemma1}
Given   any    symbol ``$m(\xi, v)$'' that satisfies the estimate \textup{(\ref{generalestimateofsymbol})} with $l=1$, 
  the following estimate holds for the bilinear form  $T_{\mu} (h, f)(t, \xi) $ defined in \textup{(\ref{octeqn1923})}, 
\be\label{octe103}
 \sup_{k \in \mathbb{Z}} 2^{ k }\| T_{\mu} (h, f)(t, \xi)\psi_k(\xi)\|_{L^\infty_\xi} \lesssim \sum_{n=0,1,2, \alpha\in \mathcal{B},  |\alpha|\leq 4} (1+|t|)^{-2+\delta}  \|  {h^\alpha}(t )\|_{X_n}   \| (1+|x|^2+|v|^2)^{20} f(t,x,v)\|_{L^2_x L^2_v}. 
\ee
Moreover,  the following $L^2$-type estimate holds, 
\[ 
  \| T_{\mu} (h, f)(t, \xi) \|_{L^2_\xi}\lesssim   \min \big\{  \sum_{n=0,1, \alpha\in \mathcal{B},  |\alpha|\leq 4}  (1+|t|)^{-1}  \|  {h^\alpha}(t )\|_{X_n}    \| (1+|x|^2+|v|^2)^{20} f(t,x,v)\|_{L^2_x L^2_v}, 
\]
\be\label{deceqn98}
  \sum_{\beta\in \mathcal{S}, |\beta|\leq 3} (1+|t|)^{-2}  \|  h(t)\|_{L^2}   \| (1+|x|^2+|v|^2)^{20}\Lambda^\beta f(t,x,v) \|_{L^2_{x,v}} \big\}.
\ee
\end{lemma}
  \begin{proof}
 See \cite{wang}[Lemma 5.3 \& Lemma 5.4]. 
 \end{proof}

\begin{lemma}\label{firstorderweighthigh1}
Given   any    symbol ``$m(\xi, v)$'' that satisfies the estimate \textup{(\ref{generalestimateofsymbol})} with $l=1$,  
  the following estimate holds for the bilinear form  $T_{\mu} (h, f)(t, \xi) $ defined in \textup{(\ref{octeqn1923})}, 
  \[
    2^{k/2}\|\nabla_\xi\big(T_{\mu}(h, f)(t, \xi) \big)\psi_k(\xi)\|_{L^2}\lesssim   \sum_{0\leq n \leq 3}\sum_{ \alpha\in\mathcal{B}, |\alpha|\leq 4} \big( (1+|t|)^{-1}2^{k_{-}   } + (1+|t|)^{-2+\delta}\big)  \| {h^\alpha}(t   )\|_{X_n} 
  \]
\be\label{octe406}
\times \| (1+|x|^2+|v|^2)^{20}  f(t, x, v)\|_{L^2_x L^2_v}.
\ee
Moreover, we have
\[
\sup_{k\in \mathbb{Z}} 2^{k/2}\|\nabla_\xi\big(T_{\mu}(h, f)(t, \xi) \big)\psi_k(\xi)\|_{L^2}\lesssim  (1+|t|)^{-2}\big(\sup_{k\in \mathbb{Z}} 2^{k} \| \widehat{h}(t, \xi)\psi_k(\xi)\|_{L^\infty_\xi} + 2^{k/2}\|\nabla_\xi \widehat{h}(t, \xi)\psi_k(\xi)\|_{L^2}\big)
\]
\be\label{octe581}
\times \big(\sum_{\beta\in \mathcal{S}, |\beta|\leq 4} \|  (1+|x|^2+|v|^2)^{20} \Lambda^\beta f(t,x,v)\|_{L^2_x L^2_v} \big).
\ee
\end{lemma}
 \begin{proof}
 See \cite{wang}[Proposition 5.3]. 
 \end{proof}
Moreover, as summarized in the following Lemma, we also have a bilinear estimate  for the Vlasov-Vlasov type interaction.  
 \begin{lemma}\label{twoaveragedistri}
For any symbols $m_1(\xi, v), m_2(\xi,v)$ that satisfy \textup{(\ref{generalestimateofsymbol})} with $l=1$, and any two distribution functions $f, g:\R_{t}\times \R_x^3\times \R_v^3\longrightarrow \mathbb{R}$, we define a   bilinear operator  as follows, 
\be\label{may15eqn60}
 K^{\mu}  (g, f)(t, \xi):= \int_{\R^3}\int_{\R^3}\int_{\R^3} e^{it|\xi|-i\mu t \hat{u}\cdot (\xi-\eta) - i  t \hat{v}\cdot \eta} m_1(\xi,   v)m_2(\xi-\eta,   u)\widehat{g}(t, \xi-\eta, u) \widehat{f}(t,\eta, v) d \eta d u d v  .
\ee
Then the following bilinear estimate holds for any fixed $k\in \mathbb{Z}, $  
\[
2^{k/2}\|\nabla_\xi\big( K^{\mu}  (g, f)(t, \xi) \big)\psi_k(\xi) \|_{L^2_\xi }
\]
\be\label{may15eqn410}
 \lesssim \sum_{\beta\in \mathcal{S}, |\beta|\leq 5} (1+|t|)^{-2}\|(1+|x|^2+|v|^2)^{20}g(t,x,v)\|_{L^2_x L^2_v}\|(1+|x|^2+|v|^2)^{20} \Lambda^{\beta}f(t,x,v)\|_{L^2_x L^2_v}.
\ee
 \end{lemma}
 \begin{proof}
 See \cite{wang}[Lemma 5.9].
 \end{proof}

With the previous preparation,  we first control the increment of the low order energy estimate of the electromagnetic field over time. More precisely, the following proposition holds. 
  \begin{proposition}\label{loworderenergyprop1}
Under the bootstrap assumption \textup{(\ref{bootstrap})}, the following estimate holds for any $t\in [1,T]$,
\be\label{octeq501}
 E_{\textup{low}}^{eb}(t) \lesssim E_{\textup{low}}^f(t) + |t|^{-1}  E_{\textup{high}}^f(t) + \epsilon_0 .
\ee
\end{proposition}
\begin{proof}
 Recall (\ref{secondorderloworder}). We first estimate the $X_n$-norm of $\p_t h_i^\alpha(t)$. Recall  (\ref{octeqn1725}) and (\ref{eqqq3}). Form the estimate of coefficients in (\ref{feb16eqn15}),  the estimate (\ref{octe0303}) in Lemma \ref{Alinearestimate}, which is used for the linear terms, and the estimate (\ref{octe364}) in Lemma \ref{bilinearinXnormed}, which is used for the quadratic terms, we have
\[
 \sum_{n=0,1,2,3}\sum_{i=1,2}\sum_{\alpha \in \mathcal{B}, |\alpha|\leq 20-3n}  (1+|t|) \| \p_t h_i^\alpha (t)\|_{X_n} + (1+|t|)^2 \| \frac{\nabla_x}{1+|\nabla_x|} \p_t   h_i^\alpha(t)\|_{X_n} \lesssim   E_{\textup{low}}^f(t) + |t|^{-1}  E_{\textup{high}}^f(t)
\]
\be\label{octeq792}
 + |t|^{-1 }  E_{\textup{high}}^f(t) E_{\textup{low}}^{eb}(t)\lesssim   E_{\textup{low}}^f(t) + |t|^{-1}  E_{\textup{high}}^f(t)+ \epsilon_0.
\ee

Now, it remains to estimate the $X_n$-norm of $h_i^\alpha(t), i\in\{1,2\}$. Recall (\ref{octe631}) and (\ref{octe633}). As a result of direct computations, we know  the symbol $\xi/\big(|\xi|(|\xi|-\hat{v}\cdot \xi)\big)$ verifies the estimate (\ref{roughestimateofsymbolgenarl2}).   From  the estimate of coefficients in (\ref{feb16eqn15}) and the estimate (\ref{octe0303}) in Lemma \ref{Alinearestimate}, we have
 \be\label{octe673}
 \sum_{n=0,1,2,3 }\sum_{|\alpha|\leq 20-3n} \|\widetilde{h_i^\alpha}(t)- h_i^\alpha(t)\|_{X_n} \lesssim   E_{\textup{low}}^{f}(t) + |t|^{-1} E_{\textup{high}}^{f}(t).
 \ee
 Therefore, it would be sufficient to estimate the $X_n$-norm     of the modified profiles $\widetilde{h_i^\alpha}(t), i \in\{1,2\}$. Recall the equations satisfied by $\p_t \widehat{\widetilde{h_i^\alpha}}(t,\xi)$ in (\ref{octe662}) and the detailed formula of the quadratic terms in (\ref{octeqn798}), (\ref{octe661}), and (\ref{octe51}), we know that $\p_t\widetilde{h^\alpha}(t,\xi) $ is a linear combination of bilinear forms defined in (\ref{octeqn1923}). Therefore, from the estimate   (\ref{octe364}) in Lemma \ref{bilinearinXnormed}, we have
\be\label{octe667}
\sum_{n=0,1,2,3 }\sum_{i=1,2}\sum_{|\alpha|\leq 20-3n} \| \p_t  \widetilde{h_i^\alpha}(t    )\|_{X_n} \lesssim (1+|t|)^{-2}   E_{\textup{low}}^{eb}(t)  E_{\textup{high}}^{f}(t)\lesssim (1+|t|)^{-2+\delta}\epsilon_1^2. 
 \ee
Hence, from the above estimate (\ref{octe667}) and the estimate (\ref{octe673}), we have
\be\label{octe671}
 \sum_{n=0,1,2,3}\sum_{i=1,2}\sum_{|\alpha|\leq 20-3n} \|\widetilde{h_i^\alpha}(t)\|_{X_n} + \| {h_i^\alpha}(t)\|_{X_n}\lesssim    E_{\textup{low}}^{f}(t) + |t|^{-1} E_{\textup{high}}^{f}(t)+\epsilon_0+ \int_1^t  |s|^{-2+\delta}  \epsilon_1^2 d s.
\ee
To sum up, our desired estimate (\ref{octeq501}) hold from the estimates (\ref{octeq792}),    and (\ref{octe671}).   
\end{proof}
 
\begin{proposition}\label{highorderenergyprop1}
Under the bootstrap assumption \textup{(\ref{bootstrap})}, the following estimate holds for any $t\in [1,T]$,
\be\label{octe691}
   E_{\textup{high}}^{eb}(t)\lesssim  E_{\textup{high}}^f(t) +(1+|t|)^{\delta}\epsilon_0. 
\ee
\end{proposition}
\begin{proof}
Recall the definition of high order energy  $ E_{\textup{high}}^{eb}(t)$ in  (\ref{highorderenergy2}). Based on the different types of norms in the high order energy of the electromagnetic field, we separate into three cases as follow.

$\bullet$ \textbf{Case} $1$:\quad  The $L^\infty_\xi$-estimate of the profiles and the modified profiles. 

 Recall (\ref{octe631}) and (\ref{octe633}). From the estimate of coefficients in  (\ref{feb16eqn15}), we know that the following estimate holds for any $\alpha\in \mathcal{B}, |\alpha|\leq N_0$,
\be\label{octe695}
\sup_{k\in \mathbb{Z}}2^{k} \|\big(\widehat{\widetilde{h_i^\alpha}}(t, \xi)-\widehat{ {h_i^\alpha}}(t, \xi)\big)\psi_k(\xi)\|_{L^\infty_\xi} \lesssim \sum_{\gamma\in \mathcal{B}, |\gamma|\leq |\alpha|} \| (1+|v|)^{5+4(|\alpha|-|\gamma|)} \widehat{g^\gamma}(t, \xi, v)\|_{L^\infty_\xi L^1_v}\lesssim E_{\textup{high}}^f(t).  
\ee

Now, it would be sufficient to estimate the $L^\infty_\xi$-norm of $\widehat{\widetilde{h_i^\alpha}}(t, \xi)$. Recall (\ref{octe662}). From the estimate (\ref{octe364}) in Lemma \ref{bilinearinXnormed}, which is used when $h_i(t)$ has relatively more derivatives, and the estimate (\ref{octe103}) in Lemma \ref{highorderbilinearlemma1}, which is used when $g(t,x,v)$ has relatively more derivatives, we have
\be\label{octe700}
\sup_{k\in \mathbb{Z}} 2^k \|\p_t \widehat{\widetilde{h_i^\alpha}}(t, \xi)\psi_k(\xi)\|_{L^\infty_\xi } \lesssim (1+|t|)^{-2+\delta}\big(  E_{\textup{high}}^{eb}(t) + E_{\textup{low}}^{eb}(t)\big)  E_{\textup{high}}^f(t) \lesssim (1+|t|)^{-2+3\delta} \epsilon_1^2.
\ee
From (\ref{octe695}) and (\ref{octe700}), we have
\be\label{octe741}
\sup_{k\in \mathbb{Z}} \sum_{|\alpha|\leq N_0}   \sum_{i=1,2}  2^{k} \|  \widehat{ \widetilde{h_i^\alpha}}(t, \xi) \psi_k(\xi)\|_{L^\infty_\xi}  \lesssim \epsilon_0, \quad \sup_{k\in \mathbb{Z}} \sum_{|\alpha|\leq N_0} \sum_{i=1,2}   2^{k} \|  \widehat{ {h_i^\alpha}}(t, \xi) \psi_k(\xi)\|_{L^\infty_\xi}  \lesssim   E_{\textup{high}}^f(t) +\epsilon_0.
 \ee
 $\bullet$ \textbf{Case} $2$:\quad  The $L^2$-estimate of the profiles and the modified profiles. 

By using the first estimate in (\ref{deceqn98}) in Lemma \ref{highorderbilinearlemma1} for the case when there are more derivatives on  the distribution function  $g(t,x,v)$  and using the second   estimate in (\ref{deceqn98}) in Lemma \ref{highorderbilinearlemma1} for the case when there are more derivatives on  the electromagnetic field, we have 
\be\label{2020feb15eqn1}
\sup_{k\in \mathbb{Z}}  \|\p_t \widehat{\widetilde{h_i^\alpha}}(t, \xi)\psi_k(\xi)\|_{L^2_\xi } \lesssim (1+t)^{-1} E_{\textup{high}}^f(t)  E_{\textup{low}}^{eb}(t) + (1+t)^{-2} E_{\textup{high}}^f(t) E_{\textup{high}}^{eb} (t) \lesssim (1+t)^{-1+\delta}\epsilon_1^2\lesssim (1+t)^{-1+\delta}\epsilon_0.
\ee
Moreover, from the estimate (\ref{octe695}), which is used at low frequencies, and the Minkowski inequality, which is used at high frequencies,  we  have 
\[
  \|\big(\widehat{\widetilde{h_i^\alpha}}(t, \xi)-\widehat{ {h_i^\alpha}}(t, \xi)\big) \|_{L^2_\xi} \lesssim \sum_{k\leq 0} 2^{k/2} E_{\textup{high}}^f(t)+  \sum_{|\gamma|\leq |\alpha|,k\geq 0 } 2^{-k} \|  (1+|v|)^{5+4(|\alpha|-|\gamma|)}\widehat{g^\gamma}(t, \xi, v)\|_{L^1_v L^2_\xi }
\]
\be\label{deceqn191}
 \lesssim E_{\textup{high}}^f(t)+ \sum_{|\gamma|\leq |\alpha|}\| \omega_{\gamma}^{\vec{0}}(t,x,v) g^\gamma(t,x,v)\|_{L^2_x L^2_v}\lesssim E_{\textup{high}}^f(t). 
 \ee
From the  estimate (\ref{2020feb15eqn1})   and the estimate (\ref{deceqn191}), we have 
 \be\label{dec30eqn21}
\sup_{k\in \mathbb{Z}  }  \sum_{\alpha\in \mathcal{B}, |\alpha|\leq N_0}   \| \widehat{ {h_i^\alpha}}(t, \xi)\psi_k(\xi)\|_{L^2_\xi } +   \| \widehat{ \widetilde{h_i^\alpha}}(t, \xi)\psi_k(\xi)\|_{L^2_\xi } \lesssim   E_{\textup{high}}^f(t) + (1+t)^{\delta } \epsilon_0. 	
\ee
 $\bullet$ \textbf{Case} $3$:\quad  The weighted $L^2$-estimate of   the modified profiles.

  Recall the equations satisfied by $ \widehat{\widetilde{h^\alpha_i}}(t, \xi)$, $i\in\{1,2\},$ in (\ref{octe662}). 

   We use different strategy for different type of nonlinearity. 
   If the total number of derivatives act on the profiles is less than ten, then we use the estimate (\ref{octe406})  in Lemma \ref{firstorderweighthigh1}. If the total number of derivatives act on the profiles is greater than ten, by using  the equalities (\ref{octe631}) and (\ref{octe633}),
 we first decompose the profiles $\widehat{h_i^\iota}(t,\xi-\eta)$, $i\in\{1,2\}$, in  (\ref{octeqn798}), (\ref{octe661}), and (\ref{octe51}) into two parts: the modified profile part and the density type function part. Then we use   the 
    estimate (\ref{octe581}) in Lemma \ref{firstorderweighthigh1}  for the modified profile part and  use the  estimate (\ref{may15eqn410}) in Lemma \ref{twoaveragedistri} for the density type function part.

     As a result, the following estimate holds  for any $\alpha\in\mathcal{B}, |\alpha|\leq N_0$, $i\in\{1,2\},$
\[
\sup_{k\in \mathbb{Z}}   2^{k/2}\|\p_t \nabla_\xi \widehat{\widetilde{h_i^\alpha}}(t, \xi)\psi_k(\xi)\|_{L^2_\xi } \lesssim \big( (1+t)^{-1}2^{k_{-}} +  (1+t)^{-2+\delta}\big) E_{\textup{low}}^{eb}(t)E_{\textup{high}}^{f}(t)  +  (1+t)^{-2 }\big(E_{\textup{high}}^{eb}(t) +  E_{\textup{high}}^{f}(t)\big) E_{\textup{high}}^{f}(t)  \]
\be
\lesssim (1+t)^{-1+\delta}2^{k_{-}}\epsilon_1^2 + (1+t)^{-2+2\delta}\epsilon_1^2\lesssim (1+t)^{-1+\delta}2^{k_{-}}\epsilon_0 + (1+t)^{-2+2\delta}\epsilon_0.
\ee
Hence, from the above estimate, we know  that the following estimate holds for any $i\in\{1,2\}$ and any fixed $k\in \mathbb{Z}$,
\be\label{octe742}
   \sum_{|\alpha|\leq N_0} 2^{k/2}\|  \nabla_\xi \widehat{\widetilde{h_i^\alpha}}(t, \xi)\psi_k(\xi)\|_{L^2_\xi } \lesssim \epsilon_0 + (1+t)^{\delta} 2^{k_{-}}\epsilon_0.  
\ee
To sum up, recall   (\ref{highorderenergy2}), our desired estimate (\ref{octe691}) holds from the estimates (\ref{octe741}), (\ref{dec30eqn21}), and (\ref{octe742}).  
\end{proof}

  \section{Energy estimates for the non-bulk terms}\label{estimateofnonbulkterm}

 In this section, we mainly finish the following two tasks: (i) Estimate the increment of the low order energy $E_{\textup{low}}^f(t)$  over time. (ii) Recall    the equation satisfied by $g_{\beta}^\alpha(t,x,v)$ in  (\ref{sepeqn43}) and the decompositions of $\textit{h.o.t}_{\beta}^\alpha(t,x,v)$ in (\ref{sepeqn180}), (\ref{sepeq90}), and (\ref{nove100}). We  estimate the high order energy of  all nonlinearities except the bulk term $\textit{bulk}_{\beta}^\alpha(t,x,v)$, see (\ref{nove120}). We refer those terms as \textit{non-bulk} terms;

  Because the issue  of losing $|v|$ caused by the bad coefficient doesn't appear  in  the low order energy estimate and the high order estimate of the  \textit{non-bulk} terms, there is little essential difference between these estimates and the corresponding estimates in the study of Vlasov-Nordstr\"om system   in \cite{wang}.   We only give  concise proofs for these estimates in this section. 

  Recall  (\ref{highorderenergy1}), (\ref{highordersecondpro}), and (\ref{sepeqn43}).   As a result of direct computations,  the following equality holds  for any fixed  $t \in[1, T]$, $\alpha\in \mathcal{B}, \beta \in \mathcal{S}$, s.t., $|\alpha| +|\beta|\leq N_0$,  
\[
 \h \| \omega_{\beta}^{\alpha}( t,x, v)  g^\alpha_\beta (t ,x,v)\|_{L^2_{x,v}}^2- \h \| \omega_{\beta}^{\alpha}(1, x, v)  g^\alpha_\beta (1,x,v)\|_{L^2_{x,v}}^2\]
 \be\label{eqq50}
 = K_{\beta}^\alpha(t)+ \textup{Re}\Big[\int_{1}^{t } \int_{\R^3} \int_{\R^3}\big(\omega_{\beta}^{\alpha}(t, x, v)\big)^2  g^\alpha_\beta (t  ,x,v) \p_t g^\alpha_\beta(t,x,v) d x d v \Big]= K_{\beta}^\alpha(t)+ \sum_{i=1,2,3,4 }  \textup{Re}[I_{\beta;i}^{\alpha}(t)], 
\ee
where
\be\label{april2eqn91}
K_{\beta}^\alpha(t) =  \int_{1}^{t }\int_{\R^3} \int_{\R^3}  \omega_{\beta}^{\alpha}(s,x, v)  \p_t \omega_{\beta}^{\alpha}(s,x, v)  \big|g^\alpha_\beta (s  ,x,v)\big|^2 d x d v d s, 
\ee
\be\label{sepeqn809}
 I_{\beta;1}^{\alpha}(t) = - \int_{1}^{t } \int_{\R^3} \int_{\R^3}\big(\omega_{\beta}^{\alpha}(s,x, v)\big)^2  g^\alpha_\beta (s ,x,v) K(s,x+\hat{v}s,v)\cdot  {D}_v g^\alpha_{\beta}(s,x,v)  d x d v ds,
\ee
\be\label{sepeqn813}
 I_{\beta;2}^{\alpha}(t) =   \int_{1}^{t } \int_{\R^3} \int_{\R^3}\big(\omega_{\beta}^{\alpha}(s, x, v)\big)^2  g^\alpha_\beta (s ,x,v)  \textit{l.o.t}_{\beta}^\alpha(s,x,v) d x  dv d s, 
 \ee
  \be\label{sepeqn810}
 I_{\beta;3}^{\alpha}(t) =   \int_{1}^{t } \int_{\R^3} \int_{\R^3}\big(\omega_{\beta}^{\alpha}(s,  x, v)\big)^2  g^\alpha_\beta (s ,x,v)  \big(\textit{h.o.t}_{\beta}^\alpha(s,x,v) - \textit{bulk}_{\beta}^\alpha(s,x,v)\big) d x  dv ds, 
 \ee
 \be\label{nov96}
 I_{\beta;4}^{\alpha}(t) =  \int_{1}^{t }  \int_{\R^3} \int_{\R^3}\big(\omega_{\beta}^{\alpha}(s, x, v)\big)^2  g^\alpha_\beta (s ,x,v)    \textit{bulk}_{\beta}^\alpha(s,x,v) d x  dv d s,
 \ee
where $\textit{bulk}_{\beta}^\alpha(t,x,v)$ is defined in (\ref{nove120}). Recall the definition of $\omega_{\beta}^\alpha(t,x,v)$ in (\ref{highorderweight}) and the estimate (\ref{april2eqn151}), we have  
\be\label{april2eqn171}
K_{\beta}^\alpha(t) \leq 0,
\ee
 which is a good sign. Hence, there is no need to estimate this term.    We defer the estimate of  bulk term $I_{\beta;4}^\alpha(t)$ to the next section and   estimate all other terms; i.e, $I_{\beta;i}^\alpha(t)$, $i\in\{1,2,3\}$ in this section. 
\begin{proposition}\label{notbulkterm1}
Under the bootstrap assumption \textup{(\ref{bootstrap})}, the following estimate holds for any $t \in [1,T],$
\be\label{sepeqn110}
\sum_{\alpha\in \mathcal{B}, \beta\in \mathcal{S}, |\alpha|+|\beta|= N_0}| I_{\beta;1}^{\alpha}(t)|\lesssim (1+t)^{2\delta}\epsilon_0, \quad   \sum_{\alpha\in \mathcal{B}, \beta\in \mathcal{S}, |\alpha|+|\beta| < N_0}| I_{\beta;1}^{\alpha}(t)|\lesssim (1+t)^{\delta }\epsilon_0
\ee
\end{proposition}
\begin{proof}
  Note that
\[
g^\alpha_\beta (t ,x,v)    {D}_v g^\alpha_{\beta}(t,x,v) = \frac{1}{2} D_v\big(g^\alpha_\beta (t ,x,v)  \big)^2, \quad D_v= \nabla_v - t\nabla_v \hat{v}\cdot \nabla_x. 
\]
Recall  (\ref{sepeqn61}). After doing integration by parts in $x$ and $v$ to move around the derivative ``$D_v$'', the following equality holds, 
\[
 I_{\beta;1}^{\alpha}(t)= \int_{1}^{t } \int_{\R^3} \int_{\R^3} \big(\omega_{\beta}^{\alpha}(s, x, v)g^\alpha_\beta (s ,x,v)\big)^2  \frac{  K(s,x+\hat{v}s,v)\cdot  {D}_v \omega_{\beta}^{\alpha}(s, x, v)}{2\omega_{\beta}^{\alpha}( s,x, v)}  d x dv d s. 	
\]
Therefore,  our 	desired estimate (\ref{sepeqn110}) holds from the $L^2_{x,v}-L^2_{x,v}-L^\infty_{x,v}$ type multilinear estimate,  the estimate (\ref{feb8eqn51}) in Lemma \ref{derivativeofweightfunction}, and the $L^\infty$ decay estimate (\ref{noveqn78}) in Lemma \ref{sharpdecaywithderivatives}.
 \end{proof}

  The main ingredients of the estimate of \textit{non-bulk} terms, i.e., the estimate of  $I_{\beta;2}^{\alpha}(t)$ and $I_{\beta;3}^{\alpha}(t)$,  are several  bilinear estimates, which have been studied and obtained  in the study of Vlasov-Nordstr\"om system in \cite{wang}. We record those bilinear estimates in the following two Lemmas respectively. 
 \begin{lemma}\label{multilinearlemma1}
Given any fixed signs $\mu, \nu\in\{+,-\}$, fixed time $t\in \R_{+}$, fixed $k_1,k_2 \in \mathbb{Z}$.  Moreover, given     any functions $f_1, f_2:\R_t\times \R_x^3 \rightarrow \mathbb{C}$, and any distribution function $g:\R_t\times \R_x^3\times \R_v^3\rightarrow \mathbb{R}$, we define a trilinear form as follows,
\be\label{jan16eqn1}
T(f_1,f_2, g):= \int_{\R^3} \int_{\R^3} e^{-i\mu t\d} P_{k_1}[f_1](t,x+\hat{v}t )e^{-i\nu t\d}  P_{k_2}[f_2](t,x+\hat{v}t ) g(t,x,v) d x d v.
\ee
Then the following estimate holds, 
\[
|T(f_1,f_2,g)| \lesssim \sum_{|\alpha|\leq 4} (1+|t|)^{-5} \|(1+|x|)^{2}(1+|v|)^{25} \nabla_v^\alpha g(t,x,v)\|_{L^1_{x,v}}\big(  2^{-k_{1,-}}\|\widehat{f_1}(t,\xi)\psi_{k_1}(\xi) \|_{L^2} 
\]
\be\label{multilinear1jan16}
 +\|\nabla_\xi \widehat{f_1}(t,\xi)\psi_{k_1}(\xi) \|_{L^2} \big) \big(  2^{-k_{2,-}}\|\widehat{f_2}(t,\xi)\psi_{k_2}(\xi) \|_{L^2}  +\|\nabla_\xi \widehat{f_2}(t,\xi)\psi_{k_2}(\xi) \|_{L^2} \big).
\ee
Moreover,  if $|k_1-k_2|\geq 5$, then the following estimate holds, 
\[
|T(f_1,f_2,g)| \lesssim \sum_{|\alpha|\leq 4} (1+|t|)^{-5} 2^{-\max\{k_1,k_2\}} \|(1+|x|)^{2}(1+|v|)^{25} \nabla_x \nabla_v^\alpha g(t,x,v)\|_{L^1_{x,v}}\big(  2^{-k_{1,-}}\|\widehat{f_1}(t,\xi)\psi_{k_1}(\xi) \|_{L^2} 
\]
\be\label{trilinearfeb20est1}
 +\|\nabla_\xi \widehat{f_1}(t,\xi)\psi_{k_1}(\xi) \|_{L^2} \big) \big(  2^{-k_{2,-}}\|\widehat{f_2}(t,\xi)\psi_{k_2}(\xi) \|_{L^2}  +\|\nabla_\xi \widehat{f_2}(t,\xi)\psi_{k_2}(\xi) \|_{L^2} \big).
\ee
 \end{lemma}
 \begin{proof}
 See \cite{wang}[Lemma 6.4].
 \end{proof}
  For any fixed sign $\mu\in\{+,-\}$, any two distribution functions $f_1(t,x,v)$ and $f_2(t,x,v)$, any fixed $k\in\mathbb{Z}$, any symbol $m(\xi,v)\in L^\infty_v \mathcal{S}^\infty_k$, and  any differentiable coefficient $c(v)$,  we define a bilinear operator as follows, 
\be\label{noveq260}
B_k  (f_1,f_2)(t,x,v):= f_1(t,x,v) E (P_{k}[f_2(t)])(x+a(v)t), 
\ee
where
\[
E(P_k[f])(t,x):=\int_{\R^3}\int_{\R^3} e^{i x \cdot \xi} e^{-i \mu t \hat{u}\cdot \xi} {c(u)  m(\xi,u) \psi_k(\xi)}  \widehat{f}(t, \xi, u) d\xi d u. 
\]
For the above defined bilinear operator, we have  
 \begin{lemma}\label{bilineardensitylemma}
For any fixed  $t\in \R, |t|\geq 1$, and any localized differentiable function $f_3(t,v):\R_t\times \R_v^3\longrightarrow \mathbb{C}$,
   the following bilinear estimate  holds for the bilinear operators defined in  \textup{(\ref{noveq260})},
 \[
\| B_k (f_1,f_2)(t,x  ,v)\|_{L^2_x L^2_v}  \lesssim \sum_{|\alpha|\leq 5} \big( \|m(\xi,v)\|_{L^\infty_v\mathcal{S}^\infty_k} + \|m(\xi,v)\|_{L^\infty_v\mathcal{S}^\infty_k}\big) \big[ |t|^{-2}2^{k} \|\big(|c(v)|+|\nabla_v c(v)|\big) f_3(t,v) \|_{L^2_v} 
     \]
 \[
+ |t|^{-3}  2^{ k }    \|(1+ |v|+|x|)^{20}c(v) f_2(t,x,v) \|_{  L^2_x L^2_v} +  |t|^{-3} \| c(v)\big( \widehat{f_2}(t,0,v ) -\nabla_v\cdot f_3(t, v)\big)\|_{L^2_v}   \big]
 \]
 \be\label{bilineardensity}
  \times  \|(1+ |v|+|x|)^{20}  \nabla_v^\alpha f_1(t,x,v)\|_{L^2_xL^2_v}, \quad \textit{if\,\,} k\in \mathbb{Z}, |t|^{-1}\lesssim  2^{k}\leq 1.
 \ee
Alternatively,  the following rough bilinear estimate holds for any $k\in \mathbb{Z}$,
 \[
\| B_k (f_1,f_2)(t,x  ,v)\|_{L^2_x L^2_v}  \lesssim \sum_{|\alpha|\leq 5}\min\{ |t|^{-3}   ,2^{3k}\}   \|m(\xi,v)\|_{L^\infty_v \mathcal{S}^\infty_k}    
 \|(1+ |v|+|x|)^{20}c(v) f_2(t,x,v) \|_{  L^2_x L^2_v}   \]
\be\label{bilineardensitylargek}
 \times \|(1+ |v|+|x|)^{20} \nabla_v^\alpha f_1(t,x,v)\|_{L^2_xL^2_v} .
\ee

 \end{lemma}
  \begin{proof}
  See \cite{wang3}[Lemma 3.2\& Lemma 3.3].
  \end{proof}
  With the above bilinear estimates, we are ready to estimate the high order energy of the \textit{non-bulk} terms. 

\begin{lemma}\label{fixedtimeestimate1}
Under the bootstrap assumption \textup{(\ref{bootstrap})}, the following estimate holds for any $t\in [1,T]$, 
\be\label{jan15eqn21}
\sum_{\alpha\in \mathcal{B}, \beta\in \mathcal{S}, |\alpha|+|\beta| = N_0} \|\omega_{\beta}^\alpha(t,x,v)\big(\textit{h.o.t}_\beta^\alpha(t,x,v)- \textit{bulk}_\beta^\alpha(t,x,v)  \big)\|_{L^2_x L^2_v}\lesssim (1+|t|)^{-1+\delta}\epsilon_1^2,\ee
\be\label{may22eqn12}
 \sum_{\alpha\in \mathcal{B}, \beta\in \mathcal{S}, |\alpha|+|\beta| < N_0}\|\omega_{\beta}^\alpha(t,x,v)\big(\textit{h.o.t}_\beta^\alpha(t,x,v)- \textit{bulk}_\beta^\alpha(t,x,v)  \big)\|_{L^2_x L^2_v}\lesssim (1+|t|)^{-1+\delta/2}\epsilon_1^2 .
\ee
\end{lemma}
\begin{proof}
 Recall the decompositions of $ \textit{h.o.t}_\beta^\alpha(t,x,v)$ in (\ref{sepeqn180}), (\ref{sepeq90}), and (\ref{nove100}). We have
\be\label{jan15eqn24}
\textit{h.o.t}_\beta^\alpha(t,x,v)- \textit{bulk}_\beta^\alpha(t,x,v) = \sum_{i=2,3}\textit{h.o.t}_{\beta;i}^\alpha(t,x,v)+  \textit{h.o.t}_{\beta;1}^{\alpha;1}(t,x,v) + \textit{error}_{\beta }^{\alpha }(t,x,v).
\ee
 Motivated from the above equality, we separate into three cases as follows. 

\noindent $\bullet$\quad The estimate of $\textit{h.o.t}_{\beta;2}^\alpha(t,x,v) $ and $\textit{h.o.t}_{\beta;3}^\alpha(t,x,v) $. 

 Recall (\ref{sepeqn311}) and (\ref{sepeqn310}). Moreover,  recall the first decomposition of $D_v$ in (\ref{summaryoftwodecomposition}) in Lemma \ref{twodecompositionlemma} , the detailed formula of $d_{\rho}(t,x,v)$ in (\ref{sepeq947}), and the detailed formula of $Y_i^\beta$ in (\ref{sepeqn522}). From the estimate of coefficients in (\ref{jan15eqn2}), (\ref{sepeqn904}),  and (\ref{march25eqn10}), the second part of the estimate (\ref{feb8eqn51}) in Lemma \ref{derivativeofweightfunction},  and the decay estimate (\ref{noveqn78}) in Lemma \ref{sharpdecaywithderivatives}, the following estimate holds from the $L^2_{x,v}-L^\infty_{x,v}$ type bilinear estimate, 
 \[
\sum_{i=2,3}\sum_{\alpha\in \mathcal{B}, \beta\in \mathcal{S}, |\alpha|+|\beta| = N_0} \|\omega_{\beta}^\alpha(t,x,v) \textit{h.o.t}_{\beta;i}^\alpha(t,x,v) \|_{L^2_{x,v}}  \lesssim  \sum_{\gamma\in \mathcal{B}, \kappa\in \mathcal{S}, |\gamma|+|\kappa|\leq N_0 } \sum_{\rho\in\mathcal{B}, |\rho|\leq 3,  u\in\{E^\rho, B^\rho\}}\| \omega_{\kappa}^\gamma(t,x,v)  g_{\kappa}^\gamma(t,x,v)\|_{L^2_{x,v}}
 \]
 \[
\times   \| (1+||t|-|x+\hat{v}t||) u(t,x+\hat{v}t)\|_{L^\infty_{x,v}}\lesssim (1+|t|)^{-1} E_{\textup{high}}^{f}(t)E_{\textup{low}}^{eb}(t)\lesssim (1+|t|)^{-1+\delta}\epsilon_1^2.
\]

\noindent $\bullet$\quad The estimate of $\textit{h.o.t}_{\beta;1}^{\alpha;1}(t,x,v) $.

Recall (\ref{jan15eqn60}). For this   term, we use the first decomposition of ``$D_v$'' (\ref{summaryoftwodecomposition}) in Lemma \ref{twodecompositionlemma}. Recall the detailed formula of $d_{\rho}(t,x,v)$ in (\ref{sepeq947}).  From the equality (\ref{sepeqn610}), the estimate of coefficients in (\ref{march13eqn10}) and (\ref{jan15eqn2}), the second part of the estimate (\ref{feb8eqn51}) in Lemma \ref{derivativeofweightfunction},  and the decay estimate (\ref{noveqn78}) in Lemma \ref{sharpdecaywithderivatives}, the following estimate holds from the $L^2_{x,v}-L^\infty_{x,v}$ type bilinear estimate, 
 \[
\sum_{\alpha\in \mathcal{B}, \beta\in \mathcal{S}, |\alpha|+|\beta| = N_0} \|\omega_{\beta}^\alpha(t,x,v) \textit{h.o.t}_{\beta;1}^{\alpha;1}(t,x,v) \|_{L^2_{x,v}}  \lesssim \sum_{\gamma\in \mathcal{B}, \kappa\in \mathcal{S}, |\gamma|+|\kappa|\leq N_0 } \sum_{\rho\in\mathcal{B}, |\rho|\leq 3,  u\in\{E^\rho, B^\rho\}} \| \omega_{\kappa}^\gamma(t,x,v)  g_{\kappa}^\gamma(t,x,v)\|_{L^2_{x,v}}
 \]
\be
\times   \| (1+||t|-|x+\hat{v}t||) u(t,x+\hat{v}t)\|_{L^\infty_{x,v}}\lesssim (1+|t|)^{-1} E_{\textup{high}}^{f}(t)E_{\textup{low}}^{eb}(t)\lesssim   (1+|t|)^{-1+\delta}\epsilon_1^2.
\ee

\noindent $\bullet$\quad The estimate of $\textit{error}_{\beta }^{\alpha }(t,x,v) $. 

Recall (\ref{nove121}) and (\ref{nove98}). We use the first decomposition of ``$D_v$'' (\ref{summaryoftwodecomposition}) in Lemma \ref{twodecompositionlemma}.  Recall the detailed formula of $d_{\rho}(t,x,v)$ in (\ref{sepeq947}). From the estimate of coefficients (\ref{jan15eqn2}), the second part of the estimate (\ref{feb8eqn51}) in Lemma \ref{derivativeofweightfunction}, and the decay estimate (\ref{noveqn78}) in Lemma \ref{sharpdecaywithderivatives},   the following estimate holds from the $L^2_{x,v}-L^\infty_{x,v}$ type bilinear estimate, 
\[
 \sum_{\alpha\in \mathcal{B}, \beta\in \mathcal{S}, |\alpha|+|\beta| = N_0} \|\omega_{\beta}^\alpha(t,x,v)\textit{error}_{\beta }^{\alpha }(t,x,v)\|_{L^2_{x,v}} \lesssim  \sum_{\gamma\in \mathcal{B}, \kappa\in \mathcal{S}, |\gamma|+|\kappa|\leq N_0 } \sum_{\rho\in\mathcal{B}, |\rho|\leq 3,  u\in\{E^\rho, B^\rho\}}\| \omega_{\kappa}^\gamma(t,x,v)  g_{\kappa}^\gamma(t,x,v)\|_{L^2_{x,v}}
\]
\be
\times   \| (1+||t|-|x+\hat{v}t||)u(t,x+\hat{v}t)\|_{L^\infty_{x,v}}\lesssim (1+|t|)^{-1} E_{\textup{high}}^{f}(t)E_{\textup{low}}^{eb}(t) \lesssim   (1+|t|)^{-1+\delta}\epsilon_1^2.
\ee
Hence finishing the proof of the desired estimate (\ref{jan15eqn21}).

 With minor modifications, our desired estimate (\ref{may22eqn12})   holds after redoing the above argument for fixed $\alpha\in \mathcal{B}, \beta\in \mathcal{S},$ s.t.,  $ |\alpha|+|\beta| < N_0$.
\end{proof}

\begin{lemma}\label{fixedtimeestimate2}
 Under the bootstrap assumption \textup{(\ref{bootstrap})}, the following estimate holds for any $t\in [1,T]$, 
\be\label{jan15eqn91}
\sum_{  |\alpha|+|\beta| = N_0}\|\omega_{\beta}^\alpha(t,x,v)  \textit{l.o.t}_\beta^\alpha(t,x,v)  \|_{L^2_x L^2_v}   \lesssim  (1+|t|)^{-1+\delta}\epsilon_1^2,
\ee 
\be\label{may22eqn200}
 \sum_{   |\alpha|+|\beta| < N_0}\|\omega_{\beta}^\alpha(t,x,v)  \textit{l.o.t}_\beta^\alpha(t,x,v)  \|_{L^2_x L^2_v}   \lesssim  (1+|t|)^{-1+\delta/2}\epsilon_1^2.
\ee
\end{lemma}
\begin{proof}
Recall (\ref{loworderterm}). Based on  the  total number derivatives act on the electromagnetic field,  we separate into two cases as follows. 

\noindent $\bullet$\quad The estimate of $ \textit{l.o.t}_{\beta;i}^\alpha(t,x,v),$ $i\in \{1,2,4\}$.

Recall (\ref{sepeqn400}), (\ref{sepeq200}), and (\ref{sepeq201}). Note that there are at most twelve derivatives hit on the electromagnetic field. Recall the commutation rule between $\Lambda^\beta$ and $X_i$ in  (\ref{noveq521}) and the equality (\ref{sepeqn610}). From the estimate of coefficients in (\ref{sepeqn524}), (\ref{sepeqn904}), and  (\ref{sepeqn88}), the following estimate holds from the linear decay estimate (\ref{noveqn78}) in Lemma \ref{sharpdecaywithderivatives} and the $L^2_{x,v}-L^\infty_{x,v}$ type bilinear estimate, 
\[
\sum_{i=1,2,4} \sum_{   |\alpha|+|\beta| =  N_0}\|\omega_{\beta}^\alpha(t,x,v)  \textit{l.o.t}_{\beta;i}^\alpha(t,x,v)  \|_{L^2_x L^2_v}\lesssim  \sum_{
\begin{subarray}{c}  
|\gamma|+|\kappa|\leq N_0, |\rho|\leq 12,
\rho, \gamma\in \mathcal{B}, \kappa\in \mathcal{S},  u\in\{E^\rho, B^\rho\}\\
\end{subarray}}  \| \omega_{\kappa}^\gamma(t,x,v)  g_{\kappa}^\gamma(t,x,v)\|_{L^2_{x,v}} \]
\be\label{jan16eqn71}
 \times \| (1+|\tilde{d}(t,x,v)|) u(t,x+\hat{v}t)\|_{L^\infty_{x,v}}  \lesssim (1+|t|)^{-1} E_{\textup{high}}^{f}(t)E_{\textup{low}}^{eb}(t)\lesssim (1+t)^{-1+\delta}\epsilon_1^2.
\ee

\noindent $\bullet$\quad The estimate of $ \textit{l.o.t}_{\beta;3}^\alpha(t,x,v)$.

Recall (\ref{sepeqn600}), (\ref{sepeqn334}), and (\ref{sepeqn335}). From the equality (\ref{sepeqn610}) in Lemma \ref{decompositionofderivatives},  the following equality holds,
 \be\label{jan15eqn81}
 \textit{l.o.t}_{\beta;3 }^\alpha(t,x,v)=  \sum_{\begin{subarray}{c}
\rho, \gamma\in \mathcal{B}, |\rho|+|\gamma|\leq |\alpha|, |\iota'|\leq |\iota|\\
\iota+\kappa=\beta, |\iota|, |\kappa|>0, \iota, \kappa\in \mathcal{S}\\
i=1,\cdots,7,|\iota|+|\rho|\geq 12	\\
\end{subarray}}   \big(\widehat{\alpha}_{\alpha;\rho, \gamma}^{\iota,i;\iota',1}(x,v)E^{\rho+\iota'}(t,x+\hat{v} t) + \widehat{\alpha}_{\alpha;\rho, \gamma}^{\iota,i;\iota',2}(x,v)B^{\rho+\iota'}(t,x+\hat{v} t)\big)  \Lambda^{\kappa}\big( X_i g^{\gamma}(t, x,v)\big),
 \ee
where  ``$\widehat{\alpha}_{\alpha;\rho, \gamma}^{\iota,i;\iota',1}(x,v)$'' and $\widehat{\alpha}_{\alpha;\rho, \gamma}^{\iota,i;\iota',2}(x,v)$ are some determined coefficients, whose explicit formulas are not pursued here. 
From the estimate of coefficients in (\ref{dec28eqn41}) and (\ref{sepeqn88}), the following rough estimate of coefficients holds, 
\be\label{march13eqn20}
|\widehat{\alpha}_{\alpha;\rho, \gamma}^{\iota,i;\iota',1}(x,v)| +|\widehat{\alpha}_{\alpha;\rho, \gamma}^{\iota,i;\iota',2}(x,v)|\lesssim (1+|x|^2+|v|^2)^{2|\iota|}.
\ee

 From the equalities (\ref{noveq521}) and  (\ref{sepeqn522}) in Lemma \ref{summaryofhighordercommutation} and the first decomposition of $D_v$ in (\ref{summaryoftwodecomposition}) in Lemma \ref{twodecompositionlemma}, we have
\[
\Lambda^{\kappa}\big(X_i g^{\gamma}(t, x,v)\big)= \big[ \alpha_i(v)\cdot D_v
\circ \Lambda^{\kappa} + [\Lambda^\kappa, X_i] \big] g^\gamma(t,x,v)= \sum_{\rho\in \mathcal{K}, |\rho|=1} \alpha_i(v)\cdot d_{\rho}(t, x, v) \Lambda^{\rho\circ \kappa} g^\gamma(t,x,v) \]
\be\label{noveq102}
+ Y_i^\kappa g^\gamma(t,x,v) + \sum_{\kappa'\in \mathcal{S}, |\kappa'|\leq |\kappa|-1} \big[ \tilde{d}(t,x,v)\tilde{e}_{\kappa,i}^{\kappa', 1}(x,v) +\tilde{e}_{\kappa,i}^{\kappa', 2}(x,v)\big]  \Lambda^{\kappa'}g^{\gamma}(t, x,v). 
\ee
 From (\ref{jan15eqn81}) and  (\ref{noveq102}), and the detailed formula of $d_{\rho}(t,x,v)$ in (\ref{sepeq947}),  we can rewrite ``$\textit{l.o.t}_{\beta;3}^\alpha(t, x, v)$'' as follows
\be\label{noveq130}
\textit{l.o.t}_{\beta;3}^\alpha(t, x, v)= \sum_{
\begin{subarray}{c}
\rho \in\mathcal{S}, \kappa_1, \kappa_2\in \mathcal{B}, u\in\{E, B\},   |\rho|\leq |\beta| \\
|\rho|+|\kappa_1|+|\kappa_2|\leq |\alpha|+|\beta|, |\kappa_2|\leq |\alpha|\\
|\rho|+|\kappa_2|\leq  |\alpha|+|\beta|-12\\
\end{subarray}}\big[\big(\tilde{d}(t,x,v)\widehat{e}^{u;1}_{\kappa_1, \kappa_2,\rho}(t,x,v)  +\widehat{e}^{u;2}_{\kappa_1, \kappa_2,\rho}(t,x,v)   \big) u^{\kappa_1}(t, x+\hat{v}t) g_{\rho}^{\kappa_2}(t,x,v), 
\ee
where the coefficients $\widehat{e}^{u;i}_{\kappa_1, \kappa_2,\rho}(t,x,v), i\in\{1,2\} $, satisfy the following estimate for any $i\in\{1,2\}$, and any $u\in\{E,B\}$,
\be\label{noveq220}
  | \widehat{e}^{u;i}_{\kappa_1, \kappa_2,\rho}(t,x,v) |\lesssim (1+|x|^2+|v|^2)^{|\alpha|+2|\beta|-2|\rho|-|\kappa_2|+10},
\ee
which can be derived from   the estimate (\ref{march13eqn20})  and the estimates (\ref{sepeqn524}) and (\ref{sepeqn904})	 in Lemma \ref{summaryofhighordercommutation}. 

The main difficulty of estimating $\textit{l.o.t}_{\beta;3}^\alpha(t,x,v)$ is that we cannot use the decay of the electromagnetic field or trade regularities for the inhomogeneous modulation because the electromagnetic field can have the maximal number of vector fields. Moreover, the loss caused by the coefficient is possible of  size ``$1+|t|$'', e.g., when $x, v\sim 1$. To get around this issue, we exploit the smallness of space-resonance set by using   the estimate   (\ref{multilinear1jan16}) in Lemma \ref{multilinearlemma1}, which allows us to gain   extra decay rate over time.

 Recall (\ref{noveq130}). We first do dyadic decomposition for the electromagnetic field. As a result, we have
\be\label{jan16eqn67}
\textit{l.o.t}_{\beta;3 }^\alpha(t,x,v)=\sum_{k\in \mathbb{Z}} H_k(t,x,v),
\ee
where
\[
 H_k(t,x,v):= \sum_{
\begin{subarray}{c}
\rho \in\mathcal{S}, \kappa_1, \kappa_2\in \mathcal{B}, u\in\{E, B\},   |\rho|\leq |\beta| \\
|\rho|+|\kappa_1|+|\kappa_2|\leq |\alpha|+|\beta|, |\kappa_2|\leq |\alpha|\\
|\rho|+|\kappa_2|\leq  |\alpha|+|\beta|-12\\
\end{subarray}}\big[\big(\tilde{d}(t,x,v)\widehat{e}^{u;1}_{\kappa_1, \kappa_2,\rho}(t,x,v)  +\widehat{e}^{u;2}_{\kappa_1, \kappa_2,\rho}(t,x,v)   \big) u_k^{\kappa_1}(t, x+\hat{v}t) g_{\rho}^{\kappa_2}(t,x,v). 
\]
Based on the possible size of $k$, we separate into two cases as follows. 

\noindent $\bullet$ \quad If $k\leq 0$. 

Recall the equalities (\ref{octe631}) and (\ref{octe633}). From the estimate of modified profiles in (\ref{octe741}), the estimate of correction terms $ \widetilde{g}_{\alpha, \gamma}(t, v)$ in (\ref{may22eqn100}),  and the estimate (\ref{may17eqn12})  in Lemma \ref{Alinearestimate}, the following estimate holds after using the volume of support of $\xi$,
\[
\sum_{\alpha \in \mathcal{B}, |\alpha|\leq N_0}\sum_{i=1,2} \| \widehat{h_i^\alpha}(t, \xi)\psi_k(\xi)\|_{L^2_\xi}\lesssim 2^{k/2}\epsilon_1 + |t|2^{3k/2}\epsilon_1.
\] 
  From   the estimate of coefficients in (\ref{noveq220}),  after  using the $L^2_x-L^\infty_x L^2_v$ type estimate, the volume of the frequency support of the electromagnetic field, and the decay estimate (\ref{densitydecay}) in Lemma \ref{decayestimateofdensity}, the following estimate holds  if $2^{k}\leq |t|^{-1}$, 
\[
\|\omega_{\beta}^\alpha(t,x,v)H_k(t,x,v) \|_{L^2_xL^2_v}\lesssim \sum_{|\rho|+|\kappa|\leq N_0- 5} 
  |t|^{-1/2}\big(2^{k/2}+ |t|2^{3k/2}  \big) \epsilon_1    \| {\omega_{\rho}^\kappa}(x-\hat{v}t ,v) g_{\rho}^\kappa(t,x-\hat{v}t,v)\|_{L^2_x L^2_v} 
	\]
  \be\label{jan16eqn61}
\lesssim  |t|^{-1/2+\delta/2}   \big(2^{k/2}+ |t|2^{3k/2}  \big) \epsilon_0.
\ee
It remains to consider  the case   $|t|^{-1}\leq 2^{k}\leq 1$.  From the decomposition (\ref{noveq432}),  the following decomposition holds for $H_k$,
\be\label{jan16eqn69}
H_k(t,x,v)= H_k^1(t,x,v) + H_k^2(t,x,v), 
\ee
where
\be\label{feb19eqn2}
 H_k^1(t,x,v) := \sum_{
\begin{subarray}{c}
\rho \in\mathcal{S}, \kappa_1, \kappa_2\in \mathcal{B}, u\in\{E, B\},   |\rho|\leq |\beta| \\
|\rho|+|\kappa_1|+|\kappa_2|\leq |\alpha|+|\beta|, |\kappa_2|\leq |\alpha|\\
|\rho|+|\kappa_2|\leq  |\alpha|+|\beta|-12\\
\end{subarray}} \big[\big(\tilde{d}(t,x,v)\widehat{e}^{u;1}_{\kappa_1, \kappa_2,\rho}(t,x,v)  +\widehat{e}^{u;2}_{\kappa_1, \kappa_2,\rho}(t,x,v)   \big) \widetilde{u_k^{\kappa_1}}(t, x+\hat{v}t)g_{\rho}^{\kappa_2}(t,x,v),
\ee
\be\label{noveq291}
 H_k^2(t,x,v) :=  \sum_{
\begin{subarray}{c}
\rho \in\mathcal{S}, \kappa_1, \kappa_2, \eta \in \mathcal{B}, u\in\{E, B\},|\eta|\leq |\kappa_1|   \\
|\rho|+|\kappa_1|+|\kappa_2|\leq |\alpha|+|\beta|, |\kappa_2|\leq |\alpha|\\
|\rho|+|\kappa_2|\leq  |\alpha|+|\beta|-12, |\rho|\leq |\beta| \\
\end{subarray}} -\big(\tilde{d}(t,x,v)\widehat{e}^1_{\kappa_1, \kappa_2,\rho}(t,x,v)  +\widehat{e}^2_{\kappa_1, \kappa_2,\rho}(t,x,v)   \big) \textup{Im}[E^{u}_{\kappa_1;\eta}(g^{\eta}_k)(t, x+\hat{v}t)] g_{\rho}^{\kappa_2}(t,x,v).
\ee
Note that the following estimate holds from the estimate of coefficients in (\ref{noveq220}),
\be\label{jan16eqn131}
 \|\omega_{\beta}^\alpha(t,x,v)H_k^1(t,x,v) \|_{L^2_xL^2_v}^2 \lesssim \sum_{\begin{subarray}{c}
\kappa, \gamma\in \mathcal{B},|\kappa|\leq N_0,
  |\rho|+|\gamma|\leq N_0-12,
 |\gamma|\leq |\alpha|, u\in \{E,B\} 
 \end{subarray}}   (1+t)^2 \int_{\R^3} \int_{\R^3} |\widetilde{u_k^\kappa}(t,x+\hat{v}t) |^2 G_{\rho}^{\gamma}(t,x,v) d x d v,
\ee
where
\[
 G_{\rho}^{\gamma}(t,x,v):=|\omega_{\rho}^\gamma(t,x,v) g_{\rho}^\gamma(t,x,v)|^2 (1+|x|^2+|v|^2)^{2|\alpha|+4|\beta|-4|\rho|-2|\gamma|+20}.
\]
Recall (\ref{noveq800}). From the estimate (\ref{jan16eqn131}),  the  multilinear estimate (\ref{multilinear1jan16}) in Lemma \ref{multilinearlemma1},  the estimates of modified profiles in (\ref{octe741}) and (\ref{octe742}), the hierarchy between the different order of weight functions and the Sobolev embedding in $v$, we have
\[
 \|\omega_{\beta}^\alpha(t,x,v)H_k^1(t,x,v) \|_{L^2_xL^2_v}\lesssim \sum_{|\rho|+|\kappa|\leq N_0- 8} 2^{-k/2} (1+t)^{-3/2} \| (1+|x|^2+|v|^2)^{|\alpha|+2|\beta|-2|\rho|-|\gamma|+30}  {\omega_{\beta}^\alpha}(t,x  ,v) g_{\rho}^\kappa(t,x ,v)\|_{L^2_x L^2_v}
\]
\be\label{jan16eqn62}
\times\big( \epsilon_0 + (1+t)^{\delta} 2^{k_{-}}\epsilon_0\big)   \lesssim   2^{-k/2} (1+t)^{-3/2+\delta/2}( 1+ (1+t)^{\delta} 2^{k_{-}})\epsilon_0.
\ee

Recall (\ref{noveq291}), (\ref{noveq432}), (\ref{noveq510}), and (\ref{noveq511}). Note that the terms inside $H_k^2(t,x,v)$ have the same structure as the bilinear form that we will define in (\ref{noveq260}). From the estimate of coefficients in (\ref{noveq220}), the estimate of correction terms $ \widetilde{g}_{\alpha,\gamma}(t, v)$ in (\ref{may22eqn100}),  and the bilinear estimate (\ref{bilineardensity}) in Lemma \ref{bilineardensitylemma}, we know  that the following estimate holds for any $k\in \mathbb{Z},$ s.t., $|t|^{-1}\leq 2^k \leq 1,$
\[
 \|\omega_{\beta}^\alpha(t,x,v)H_k^2(t,x,v) \|_{L^2_xL^2_v} \lesssim  \sum_{|\rho|+|\kappa|\leq N_0- 5} \big(|t|^{-1}+ |t|^{-2}2^{-k} +|t|^{-2+\delta}   \big)\epsilon_1\| {\omega_{\rho}^\kappa}(x  ,v) g_{\rho}^\kappa(t,x ,v)\|_{L^2_x L^2_v} 
\]
\be\label{jan16eqn64}
   \lesssim   \big(|t|^{-1+\delta/2}+ |t|^{-2+\delta/2}2^{-k} +|t|^{-2+2\delta}   \big)\epsilon_0.
\ee
To sum up, from the decompositions (\ref{jan16eqn67}) and (\ref{jan16eqn69}) and the  estimates (\ref{jan16eqn61}), (\ref{jan16eqn62}), and (\ref{jan16eqn64}), we have
\[
\sum_{|\alpha|+|\beta|\leq N_0}\sum_{k\in\mathbb{Z}, k\leq 0} \|\omega_{\beta}^\alpha(t,x,v)H_k(t,x,v) \|_{L^2_xL^2_v}  \lesssim \sum_{   2^{k}\leq  |t|^{-1 }}    |t|^{-1/2+\delta/2}   \big(2^{k/2}+ |t|2^{3k/2}  \big) \epsilon_0+  \sum_{  |t|^{-1 }\leq 2^k\leq 1}     \big(|t|^{-1+\delta/2}
\]
 \be\label{jan16eqn72}
+ |t|^{-3/2+\delta/2}2^{-k/2}+ |t|^{-2+\delta/2}2^{-k} +|t|^{-3/2+2\delta}2^{k/2} +|t|^{-2+2\delta} \big)\epsilon_0\lesssim  (1+t)^{-1+\delta/2}\log(1+t)\epsilon_0 .
 \ee

 \noindent $\bullet$\quad If $k\geq 0$. 

 From the estimate of coefficients in (\ref{noveq220})  and the bilinear estimate (\ref{bilineardensitylargek}) in Lemma \ref{bilineardensitylemma},  we have 
\be\label{feb12eqn161}
\sum_{k\geq 0, k\in \mathbb{Z}} \|\omega_{\beta}^\alpha(t,x,v)H_k^2(t,x,v) \|_{L^2_xL^2_v} \lesssim (1+t)^{-2 } \big(E_{\textup{high}}^{f}(t)     \big)^2\lesssim (1+t)^{-2 +2\delta}\epsilon_0.
\ee
Now, it remains to estimate ``$H_k^1(t,x,v)$''. Recall    (\ref{feb19eqn2}), we have
 \be\label{feb12eqn201}
 \|\sum_{k\in \mathbb{Z}, k\geq 0}\omega_{\beta}^\alpha(t,x,v)H_k^1(t,x,v) \|_{L^2_xL^2_v}^2 \lesssim   \sum_{
\begin{subarray}{c}
  k_1, k_2\in \mathbb{Z},   k_1, k_2\geq 0 \\
\end{subarray} 
  } (1+t)^2  K_{k_1,k_2} , 
  \ee
  where
\be\label{feb19eqn12}
  K_{k_1,k_2}:=\sum_{\begin{subarray}{c}
\kappa, \gamma\in \mathcal{B},|\kappa|\leq N_0,u_1, u_2\in\{E, B\}\\ 
  |\rho|+|\gamma|\leq N_0-12,  |\gamma|\leq |\alpha|\\
 \end{subarray}}  \Big| \int_{\R^3} \int_{\R^3}  G_{\rho;u_1}^{\gamma;u_2}(t,x,v)      (  \widetilde{u_1^\kappa})_{k_1}(t,x+\hat{v}t) (  \widetilde{u_2^\kappa})_{k_2}(t,x+\hat{v}t)  d x d v \Big|,
\ee
where $ G_{\rho;u_1}^{\gamma;u_2}(t,x,v)$, $u_1, u_2\in\{E, B\}$, are some determined function that satisfies the following estimate, 
\be\label{feb13eqn1}
\sum_{ u_1, u_2\in\{E, B\}}\sum_{\iota\in \mathcal{S}, |\iota|\leq 5}|\Lambda^\iota\big(G_{\rho;u_1}^{\gamma;u_2}(t,x,v)\big)|\lesssim\sum_{\iota\in\mathcal{S}, |\iota|\leq 5} |\omega_{\beta}^\alpha (t,x,v) g_{\iota\circ\rho}^\gamma(t,x,v)|^2 (1+|x|^2+|v|^2)^{2|\alpha|+4|\beta|-4|\rho|-2|\gamma|+30}. 
\ee

We first consider the case when $|k_1-k_2|\geq 10$. Recall (\ref{feb19eqn12}).  From the above estimate (\ref{feb13eqn1}),  the trilinear estimate (\ref{trilinearfeb20est1}) in Lemma \ref{multilinearlemma1} and the Sobolev embedding in ``$v$'', we know that the following  estimate holds, 
\[
\sum_{k_1,k_2\in \mathbb{Z}, k_1,k_2\geq 0,  |k_1-k_2|\geq 10 } | K_{k_1,k_2}|\lesssim \sum_{k_1,k_2\in \mathbb{Z}, k_1,k_2\geq 0,  |k_1-k_2|\geq 10 }\sum_{|\rho|+|\gamma|\leq N_0-12}\sum_{ u_1, u_2\in\{E, B\}}\sum_{|\alpha|\leq 4} 2^{-\max\{k_1,k_2\}}   \]
\[
\times (1+|t|)^{-5}  \big(E_{\textup{high}}^{eb}(t)\big)^2 \|(1+|x|^2)(1+|v|^{25}) \nabla_x \nabla_{v}^\alpha   G_{\rho; u_1}^{\gamma;u_2}(t,x,v)  \|_{L^1_{x,v}}
\]
\be\label{feb13eqn9} 
 \lesssim \sum_{\rho\in \mathcal{S}, \gamma\in \mathcal{B}, |\alpha|+|\rho|\leq N_0-5}   (1+|t|)^{-5}  \|  {\omega_{\rho}^\gamma}(x,v) g_{\rho}^\gamma(t,x,v) \|_{L^2_{x}L^2_v}^2  \big(E_{\textup{high}}^{eb}(t)\big)^2\lesssim   (1+|t|)^{-5+4\delta} \epsilon_0^2.  
\ee
Lastly, we consider the case when $|k_1-k_2|\leq 10$. Recall (\ref{feb19eqn12}). Again, from   the   estimate (\ref{feb13eqn1}),  the trilinear estimate (\ref{multilinear1jan16}) in Lemma \ref{multilinearlemma1}, the Cauchy-Schwarz inequality,  and the Sobolev embedding in ``$v$'', we know that the following  estimate holds, 
\[
\sum_{ k_1,k_2\in \mathbb{Z}, k_1,k_2\geq 0, |k_1-k_2|\leq 10} |K_{k_1,k_2}| \lesssim \sum_{
\begin{subarray}{c}
k_1,k_2\geq 0, |k_1-k_2|\leq 10 \\
 |\alpha|\leq 4, u_1, u_2\in\{E, B\}\\
 i=1,2,3,|\rho|+|\gamma|\leq N_0-12\\
\end{subarray}
} (1+|t|)^{-5} \|(1+|x|^2)(1+|v|^{25})  \nabla_{v}^\alpha  G_{\rho; u_1}^{\gamma;u_2}(t,x,v)  \|_{L^1_{x,v}}
\]
\[
\times \big(2^{-k_{1}/2} E_{\textup{high}}^{eb}(t) +\sum_{\iota\in \mathcal{B}, |\iota|\leq N_0} \| \widehat{h^\iota}(t, \xi)\psi_{k_1}(\xi)\|_{L^2} )   \big(2^{-k_{2}/2} E_{\textup{high}}^{eb}(t) +\sum_{\iota\in \mathcal{B}, |\iota|\leq N_0} \| \widehat{h^\iota}(t, \xi)\psi_{k_2}(\xi)\|_{L^2} ) 
\]
\be\label{feb13eqn10}
  \lesssim \sum_{\rho\in \mathcal{S}, \gamma\in \mathcal{B}, |\alpha|+|\rho|\leq N_0-5}   (1+|t|)^{-5}  \|  {\omega_{\rho}^\gamma}(x,v) g_{\rho}^\gamma(t,x,v) \|_{L^2_{x}L^2_v}^2  \big(E_{\textup{high}}^{eb}(t)\big)^2\lesssim   (1+|t|)^{-5+4\delta} \epsilon_0^2.
\ee
From   the estimates (\ref{feb12eqn201}), (\ref{feb13eqn9}), and (\ref{feb13eqn10}), it is easy to see that the following estimate holds, 
\be\label{feb13eqn19}
   \|\sum_{k\in \mathbb{Z}, k\geq 0}\omega_{\beta}^\alpha(t,x,v)H_k^1(t,x,v) \|_{L^2_xL^2_v} \lesssim  (1+|t|)^{-3/2+2\delta} \epsilon_0. 
\ee
Recall the decompositions (\ref{jan16eqn67}) and (\ref{jan16eqn69}). From the estimates (\ref{jan16eqn72}), (\ref{feb12eqn161}), and (\ref{feb13eqn19}), we have 
\be\label{feb13eqn24}
\sum_{\alpha\in\mathcal{B}, \beta\in \mathcal{S}, |\alpha|+|\beta|\leq N_0} \| \omega_{\beta}^\alpha(t,x,v) \textit{l.o.t}_{\beta;3}^\alpha(t,x,v) \|_{L^2_{x,v}}\lesssim (1+t)^{-1+\delta/2}\log(1+t)\epsilon_0 .
\ee
Recall the decomposition in  (\ref{loworderterm}). Our desired estimate  (\ref{jan15eqn91}) holds from the estimates (\ref{jan16eqn71}) and (\ref{feb13eqn24}).

Recall (\ref{correctionterm}). Since the correction term $\widetilde{g}_{\alpha,\gamma}(t, v)$, which contributes the logarithmic growth in the estimate (\ref{jan16eqn72}),  equals zero if $|\alpha|+|\gamma|< N_0$, with minor modifications in the above argument, the desired estimate (\ref{may22eqn200}) holds   similarly.

\end{proof}

 \begin{proposition}\label{notbulkterm2}
 Under the bootstrap assumption \textup{(\ref{bootstrap})}, the following estimate holds for any $t\in [1,T]$, 
 \be\label{nov102}
 \sum_{\alpha\in \mathcal{B}, \beta\in \mathcal{S}, |\alpha|+|\beta|=  N_0}|I_{\beta;2}^\alpha(t) |+ |I_{\beta;3}^\alpha(t) |  \lesssim (1+t)^{2\delta} \epsilon_0,
 \ee
  \be\label{may22eqn41}
 \sum_{\alpha\in \mathcal{B}, \beta\in \mathcal{S}, |\alpha|+|\beta| <  N_0}|I_{\beta;2}^\alpha(t) |+ |I_{\beta;3}^\alpha(t) |  \lesssim (1+t)^{\delta } \epsilon_0.
 \ee
 \end{proposition} 
 \begin{proof}
 Recall (\ref{sepeqn813}) and (\ref{sepeqn810}). From the estimate (\ref{jan15eqn21}) in Lemma \ref{fixedtimeestimate1} and the estimate (\ref{jan15eqn91}) in Lemma \ref{fixedtimeestimate2}, we know  that the following estimate holds from the $L^2_{x,v}-L^2_{x,v}$ type estimate,
 \[
\sum_{\alpha\in \mathcal{B}, \beta\in \mathcal{S}, |\alpha|+|\beta|= N_0}|I_{\beta;2}^\alpha |+ |I_{\beta;3}^\alpha |\lesssim \sum_{\alpha\in \mathcal{B}, \beta\in \mathcal{S}, |\alpha|+|\beta|=  N_0}  \int_{1}^{t} \| \omega_{\beta}^{\alpha}(s, x, v)g^\alpha_\beta (s ,x,v)\|_{L^2_x L^2_v}  \big[ \|\omega_{\beta}^\alpha(s,x,v)  \textit{l.o.t}_\beta^\alpha(s,x,v)  \|_{L^2_x L^2_v} 
 \]
 \[
  +  \|\omega_{\beta}^\alpha(s,x,v)\big(\textit{h.o.t}_\beta^\alpha(s,x,v)- \textit{bulk}_\beta^\alpha(s,x,v)  \big)\|_{L^2_x L^2_v} \big] d s \lesssim \int_{1}^{t} (1+s)^{-1+2\delta}\epsilon_0 ds \lesssim (1+t)^{2\delta}\epsilon_0. 
 \]
 Hence finishing the proof of the desired estimate (\ref{nov102}). With minor modifications, the desired estimate (\ref{may22eqn41}) holds similarly from the estimate  (\ref{may22eqn12}) in Lemma  \ref{fixedtimeestimate1} and   the estimate (\ref{may22eqn200}) in Lemma \ref{fixedtimeestimate2}.   
 \end{proof}

As  a natural generalization of the aforementioned methods used in the estimate of \textit{non-bulk terms}, we prove the following two lemmas,  which will be helpful  in the   estimate of the \textit{bulk term} in the next section.  
\begin{lemma}\label{fixtimehighorder2}
  Under the bootstrap assumption \textup{(\ref{bootstrap})}, the following estimate holds for any $t\in [1,T]$, 
\be\label{jan15eqn22}
\sum_{\alpha\in \mathcal{B}, \kappa, \rho\in \mathcal{S}, |\rho|=1,  |\alpha|+|\kappa|\leq N_0-1}\|\omega_{\rho\circ \kappa}^\alpha(t,x,v)\Lambda^{\rho} \big(\textit{h.o.t}_\kappa^\alpha(t,x,v)- \textit{bulk}_\kappa^\alpha(t,x,v)  \big)\|_{L^2_x L^2_v}\lesssim (1+|t|)^{-1+\delta} \epsilon_0,
\ee
\be\label{may22eqn51}
\sum_{\alpha\in \mathcal{B}, \kappa, \rho\in \mathcal{S}, |\rho|=1,  |\alpha|+|\kappa|\leq N_0-2}\|\omega_{\rho\circ \kappa}^\alpha(t,x,v)\Lambda^{\rho} \big(\textit{h.o.t}_\kappa^\alpha(t,x,v)- \textit{bulk}_\kappa^\alpha(t,x,v)  \big)\|_{L^2_x L^2_v}\lesssim (1+|t|)^{-1+\delta/2} \epsilon_0,
\ee 
\end{lemma}
\begin{proof}
Recall the decomposition  (\ref{jan15eqn24}) and  the corresponding detailed formulas in  (\ref{sepeqn311}), (\ref{sepeqn310}),  (\ref{jan15eqn60}), (\ref{nove100}), and (\ref{nove121}). Moreover, we recall the detailed formula of $Y_i^\beta$ in (\ref{sepeqn522}).  Based on the order of derivatives acting  on the profile $g(t,x,v)$, we    decompose     $\Lambda^\rho\big( \textit{h.o.t}_{\kappa;i}^\alpha\big)$,  $i\in \{ 2,3\}$,   $\Lambda^{\rho}\big(\textit{h.o.t}_{\kappa;1}^{\alpha;1}\big)$, and  $\Lambda^{\rho}\big(\textit{error}_{\kappa }^{\alpha }\big)$ as follows, 
\be\label{noveq41}
\Lambda^{\rho} \big(\textit{h.o.t}_{\kappa;i}^\alpha\big)(t,x,v)=\widetilde{ \textit{h.o.t}_{\kappa;i}^{\alpha;\rho } }(t,x,v) + \widetilde{ \textit{l.o.t}_{\kappa;i}^{\alpha;\rho} }(t,x,v),\quad i\in\{ 2,3\},
\ee
\be\label{2020feb15eqn11}
\Lambda^{\rho}\big(\textit{h.o.t}_{\kappa;1}^{\alpha;1}\big)(t,x,v)=\widetilde{ \textit{h.o.t}_{\kappa;1}^{\alpha;\rho} }(t,x,v)+ \widetilde{ \textit{l.o.t}_{\kappa;1}^{\alpha;\rho} }(t,x,v),
\ee
\be\label{2020feb15eqn12}
\Lambda^{\rho}\big(\textit{error}_{\kappa}^\alpha\big)(t,x,v) = \widetilde{\textit{error}_{\kappa;1}^{\alpha; \rho }}(t,x,v) + \widetilde{\textit{error}_{\kappa;2}^{\alpha;\rho }}(t,x,v),  
\ee
where
\be\label{noveq45}
\widetilde{ \textit{h.o.t}_{\kappa;2}^{\alpha;\rho} }(t,x,v)= \sum_{ 
 \begin{subarray}{c}
  |\iota|  \leq 1,|\gamma|=|\alpha|-1\\
  \end{subarray}}\sum_{  i=1,\cdots, 7}   K_{\alpha;\iota, \gamma}^i(t,x,v) \alpha_i(v)\cdot D_v  \Lambda^{ \rho } g_{\kappa}^\gamma(t,x,v),
\ee
\be\label{noveq47}
\widetilde{ \textit{l.o.t}_{\kappa;2}^{\alpha;\rho} }(t,x,v)=   \sum_{ 
 \begin{subarray}{c}
  |\iota|  \leq 1,|\gamma|=|\alpha|-1\\
  \end{subarray}}\sum_{  i=1,\cdots, 7}      \Lambda^{ \rho}\big( K_{\alpha;\iota, \gamma}^i(t,x,v)  \big) X_i  g_{\kappa}^\gamma(t,x,v) +   K_{\alpha;\iota, \gamma}^i(t,x,v)     [\Lambda^\rho, X_i] g_{\kappa}^\gamma(t,x,v) ,
 \ee
\be\label{noveq48}
 \widetilde{ \textit{h.o.t}_{\kappa;3}^{\alpha;\rho} }(t,x,v)=  \sum_{
\begin{subarray}{c}
 \iota\in \mathcal{S}, |\iota|=|\kappa|, 
  |i(\iota)- i(\kappa)|\leq 1\\
  \end{subarray}} \sum_{  	 i=1,\cdots, 7 }   K^i(t,x, v)   \big(\tilde{d}(t,x,v)\tilde{e}_{\kappa,i}^{\iota, 1}(x,v) +\tilde{e}_{\kappa,i}^{\iota, 2}(x,v) \big) \Lambda^{
\rho } g_{\iota}^\alpha(t,x,v),
\ee
\be\label{noveq49}
 \widetilde{ \textit{l.o.t}_{\kappa;3}^{\alpha;\rho} }(t,x,v)=   \sum_{
\begin{subarray}{c}
 \iota\in \mathcal{S}, |\iota|=|\kappa|, 
  |i(\iota)- i(\kappa)|\leq 1\\
  \end{subarray}} \sum_{  	 i=1,\cdots, 7 }    \Lambda^\rho\big[ K^i(t,x, v)   \big( \tilde{d}(t,x,v)\tilde{e}_{\kappa,i}^{\iota, 1}(x,v) +\tilde{e}_{\kappa,i}^{\iota, 2}(x,v)  \big)\big] g_{\iota}^\alpha(t,x,v),
\ee
 \be\label{noveq42}
 \widetilde{ \textit{h.o.t}_{\kappa;1}^{\alpha;\rho} }(t,x,v)=  \sum_{\begin{subarray}{l}
j=1,2,3,
i=1,\cdots,7\\
\end{subarray}} \sum_{
\begin{subarray}{c}
\iota+\kappa=\beta,\iota, \kappa\in \mathcal{S}, 
 |\iota|=1, \Lambda^\iota \nsim \psi_{\geq 1}(|v|) \widehat{\Omega}^v_j \textup{or}\,\psi_{\geq 1}(|v|) \Omega_j^x  \\
\end{subarray}}  \Lambda^\iota(K^i(t,x,v)) \alpha_i(v)\cdot D_v \Lambda^\rho g_{\kappa}^\alpha(t,x,v),
\ee
\[
\widetilde{ \textit{l.o.t}_{\kappa;1}^{\alpha;\rho} }(t,x,v)= \sum_{\begin{subarray}{l}
j=1,2,3,\\
i=1,\cdots,7 
\end{subarray}} \sum_{
\begin{subarray}{c}
\iota+\kappa=\beta,\iota, \kappa\in \mathcal{S},|\iota|=1,   \\
 \Lambda^\iota \nsim \psi_{\geq 1}(|v|) \widehat{\Omega}^v_j \textup{or}\,\psi_{\geq 1}(|v|) \Omega_j^x  \\
\end{subarray}}   \Lambda^{  \rho\circ \iota}\big(  K^i(t,x,v)   \big)  X_i g^\alpha_\kappa(t,x, v)  + 	   \Lambda^\iota(K^i(t,x,v))    [\Lambda^\rho, X_i] g_{\kappa}^\alpha(t,x,v),
\]
\[
 \widetilde{\textit{error}_{\kappa;1}^{\alpha; \rho }}(t,x,v) =\sum_{\begin{subarray}{l}
j=1,2,3,
i=1,\cdots,7 
\end{subarray}} \sum_{
\begin{subarray}{c}
\iota+\kappa=\beta,\iota, \kappa\in \mathcal{S},
 |\iota|=1, \Lambda^\iota \sim \psi_{\geq 1}(|v|) \widehat{\Omega}^v_j \textup{or}\,\psi_{\geq 1}(|v|) \Omega_j^x  \\
\end{subarray}}  K_{\iota;2}^i(t,x,v)\alpha_i(v)\cdot D_v \Lambda^\rho g^\alpha_\kappa(t,x, v),	
\]
\[
\widetilde{\textit{error}_{\kappa;2}^{\alpha; \rho }}(t,x,v)=\sum_{\begin{subarray}{c}
j=1,2,3\\
i=1,\cdots,7\\
\end{subarray}} \sum_{
\begin{subarray}{c}
\iota+\kappa=\beta,\iota, \kappa\in \mathcal{S}, |\iota|=1  \\
\Lambda^\iota \sim \psi_{\geq 1}(|v|) \widehat{\Omega}^v_j \textup{or}\,\psi_{\geq 1}(|v|) \Omega_j^x  \\
\end{subarray}} \Lambda^\rho\big( K_{\iota;2}^i(t,x,v) \big)\alpha_i(v)\cdot D_v  g^\alpha_\kappa(t,x, v) 	+  K_{\iota;2}^i(t,x,v)[\Lambda^\rho, X_i]  g^\alpha_\kappa(t,x, v),
\]
where $ K_{\iota;2}^i(t,x,v)$ is defined in  (\ref{nove98}).  

We  use the  the first decomposition of $D_v$ (\ref{summaryoftwodecomposition}) in Lemma \ref{twodecompositionlemma} and the equality (\ref{sepeqn610}) in Lemma \ref{decompositionofderivatives} for  $\widetilde{ \textit{h.o.t}_{\kappa;i}^{\alpha;\rho} }(t,x,v)$, $i\in\{1,2,3\}$, and $ \widetilde{\textit{error}_{\kappa;1}^{\alpha; \rho }}(t,x,v)$. Then from the linear decay estimate (\ref{noveqn78}) in Lemma \ref{sharpdecaywithderivatives}, the estimate of coefficients in (\ref{jan15eqn2}), (\ref{sepeqn904}), (\ref{march25eqn10}), and (\ref{march13eqn10}), the second part of the estimate (\ref{feb8eqn51}) in Lemma \ref{derivativeofweightfunction},  and the $L^2_{x,v}-L^\infty_{x,v}$ type bilinear estimate, we have 
\[
\sum_{\alpha\in \mathcal{B}, \kappa, \rho\in \mathcal{S}, |\rho|=1,  |\alpha|+|\kappa|\leq N_0-1}\sum_{i=1,2,3} \| \omega_{\rho\circ \kappa}^\alpha(t,x,v) \widetilde{ \textit{h.o.t}_{\kappa;i}^{\alpha;\rho} }(t,x,v)\|_{L^2_{x,v}} +     \| \omega_{\rho\circ \kappa}^\alpha(t,x,v) \widetilde{ \textit{error}_{\kappa;1}^{\alpha;\rho} }(t,x,v)\|_{L^2_{x,v}}
\]
\[ 
  \lesssim \sum_{ 
\begin{subarray}{c}
  \gamma\in \mathcal{B}, \kappa\in \mathcal{S}, |\gamma|+|\kappa|\leq N_0, \rho\in \mathcal{B},|\rho|\leq 3, u\in\{E^\rho, B^\rho\}
 \end{subarray} }  \| \omega_{\kappa}^\gamma(t,x,v)  g_{\kappa}^\gamma(t,x,v)\|_{L^2_{x,v}}\| (1+||t|-|x+\hat{v}t||) u(t,x+\hat{v}t)\|_{L^\infty_{x,v}} 
\] 
\be 
 \lesssim (1+|t|)^{-1} E_{\textup{high}}^{f}(t)E_{\textup{low}}^{eb}(t)\lesssim  (1+|t|)^{-1+\delta} \epsilon_1^2.
 \ee
 
Recall (\ref{noveq521}), (\ref{dec26eqn1}), (\ref{sepeqn610}), (\ref{sepeqn334}), and (\ref{sepeqn335}).  From the estimates of coefficients in (\ref{jan23eqn11}),  (\ref{jan15eqn41}), (\ref{sepeqn524}), (\ref{sepeqn904}), and (\ref{sepeqn88}), the following estimate holds from the $L^2_{x,v}-L^\infty_{x,v}$ type bilinear estimate, 
\[
\sum_{\alpha\in \mathcal{B}, \kappa, \rho\in \mathcal{S}, |\rho|=1,  |\alpha|+|\kappa|\leq N_0-1}\sum_{i=1,2,3} \| \omega_{\rho\circ \kappa}^\alpha(t,x,v) \widetilde{ \textit{l.o.t}_{\kappa;i}^{\alpha;\rho} }(t,x,v)\|_{L^2_{x,v}}  + \| \omega_{\rho\circ \kappa}^\alpha(t,x,v) \widetilde{ \textit{error}_{\kappa;2}^{\alpha;\rho} }(t,x,v)\|_{L^2_{x,v}} 
\]
\[
  \lesssim    \sum_{ 
\begin{subarray}{c}
  \gamma\in \mathcal{B}, \kappa\in \mathcal{S}, |\gamma|+|\kappa|\leq N_0, \rho\in \mathcal{B},|\rho|\leq 3, u\in\{E^\rho, B^\rho\}
 \end{subarray} }  \| \omega_{\kappa}^\gamma(t,x,v)  g_{\kappa}^\gamma(t,x,v)\|_{L^2_{x,v}}\| (1+|\tilde{d}(t,x,v)|) u(t,x+\hat{v}t)\|_{L^\infty_{x,v}}
 \]
 \[
   \lesssim (1+|t|)^{-1} E_{\textup{high}}^{f}(t)E_{\textup{low}}^{eb}(t)\lesssim  (1+|t|)^{-1+\delta} \epsilon_1^2.
 \]
 Hence, our desired estimate (\ref{jan15eqn22}) follows from the above estimate and the decompositions in (\ref{noveq41}), (\ref{2020feb15eqn11}), and (\ref{2020feb15eqn12}). With minor modifications, our  desired estimate (\ref{may22eqn51}) holds very similarly as we only allow $E_{\textup{high}}^{f;2}(t)$ grows at rate $(1+t)^{\delta/2}$ over time.  
 \end{proof}

\begin{lemma}\label{fixtimeloworder2}
  Under the bootstrap assumption \textup{(\ref{bootstrap})}, the following estimate holds for any $t\in [1,T]$, 
\be\label{march14eqn21}
\sum_{  \rho\in \mathcal{S}, |\rho|=1, |\alpha|+|\beta|\leq N_0-1} \|\omega_{\rho\circ \beta}^\alpha(t,x,v)\Lambda^{\rho} \big(\textit{l.o.t}_\beta^\alpha(t,x,v)   \big)\|_{L^2_x L^2_v}\lesssim   (1+|t|)^{-1+\delta} \epsilon_1^2 , 
\ee
\be\label{may22eqn65}
\sum_{  \rho\in \mathcal{S}, |\rho|=1, |\alpha|+|\beta|\leq N_0-2} \|\omega_{\rho\circ \beta}^\alpha(t,x,v)\Lambda^{\rho} \big(\textit{l.o.t}_\beta^\alpha(t,x,v)   \big)\|_{L^2_x L^2_v}\lesssim   (1+|t|)^{-1+\delta/2} \epsilon_1^2.
\ee
\end{lemma}
\begin{proof}
Recall (\ref{loworderterm}), (\ref{sepeqn400}), (\ref{sepeq200}), (\ref{sepeqn600}), and (\ref{sepeq201}). We have
\be\label{noveq739}
\Lambda^{\rho}(\textit{l.o.t}_{\beta}^\alpha)(t,x,v)= \sum_{i=1,2,3,4}   {\textit{l.o.t}_{\beta;i}^{\alpha;\rho}}(t,x,v),
\ee
where
\be 
\textit{l.o.t}_{\beta;1}^\alpha(t,x,v)= \sum_{i=1,\cdots, 7}  \sum_{\kappa\in \mathcal{S}, |\kappa|\leq |\beta|-1}  \Lambda^\rho\big( K^i(t,x+\hat{v}t,v)  \big[ \tilde{d}(t,x,v)\tilde{e}_{\beta,i}^{\kappa, 1}(x,v) +\tilde{e}_{\beta,i}^{\kappa, 2}(x,v)\big]  \Lambda^\kappa g^\alpha(t,x,v)\big),
 \ee
 \[
\textit{l.o.t}_{\beta;2 }^\alpha(t,x,v)= \sum_{
\begin{subarray}{c}
\iota+\kappa=\beta, |\iota|=1\\
i=1,\cdots,7,\iota, \kappa\in \mathcal{S}
\end{subarray}} \Lambda^\rho \big[   \Lambda^\iota(K^i(t,x+\hat{v}t,v)) [ \Lambda^\kappa,X_i] g^\alpha (t,x, v)  + \sum_{  |\gamma|\leq |\alpha|-1} \Lambda^\iota(K^i_{\alpha;\vec{0}, \gamma}(t,x+\hat{v}t,v))   \Lambda^\kappa  X_i g^\gamma(t,x,v)\big] 
\]
\be 
   +\sum_{|\rho|  \leq 1 } \Lambda^\rho \big[\sum_{|\gamma|=|\alpha|-1} K_{\alpha;\rho, \gamma}^i(t,x+\hat{v}t,v)\big(  [\Lambda^\beta, X_i] g^{\gamma}_{ }(t, x,v)\big)+ \sum_{  |\gamma|\leq |\alpha|-2}  K_{\alpha;\rho, \gamma}^i(t,x+\hat{v}t,v) \Lambda^\beta\big(X_i g^{\gamma} \big)(t,x,v)\big],
\ee
\be 
\textit{l.o.t}_{\beta;3 }^\alpha(t,x,v)=  \sum_{\begin{subarray}{c}
\rho, \gamma\in \mathcal{B}, |\rho|+|\gamma|\leq |\alpha|,
\iota+\kappa=\beta, \iota, \kappa\in \mathcal{S}\\
i=1,\cdots,7,|\iota|+|\rho|\geq 12\\
\end{subarray}} \Lambda^\rho\big[ \big(\Lambda^{\iota}K^i_{\alpha;\rho,\gamma}(t,x+\hat{v}t,v)\big)  \Lambda^{\kappa}\big(X_i g^{\gamma}(t, x,v)\big)\big],
\ee
\be 
\textit{l.o.t}_{\beta;4 }^\alpha(t,x,v)=   \sum_{\begin{subarray}{c}
\rho, \gamma\in \mathcal{B}, |\rho|+|\gamma|\leq |\alpha|,
\iota+\kappa=\beta,   \iota, \kappa\in \mathcal{S}\\
i=1,\cdots,7,1< |\iota|+|\rho|<  12\\
\end{subarray}} \Lambda^\rho\big[\big(\Lambda^{\iota}K^i_{\alpha;\rho,\gamma}(t,x+\hat{v}t,v)\big)\Lambda^{\kappa}\big(X_i g^{\gamma}(t, x,v)\big)\big].
\ee
Same as we did in the proof of the estimate (\ref{jan15eqn91}), we separate into two cases as follows. 

\noindent $\bullet$\quad The estimate of $\textit{l.o.t}_{\beta;i }^{\alpha;\rho}(t,x,v), i\in\{1,2,4\}. $

With minor modifications in the proof of the estimate (\ref{jan16eqn71}), we obtain the following estimate, 
\[
\sum_{\alpha\in \mathcal{B}, \beta, \rho\in \mathcal{S}, |\rho|=1,  |\alpha|+|\beta|\leq N_0-1} \sum_{i=1,2,4}\|\omega_{\rho\circ \beta}^\alpha(t,x,v)\Lambda^{\rho} \textit{l.o.t}_{\beta;i }^{\alpha;\rho}(t,x,v)\|_{L^2_x L^2_v} \lesssim (1+t)^{-1+\delta}\epsilon_1^2. 
\]

\noindent $\bullet$\quad The estimate of $\textit{l.o.t}_{\beta;3 }^{\alpha;\rho}(t,x,v). $

With minor modifications in the proof of the estimate (\ref{feb13eqn24}), we obtain the following estimate, 
\[
\sum_{\alpha\in \mathcal{B}, \beta, \rho\in \mathcal{S}, |\rho|=1,  |\alpha|+|\beta|\leq N_0-1} \|\omega_{\rho\circ \beta}^\alpha(t,x,v)\Lambda^{\rho} \textit{l.o.t}_{\beta;3 }^{\alpha;\rho}(t,x,v)\|_{L^2_x L^2_v} \lesssim   (1+t)^{-1+\delta/2}\epsilon_0.
\]
Hence finishing the desired estimate (\ref{march14eqn21}). With minor modifications, the desired estimate (\ref{may22eqn65}) holds very similarly because we only allow $E_{\textup{high}}^{f;2}(t)$ grows at rate $(1+t)^{\delta/2}$ over time and the correction term $\widetilde{g}_{\alpha,\gamma}(t, v)$, which contributes the logarithmic growth in the estimate (\ref{jan16eqn72}),  equals zero if  $|\alpha|+|\gamma|< N_0$, see (\ref{correctionterm}). 
\end{proof}
 
 Lastly, we estimate the increment of the low order energy  of the Vlasov part over time  as follows. 

   \begin{proposition}\label{loworderenergy}
Under the bootstrap assumption \textup{(\ref{bootstrap})}, the following estimate holds for any $t\in [1,T]$, 
\be\label{loworderenergyvlasov}
   E_{\textup{low}}^{f}(t )  \lesssim  \epsilon_0	+ \int_{1}^{t}(1+s)^{-3/2+4\delta }\epsilon_1^2  ds \lesssim\epsilon_0.
\ee
 
\end{proposition} 
\begin{proof}
  Recall the definition of the low order energy $E_{\textup{low}}^{f}(t)$  in  (\ref{octeqn1896})  and the definition of the correction term $\widetilde{g}_{\alpha,\gamma}(t, v)$ in  (\ref{correctionterm}). Since the set-up of the low order energy estimate is same as we did in the Vlasov-Nordstr\"om system setting  and the issue of losing ``$|v|$'' plays no role in  the low order energy, with minor modification in the proof of \cite{wang}[Proposition 6.2], our desired estimate (\ref{loworderenergyvlasov}) holds very similarly. 

To give a sense, we summarize the key idea of the proof here.   The key idea is that  the decay rate of electromagnetic field is improved because of the extra spatial derivative  in the worst scenario.  Recall (\ref{sepeqn31}). Intuitively speaking,  in  the equation satisfied by  $ \p_t \big(\nabla_v^\alpha \widehat{ g^\gamma }(t,0,v)-\nabla_v\cdot  \widetilde{g}_{\alpha, \gamma}(t, v) \big)$,  we can move   the spatial derivative $\nabla_x$ in front of ``$t\nabla_v\hat{v}\cdot \nabla_x g^{\iota}(t, x,v)$'' in $K_{\gamma;\beta, \iota}^i(t,x+\hat{v}t,v)\cdot X_i g^{\iota}(t, x,v)$   to the electromagnetic field by doing integration by parts in ``$x$''  .  Hence, comparing with the sub-polynomial growth of the high order energy, the low order energy doesn't grow over time.  
\end{proof}

 \section{The high order energy estimate of the bulk terms }\label{corelemmaproof1}

This section is devoted to control the bulk term of the high order energy estimate,  $I_{\beta;4}^\alpha(t)$ (see (\ref{nov96})),  which is also the last term to be estimated in the high order energy estimate.

The essential new ingredient of controlling the \textit{bulk terms} is the  \textit{hidden null structure} we mentioned in the subsection \ref{hiddennullstructure}. In this section, we will explain in what sense the \textit{hidden null structure} means and  how to make use of the   \textit{hidden null structure}. More precisely, we will lay out a step by step strategy to control $I_{\beta;4}^\alpha(t)$ and reduce the estimate of \textit{bulk} terms to the proof of a multilinear estimate in Lemma \ref{trilinearestimate1}, which will be carried out in the section \ref{proofofmainlemmalinear}. 

From the decay estimate (\ref{noveqn78})  in Lemma \ref{sharpdecaywithderivatives}, we know that  the electromagnetic field decays faster in time  if localized far away from the light cone.  We can reduce the estimate of \textit{bulk} terms further by ruling out the far away from the light cone case, e.g.,    $||t|-|x+\hat{v}t|\big|\geq 2^{-10}|t|$, so that we can focus  on the near light cone case later. More precisely, we decompose $I_{\beta;4}^\alpha(t)$ into two parts as follows, 
\be\label{nov104}
I_{\beta;4}^{\alpha }(t)= \widetilde{I}_{\beta;1}^\alpha(t) + \widetilde{I}_{\beta;2}^\alpha(t) ,
\ee
where
\[
\widetilde{I}_{\beta;1}^\alpha(t) =\sum_{\begin{subarray}{c}
j=1,2,3\\
i=1,\cdots,7\\
\end{subarray}}\sum_{
\begin{subarray}{c}
\iota+\kappa=\beta,\iota, \kappa\in \mathcal{S},|\iota|=1 \\
 \Lambda^\iota \thicksim \psi_{\geq 1}(|v|)\widehat{\Omega}^v_j \textup{or}\,\psi_{\geq 1}(|v|) \Omega_j^x  \\
\end{subarray}}  \int_{1}^{t } \int_{\R^3}\int_{\R^3}  \big(\omega_{\beta}^{\alpha}(s, x, v)\big)^2  g^\alpha_\beta (s ,x,v) \big(\sqrt{1+|v|^2}\tilde{d}(s, x,v)\big)^{1-c(\iota)}\psi_{\geq 1}(|v|)\]
\be\label{nov106}
 \times \psi_{\leq -10}(1-|x+\hat{v}s|/|s|)   \alpha_i(v)\cdot\Omega^x_j  \big( E(s, x+\hat{v}s) + \hat{v}\times  B(s, x+\hat{v}s) \big) \alpha_i(v) \cdot D_v  
g^\alpha_\kappa(s,x, v) d x d v d s,
\ee
\[
\widetilde{I}_{\beta;2}^\alpha(t) = \sum_{\begin{subarray}{c}
j=1,2,3\\
i=1,\cdots,7\\
\end{subarray}}\sum_{
\begin{subarray}{c}
\iota+\kappa=\beta,\iota, \kappa\in \mathcal{S}, |\iota|=1  \\
 \Lambda^\iota \thicksim \psi_{\geq 1}(|v|)\widehat{\Omega}^v_j \textup{or}\,\psi_{\geq 1}(|v|) \Omega_j^x  \\
\end{subarray}}  \int_{1}^{t} \int_{\R^3}\int_{\R^3}  \big(\omega_{\beta}^{\alpha}(s, x, v)\big)^2  g^\alpha_\beta (s ,x,v) \big(\sqrt{1+|v|^2}\tilde{d}(s, x,v)\big)^{1-c(\iota)}\psi_{\geq 1}(|v|) \]
\be\label{nov107}
 \times  \psi_{\geq -9}(1-|x+\hat{v}s|/|s|)   \alpha_i(v)\cdot\Omega^x_j  \big( E(s, x+\hat{v}s) + \hat{v}\times  B(s, x+\hat{v}s) \big) \alpha_i(v) \cdot D_v  
g^\alpha_\kappa(s,x, v) d x d v d s.
\ee
 
\begin{lemma}\label{awayfromthecone}
  Under the bootstrap assumption \textup{(\ref{bootstrap})}, the following estimate holds for any $t\in [1,T]$, 
\be\label{nov418}
\sum_{\alpha\in \mathcal{B}, \beta\in \mathcal{S}, |\alpha|+|\beta|=  N_0} | \widetilde{I}_{\beta;2}^\alpha(t)|  \lesssim  (1+t)^{2\delta}\epsilon_0, \quad \sum_{\alpha\in \mathcal{B}, \beta\in \mathcal{S}, |\alpha|+|\beta| < N_0} | \widetilde{I}_{\beta;2}^\alpha(t)|  \lesssim  (1+t)^{\delta }\epsilon_0.
\ee
\end{lemma}
\begin{proof}
Recall the second decomposition of $D_v$ in     (\ref{summaryoftwodecomposition})  in Lemma \ref{twodecompositionlemma}, we have
\[
\widetilde{I}_{\beta;2}^\alpha(t) =\sum_{\begin{subarray}{c}
j=1,2,3\\
i=1,\cdots,7\\
\end{subarray}}\sum_{
\begin{subarray}{c}
\iota+\kappa=\beta,\iota,\rho, \kappa\in \mathcal{S},  |\iota|=|\rho|=1, \\
 \Lambda^\iota \thicksim \psi_{\geq 1}(|v|)\widehat{\Omega}^v_j \textup{or}\,\psi_{\geq 1}(|v|) \Omega_j^x  \\
\end{subarray}}   \int_{1}^{t } \int_{\R^3}\int_{\R^3}  \big(\omega_{\beta}^{\alpha}( s,x, v)\big)^2  g^\alpha_\beta (s ,x,v) \big(\sqrt{1+|v|^2}\tilde{d}(s, x,v)\big)^{1-c(\iota)} \psi_{\geq 1}(|v|)    \]
\be
 \times    \psi_{\geq -9}(1-|x+\hat{v}s|/|s|)   \alpha_i(v)\cdot\Omega^x_j  \big( E(s, x+\hat{v}s) + \hat{v}\times  B(s, x+\hat{v}s) \big) \alpha_i(v) \cdot e_{\rho}(s, x,v)  \Lambda^\rho 
g^\alpha_\kappa(s,x, v) d x d v d s,
\ee
where $e_{\rho}(s, x,v)$ is defined in (\ref{sepeq932}). Recall (\ref{highorderweight}). For any $\iota\in \mathcal{K}, \kappa\in \mathcal{S}$, s.t., $|\iota|=1$ and $\iota+\kappa=\beta$, the following estimate holds, 
\be\label{jan1eqn1}
\Big| \frac{\omega_{\beta}^\alpha(s,x,v)}{\omega_{\rho\circ\kappa}^\alpha(s , x,v)} \Big|\lesssim (1+|v|)^{c(\iota)-c(\rho)}(\phi(s,x,v))^{\iota(\iota)-\iota(\rho)}. 
\ee
Note that the following estimate   holds inside the support of the  cutoff function  ``$\psi_{\geq -9}(1-|x+\hat{v}s|/|s|)$'' in $\widetilde{I}_{\beta;2}^\alpha(t)$,   
\be\label{noveqn100}
  |s-|x+\hat{v} s||\sim |x|+|s|.
\ee
Hence, from the above estimate and  the linear decay estimate (\ref{noveqn78}) 
in Lemma \ref{sharpdecaywithderivatives},  we have
\be\label{noveqn101}
\big(|\nabla_x E(s, x+\hat{v}s)| + |\nabla_x B(s, x+\hat{v}s)| \big) \psi_{\geq -9}(1-|x+\hat{v}s|/|s|)\lesssim \big(|x|+|s|\big)^{-2}(1+s)^{-1} \epsilon_1.
\ee
From the   estimates (\ref{jan1eqn1}), (\ref{noveqn100}), and (\ref{noveqn101}), the second   estimate in \ref{feb8eqn51}   in Lemma \ref{derivativeofweightfunction}, and the detailed formulas of coefficients $e_{\rho}(s,x,v)$, $\rho\in\mathcal{K}$, in (\ref{sepeq932}),   we know that our desired   estimate   (\ref{nov418})  holds from the $L^2_{x,v}-L^2_{x,v}-L^\infty_{x,v}$ type multilinear estimate.  
\end{proof}

Finally, the high order energy estimate is reduced to the estimate of bulk term $\widetilde{I}_{\beta;1}^\alpha(t)$. To be precise about the size of frequencies of the electromagnetic field and the size of    the distance with respect to the light cone ``$||t|-|x+\hat{v} t||$'', for any fixed $\alpha\in \mathcal{B}, \beta\in \mathcal{S}$,  we localize both the frequencies of the electromagnetic field and     the distance with respect to the light cone ``$||t|-|x+\hat{v} t||$'' for $ \widetilde{I}_{\beta;1}^\alpha(t) $   as follows, 
 \be\label{nov112}
\widetilde{I}_{\beta;1 }^\alpha(t) = \sum_{d\in \mathbb{Z}, d\geq 0} \sum_{k\in \mathbb{Z}} H_{k,d}(t), 
 \ee
where
\[
H_{k,d}(t)= \sum_{\begin{subarray}{c}
j=1,2,3\\
i=1,\cdots,7\\
\end{subarray}}\sum_{
\begin{subarray}{c}
\iota+\kappa=\beta,\iota, \kappa\in \mathcal{S}, |\iota|=1  \\
 \Lambda^\iota \thicksim \psi_{\geq 1}(|v|)\widehat{\Omega}^v_j \textup{or}\,\psi_{\geq 1}(|v|) \Omega_j^x  \\
\end{subarray}}  \int_{1}^{t } \int_{\R^3}\int_{\R^3}  \big(\omega_{\beta}^{\alpha}(s, x, v)\big)^2  g^\alpha_\beta (s ,x,v) \big(\sqrt{1+|v|^2}\tilde{d}(s, x,v)\big)^{1-c(\iota)}\psi_{\geq 1}(|v|) \]
\be\label{nov241}
\times \psi_{\leq -10}(1-|x+\hat{v}s|/|s|) \varphi_{d}\big(||s|-|x+\hat{v}s||\big)   \alpha_i(v)\cdot\Omega^x_j  \big( P_k[E](s, x+\hat{v}s) + \hat{v}\times  P_k[B](s, x+\hat{v}s) \big) \alpha_i(v) \cdot D_v  
g^\alpha_\kappa(s,x, v) d x d v d s,
\ee
 where the cutoff function ``$\varphi_{d}(\cdot)$'' is defined in (\ref{cutoffwiththreshold100}). 

For any fixed $k,d$, s.t., $k\in \mathbb{Z}, d\in \mathbb{Z}_{+}$, our strategy is to prove two estimates for $H_{k,d}$, which are stated in Lemma  \ref{firstprop1}  and   Lemma  \ref{secondprop1}. Those two estimates will help us  to get around a summability issue  with respect to the frequency of the electromagnetic field , which equivalents to an issue of logarithmic growth in time.

To improve presentation,  we define the following quantity, which measures the energy of profiles and the energy of \textit{non-bulk  terms} in a    region with     the distance  localized with respect to the light cone $C_{t}=\{(x,v): x, v\in \R^3, |t|-|x+\hat{v} t|=0\}$, 
\[
E_{\beta;d}^{\alpha}(t):=     \sum_{
\begin{subarray}{c}
\iota,   \kappa, \rho \in \mathcal{S},\iota+\kappa=\beta,
|\rho|=|\iota|=1   \\
\end{subarray}}  \|\omega_\beta^\alpha(t,x,v) g_{\beta}^\alpha(t, x,v)\varphi_{[d-1,d+1]} \big(||t|-|x+\hat{v}t||\big) \|_{L^2_{x}L^2_v}^2
\]
\[ 
+ (1+t)^2  \big[ \|\omega_{\beta}^\alpha(t,x, v)  \big(\textit{h.o.t}_{\beta}^\alpha(t,x,v) -  \textit{bulk}_{\beta}^\alpha(t,x,v)\big) \varphi_{[d-1,d+1]}\big(||t|-|x+\hat{v}t||\big) \|_{L^2_x L^2_v }^2 \]
\[
+\|\omega_{\rho\circ\kappa}^\alpha(t,x, v) \Lambda^\rho \big(\textit{h.o.t}_{\kappa}^\alpha(t,x,v) -  \textit{bulk}_{\kappa}^\alpha(t,x,v)\big) \varphi_{[d-1,d+1]}\big(||t|-|x+\hat{v}t||\big) \|_{L^2_x L^2_v }^2 +\|\omega_{\beta}^\alpha(t,x, v)  \big(\textit{l.o.t}_{\beta}^\alpha(t,x,v)  \big)\]
\be\label{jan11eqn80}
\times  \varphi_{[d-1,d+1]}\big(||t|-|x+\hat{v}t||\big)\|_{L^2_x L^2_v }^2+  \|\omega_{\rho\circ\kappa}^\alpha(t,x, v) \Lambda^\rho \big(\textit{l.o.t}_{\kappa}^\alpha(t,x,v) \big)\varphi_{[d-1,d+1]}\big(||t|-|x+\hat{v}t||\big) \|_{L^2_x L^2_v }^2 \big].
\ee

   Due to the fully nonlinear nature of the problem, the non-bulk terms $\Lambda^\rho \big(\textit{h.o.t}_{\kappa}^\alpha(t,x,v) -  \textit{bulk}_{\kappa}^\alpha(t,x,v)\big)$ and $ \Lambda^\rho \big(\textit{l.o.t}_{\kappa}^\alpha(t,x,v) \big)$ in (\ref{jan11eqn80}) will appear  when we utilize the \textit{hidden null structure} by doing  integration by parts in time once in the later argument, see section \ref{proofofmainlemmalinear}. We separate  out the localized energy ``$E_{\beta;d}^{\alpha}(t)$'' to help us identify the main enemy when estimating ``$H_{k,d}(t)$''.

\begin{lemma}\label{firstprop1}
For any $k \in \mathbb{Z}$,  $d\in \mathbb{N}_{+}, t\in [1,T]$, we have the following estimate,
\be\label{nov401}
|H_{k,d}(t)|\lesssim  (2^{k/2+d/2}+2^{2k+2d}  )2^{-4k_{+}}\epsilon_1 \big[ \sum_{ 
  \tau \in\{1,t\}  
} E_{\beta;d}^{\alpha}(\tau)  + \int_{1}^{t} (1+s)^{-1}   E_{\beta;d}^{\alpha}(s ) d s\big].
\ee

\end{lemma}
\begin{proof}
See subsection \ref{reductionoflemmabulk}.
\end{proof}
\begin{lemma}\label{secondprop1}
For any $k, \in \mathbb{Z}$, $d\in \mathbb{N}_{+}$,  $d\geq 10, t\in [1,T]$, we have the following estimate 
\be\label{nov402}
|H_{k,d}(t)|\lesssim    (2^{-k-d}+2^{-7k/2-7d/2} ) 2^{-4k_{+}}\epsilon_1   \big[ \sum_{ 
  \tau \in\{1,t\}  
} E_{\beta;d}^{\alpha}(\tau)  + \int_{1}^{t} (1+s)^{-1}   E_{\beta;d}^{\alpha}(s ) d s\big].
\ee
 
\end{lemma}
\begin{proof}
See subsection \ref{reductionoflemmabulk}.
\end{proof}
  Assuming  the validities of  the estimate (\ref{nov401}) in Lemma \ref{firstprop1} and the estimate (\ref{nov402}) in Lemma \ref{secondprop1}, as summarized in the following Lemma, we finish the estimate of the last term $\widetilde{I}_{\beta;1 }^\alpha(t)$. 
\begin{lemma}\label{dec30proposition1}
Under the  assumption  that  the   Lemma \textup{\ref{firstprop1}} and  the   Lemma  \textup{\ref{secondprop1}} hold, we have
\be\label{nov412}
\sum_{\alpha\in \mathcal{B}, \beta\in \mathcal{S}, |\alpha|+|\beta|=  N_0}| \widetilde{I}_{\beta;1}^\alpha(t)| \lesssim  (1+t)^{2\delta}\epsilon_0, \quad  \sum_{\alpha\in \mathcal{B}, \beta\in \mathcal{S}, |\alpha|+|\beta| <  N_0}| \widetilde{I}_{\beta;1}^\alpha(t)| \lesssim  (1+t)^{\delta }\epsilon_0.
\ee
\end{lemma}
\begin{proof}
Recall (\ref{nov112}). From the estimate (\ref{nov401}) in Lemma \ref{firstprop1} and the estimate (\ref{nov402}) in Lemma \ref{secondprop1}, the following estimate holds, 
\be\label{jan17eqn81}
 | \widetilde{I}_{\beta;1}^\alpha(t)|\lesssim \sum_{d\geq 0}   \big(   \sum_{k\leq -d  } (2^{k+d}+2^{2k+2d}  )+ \sum_{k\geq -d   }  (2^{-k-d}+2^{-7k/2-7d/2} )\big) \epsilon_1   \big[ \sum_{ 
  \tau \in\{1,t\}  
} E_{\beta;d}^{\alpha}(\tau)  + \int_{1}^{t} (1+s)^{-1}   E_{\beta;d}^{\alpha}(s ) d s\big] .
 \ee
Recall (\ref{jan11eqn80}). From the estimates (\ref{jan15eqn21}) and (\ref{may22eqn12})  in Lemma \ref{fixedtimeestimate1}, the estimates (\ref{jan15eqn22}) and (\ref{may22eqn51}) in Lemma \ref{fixtimehighorder2}, the estimates (\ref{jan15eqn91}) and (\ref{may22eqn200}) in Lemma \ref{fixedtimeestimate2}  and the estimates (\ref{march14eqn21}) and  (\ref{may22eqn65}) in Lemma \ref{fixtimeloworder2}, we have
  \be\label{march15eqn1}
\sum_{d\in \mathbb{N}_{+}}\sum_{\alpha\in \mathcal{B}, \beta\in \mathcal{S}, |\alpha|+|\beta|=N_0}  E_{\beta;d}^{\alpha}(t) \lesssim (1+t)^{2\delta}\epsilon_1^2, \quad \sum_{d\in \mathbb{N}_{+}}\sum_{\alpha\in \mathcal{B}, \beta\in \mathcal{S}, |\alpha|+|\beta|< N_0}  E_{\beta;d}^{\alpha}(t) \lesssim (1+t)^{ \delta}\epsilon_1^2. 
  \ee
Therefore, our desired estimate (\ref{nov412}) holds  from the estimates (\ref{jan17eqn81}) and (\ref{march15eqn1}).  
\end{proof}

\subsection{Reduction of the proof of Lemma \ref{firstprop1} and the proof of Lemma \ref{secondprop1}}\label{reductionoflemmabulk}

This section is devoted to lay out a strategy to prove the estimate (\ref{nov401}) in  Lemma \ref{firstprop1} and the estimate (\ref{nov402}) in  Lemma \ref{secondprop1}.

 Intuitively speaking, there are two main ingredients  in  proving these two desired estimates.  Firstly, by doing integration by parts in time,  we exploit the hidden null structure by taking the advantage of high   oscillation of phase in time, which solely depends on the electromagnetic field.  Secondly,  by using the equality (\ref{noveqn171}) in Lemma \ref{tradethreetimes},  we can trade the spatial derivative for   the decay  rate of the distance  with respect to the light cone ``$||t|-|x+\hat{v}t||$''. After comparing the gain and the loss, we decide whether to do the trading process. More precisely, to prove our desired estimate (\ref{nov401}) in Lemma \ref{firstprop1}, we don't do the trading process. However, to  prove the desired estimate (\ref{nov402}) in Lemma \ref{secondprop1}, we do the trading process.

 To better explain our strategy,  as an example, we use the following term inside $H_{k,d}(t)$, 
\[
   \int_{1}^{t} \int_{\R^3}\int_{\R^3}  \big(\omega_{\beta}^{\alpha}( s,x, v)\big)^2  g^\alpha_\beta (s ,x,v)  \psi_{\leq -10}(1-|x+\hat{v}s|/|s|) \sqrt{1+|v|^2}\tilde{d}(s, x,v)\psi_{\geq 1}(|v|) \alpha_i(v) \cdot D_v  
g^\alpha_\kappa(s,x, v) \]
\be\label{jan3eqn1}
\times \varphi_{d}\big(||s|-|x+\hat{v}s||\big)  \alpha_i(v)\cdot\Omega^x_j  \big( P_k[E](s, x+\hat{v}s) + \hat{v}\times  P_k[B](s, x+\hat{v}s) \big)  d x d v d s ,\quad 
\ee
where $\beta=\kappa+\iota$ and $\Lambda^{\iota}\thicksim \psi_{\geq 1}(|v|) \widehat{\Omega}_j^v$, see (\ref{nov241}). To make   the coefficient  $\sqrt{1+|v|^2}$ in (\ref{jan3eqn1}) controllable when  ``$|v|$'' is extremely large, we use the second decomposition of $D_v$ in (\ref{summaryoftwodecomposition}). After replacing $D_v$ in (\ref{jan3eqn1}) by the second decomposition of $D_v$ in (\ref{summaryoftwodecomposition}),   as an example, we consider the following   term, 
\[
   \int_{1}^{t  } \int_{\R^3}\int_{\R^3}  \big(\omega_{\beta}^{\alpha}( s, x, v)\big)^2  g^\alpha_\beta (s ,x,v)   \psi_{\leq -10}(1-|x+\hat{v}s|/|s|) \sqrt{1+|v|^2}\tilde{d}(s, x,v)  \psi_{\geq 1}(|v|)\varphi_{d}\big(||s|-|x+\hat{v}s||\big)  \alpha_i(v) \cdot e_{\rho}(s,x,v)  \]
\be\label{jan3eqn2}
\times \Lambda^\rho  
g^\alpha_\kappa(s,x, v)  \alpha_i(v)\cdot\Omega^x_j  \big( P_k[E](s, x+\hat{v}s) + \hat{v}\times  P_k[B](s, x+\hat{v}s) \big)  d x d v d s, 
\ee
where  $\Lambda^\rho\thicksim \Omega_i^x,  i\in\{1,2,3\}$.

 Since $\Lambda^\rho$ is a \textit{good} derivative, from (\ref{highorderweight}) and the second part of the estimate (\ref{feb8eqn51}) in Lemma \ref{derivativeofweightfunction}, we know  that  the following estimate holds for the case we are considering, 
\be\label{jan4eqn1}
\Big|\frac{\omega_{\beta}^{\alpha}(s,x,v)}{\omega_{\rho\circ\kappa}^\alpha(s,x,v)}\Big| \Big|\frac{\tilde{d}(s,x,v) }{1+||s |-|x+\hat{v} s||}\Big|\sim \frac{1}{1+|v|}. 
\ee

Thanks to the dyadic localization of the distance with respect to the light cone, we know that the size of ``$||s|-|x+\hat{v} s||$'' is at most ``$2^{d+2}$'', $d\in \mathbb{N}_{+}$. Let 
\be
F^{\alpha}_\beta(t,x,v):=  2^{-d} \big(\omega_{\beta}^{\alpha}( t,x, v)\big)^2  g^\alpha_\beta (t ,x,v) \Lambda^\rho g^\alpha_\kappa(t,x, v)      \psi_{\leq -10}(1-|x+\hat{v}t|/|t|) \sqrt{1+|v|^2}\tilde{d}(t, x,v) \varphi_{d}\big(||t|-|x+\hat{v}t||\big).
\ee
From the above definition and the estimate (\ref{jan4eqn1}), we have 
\be\label{jan4eqn2}
\| F^{\alpha}_\beta(t,x,v)\|_{L^1_x L^1_v} \lesssim \sum_{\alpha\in \mathcal{B}, \beta\in \mathcal{S}} \|\omega_{\beta}^{\alpha}(t, x, v)   g^\alpha_\beta (t ,x,v)  \varphi_{d}\big(||t|-|x+\hat{v}t||\big) \|_{L^2_{x,v}}^2. 
\ee

From the definition of $F_\beta^\alpha(t,x,v)$ and the detailed formula of $e_{\rho}(s,x,v)$ in  (\ref{sepeq932}), we can rewrite the integral in (\ref{jan3eqn2}) as follows, 
\be\label{jan4eqn3}
- \int_{1}^{t }\int_{\R^3} \int_{\R^3} 2^{d} F_\beta^\alpha(s,x,v) \alpha_i(v)\cdot \tilde{V}_j  \psi_{\geq 1}(|v|) |v|^{-1} (X_j+\hat{V}_j s )\cdot \tilde{V}_i \alpha_i(v)\cdot\Omega^x_j  \big( P_k[E](s, x+\hat{v}s) + \hat{v}\times  P_k[B](s, x+\hat{v}s) \big) d x d v d s. 
\ee

Since the magnetic field can be handled in the same way as the electric field, it would be sufficient to estimate the following term,
\be\label{jan4eqn4}
\int_{1}^{t  } \int_{\R^3} \int_{\R^3} \widetilde{F}_{\beta}^\alpha(s,x,v)    \psi_{\geq 1}(|v|) |v|^{-1}  
   \cdot\Omega^x_j   P_k[E](s, x+\hat{v}s) d x d v d s,
\ee
where 
\[
 \widetilde{F}_{\beta}^\alpha(s,x,v):= F_\beta^\alpha(s,x,v)  (X_j+\hat{V}_j s )\cdot \tilde{V}_i 2^{d} \alpha_i(v)\cdot \tilde{V}_j \alpha_i(v).
\]

 To better see the hidden null structure,  we write the integral in (\ref{jan4eqn4}) on the Fourier side. Here, we do Fourier transform in ``$x$'' and view ``$t$'' and ``$v$'' as fixed parameters. As a result, we can reduce the integral in (\ref{jan4eqn4}) as follows, 
 \be\label{jan4eqn6}
\sum_{\mu\in \{+,-\}} i c_{\mu} 	\int_{1}^{t } \int_{\R^3} \int_{\R^3}  \overline{\mathcal{F}_{x}[\widetilde{F}_{\beta}^\alpha](s,\xi,v) }   e^{i s \hat{v}\cdot \xi - i \mu s|\xi|} \psi_{\geq - 2}(|v|) \tilde{V}_j \cdot \xi  |v|^{-1}  
\psi_{k}(\xi) \widehat{h_1}(s, \xi)  d \xi d v d s,	
 \ee
 where $h_1(t)$ is the profile of electric field $E(t)$, see (\ref{octeqn1054}) and (\ref{octeqn1055}).  Note that 
 \be\label{jan4eqn8}
 |\xi|- \mu \hat{v}\cdot \xi \gtrsim |\xi|\big( \frac{1}{1+|v|^2} + (1-\cos(\mu v , \xi )) \gtrsim \big|\frac{\tilde{V}_j\cdot \xi }{1+|v|}\big|. 
 \ee
From the above estimate, we know  that  the price of doing integration by parts in time can be paid exactly by the symbol ``$\psi_{\geq 1}(|v|) \tilde{V}_j \cdot \xi  |v|^{-1} $'' in (\ref{jan4eqn6}). As a result, the integral over time actually doesn't grow dramatically. 

	Moreover, from the estimate (\ref{jan4eqn8}),  it is easy to see that both $\angle(\mu v, \xi)$ and $(1+|v|)^{-1}$ acts like the null structure. Note that $\widehat{S}^v$ contributes the smallness of  $(1+|v|)^{-3}$  and $\Omega_i^x$ contributes a symbol $\tilde{V}_i \cdot \xi \sim |\xi|\angle(\mu v, \xi)$ when these derivatives hit the pulled-back electromagnetic field $u(t,x+\hat{v}t)$, where $u\in \{E,B\}$. Because of this fact, we call these derivatives as ``\textit{good derivatives}''. 

In practice, we will use a more delicate version integration by parts in time.  Instead of doing integration by parts in time on Fourier side directly,  we will compare the size of phase ``$|\xi|-\mu \hat{v}\cdot \xi$'' with the size of ``$t$''. Moreover, since it is more convenient to work in the physical space due to the presence of complicated weight function   associated with the energy, we will formulate the Fourier based integration by parts in time into an equality on the physical space, which is the equality (\ref{nove91}) in 
Lemma \ref{integrationbypartsintimephysical}.

\begin{definition}\label{definitionofbilinearoperators}
For any  given Fourier  	  symbol $m(\xi)$ and any function $h(t,x)\in \{h^\alpha_i(t,x), i\in\{1,2\}, \alpha\in\mathcal{B}, |\alpha|\leq 10\}$,  we define 
\be\label{noveqn89}
  T_{k}^{\mu}(m(\xi), h) (t, x+\hat{v}t,v):=  \int_{\R^3} e^{i(x+t\hat{v})\cdot \xi - i t\mu |\xi|} \frac{    -i m(\xi)  \psi_k(\xi)}{\hat{v}\cdot \xi - \mu |\xi|} \widehat{h }(t, \xi) \psi_{> 10}\big(    t (|\xi|-\mu\hat{v}\cdot \xi  )\big)   d \xi, 
\ee
\be\label{noveqn90}
H_k^\mu(m(\xi), h)(t, x+\hat{v}t, v):=  \int_{\R^3} e^{i(x+t\hat{v})\cdot \xi - i t\mu |\xi|}   \frac{ i  m(\xi) \psi_k(\xi) }{\hat{v}\cdot \xi - \mu |\xi|} \p_t   \widehat{h  }(t, \xi)\psi_{> 10} (    t (|\xi|-\mu\hat{v}\cdot \xi ) ) d \xi, 
\ee
\be\label{noveqn96}
K_k^\mu(m(\xi), h)(t, x+\hat{v}t, v):=  \int_{\R^3} e^{i(x+t\hat{v})\cdot \xi - i t\mu |\xi|}  m(\xi) \psi_k(\xi)      \widehat{h }(t, \xi) \big[  -i\mu \psi_{> 10}' (    t (|\xi|-\mu\hat{v}\cdot \xi ) )+ \psi_{\leq 10}\big(    t (|\xi|-\mu\hat{v}\cdot \xi  )\big)\big]d \xi. 
\ee
\end{definition}

With the above definition, now we can decompose the good derivative of the electromagnetic field into  \textit{good errors} and the \textit{time derivative of a linear operator}. More precisely, we have
\begin{lemma}\label{integrationbypartsintimephysical}
For any $\alpha\in \mathcal{B}$, $u\in \{E^\alpha,B^\alpha\}$, $k\in \mathbb{Z},$ $j\in \{1,2,3\}$, and any Fourier multiplier operator $T$ with symbol $m(\xi)$,  the following equalities hold for some $l\in \{1,2\}$, 
\[
\Omega_j^x T_k[u^\alpha](t, x+\hat{v}t) =\sum_{\mu\in \{+,-\}} c_{\mu}\big[\p_t T_{k}^\mu(i\tilde{V}_j\cdot \xi m(\xi), (h_l^\alpha)^\mu)(t,x+\hat{v}t,v)+ H_{k}^\mu(i\tilde{V}_j\cdot \xi  m(\xi), (h_l^\alpha)^\mu)(t,x+\hat{v}t,v)
\]
\be\label{nove91}
+ K_{k}^\mu(i\tilde{V}_j\cdot \xi  m(\xi),  (h_l^\alpha)^\mu)(t,x+\hat{v}t,v)\big],
\ee
\[
\Omega_j^x T_k[\p_t u^\alpha](t, x+\hat{v}t) =\sum_{\mu\in \{+,-\}}  \h\big[\p_t T_{k}^\mu(i\tilde{V}_j\cdot \xi |\xi| m(\xi), (h_l^\alpha)^\mu)(t,x+\hat{v}t,v)+ H_{k}^\mu(i\tilde{V}_j\cdot \xi |\xi| m(\xi), (h_l^\alpha)^\mu)(t,x+\hat{v}t,v)
\]
\be\label{jan5eqn1}
+ K_{k}^\mu(i\tilde{V}_j\cdot \xi |\xi|   m(\xi), (h_l^\alpha)^\mu)(t,x+\hat{v}t,v)\big].
\ee
\end{lemma}
\begin{proof}
Recall (\ref{jan5eqn3})    and (\ref{octeqn1055}). 
Note that, for any $j\in\{1,2,3\}$, $\alpha\in \mathcal{B}$, and $u\in \{E^\alpha, B^\alpha\}$, the following equalities hold for some $l\in \{1,2\},$
\be\label{noveqn88} 
\Omega_j^x  T_k[u^\alpha](t, x+\hat{v}t)=
\sum_{\mu\in \{+,-\}}\int_{\R^3} e^{i(x+t\hat{v})\cdot \xi - i t\mu |\xi|}    i c_{\mu} \tilde{V}_j\cdot \xi m(\xi) \widehat{(h_l^\alpha)^\mu }(t, \xi)\psi_k(\xi) d \xi, 
\ee
\be\label{jan5eqn2} 
\Omega_j^x  T_k[\p_t u^\alpha](t, x+\hat{v}t)=
\sum_{\mu\in \{+,-\}}\int_{\R^3} e^{i(x+t\hat{v})\cdot \xi - i t\mu |\xi|}    \frac{i}{2}   \tilde{V}_j\cdot \xi |\xi| m(\xi) \widehat{(h_l^\alpha)^\mu}(t, \xi)\psi_k(\xi) d \xi .
\ee
Hence,   our desired equalities (\ref{nove91}) and (\ref{jan5eqn1}) hold from (\ref{noveqn88}), (\ref{jan5eqn2}), (\ref{noveqn89}), (\ref{noveqn90}), and (\ref{noveqn96}).
\end{proof}

With the above preparation, we are ready to lay out the strategy for the proof of the desired estimate (\ref{nov401}) in Lemma \ref{firstprop1}. Recall (\ref{nov241}). Firstly, we will using the equality (\ref{nove91}) for  $\Omega_j^x P_k[u](t, x+\hat{v}t)$, $u\in\{E,B\}$. Then, we  do integration by parts in time once to move the time derivative in front of $ \p_t T_{k}^\mu(i\tilde{V}_j\cdot \xi, h_l)(t,x+\hat{v}t) $. The rest of terms in the equality (\ref{nove91}) will be good error terms.

Now, we proceed to lay out our  strategy for the proof of the desired estimate (\ref{nov402}) in Lemma \ref{secondprop1}. Same as  the proof of the desired estimate (\ref{nov401}) in Lemma \ref{firstprop1},  we will also use the oscillation in time for the electromagnetic field. The only extra procedure we will do is trading the spatial derivatives for the decay of the distance with respect to the light cone, which will provide  the factor of $2^{-3k-3d}+2^{-4k-4d}$ and also explains the difference between the desired estimates (\ref{nov401}) and  (\ref{nov402}).

We summarize the main result of the trading process in  the following Lemma. 
\begin{lemma}\label{tradingreg4mod}
For any $j\in\{1,2,3\}$, $u\in \{E,B\}$, and $k\in \mathbb{Z}$, the following decomposition holds after trading the spatial derivatives for  the decay of the distance with respect to the light cone,
\be\label{jan6eqn10}
\Omega_j^x(u_k)(t,x+\hat{v}t)=L^1_{k,j}[u](t,x+\hat{v}t) +  \widetilde{L_{k,j}}[u](t,x+\hat{v}t,v) + \sum_{i=1,\cdots,5}  {\textup{E}}_{k,j}^i[u](t,x+\hat{v}t,v),
\ee
where the leading terms $L_{k,j}^1[u](t,x+\hat{v}t)$ and $\widetilde{L_{k,j}}[u](t,x+\hat{v}t,v) $  are  given  in \textup{(\ref{nove611})}  and  \textup{(\ref{april3eqn40})} respectively, and the error terms $ {\textup{E}}_{k,j}^i[u](t,x+\hat{v}t,v)$, $i\in\{1,\cdots,5\}$, are given in \textup{(\ref{nove510})}  and  \textup{(\ref{april4eqn2})} respectively.

\end{lemma}
 \begin{proof}
From the equality (\ref{noveqn171}) in Lemma \ref{tradethreetimes},  we can rewrite $\Omega_{j}^x(u_k(t,x+\hat{v}t)$, $u\in \{E, B\}$, as follows, 
\[
\Omega_j^x(u_k(t,x+\hat{v}t))= (|t|-|x+\hat{v}t|)^{-3} \Omega_j^x\big( (|t|-|x+\hat{v}t|)^3 u_k(t,x+\hat{v}t) \big) 
-  3\big(\Omega_j^x   (|t|-|x+\hat{v}t|)\big) (  (|t|-|x+\hat{v}t|))^{-1} u_k(t,x+\hat{v}t)
\]
\[
= \sum_{\alpha\in \mathcal{B}, |\alpha|\leq 3} \sum_{i=0,1,2 } ((|t|-|x+\hat{v}t|))^{-3} \Omega_j^x\big[ {\tilde{c}_{\alpha}^{i}}(t,x+\hat{v}t) \tilde{T}^i_{k,\alpha}(\p_t^i u^\alpha)(t, x+\hat{v}t)+
(|t|-|x+\hat{v}t|) {e}_{\alpha}(t, x+\hat{v}t)  \]
\[
\times \tilde{T}_{k,\alpha}^3((\p_t^2-\Delta)u)(t, x+\hat{v}t) \big]-  3\big(\Omega_j^x  (|t|-|x+\hat{v}t|)\big)  (|t|-|x+\hat{v}t|)^{-4}\big[ {\tilde{c}_{\alpha}^{i}}(t,x+\hat{v}t) \tilde{T}^i_{k,\alpha}(\p_t^i u^\alpha)(t, x+\hat{v}t)\]
\[
+   (|t|-|x+\hat{v}t|) {e}_{\alpha}(t, x+\hat{v}t) \tilde{T}_{k,\alpha}^3((\p_t^2-\Delta)u)(t, x+\hat{v}t) 
\big] 
\]
\[
= \sum_{\alpha\in \mathcal{B}, |\alpha|\leq 3} \sum_{i=0,1,2 }    (|t|-|x+\hat{v}t|)^{-3} {\tilde{c}_{\alpha}^{i}}(t,x+\hat{v}t) \Omega_j^x\big( \tilde{T}^i_{k,\alpha}(\p_t^i u^\alpha)(t, x+\hat{v}t)\big) +  {c}^j_{\alpha;i}(t,x+\hat{v}t) \tilde{T}^i_{k,\alpha}(\p_t^i u^\alpha)(t, x+\hat{v}t) \]
\be\label{nove321}
 +   \widehat{e}_{j,1}^{\alpha}(t, x+\hat{v}t)\tilde{T}_{k,\alpha}^3((\p_t^2-\Delta)u)(t, x+\hat{v}t)  + \widehat{e}_{j,2}^{\alpha}(t, x+\hat{v}t)\Omega_j^x \big(\tilde{T}_{k,\alpha}^3((\p_t^2-\Delta)u)(t, x+\hat{v}t) \big),
\ee
where
\be
  {c}_{\alpha;i}^j(t, x+\hat{v}t)=    (|t|-|x+\hat{v}t|)^{-3} \Omega_j^x\big(  {\tilde{c}_{\alpha}^{i}}(t,x+\hat{v}t) \big) - 3\big(\Omega_j^x    (|t|-|x+\hat{v}t|)\big)   (|t|-|x+\hat{v}t|)^{-4}  {\tilde{c}_{\alpha}^{i}}(t,x+\hat{v}t),
\ee
\be
 \widehat{e}_{j,1}^{\alpha}(t, x+\hat{v}t) =    (|t|-|x+\hat{v}t|)^{-3} \Omega_j^x\big(   (|t|-|x+\hat{v}t|) {e}_{\alpha}(t, x+\hat{v}t)  \big) - 3\big(\Omega_j^x  (|t|-|x+\hat{v}t|)\big)    (|t|-|x+\hat{v}t|)^{-3	} {e}_{\alpha} (t, x+\hat{v}t),
\ee
\be\label{nove643}
\widehat{e}_{j,2}^{\alpha}(t, x+\hat{v}t)=     (|t|-|x+\hat{v}t|)^{-2}{e}_{\alpha}(t, x+\hat{v}t).
\ee

To better estimate the coefficients,  we classify and  decompose $ \widehat{c}_{\alpha;i}^j(t,x,v)$, $ i\in\{0,1,2\},$  and $   \widehat{e}_{j,1}^{\alpha}(t, x, v)$ into two parts as follows,
\[
 {c}_{\alpha;i}^j(t, x+\hat{v}t)=\Omega_j^x   (|t|-|x+\hat{v}t|)  \widehat{c}_{\alpha;i}^{ 1}(t, x+\hat{v}t) + \widehat{c}_{\alpha;i}^{j;2}(t, x+\hat{v}t),\quad 
 \widehat{e}_{j,1}^{\alpha}(t, x+\hat{v}t)=  \Omega_j^x     (|t|-|x+\hat{v}t|) \widehat{e}_{ 1}^{\alpha;1}(t, x+\hat{v}t)+  \widehat{e}_{j,1}^{\alpha;2}(t, x+\hat{v}t),
\]
where 
\be\label{jan9eqn10}
 \widehat{c}_{\alpha;i}^{ 1}(t, x+\hat{v}t)=- 3    (|t|-|x+\hat{v}t|)^{-4}  {\tilde{c}_{\alpha}^{i}}(t,x+\hat{v}t), \quad \widehat{c}_{\alpha;i}^{j;2}(t,x,v)=    (|t|-|x+\hat{v}t|)^{-3}  \Omega_j^x\big( {\tilde{c}_{\alpha}^{i}}(t,x+\hat{v}t) \big),
\ee
\be\label{jan9eqn11}
   \widehat{e}_{ 1}^{\alpha;1} (t, x+\hat{v}t)= - 2   (|t|-|x+\hat{v}t|)^{-3	}  {e}_{\alpha}(t, x+\hat{v}t), \quad  \widehat{e}_{j,1}^{\alpha;2}(t, x, v)=    (|t|-|x+\hat{v}t|)^{-2} \Omega_j^x\big(   {e}_{\alpha}(t, x+\hat{v}t)  \big). 
 \ee
 Note that
 \be\label{april3eqn21}
\Omega_j^x (|t|-|x+\hat{v}t|)= - \frac{\tilde{V}_j\cdot x }{|x+\hat{v}t|}.
 \ee
As a result of  direct computations and the estimate   (\ref{april3eqn1})   in Lemma \ref{tradethreetimes}, the following  estimate   holds, 
\be\label{nove353}
 \sum_{i=0,1,2}  (|t|-|x+\hat{v}t|)^{ 3} | \widehat{c}_{\alpha;i}^{j;2} (t,x,v)| +   (|t|-|x+\hat{v}t|)^{ 2}  | \widehat{e}_{j,1}^{\alpha;2}(t,x,v)|  \lesssim  \big( \frac{1  }{t+|x+\hat{v}t| }\big)  .
\ee

Recall (\ref{nove321}).   For  the term $\tilde{T}^2_{\alpha}(\p_t^2 u^\alpha)(t, x+\hat{v}t)$,  we  decompose it further into two parts as follows, 
\be\label{nove324}
\tilde{T}^2_{k,\alpha}(\p_t^2 u^\alpha)(t, x+\hat{v}t)= \tilde{T}^2_{k,\alpha}(\Delta u^\alpha)(t, x+\hat{v}t) + \tilde{T}^2_{k,\alpha}((\p_t^2-\Delta) u^\alpha)(t, x+\hat{v}t).
\ee
From the equalities (\ref{nove321}) and  (\ref{nove324}), we   identify the leading terms of $\Omega_j^x u_k(t,x+\hat{v}t)$ and classify the error terms into four parts as follows,
\be\label{nove510}
\Omega_j^x(u_k)(t,x+\hat{v}t)=\sum_{i=1,2}L^i_{k,j}[u](t,x,v)  + \textup{Error}_{k,j}[u](t,x+\hat{v}t), \,\, \textup{Error}_{k,j}[u](t,x+\hat{v}t) = \sum_{i=1,\cdots, 4}E_{k,j}^i[u](t,x+\hat{v}t),
\ee
where
\be\label{nove611}
 L_{k,j}^1 [u](t,x,v) =\sum_{\alpha\in \mathcal{B}, |\alpha|\leq 3} (|t|-|x+\hat{v}t|)^{-3}\big[ \sum_{i=0,1} \tilde{c}_{\alpha}^i(t,x+\hat{v}t) \Omega_j^x\big( \tilde{T}^i_{k,\alpha}(\p_t^i u^\alpha)(t, x+\hat{v}t)\big)+ \tilde{c}_{\alpha}^2(t,x+\hat{v}t) \Omega_j^x\big( \tilde{T}^2_{k,\alpha}(\Delta u^\alpha)(t, x+\hat{v}t)\big)\big],
\ee
\be\label{nove612}
L_{k,j}^2[u] (t,x+\hat{v}t)= \sum_{\alpha\in \mathcal{B}, |\alpha|\leq 3} \sum_{i=0,1}\Omega_j^x   (|t|-|x+\hat{v}t|) \big[\widehat{c}_{\alpha;i}^{ 1}(t, x+\hat{v}t) \tilde{T}^i_{k,\alpha}(\p_t^i u^\alpha)(t, x+\hat{v}t) + \widehat{c}^{ 1}_{\alpha;2}(t, x+\hat{v}t) \tilde{T}^2_{k,\alpha}(\Delta u^\alpha)(t, x+\hat{v}t)\big],
\ee
\be\label{nove613}
E_{k,j}^1[u](t,x+\hat{v}t)= \sum_{\alpha\in \mathcal{B}, |\alpha|\leq 3} \sum_{i=0,1} \widehat{c}^{j;2}_{\alpha;i}(t, x+\hat{v}t) \tilde{T}^i_{k,\alpha}(\p_t^i u^\alpha)(t, x+\hat{v}t) + \widehat{c}^{j;2}_{\alpha;2}(t, x+\hat{v}t) \tilde{T}^2_{k,\alpha}(\Delta u^\alpha)(t, x+\hat{v}t),
\ee
\[
E_{k,j}^2[u](t,x+\hat{v}t)= (|t|-|x+\hat{v}t|)^{-3}   {\tilde{c}_{\alpha}^{2}}(t,x+\hat{v}t) \Omega_j^x\big( \tilde{T}^2_{k,\alpha}( (\p_t^2-\Delta) u^\alpha)(t, x+\hat{v}t)\big)
\]
\be\label{nove614}
 +  \widehat{e}_{j,2}^{\alpha}(t, x+\hat{v}t)\Omega_j^x \big(\tilde{T}_{k,\alpha}^3((\p_t^2-\Delta)u)(t, x+\hat{v}t) \big),
\ee
\[
E_{k,j}^3[u](t,x+\hat{v}t)= \Omega_j^x   (|t|-|x+\hat{v}t|)  \widehat{c}_{\alpha;i}^{ 1}(t, x+\hat{v}t)   \tilde{T}^2_{k,\alpha}( (\p_t^2-\Delta) u^\alpha)(t, x+\hat{v}t) 
\]
\be\label{nove615}
 +   \Omega_j^x     (|t|-|x+\hat{v}t|)\widehat{e}_{ 1}^{\alpha;1}(t, x+\hat{v}t) \tilde{T}_{k,\alpha}^3((\p_t^2-\Delta)u)(t, x+\hat{v}t),
\ee
\be\label{nove616}
E_{k,j}^4[u](t,x+\hat{v}t)=\widehat{c}^{j;2}_{\alpha;2}(t, x+\hat{v}t)  \tilde{T}^2_{k,\alpha}( (\p_t^2-\Delta) u^\alpha)(t, x+\hat{v}t) +  \widehat{e}_{j,1}^{\alpha;2}(t, x+\hat{v}t)\tilde{T}_{k,\alpha}^3((\p_t^2-\Delta)u)(t, x+\hat{v}t).
\ee
Moreover, from the equality(\ref{april3eqn21}), we can split $L_{k,j}^2$ in (\ref{nove612}) into two parts on the Fourier side as follows,
\be\label{april3eqn31}
L_{k,j}^2[u](t,x+\hat{v}t) =\widetilde{L_{k,j}}[u](t,x+\hat{v}t,v)+ E_{k,j}^5[u](t,x+\hat{v}t,v),
\ee
\[
\widetilde{L_{k,j}}[u](t,x+\hat{v}t,v)= \sum_{\alpha\in \mathcal{B}, |\alpha|\leq 3} \sum_{\mu\in\{+,-\}} \sum_{i=0,1}  \int_{\R^3} e^{i(x+\hat{v}t)\cdot \xi-it \mu|\xi|}\frac{- t\mu \xi\cdot \tilde{V}_j}{|x+\hat{v}t||\xi|}      \big( c_{\mu}\widehat{c}^{ 1}_{\alpha;0}(t, x+\hat{v}t) \tilde{m}_{k,\alpha}^0(\xi)  \]
\be\label{april3eqn40}
- c_{\mu}\widehat{c}^{ 1}_{\alpha;2}(t, x+\hat{v}t) \tilde{m}_{k,\alpha}^2(\xi)|\xi|^2+  \h \widehat{c}^{ 1}_{\alpha;1}(t, x+\hat{v}t) \tilde{m}_{k,\alpha}^1(\xi)|\xi|\big) \widehat{P_{\mu}[h_l^{\alpha}]}(t,\xi)  d \xi,
\ee
 \[
E_{k,j}^5[u](t,x+\hat{v}t,v)= \sum_{\alpha\in \mathcal{B}, |\alpha|\leq 3} \sum_{\mu\in\{+,-\}} \sum_{i=0,1}  \int_{\R^3} e^{i(x+\hat{v}t)\cdot \xi-it \mu|\xi|}\frac{-1}{|x+\hat{v}t|} \big(x -t\mu\frac{\xi}{|\xi|}  \big)\cdot \tilde{V}_j  \big( c_{\mu}\widehat{c}^{ 1}_{\alpha;0}(t, x+\hat{v}t) \tilde{m}_{k,\alpha}^0(\xi)  \]
\be\label{april4eqn2}
  - c_{\mu}\widehat{c}^{ 1}_{\alpha;2}(t, x+\hat{v}t) \tilde{m}_{k,\alpha}^2(\xi)|\xi|^2+  \h \widehat{c}^{ 1}_{\alpha;1}(t, x+\hat{v}t) \tilde{m}_{k,\alpha}^1(\xi)|\xi|\big) \widehat{P_{\mu}[h_l^{\alpha}]}(t,\xi)  d \xi,
\ee
where $l\in\{1,2\}$ is uniquely determined  by the type of input $u\in\{E,B\}$.

To sum up,  our desired decomposition (\ref{jan6eqn10}) holds from the decompositions (\ref{nove510}) and  (\ref{april3eqn31}). 
\end{proof}

Motivated from the decomposition (\ref{jan6eqn10}) in Lemma \ref{tradingreg4mod},    we decompose $H_{k,d}$  similarly (see (\ref{nov241})) into three terms after trading the spatial derivatives   for the decay of the distance with respect to the light cone as follows, 
\be\label{jan9eqn31}
H_{k,d}(t)= \widetilde{H}_{k,d}^1(t)+  \widetilde{H}_{k,d}^2(t)+ \widetilde{\textup{Error}}_{k,d}(t),\quad \widetilde{\textup{Error}}_{k,d}(t)=\sum_{l=1,\cdots,5}\widetilde{\textup{Error}}_{k,d}^l(t),
 \ee
where
\[
\widetilde{H}^1_{k,d}(t):= \sum_{
\begin{subarray}{c}
\iota+\kappa=\beta,  |\iota|=1,   j=1,2,3,
i=1,\cdots,7  \\
\Lambda^\iota \thicksim \psi_{\geq 1}(|v|)\widehat{\Omega}^v_j \textup{or}\,\psi_{\geq 1}(|v|) \Omega_j^x  \\
\end{subarray}}  \int_{1}^{t} \int_{\R^3}\int_{\R^3}  \big(\omega_{\beta}^{\alpha}(s, x, v)\big)^2  g^\alpha_\beta (s ,x,v) \big(\sqrt{1+|v|^2}\tilde{d}(s, x,v)\big)^{1-c(\iota)}   \psi_{\leq -10}(1-|x+\hat{v}s|/|s|) \]
\be\label{jan6eqn11}
 \times\psi_{\geq 1}(|v|) \varphi_{d}\big( |s|-|x+\hat{v}s|  \big)  \alpha_i(v)\cdot \big(  L^1_{k,j}[E](s,x+\hat{v}s)   + \hat{v}\times \big(L^1_{k,j}[B](s,x+\hat{v}s)    \big)\big)  \alpha_i(v) \cdot D_v  
g^\alpha_\kappa(s,x, v) d x d v d s,
\ee
\[
\widetilde{H}^2_{k,d}:=  \sum_{
\begin{subarray}{c}
\iota+\kappa=\beta,  |\iota|=1 , j=1,2,3,
i=1,\cdots,7  \\
  \Lambda^\iota \thicksim \psi_{\geq 1}(|v|)\widehat{\Omega}^v_j \textup{or}\,\psi_{\geq 1}(|v|) \Omega_j^x  \\
\end{subarray}}   \int_{1}^{t} \int_{\R^3}\int_{\R^3}  \big(\omega_{\beta}^{\alpha}( s,x, v)\big)^2  g^\alpha_\beta (s ,x,v) \big(\sqrt{1+|v|^2}\tilde{d}(s, x,v)\big)^{1-c(\iota)}    \psi_{\leq -10}(1-|x+\hat{v}s|/|s|) \]
\be\label{april3eqn41}
 \times\psi_{\geq 1}(|v|)   \varphi_{d}\big( |s|-|x+\hat{v}s|  \big) \alpha_i(v)\cdot \big(    \widetilde{L_{k,j}}[E](s,x+\hat{v}s,v)  + \hat{v}\times \big(  \widetilde{L_{k,j}}[B](s,x+\hat{v}s,v)  \big)\big)  \alpha_i(v) \cdot D_v  
g^\alpha_\kappa(s,x, v) d x d v d s,
\ee
\[
\widetilde{\textup{Error}}_{k,d}^l:=  \sum_{
\begin{subarray}{c}
\iota+\kappa=\beta,|\iota|=1 , j=1,2,3,
i=1,\cdots,7  \\
   \Lambda^\iota \thicksim \psi_{\geq 1}(|v|)\widehat{\Omega}^v_j \textup{or}\,\psi_{\geq 1}(|v|) \Omega_j^x  \\
\end{subarray}} \int_{1}^{t } \int_{\R^3}\int_{\R^3}  \big(\omega_{\beta}^{\alpha}(s, x, v)\big)^2  g^\alpha_\beta (s ,x,v) \big(\sqrt{1+|v|^2}\tilde{d}(s, x,v)\big)^{1-c(\iota)}   \psi_{\leq -10}(1-|x+\hat{v}s|/|s|) \]
\be\label{jan6eqn12}
 \times \psi_{\geq 1}(|v|) \varphi_{d}\big( |s|-|x+\hat{v}s|  \big) \alpha_i(v)\cdot \big( \textup{E}_{k,j}^l [E](s,x+\hat{v}s,v) + \hat{v}\times    \textup{E}_{k,j}^l [B](s,x+\hat{v}s,v)\big)  \alpha_i(v) \cdot D_v  
g^\alpha_\kappa(s,x, v) d x d v d s.
\ee

We summarize the estimate of error term $\widetilde{\textup{Error}}_{k,d}(t)$ in the following Lemma. 
\begin{lemma}\label{remainderestimatejan6}
 The following estimate holds, 
\be\label{jan9eqn2}
 |\widetilde{\textup{Error}}_{k,d}|\lesssim   \big( 2^{-3k-3d} +2^{-k-d}\big) 2^{-4k_{+}}\int_{1}^{t }   (1+|s|)^{-1}   E_{\textup{low}}^{eb}(s)  E_{\beta;d}^{\alpha}(s) d s.
  \ee 
\end{lemma}
\begin{proof}
Postponed to subsection \ref{proofoferrorterm}.
\end{proof}

From the equalities (\ref{nove611}), (\ref{april3eqn40}), (\ref{nove91}), and (\ref{jan5eqn1}), modulo the coefficients and the symbol of the Fourier multipliers, we know that both the leading term $L_{k,j}[u](t,x+\hat{v}t)$ and $\widetilde{L_{k,j}}[u](t,x+\hat{v}t,v)$ , which appears in $\widetilde{H}_{k,d}^1$ and $\widetilde{H}_{k,d}^2$ , and $\Omega_j^x(u_k)(t,x+\hat{v}t)$, which appears in $H_{k,d} $ , can be   viewed as a good derivative $\Omega_j^x$ acting on a Fourier multiplier operator. Motivated from this observation and  the decompositions   (\ref{nove91}) and (\ref{jan5eqn1}) for the good derivative ``$\Omega_j^x$'' acting on a Fourier multiplier,   we  define the following multilinear operator.  
\begin{definition}\label{trilinearoperators}
For any fixed  $t_1, t_2\in \mathbb{R}$,  $i\in\{1,\cdots,7\}$, $j\in\{1,2,3\}$, $\alpha\in\mathcal{B}$, $\beta,\iota,\kappa\in \mathcal{S}$, $\mu\in\{+,-\}$, s.t., $\iota+\kappa=\beta$, $|\iota|=1$, and $ \Lambda^\iota \sim \psi_{\geq 1}(|v|)\widehat{\Omega}^v_j \textup{or}\, \psi_{\geq 1}(|v|) \Omega_j^x $, any given Fourier multiplier $m(\xi)$, any given coefficients $a:\R_x\rightarrow \mathbb{C}$, s.t., $a'(x)=0$ if $|x|\leq 2^{-5},$ and $c:\R_t\times \R_x^3\times\R_v^3\rightarrow \mathbb{C}$, and any given profile $h(t,x)\in\{h_i^\alpha(t,x),i\in\{1,2\},\alpha\in\mathcal{B},|\alpha|\leq 10\}$,  we define three multilinear forms as follows, 
\[
T (m , a,  c , h):= \sum_{\Lambda^\iota \sim \psi_{\geq 1}(|v|)\widehat{\Omega}^v_j \textup{or}\, \psi_{\geq 1}(|v|) \Omega_j^x}   \int_{1}^{t } \int_{\R^3}\int_{\R^3}  \big(\omega_{\beta}^{\alpha}(s, x, v)\big)^2  g^\alpha_\beta (s ,x,v) \big(\sqrt{1+|v|^2}\tilde{d}(s, x,v)\big)^{1-c(\iota)}  
\]
 \be\label{jan6eqn1}
\times     C_d(s,x,v) \p_s T_{k}^\mu(\tilde{V}_j\cdot\xi m(\xi), h)(s,x+\hat{v}s,v) \alpha_i(v) \cdot D_v  
g^\alpha_\kappa(s,x, v) d x d v d s,
\ee
\[
H   (m , a,  c , h):= \sum_{\Lambda^\iota \sim \psi_{\geq 1}(|v|)\widehat{\Omega}^v_j \textup{or}\, \psi_{\geq 1}(|v|) \Omega_j^x}    \int_{1}^{t } \int_{\R^3}\int_{\R^3}  \big(\omega_{\beta}^{\alpha}(s, x, v)\big)^2  g^\alpha_\beta (s  ,x,v) \big(\sqrt{1+|v|^2}\tilde{d}(s, x,v)\big)^{1-c(\iota)} 
\]
\be\label{jan6eqn2}
\times   C_d(s,x,v) H_{k}^\mu(\tilde{V}_j\cdot\xi m(\xi), h)(s,x+\hat{v}s,v)  \alpha_i(v) \cdot D_v  
g^\alpha_\kappa(s,x, v) d x d v d s,
\ee
\[
K (m , a, c , h):=  \sum_{\Lambda^\iota \sim \psi_{\geq 1}(|v|)\widehat{\Omega}^v_j \textup{or}\, \psi_{\geq 1}(|v|) \Omega_j^x}   \int_{1}^{t } \int_{\R^3}\int_{\R^3}  \big(\omega_{\beta}^{\alpha}(s, x, v)\big)^2  g^\alpha_\beta (s ,x,v) \big(\sqrt{1+|v|^2}\tilde{d}(s, x,v)\big)^{1-c(\iota)}  
\]
\be\label{jan6eqn3}
\times     C_d(s,x,v) K_{k}^\mu(\tilde{V}_j\cdot\xi m(\xi), h)(s,x+\hat{v}s,v)  \alpha_i(v) \cdot D_v  
g^\alpha_\kappa(s,x, v) d x d v d s,
\ee
 where the bilinear operators $T_{k}^\mu(\cdot, \cdot)$, $H_{k}^\mu(\cdot, \cdot)$, $K_{k}^\mu(\cdot, \cdot)$ are defined in Definition \ref{definitionofbilinearoperators} and the coefficient $C_{d}(s,x,v)$ is defined as follows,
\be\label{2020feb16eqn11}
C_{d}(s,x,v):=a(||s|-|x+\hat{v}s||)  c(s,x,v) \psi_{\leq -10}(1-|x+\hat{v}s|/|s|)   \varphi_{d}(||s|-|x+\hat{v}s||) \psi_{\geq 1}(|v|) .
\ee
\end{definition}

Now our goal is to show that the following two Lemmas hold for the above defined three multilinear  operators.

\begin{lemma}\label{trilinearestimate2}
For any given Fourier multiplier $m(\xi)$,  any given coefficients $a:\R_x\rightarrow \mathbb{C}$ , s.t., $a'(x)=0$ if $|x|\leq 2^{-5},$ and $c:\R_t\times \R_x^3\times\R_v^3\rightarrow \mathbb{C}$,  and any given profile  $h(t,x)\in\{h_i^\alpha(t,x),i\in\{1,2\},\alpha\in\mathcal{B},|\alpha|\leq 10\}$, the following estimate holds, 
\[
|H   (m ,a, c , h)| + |K  (m ,a, c , h)|
\]
  \be\label{jan6eqn21}
 \lesssim    (2^{k+d}+2^{2k+2d}) 2^{-4k_{+}} \|a\|_{Y }  \|m(\xi)\|_{\mathcal{S}_k^\infty} \big[   
  \int_{1}^{t } (1+s)^{-1}  \|c(s,x,v)\|_{L^\infty_{x,v}}   E_{\textup{low}}^{eb}(s)E_{\beta;d}^{\alpha}(s) d s\big],
\ee
where$\| a\|_{Y }:=   \sup_{x\in \R } |a(x)|+ |xa'(x)|$. 
\end{lemma}
\begin{proof}
Postponed to subsection \ref{errormultilinearestimate}. 
\end{proof}

\begin{lemma}\label{trilinearestimate1}
For any given Fourier multiplier $m(\xi)$,  any given coefficients $a:\R_x\rightarrow \mathbb{C}$, s.t., $a'(x)=0$ if $|x|\leq 2^{-5},$ and $c:\R_t\times \R_x^3\times\R_v^3\rightarrow \mathbb{C}$, and   any given profile  $h(t,x)\in\{h_i^\alpha(t,x),i\in\{1,2\},\alpha\in\mathcal{B},|\alpha|\leq 10\}$, the following estimate holds, 
\[
|T  (m ,a, c , h)| \lesssim    (2^{k/2+d/2}+2^{2k+2d})2^{-4k_{+}}\|a\|_{Y}   \int_{1}^{t } (1+|s|)^{-1} \big( \|c(s,x,v)\|_{L^\infty_{x,v}} + s\|\p_s c(s,x,v)\|_{L^\infty_{x,v}}
\]
\be\label{jan6eqn20}
 +\| D_v c(s,x,v)	\|_{L^\infty_{x,v}}\big)  \|m(\xi)\|_{\mathcal{S}_k^\infty} E_{\textup{low}}^{eb}(s)\big(1+E_{\textup{low}}^{eb}(s)\big)  E_{\beta;d}^{\alpha}(s) d s    .
  \ee
 
\end{lemma}
\begin{proof}
Postponed to section \ref{proofofmainlemmalinear}. 
\end{proof}

Assuming that  the estimates in   Lemma \ref{trilinearestimate1},   Lemma \ref{trilinearestimate2}, and   Lemma \ref{remainderestimatejan6}  hold,   we can prove the desired estimate (\ref{nov401}) in Lemma \ref{firstprop1} and the desired estimate (\ref{nov402}) in Lemma \ref{secondprop1}. 

\vo

\noindent \textit{Proof of Lemma \textup{\ref{firstprop1}}}:

Recall (\ref{nov241}). From the equality (\ref{nove91}) in Lemma \ref{integrationbypartsintimephysical}, we know that $H_{k,d}$ is   a linear combination of the trilinear forms $T(1,1, a_i(v), h_i)$, $H(1,1, a_i(v), h_i)$, and $K(1,1, a_i(v), h_i)$, $i\in\{1,2\}$, where $a_i(v)$, $i\in\{1,2\}$, are some explicit coefficients that satisfies the following estimate, 
\[
\sum_{i=1,2}\|a_i(v)\|_{L^\infty_v}+ \|(1+|v|)\nabla_v a_i(v) \|_{L^\infty_v}\lesssim 1. 
\]
Therefore, the desired estimate (\ref{nov401}) follows directly from the estimate (\ref{jan6eqn20}) in Lemma \ref{trilinearestimate1} and the estimate (\ref{jan6eqn21}) in Lemma \ref{trilinearestimate2}.
\qed

\vo

\noindent \textit{Proof of Lemma \textup{\ref{secondprop1}}}:

Note that  we have $d\geq 10$ for the case we are considering.    Recall the decomposition (\ref{jan9eqn31}) and the equations (\ref{jan6eqn11}) and (\ref{april3eqn41}).  From the equalities (\ref{nove91}) and (\ref{jan5eqn1}) in Lemma \ref{integrationbypartsintimephysical}, the detailed formulas of $L^1_{k,j}[u](t,x+\hat{v}t)$ in  (\ref{nove611}) and $\widetilde{L_{k,j}}[u](t,x+\hat{v}t)$ in  (\ref{april3eqn40}) and the formulas of the coefficients $\widehat{c}^{ 1}_{\alpha;i}(t, x+\hat{v}t)$, $i\in\{1,2,3\}$, in (\ref{jan9eqn10}),  we know that we can write $\widetilde{H}_{k,d}^1$ and $\widetilde{H}_{k,d}^2$ as   linear combinations of multilinear forms as follows,
\[
\widetilde{H}_{k,d}^1= \sum_{\alpha\in \mathcal{B}, |\alpha|\leq 3}\sum_{i=0,1,2, l=1,2}T\big(\hat{m}_{\alpha}^i(\xi), |x|^{-3}\psi_{[d-2,d+2]}(x),    {\tilde{c}_{\alpha}^{i}}(t,x+\hat{v}t) a^1_{i,l}(v) , h_l^\alpha(t)\big)+ H \big(\hat{m}_{\alpha}^i(\xi), |x|^{-3}\psi_{[d-2,d+2]}(x),\]
\[
  {\tilde{c}_{\alpha}^{i}}(t,x+\hat{v}t) a^1_{i,l}(v), h_l^\alpha(t)\big) 
+ K\big(\hat{m}_{\alpha}^i(\xi), |x|^{-3}\psi_{[d-2,d+2]}(x),   {\tilde{c}_{\alpha}^{i}}(t,x+\hat{v}t)a^1_{i,l}(v), h_l^\alpha(t)\big),
\]
\[
\widetilde{H}_{k,d}^2= \sum_{\alpha\in \mathcal{B}, |\alpha|\leq 3}\sum_{i=0,1,2, l=1,2}T \big(\hat{m}_{\alpha}^i(\xi)|\xi|^{-1}, |x|^{-4}\psi_{[d-2,d+2]}(x),   \widehat{c_{\alpha}^i}(t,x+\hat{v}t)a^2_{i,l}(v) , h_l^\alpha(t)\big)+ H  \big(\hat{m}_{\alpha}^i(\xi)|\xi|^{-1} ,\]
\[
   |x|^{-4}\psi_{[d-2,d+2]}(x), \widehat{c_{\alpha}^i}(t,x+\hat{v}t) a^2_{i,l}(v), h_l^\alpha(t)\big) 
+ K \big(\hat{m}_{\alpha}^i(\xi)|\xi|^{-1}, |x|^{-4}\psi_{[d-2,d+2]}(x),    \widehat{c_{\alpha}^i}(t,x+\hat{v}t) a^2_{i,l}(v) , h_l^\alpha(t)\big),
\]
where the symbols $\hat{m}_{\alpha}^i(\xi)$ and the coefficients $\widehat{c^i_\alpha}(t,x+\hat{v}t)$, $i\in\{0,1,2\},$  are defined as follows, 
\be\label{march16eqn2}
\hat{m}_{\alpha}^0(\xi)= \tilde{m}_{\alpha}^0(\xi),\quad  \hat{m}_{\alpha}^1(\xi)=\tilde{m}_{\alpha}^1(\xi) |\xi|,\quad \hat{m}_{\alpha}^2(\xi)=\tilde{m}_{\alpha}^2(\xi) |\xi|^2,\,\, \widehat{c_{\alpha}^i}(t,x+\hat{v}t):= \frac{t}{|x+\hat{v}t|} {\tilde{c}_{\alpha}^{i}}(t,x+\hat{v}t)\psi_{\leq -5}(1-|x+\hat{v}t|/|t|),
\ee
and $a^n_{i,l}(v)$, $i\in\{0,1,2\}, l,n\in\{1,2\}$, are some explicit coefficients that satisfy the following estimate, 
\[
\sum_{i=0,1,2,n,l=1,2}\|a^n_{i,l}(v)\|_{L^\infty_v} +\|(1+|v|)\nabla_v a^n_{i,l}(v) \|_{L^\infty_v} \lesssim 1. 
\]
From (\ref{march16eqn2}) and the estimate (\ref{noveqn141}), we know that the following estimate holds, 
\be\label{jan9eqn3}
\sum_{i=0,1,2} \|\hat{m}_{\alpha}^i(\xi)\|_{\mathcal{S}^\infty_k}\lesssim 2^{-3k}.
\ee 
Recall the definition of $Y$-norm in Lemma \ref{trilinearestimate2}. We have 
\be\label{jan9eqn4}
\| |x|^{-3} \psi_{[d-2,d+2]}(x) \|_{Y }\lesssim 2^{-3d},\quad \| |x|^{-4}\psi_{[d-2,d+2]}(x)\|_{Y }\lesssim 2^{-4d}, \quad  \textup{when\,} d\geq 5.
\ee
 From the above estimates (\ref{jan9eqn3}) and (\ref{jan9eqn4}) and the estimate (\ref{april3eqn1}) for the coefficients $ {\tilde{c}_{\alpha}^{i}}(t,x,v)$, $i\in\{0,1,2\}$, we know  that  the desired estimate (\ref{nov402}) follows directly from the estimate (\ref{jan6eqn20}) in Lemma \ref{trilinearestimate1}, the estimate (\ref{jan6eqn21}) in Lemma \ref{trilinearestimate2}, and the estimate  (\ref{jan9eqn2}) in Lemma \ref{remainderestimatejan6}.
\qed

To sum up, we reduce the proofs of Lemma \ref{firstprop1} and   Lemma \ref{secondprop1}   to the proofs of    Lemma \ref{remainderestimatejan6}, Lemma \ref{trilinearestimate2}, and Lemma \ref{trilinearestimate1}.  We will prove Lemma \ref{remainderestimatejan6} and Lemma \ref{trilinearestimate2}  in next two subsections. For clarity, the proof of  Lemma \ref{trilinearestimate1}, which is more complicated,  is postponed to section \ref{proofofmainlemmalinear}.

 \subsection{Proof of Lemma \ref{remainderestimatejan6}}\label{proofoferrorterm}
In this subsection, we estimate the error term which arises from the process of trading the spatial derivative for the decay of modulations and finish the proof of Lemma \ref{remainderestimatejan6}. The main ingredient of the proof is the following Lemma. 

\begin{lemma}\label{tradeoffregularity}

The following estimate holds for any $j\in\{1,2,3\},$  $i\in\{1,\cdots,5\}$, $\rho\in\mathcal{K}, |\rho|=1$, $d\in \mathbb{N}_{+}, d\geq 5$,
\[
 \| (1+|v|)^{1-c(\rho)}e_{\rho}(t,x,v) E_{k,j}^i[u](t,x+\hat{v}t,v)\varphi_{d}\big( |t|-|x+\hat{v}t|  \big)\psi_{\leq -10}(1-|x+\hat{v}t|/|t|) \|_{L^\infty_{x,v}}
\]
\be\label{nov291} 
  \lesssim  (1+t)^{-1} \big( 2^{-4d-4k} +2^{-2d-2k}\big) 2^{k-4k_{+}}  E_{\textup{low}}^{eb}(t). 
\ee
\end{lemma}
\begin{proof}
   Note that the following estimate holds from the detailed formula of $e_{\rho}(t,x,v)$ in (\ref{sepeq932}),
\be\label{nov318}
\sum_{\rho\in \mathcal{K}, |\rho|=1}\| (1+|v|)^{1-c(\rho)}e_{\rho}(t,x,v)\psi_{\leq -10}(1-|x+\hat{v}t|/|t|)\|_{L^\infty_{x,v}} \lesssim  (1+|t|).
\ee

Recall (\ref{nove613}), (\ref{nove614}), and (\ref{nove643}). From the estimates of coefficients in (\ref{nove353}) and (\ref{april3eqn1}), the estimates of the symbols $\tilde{m}_{k,\alpha}^i(\xi)$ of the linear operator $\tilde{T}_{k,\alpha}^i(\cdot)$ in (\ref{noveqn141}),  and the linear decay estimate (\ref{noveqn555}) in Lemma \ref{twistedlineardecay}, the following estimate holds, 
\be\label{jan9eqn21}
\sum_{u\in\{E,B\}, i=1,2}
 \|   E_{k,j}^i[u](t,x+\hat{v}t) \varphi_{d}\big( |t|-|x+\hat{v}t|  \big) \|_{L^\infty_{x,v}} \lesssim(1+t)^{-2}\big(2^{-2d-2k}+ 2^{-3d-3k} \big) 2^{ k-4k_{+}} E_{\textup{low}}^{eb}(t). 
\ee

Now, we proceed to estimate $E_{k,j}^3[u](t,x+\hat{v}t)$ and $E_{k,j}^4[u](t,x+\hat{v}t)$,$ u\in\{E,B\}$. Recall their detailed formulas in (\ref{nove615}) and (\ref{nove616}) and the detailed formulas of corresponding coefficients in (\ref{jan9eqn10}) and (\ref{jan9eqn11}). From the equality (\ref{april3eqn21}),  the estimates of the symbols $\tilde{m}_{k,\alpha}^i(\xi)$ of the linear operator $\tilde{T}_{k,\alpha}^i(\cdot)$ in (\ref{noveqn141}), the estimate of coefficients in (\ref{nove353}), and the definition of low order energy in  (\ref{secondorderloworder}), we have
\be\label{april4eqn1}
\sum_{u\in\{E,B\}, i=3,4}\|   E_{k,j}^3[u](t,x+\hat{v}t) \varphi_{d}\big( |t|-|x+\hat{v}t|  \big) \|_{L^\infty_{x,v}}    \lesssim (1+|t|)^{-2}(2^{-4d-4k}+2^{-3d-3k}+2^{-2d-2k}) 2^{k-4k_{+}} E_{\textup{low}}^{eb}(t).
\ee

 Lastly, we estimate $E_{k,j}^5[u](t,x+\hat{v}t,v), u\in\{E,B\}$. Recall its detailed formula in (\ref{april4eqn2}). Note that the following equality holds, 
 \[
 e^{i(x+\hat{v}t)\cdot \xi-it \mu|\xi|}  \big(x -t\mu\frac{\xi}{|\xi|}  \big)\cdot \tilde{V}_j = -i\tilde{V}_j\cdot \nabla_\xi\big(  e^{i(x+\hat{v}t)\cdot \xi-it \mu|\xi|}\big). 
 \]
 Hence, after doing integration by parts for $\xi$ in $\tilde{V}_j$ direction, we  have 
  \[
E_{k,j}^5[u](t,x+\hat{v}t,v)= \sum_{\alpha\in \mathcal{B}, |\alpha|\leq 3} \sum_{\mu\in\{+,-\}} \sum_{i=0,1}  \int_{\R^3} e^{i(x+\hat{v}t)\cdot \xi-it \mu|\xi|}\frac{-i}{|x+\hat{v}t|} \tilde{V}_j\cdot \nabla_\xi \big[\big( c_{\mu}\widehat{c}^{ 1}_{\alpha;0}(t, x+\hat{v}t)\tilde{m}_{k,\alpha}^0(\xi)   \]
\[
  - c_{\mu}\widehat{c}^{ 1}_{\alpha;2}(t, x+\hat{v}t) \tilde{m}_{k,\alpha}^2(\xi)|\xi|^2+  \h \widehat{c}^{ 1}_{\alpha;1}(t, x+\hat{v}t) \tilde{m}_{k,\alpha}^1(\xi)|\xi|\big) \widehat{P_{\mu}[h_l^{\alpha}]}(t,\xi) \big] d \xi.
\]
From the above formula,  the detailed formula  of coefficients in (\ref{jan9eqn10}), the estimate of coefficients in (\ref{april3eqn1}), the estimate of symbols in (\ref{noveqn141}), the linear decay estimate (\ref{noveqn555}) in Lemma \ref{twistedlineardecay}, the following estimate holds, 
\be\label{april4eqn31}
\||E_{k,j}^5[u](t,x+\hat{v}t,v)\psi_{\leq -5}(1-|x+\hat{v}t|/|t|)\varphi_{d}\big( |t|-|x+\hat{v}t|  \big)  \|_{L^\infty_{x,v}} \lesssim (1+|t|)^{-2} 2^{-4d-4k} 2^{k-4k_{+}} E_{\textup{low}}^{eb}(t).
\ee

 To sum up, our desired estimate (\ref{nov291}) holds from the estimates (\ref{nov318}), (\ref{jan9eqn21}), (\ref{april4eqn1}), and (\ref{april4eqn31}). 
\end{proof}

\noindent \textit{Proof of   Lemma} \textup{\ref{remainderestimatejan6}}:

Recall (\ref{jan9eqn31}) and (\ref{jan6eqn12}). We use the second decomposition of ``$D_v$'' in (\ref{summaryoftwodecomposition}) in Lemma \ref{twodecompositionlemma}.  From the   estimate (\ref{nov291}) in Lemma \ref{tradeoffregularity}, the second part of the estimate (\ref{feb8eqn51}) in Lemma \ref{derivativeofweightfunction}, and the $L^2_{x,v}-L^2_{x,v}-L^\infty_{x,v}$ type multi-linear estimate, we have
 \be\label{nov261}
 \sum_{l=1,\cdots,5}| \widetilde{\textup{Error}}_{k,d}^l  |  \lesssim  \big( 2^{-3k-3d} +2^{-k-d}\big) 2^{-4k_{+}}  \int_{1}^{t}   (1+|s|)^{-1}  E_{\textup{low}}^{eb}(s)  E_{\beta;d}^{\alpha}(s) d s. 
 \ee
 Hence finishing the proof of the desired estimate (\ref{jan9eqn2}).  
\qed  

\subsection{Proof of Lemma \ref{trilinearestimate2}}\label{errormultilinearestimate}

 In this subsection, we prove our desired estimate (\ref{jan6eqn21}) in  Lemma \ref{trilinearestimate2}.  The main tools that we will use are two linear estimates associated with the bilinear operators defined in (\ref{noveqn90}) and (\ref{noveqn96}), which will be stated in Lemma \ref{estimateremainder1st} and Lemma \ref{auxillaryestimate5linear}. We will first derive these two estimates and then use these two  estimates to prove the desired estimate (\ref{jan6eqn21}).

\begin{lemma}\label{estimateremainder1st}
The following estimate  hold, 
\be\label{noveqn493}
 \|  H_{k}^\mu(\tilde{V}_j\cdot\xi m(\xi), h)(t,x+\hat{v}t,v) \|_{L^\infty_{x,v}}\lesssim \min\{ (1+|t|)^{-1} 2^{ 2 k}  , (1+|t|)^{-2} 2^{  k} \} 2^{-4k_{+}} \|m(\xi)\|_{\mathcal{S}^\infty_k}   E_{\textup{low}}^{eb}(t). 
\ee
\end{lemma}
\begin{proof}
Recall (\ref{noveqn90}). For any fixed $x$ and $v$, we do dyadic decomposition for the angle between $\xi$ and $\mu v$. As a result,   we have
\[
H_{k}^\mu(\tilde{V}_j\cdot\xi m(\xi), h)(t,x+\hat{v}t,v) = \sum_{n\in \mathbb{Z}, n\leq 2}\int_{\R^3} e^{i(x+t\hat{v})\cdot \xi - i t\mu |\xi|} \frac{   i \tilde{V}_j \cdot \xi  m(\xi) \psi_k(\xi) }{\hat{v}\cdot \xi - \mu |\xi|}\]
\[
\times  \psi_{> 10} (    t (|\xi|-\mu\hat{v}\cdot \xi ) )  \p_t \widehat{h }(t, \xi)\psi_k(\xi) \psi_{n}(\angle(\xi, \nu v)) d \xi.
\]
Hence, after  using   the volume of support of $\xi$ and the definition of the low order energy $E_{\textup{low}}^{eb}(\phi)$ in (\ref{secondorderloworder}), we have
\[
| H_{k}^\mu(\tilde{V}_j\cdot\xi m(\xi), h)(t,x+\hat{v}t,v) | \lesssim \sum_{n\in \mathbb{Z}, n\leq 2} 2^{3k+n} \| \p_t \widehat{h}(t, \xi)\psi_k(\xi)\|_{L^\infty_\xi} \|m(\xi)\|_{\mathcal{S}^\infty_k} 
\]
\be\label{noveqn91}
\lesssim \min\{ (1+|t|)^{-1} 2^{ 2 k}  , (1+|t|)^{-2} 2^{  k} \} 2^{-4k_{+}} \|m(\xi)\|_{\mathcal{S}^\infty_k}  E_{\textup{low}}^{eb}(t). 
\ee
\end{proof}
\begin{lemma}\label{auxillaryestimate5linear}
The following estimates hold for any $t\in[2^{m-1}, 2^{m}]$, $m\in\mathbb{Z}_{+}$, $i,j\in\{1,2,3\}$,
\be\label{noveqn93}
\| (1+|v|)^{-1} K_{k}^\mu(\tilde{V}_j\cdot\xi m(\xi), h)(t,x+\hat{v}t,v)\|_{L^\infty_{x,v}} \lesssim 2^{-2m+k-4k_{+}} \| m(\xi)\|_{\mathcal{S}^\infty_k}  E_{\textup{low}}^{eb}(t), 
\ee  
\be\label{jan9eqn41}
\|  K_{k}^\mu(\tilde{V}_j\cdot\xi m(\xi), h)(t,x+\hat{v}t,v) \|_{L^\infty_{x,v}}\lesssim 2^{-m+2k-4k_{+}} \| m(\xi)\|_{\mathcal{S}^\infty_k}  E_{\textup{low}}^{eb}(t) , 
\ee
\be\label{noveqn99}
 \| (X_i \cdot \tilde{v})   K_{k}^\mu(\tilde{V}_j\cdot\xi m(\xi), h)(t,x+\hat{v}t,v)\|_{L^\infty_{x,v}} \lesssim 2^{- m+k-4k_{+}} \| m(\xi)\|_{\mathcal{S}^\infty_k} E_{\textup{low}}^{eb}(t). 
\ee
\end{lemma}
\begin{proof}

Recall (\ref{noveqn96}). 
From the support of $\xi$ and the estimate (\ref{jan4eqn8}), we know that   the size of $|\xi|-\mu\hat{v}\cdot \xi$ is less than $2^{-m}$, which implies that the angle between $\xi$ and $\mu v$ is less than $2^{-m/2-k/2}$ and the size of $v$ is greater than $2^{m/2+k/2}$. 	 As a result, the following estimate holds, 
\[
 (1+|v|)^{-1}| K_{k}^\mu(\tilde{V}_j\cdot\xi m(\xi), h)(t,x+\hat{v}t,v)| \lesssim \sum_{l \leq -m/2-k/2}  2^{-m/2-k/2} 2^{k+l} 2^{3k+2l} \|\widehat{h}  (t, \xi)\psi_k(\xi)\|_{L^\infty_\xi}\|m(\xi)\|_{\mathcal{S}^\infty_k}
\]
\be\label{noveqn92}
  \lesssim  2^{-2m+k-4k_{+}} \|m(\xi)\|_{\mathcal{S}^\infty_k}  E_{\textup{low}}^{eb}(t) .
\ee
Similarly, we have
\be
  | K_{k}^\mu(\tilde{V}_j\cdot\xi m(\xi), h)(t,x+\hat{v}t,v)|   \lesssim \sum_{l \leq -m/2-k/2}    2^{k} 2^{3k+2l} \|\widehat{h} (t, \xi)\psi_k(\xi)\|_{L^\infty_\xi} \| m(\xi)\|_{\mathcal{S}^\infty_k} 
\lesssim 2^{-m+2k-4k_{+}} \| m(\xi)\|_{\mathcal{S}^\infty_k}  E_{\textup{low}}^{eb}(t).
\ee
Hence finishing the proofs of the desired estimates (\ref{noveqn93}) and (\ref{jan9eqn41}).

Lastly, we prove the desired estimate  (\ref{noveqn99}).  Recall (\ref{noveqn96}).  We have 
\[
(X_i \cdot \tilde{v})   K_{k}^\mu(\tilde{V}_j\cdot\xi m(\xi), h)(t,x+\hat{v}t,v)= I_1 + I_2,
 \]
where
\be\label{march30eqn21}
I_1 := \int_{\R^3} e^{i(x+t\hat{v})\cdot \xi - i t\mu |\xi|} i \tilde{V}_j\cdot \xi m(\xi)\big((X_i + t\hat{V}_i )- i t\mu \frac{e_i\times \xi }{|\xi|} \big)\cdot \tilde{v} \big[   - i\mu   \psi_{> 10}' (    t (|\xi|-\mu\hat{v}\cdot \xi ) )\big) +       \psi_{\leq 10}\big(    t (|\xi|-\mu\hat{v}\cdot \xi  )\big) \big]    \widehat{h }(t, \xi) \psi_k(\xi)      d \xi,
\ee
\[I_2 := \int_{\R^3} e^{i(x+t\hat{v})\cdot \xi - i t\mu |\xi|} i \tilde{V}_j\cdot \xi m(\xi)\big(  i t\mu \frac{
e_i\times \xi}{|\xi|} \big)\cdot \tilde{v}    \big[   - i\mu     \psi_{> 10}' (    t (|\xi|-\mu\hat{v}\cdot \xi ) ) 
 +        \psi_{\leq 10}\big(    t (|\xi|-\mu\hat{v}\cdot \xi  )\big) \big] \psi_k(\xi)       \widehat{h }(t, \xi) d \xi.
\]
From the volume of support of ``$\xi$'',  we have
\be\label{jan9eqn44}
 |I_2|\lesssim  \sum_{l\leq -m/2-k/2}   2^{m} 2^{k+2l} 2^{3k+2l}\|\widehat{h} (t, \xi)\psi_k(\xi)\|_{L^\infty_\xi}\|m(\xi)\|_{\mathcal{S}^\infty_k} \lesssim   2^{- m+k-4k_{+}}  \|m(\xi)\|_{\mathcal{S}^\infty_k} E_{\textup{low}}^{eb}(t).   
\ee

Note that
\[
\big((X_i + t\hat{V}_i )- i t\mu \frac{e_i\times \xi}{|\xi|} \big)\cdot \tilde{v} = \big( x+t\hat{v} -it \mu \frac{\xi}{|\xi|} \big)\cdot \tilde{V}_i.
\]
Hence, we can do integration by parts in $\xi$ once in the $\tilde{V}_i$ direction for $I_1$ in (\ref{march30eqn21}). As a result, we have
\[
I_1  = \int_{\R^3} e^{i(x+t\hat{v})\cdot \xi - i t\mu |\xi|}   \tilde{V}_i\cdot \nabla_\xi\big[\tilde{V}_j\cdot \xi m(\xi)  \psi_k(\xi)       \widehat{h }(t, \xi)  \big(   - i\mu  \psi_{> 10}' (    t (|\xi|-\mu\hat{v}\cdot \xi ) ) 
   +      \psi_{\leq 10} (    t (|\xi|-\mu\hat{v}\cdot \xi  )) \big) \big] d \xi.
\]
By using the volume of support of ``$\xi$'', we have
\[
|I_1|\lesssim  \sum_{l\leq -m/2-k/2} \big(2^{m+k+2l} + 1\big) 2^{3k+2l}\| \widehat{h} (t, \xi)\psi_k(\xi)\|_{L^\infty_\xi} \|m(\xi)\|_{\mathcal{S}^\infty_k} + 2^{4k+3l} \|\nabla_\xi \widehat{h} (t, \xi)\psi_k(\xi)\|_{L^\infty_\xi}  \|m(\xi)\|_{\mathcal{S}^\infty_k}
\]
\be\label{jan9eqn47}
 \lesssim   2^{-m+k-4k_{+}} \|m(\xi)\|_{\mathcal{S}^\infty_k} E_{\textup{low}}^{eb}(t). 
\ee

Therefore, our desired estimate      (\ref{noveqn99}) holds from the estimates (\ref{jan9eqn44}) and (\ref{jan9eqn47}). Hence finishing the proof.
 
\end{proof}

 \vo 
 
\noindent \textit{Proof of Lemma \textup{\ref{trilinearestimate2}}:}	

$\bullet$\quad The estimate of $H(m,a,c,h)$. 

Recall (\ref{jan6eqn2}). For this case, we use the second decomposition of ``$D_v$'' (\ref{summaryoftwodecomposition}) in Lemma \ref{twodecompositionlemma} for the term ``$D_v g_{\kappa}^\alpha(t,x,v)$'' in  (\ref{jan6eqn2}). From the   estimate (\ref{nov318}), the second part of the estimate (\ref{feb8eqn51}) in Lemma \ref{derivativeofweightfunction}, and   the estimate (\ref{noveqn493}) in Lemma 
\ref{estimateremainder1st}, the following estimate holds after using the $L^2_{x,v}-L^2_{x,v}-L^\infty_{x,v}$ type multilinear estimate, 
\[
|H(m,a,c,h)|\lesssim   \int_{1}^{t }(1+|s|)2^{d}\|a\|_{Y } \|  H_{k}^\mu(\tilde{V}_j\cdot\xi m(\xi), h)(s,x+\hat{v}s,v) \|_{L^\infty_{x,v}}	 \| c(s,x,v)\|_{L^\infty_{x,v}}  E_{\beta;d}^{\alpha}(s) d s
\]
\be\label{jan10eqn37} 
\lesssim     2^{k+d-4k_{+}} \|a\|_{Y }\|m(\xi)\|_{\mathcal{S}^\infty_k}    \int_{1}^{t }(1+|s|)^{-1} \| c(s,x,v)\|_{L^\infty_{x,v}}   E_{\textup{low}}^{eb}(s) E_{\beta;d}^{\alpha}(s)d s.
\ee

$\bullet$\quad The estimate of $K(m,a,c,h)$. 

Recall (\ref{jan6eqn3}).  For this case, we use the second decomposition of ``$D_v$'' (\ref{summaryoftwodecomposition}) in Lemma \ref{twodecompositionlemma} for the term ``$D_v g_{\kappa}^\alpha(t,x,v)$'' in (\ref{jan6eqn3}). Recall the detailed formulas of $e_{\rho}(t,x,v)$, $\rho\in\mathcal{K},|\rho|=1$, in (\ref{sepeq932}). From the estimates (\ref{noveqn93}), (\ref{jan9eqn41}), (\ref{noveqn99}) in Lemma 
\ref{auxillaryestimate5linear}, the following estimate holds, 
\[
\sum_{\rho\in \mathcal{K}, |\rho|=1}\|(1+|v|)^{1-c(\rho)} e_{\rho}(t,x,v) K_{k}^\mu(\tilde{V}_j\cdot\xi m(\xi), h)(t,x+\hat{v}t,v) \varphi_{d}\big( |t|-|x+\hat{v}t|  \big) \psi_{\leq -10}(1-|x+\hat{v}t|/|t|)\|_{L^\infty_{x,v}}
\]
\be
\lesssim (1+|t|)^{-1} (2^{k }+2^{2k+ d}) 2^{-4k_{+}}\|m(\xi)\|_{\mathcal{S}^\infty_k}  E_{\textup{low}}^{eb}(t).
\ee
Therefore, from the above estimate, the second part of the estimate (\ref{feb8eqn51}) in Lemma \ref{derivativeofweightfunction}, and  the $L^2_{x,v}-L^2_{x,v}-L^\infty_{x,v}$ type multilinear estimate, we have
 \be\label{jan10eqn16}
|K(m,a,c,h)|\lesssim   (2^{k +d}+2^{2k+ 2d}) 2^{-4k_{+}} \|a\|_{Y }\|m(\xi)\|_{\mathcal{S}^\infty_k}    \int_{1}^{t }(1+|s|)^{-1} \| c(s,x,v)\|_{L^\infty_{x,v}}   E_{\textup{low}}^{eb}(s) E_{\beta;d}^{\alpha}(s)d t.
\ee
To sum up,  our desired estimate (\ref{jan6eqn21}) holds from the estimates (\ref{jan10eqn37}) and (\ref{jan10eqn16}). 
\qed

\section{Proof of Lemma \ref{trilinearestimate1}}\label{proofofmainlemmalinear}
The main goal of this section is devoted to prove our last desired estimate (\ref{jan6eqn20}) in Lemma \ref{trilinearestimate1}. Hence finishing the whole argument.

Recall (\ref{jan6eqn1}). To take the advantage of oscillation in time for the electromagnetic field, we do integration by parts in time to move around the time derivative in front of ``$\p_t  T_{k}^\mu(\tilde{V}_j\cdot\xi m(\xi), h)(t,x+\hat{v}t)$''.  The proof of Lemma \ref{trilinearestimate1}  will be very complicated by the fact that many terms will be created when ``$\p_t$'' hits $g_{\beta}^\alpha(t,x,v)$ or $D_v g_{\kappa}^\alpha(t,x,v)$,   e.g., see the equation satisfied by $\p_t g_{\beta}^\alpha(t,x,v)$ in (\ref{sepeqn43}).

For clarity,   we first classify those terms. More precisely, after doing integration by parts in time for (\ref{jan6eqn1}), we have
\be\label{jan23eqn16}
T(m,a,c,h)= \tilde{T}_1(m,a,c,h) + \tilde{T}_2(m,a,c,h) + \textup{Error} , 
\ee
where
\[
\tilde{T}_1(m,a,c,h)=  \sum_{
\begin{subarray}{c}
 \Lambda^\iota \sim \psi_{\geq 1}(|v|) \widehat{\Omega}^v_j \textup{or}\,\psi_{\geq 1}(|v|)\Omega_j^x  \\
\end{subarray}}  - 	\int_{1}^{t  } \int_{\R^3}\int_{\R^3}  \big(\omega_{\beta}^{\alpha}(s, x, v)\big)^2 \p_s g^\alpha_\beta (s ,x,v) \big(\sqrt{1+|v|^2}\tilde{d}(s, x,v)\big)^{1-c(\iota)}   \]
\be\label{jan10eqn7100}
\times    C_{d}(s,x,v)     T_{k}^\mu(\tilde{V}_j\cdot\xi m(\xi), h)(s,x+\hat{v}s,v)  \alpha_i(v) \cdot D_v  
g^\alpha_\kappa(s,x, v) d x d v d s,
\ee

\[
\tilde{T}_2(m,a,c,h)=   \sum_{
\begin{subarray}{c}
 \Lambda^\iota \sim \psi_{\geq 1}(|v|) \widehat{\Omega}^v_j \textup{or}\,\psi_{\geq 1}(|v|)\Omega_j^x  \\
\end{subarray}}  - 	\int_{1  }^{t } \int_{\R^3}\int_{\R^3}  \big(\omega_{\beta}^{\alpha}(s, x, v)\big)^2  g^\alpha_\beta (s  ,x,v) \big(\sqrt{1+|v|^2}\tilde{d}(s, x,v)\big)^{1-c(\iota)}    \]
\be\label{jan10eqn71}
\times    C_{d}(s,x,v) T_{k}^\mu(\tilde{V}_j\cdot\xi m(\xi), h)(s,x+\hat{v}s,v) \alpha_i(v) \cdot \p_s\big(D_v  
g^\alpha_\kappa(s,x, v)\big) d x d v d s,
\ee
\[
\textup{Error} =    \sum_{
\begin{subarray}{c}
 \Lambda^\iota \sim \psi_{\geq 1}(|v|) \widehat{\Omega}^v_j \textup{or}\,\psi_{\geq 1}(|v|)\Omega_j^x  \\
\end{subarray}}  - 	\int_{1}^{t } \int_{\R^3}\int_{\R^3}  \big(\omega_{\beta}^{\alpha}(s, x, v)\big)^2   g^\alpha_\beta (s ,x,v)  T_{k}^\mu(\tilde{V}_j\cdot\xi m(\xi), h)(s,x+\hat{v}s,v)\]
\[
\times \p_s\big[ \big(\sqrt{1+|v|^2}\tilde{d}(s, x,v)\big)^{1-c(\iota)}   C_{d}(s,x,v)\big]   \alpha_i(v) \cdot D_v  
g^\alpha_\kappa(s,x, v)  d x d v d s, 
\]
\[
+ \sum_{i=1,2,t_1=1,t_2=t} (-1)^{i} \int_{\R^3}\int_{\R^3}  \big(\omega_{\beta}^{\alpha}(t_i, x, v)\big)^2    T_{k}^\mu(\tilde{V}_j\cdot\xi m(\xi), h)(t_i,x+\hat{v}t_i,v)\big(\sqrt{1+|v|^2}\tilde{d}(t_i, x,v)\big)^{1-c(\iota)}    \]
\be\label{jan10eqn181}
\times C_{d}(t_i,x,v) g^\alpha_\beta (t_i ,x,v)
    \alpha_i(v) \cdot D_v  
g^\alpha_\kappa(t_i,x, v) d x d v.
\ee

Recall the equation satisfied by $g_\beta^\alpha(t,x,v)$ in (\ref{sepeqn43}),  the classification of $\textit{h.o.t}_{\beta}^\alpha$ in decompositions  (\ref{sepeqn180}),  (\ref{sepeq90}), and (\ref{nove100}). We decompose $\tilde{T}_1(m,a,c,h)$ into four parts 	as follows, 
\be\label{jan23eqn17}
\tilde{T}_1(m,a,c,h)= \sum_{i=1,\cdots,4} \tilde{T}_1^i(m,a,c,h), 
\ee
where
\[
\tilde{T}_1^1(m,a,c,h)=      \sum_{
\begin{subarray}{c}
 \Lambda^\iota \sim \psi_{\geq 1}(|v|) \widehat{\Omega}^v_j \\
 \textup{or}\,\psi_{\geq 1}(|v|)\Omega_j^x  \\
\end{subarray}} \int_{1}^{t } \int_{\R^3}\int_{\R^3}   \big(\omega_{\beta}^{\alpha}( s,x, v)\big)^2  \big(  K(s, x+\hat{v}s,v)\cdot D_v g_\beta^\alpha(s,x,v) \alpha_i(v) \cdot D_v  
g^\alpha_\kappa(s ,x, v) \big)     \]
 \be\label{jan11eqn1}
 \times   C_{d}(s,x,v) \big(\sqrt{1+|v|^2}\tilde{d}(s , x,v)\big)^{1-c(\iota)}          T_{k}^\mu(\tilde{V}_j\cdot\xi m(\xi), h)(s,x+\hat{v}s ,v) d x d v d s, 
\ee
 \[
\tilde{T}_1^2(m,a,c,h)=      \sum_{
\begin{subarray}{c}
 \Lambda^\iota \sim \psi_{\geq 1}(|v|) \widehat{\Omega}^v_j \\
 \textup{or}\,\psi_{\geq 1}(|v|)\Omega_j^x  \\
\end{subarray}} \int_{1 }^{t } \int_{\R^3}\int_{\R^3}  -\big(\omega_{\beta}^{\alpha}(s, x, v)\big)^2\alpha_i(v) \cdot D_v  
g^\alpha_\kappa(s ,x, v)\big(\sqrt{1+|v|^2}\tilde{d}(s , x,v)\big)^{1-c(\iota)}        \]
\be\label{jan11eqn63}
 \times     C_{d}(s,x,v)    T_{k}^\mu(\tilde{V}_j\cdot\xi m(\xi), h)(s,x+\hat{v}s,v)     \textit{bulk}_{\beta}^\alpha(s,x,v)   d x d v d s , 
\ee
  \[
\tilde{T}_1^3(m,a,c,h)= \sum_{
\begin{subarray}{c}
 \Lambda^\iota \sim \psi_{\geq 1}(|v|) \widehat{\Omega}^v_j \\
 \textup{or}\,\psi_{\geq 1}(|v|)\Omega_j^x  \\
\end{subarray}} \int_{1}^{t } \int_{\R^3}\int_{\R^3}  -\big(\omega_{\beta}^{\alpha}( s,x, v)\big)^2\alpha_i(v) \cdot D_v  
g^\alpha_\kappa(s ,x, v)  \big(\sqrt{1+|v|^2}\tilde{d}(s , x,v)\big)^{1-c(\iota)}    \]
\be\label{jan11eqn60}
 \times     C_{d}(s,x,v)     T_{k}^\mu(\tilde{V}_j\cdot\xi m(\xi), h)(s,x+\hat{v}s,v)   \big(   \textit{h.o.t}_{\beta}^\alpha(s,x,v) -  \textit{bulk}_{\beta}^\alpha(s,x,v)\big)  d x d v d s, 
\ee
\[
\tilde{T}_1^4(m,a,c,h)=     \sum_{
\begin{subarray}{c}
 \Lambda^\iota \sim \psi_{\geq 1}(|v|) \widehat{\Omega}^v_j \\
 \textup{or}\,\psi_{\geq 1}(|v|)\Omega_j^x  \\
\end{subarray}}  \int_{1}^{t } \int_{\R^3}\int_{\R^3}  -\big(\omega_{\beta}^{\alpha}(s, x, v)\big)^2  \alpha_i(v) \cdot D_v  
g^\alpha_\kappa(s ,x, v)     \big(\sqrt{1+|v|^2}\tilde{d}(s  , x,v)\big)^{1-c(\iota)}   \]
\be\label{jan11eqn61}
 \times       C_{d}(s,x,v)        T_{k}^\mu(\tilde{V}_j\cdot\xi m(\xi), h)(s,x+\hat{v}s,v)     \textit{l.o.t}_{\beta}^\alpha(s,x,v) d x d v d s.
\ee
 It remains to classify terms inside $\tilde{T}_2(m,a,c,h)$. Recall (\ref{jan10eqn71}).  Note that
 \be\label{jan10eqn135}
[D_v, \p_t] = [\nabla_v - t \nabla_v\hat{v}\cdot \nabla_x, \p_t] = \nabla_v \hat{v}\cdot \nabla_x, \Longrightarrow \p_t(D_v g_{\kappa}^\alpha )= D_v(\p_t g_{\kappa}^\alpha ) - \nabla_v \hat{v}\cdot \nabla_x g_{\kappa}^\alpha. 
 \ee

Recall the equation satisfied by $\p_t g_{\beta}^\alpha(t,x,v)$ in (\ref{sepeqn43}).  Since the most problematic term inside $\textit{h.o.t}_{\beta}^\alpha(t,x,v)$ is $\textit{bulk}_{\beta}^\alpha(t,x,v)$, we first study the structure of  ``$D_v \textit{bulk}_{\beta}^\alpha(t,x,v)$'', which is summarized in the following Lemma. 
\begin{lemma}\label{bulkdecomposition}
The following equality holds, 
\be\label{march17eqn25}
D_v( \textit{bulk}_{\kappa}^\alpha(t,x,v))= \widetilde{bulk}_{\kappa}^{\alpha }(t,x,v)+ \widetilde{error}_{\kappa}^{\alpha }(t,x,v),
\ee
where the detailed formula of  $ \widetilde{bulk}_{\kappa}^{\alpha }(t,x,v) $ and $ \widetilde{error}_{\kappa}^{\alpha }(t,x,v)$ are given in \textup{(\ref{jan12eqn1})} and \textup{(\ref{jan16eqn100})}respectively. 
\end{lemma}
\begin{proof}
Recall (\ref{nove120}). We have
\[
D_v( \textit{bulk}_{\kappa}^\alpha(t,x,v))= \widetilde{bulk}_{\kappa}^{\alpha  }(t,x,v)+ \widetilde{error}_{\kappa}^{\alpha }(t,x,v),
\]
where
\be\label{jan12eqn1}
\widetilde{bulk}_{\kappa}^{\alpha }(t,x,v)=\sum_{\begin{subarray}{l}
j=1,2,3,
i=1,\cdots,7 
\end{subarray}} \sum_{
\begin{subarray}{c}
\iota'+\kappa'=\kappa,\iota', \kappa'\in \mathcal{S},
 |\iota'|=1, \Lambda^{\iota'} \sim \psi_{\geq 1}(|v|)\widehat{\Omega}^v_j \textup{or}\,\psi_{\geq 1}(|v|) \Omega_j^x 
\end{subarray}}  K_{\iota';1}^i(t,x+\hat{v}t,v)\alpha_i(v)\cdot D_v D_v g^\alpha_{\kappa'}(t,x, v), 
\ee
\be\label{jan12eqn4}
 \widetilde{error}_{\kappa}^{\alpha  }(t,x,v)= \sum_{\begin{subarray}{l}
j=1,2,3,
i=1,\cdots,7 
\end{subarray}} \sum_{
\begin{subarray}{c}
\iota'+\kappa'=\kappa,\iota', \kappa'\in \mathcal{S},
 |\iota'|=1, \Lambda^{\iota'} \sim \psi_{\geq 1}(|v|) \widehat{\Omega}^v_j \textup{or}\,\psi_{\geq 1}(|v|) \Omega_j^x  \\
\end{subarray}}  D_v\big( K_{\iota';1}^i(t,x+\hat{v}t,v)\alpha_i(v)\big) \cdot D_v  g^\alpha_{\kappa'}(t,x, v). 
\ee
Note that we used the fact that $[D_{v_m}, D_{v_n}]=0,$ for any $ m, n\in \{1,2,3\}.$ After using the first decomposition of ``$D_v$'' in (\ref{summaryoftwodecomposition}) in Lemma \ref{twodecompositionlemma} and the detailed formula of $K_{\iota';1}^i(t,x+\hat{v}t, v)$ in (\ref{nove99}),   we have 
\[
 \widetilde{error}_{\kappa}^{\alpha  }(t,x,v)= \sum_{\begin{subarray}{l}
j=1,2,3,
i=1,\cdots,7\\
\rho_1, \rho_2\in \mathcal{K}, |\rho_1|=|\rho_2|=1
\end{subarray}} \sum_{
\begin{subarray}{c}
\iota'+\kappa'=\kappa,\iota', \kappa'\in \mathcal{S}  \\
 |\iota'|=1, \Lambda^{\iota'} \sim \psi_{\geq 1}(|v|) \widehat{\Omega}^v_j \textup{or}\,\psi_{\geq 1}(|v|) \Omega_j^x  \\
\end{subarray}}    d_{\rho_1}(t,x,v) d_{\rho_2}(t,x,v) \Lambda^{\rho_2}(g_{\kappa'}^\alpha(t,x,v))\]
\[
 \times \big[\Lambda^{\rho_1} \big(\alpha_i(v)\big(\sqrt{1+|v|^2}\tilde{d}(t,x,v) \big)^{1-c(\iota)}\alpha_i(v)\big)\cdot \Omega_j^x \big(E(t,x+\hat{v}t) + \hat{v}\times B(t,x+\hat{v}t) \big)
\]
\be\label{jan16eqn100}
 + \alpha_i(v)  \big(\sqrt{1+|v|^2}\tilde{d}(t,x,v) \big)^{1-c(\iota)}\alpha_i(v) \cdot  \Lambda^{\rho_1} \big(\Omega_j^x \big(E(t,x+\hat{v}t) + \hat{v}\times B(t,x+\hat{v}t) \big)\big]. 
\ee
 \end{proof}

We will show that both ``$D_v(\textit{h.o.t}_{\kappa}^\alpha(t,x,v)-\textit{bulk}_{\kappa}^\alpha(t,x,v) )$'' and ``$D_v(\textit{l.o.t}_{\kappa}^\alpha(t,x,v))$'' are \textit{non-bulk  terms}. Motivated from this expectation and  the decomposition (\ref{march17eqn25}) in Lemma \ref{bulkdecomposition},  we decompose $ \tilde{T}_2  (m,a,c,h)$ into five parts as follows, 
\be\label{jan23eqn19}
\tilde{T}_2(m,a,c,h)=\sum_{i=1,\cdots,5} \tilde{T}_2^i(m,a,c,h),
\ee
where
\[
\tilde{T}_2^1(m,a,c,h)=    \sum_{
\begin{subarray}{c}
 \Lambda^\iota \sim \psi_{\geq 1}(|v|) \widehat{\Omega}^v_j \\
 \textup{or}\,\psi_{\geq 1}(|v|)\Omega_j^x  \\
\end{subarray}}   	\int_{1}^{t } \int_{\R^3}\int_{\R^3}  \big(\omega_{\beta}^{\alpha}(s, x, v)\big)^2  g^\alpha_\beta (s ,x,v) \big(\sqrt{1+|v|^2}\tilde{d}(s, x,v)\big)^{1-c(\iota)} C_{d}(s,x,v)\]
\be\label{jan10eqn162}
\times     T_{k}^\mu(\tilde{V}_j\cdot\xi m(\xi), h)(s,x+\hat{v}s,v)  \alpha_i(v) \cdot D_v\big( 
 K(s,x+\hat{v}s,v)\cdot D_v g_{\kappa}^\alpha(s,x,v) \big) d x d v d s,
\ee

\[
\tilde{T}_2^2(m,a,c,h)=   \sum_{
\begin{subarray}{c}
 \Lambda^\iota \sim \psi_{\geq 1}(|v|) \widehat{\Omega}^v_j \\
 \textup{or}\,\psi_{\geq 1}(|v|)\Omega_j^x  \\
\end{subarray}}   	\int_{1}^{t } \int_{\R^3}\int_{\R^3} - \big(\omega_{\beta}^{\alpha}(s, x, v)\big)^2  g^\alpha_\beta (s ,x,v) \big(\sqrt{1+|v|^2}\tilde{d}(s, x,v)\big)^{1-c(\iota)} C_{d}(s,x,v) \]
\be\label{jan10eqn163}
\times   T_{k}^\mu(\tilde{V}_j\cdot\xi m(\xi), h)(s,x+\hat{v}s,v) \alpha_i(v) \cdot  \widetilde{bulk}_{\kappa}^{\alpha }(s,x,v)  d x d v d s,
\ee
\[
\tilde{T}_2^3(m,a,c,h)=  \sum_{
\begin{subarray}{c}
 \Lambda^\iota \sim \psi_{\geq 1}(|v|) \widehat{\Omega}^v_j \\
 \textup{or}\,\psi_{\geq 1}(|v|)\Omega_j^x  \\
\end{subarray}}   	\int_{t_1}^{t_2} \int_{\R^3}\int_{\R^3} - \big(\omega_{\beta}^{\alpha}( s, x, v)\big)^2  g^\alpha_\beta (s ,x,v) \big(\sqrt{1+|v|^2}\tilde{d}(s, x,v)\big)^{1-c(\iota)}    \]
\be\label{jan10eqn164}
 \times   C_{d}(s,x,v)  T_{k}^\mu(\tilde{V}_j\cdot\xi m(\xi), h)(s,x+\hat{v}s,v) \alpha_i(v) \cdot  D_v\big(  \textit{h.o.t}_{\kappa}^\alpha(s,x,v) - \textit{bulk}_{\kappa}^\alpha(s,x,v)\big)    d x d v d s,
\ee
\[
\tilde{T}_2^4(m,a,c,h)= \sum_{
\begin{subarray}{c}
 \Lambda^\iota \sim \psi_{\geq 1}(|v|) \widehat{\Omega}^v_j \\
 \textup{or}\,\psi_{\geq 1}(|v|)\Omega_j^x  \\
\end{subarray}} 	\int_{1}^{t } \int_{\R^3}\int_{\R^3} - \big(\omega_{\beta}^{\alpha}(s, x, v)\big)^2  g^\alpha_\beta (t ,x,v) \big(\sqrt{1+|v|^2}\tilde{d}(s, x,v)\big)^{1-c(\iota)}  \]
\be\label{jan10eqn161}
\times  C_{d}(s,x,v)     T_{k}^\mu(\tilde{V}_j\cdot\xi m(\xi), h)(s,x+\hat{v}s,v)  \alpha_i(v) \cdot  D_v\big( \textit{l.o.t}_{\kappa}^\alpha(s,x,v) \big)   d x d v d s,
\ee
\[
\tilde{T}_2^5(m,a,c,h)=  \sum_{
\begin{subarray}{c}
 \Lambda^\iota \sim \psi_{\geq 1}(|v|) \widehat{\Omega}^v_j \\
 \textup{or}\,\psi_{\geq 1}(|v|)\Omega_j^x  \\
\end{subarray}}  	\int_{1}^{t } \int_{\R^3}\int_{\R^3}   \big(\omega_{\beta}^{\alpha}( s, x, v)\big)^2  g^\alpha_\beta (s ,x,v) \big(\sqrt{1+|v|^2}\tilde{d}(s, x,v)\big)^{1-c(\iota)}   C_{d}(s,x,v) \]
\be\label{jan11eqn91}
\times    T_{k}^\mu(\tilde{V}_j\cdot\xi m(\xi), h)(s,x+\hat{v}s,v) \alpha_i(v) \cdot \big[  \nabla_v \hat{v}\cdot \nabla_x g_\kappa^\alpha(s,x,v) -     \widetilde{error}_{\kappa}^{\alpha }(s,x,v)  \big] d x d v d s.
\ee

Since there are many terms to be controlled to finish the proof of Lemma \ref{trilinearestimate1}, for the sake of readers, we provide the following plan of this section. 

\begin{enumerate}
\item[$\bullet$] Under the assumption that the $L^\infty_x$-type decay estimate (\ref{nove501}) in Lemma \ref{lineardecaylemma3}  holds for  the operator $T_k^\mu(\cdot, \cdot)$ and the validity of Lemma \ref{bulktermbilinearestimate1}, we finish the proof of Lemma \ref{trilinearestimate1} by estimating terms in the decompositions (\ref{jan23eqn16}), (\ref{jan23eqn17}), and (\ref{jan23eqn19}) one by one in subsection \ref{feb172020prf1}. 

\item[$\bullet$]Under the assumption that the $L^\infty_x$-type decay estimate (\ref{nove501}) in Lemma \ref{lineardecaylemma3}  holds, we finish the proof of Lemma  \ref{bulktermbilinearestimate1} in subsection \ref{proofofbulklemma}. 

\item[$\bullet$] We finish the proof of Lemma \ref{lineardecaylemma3} in subsection \ref{feb172020prf2}. Hence complete the whole proof.
\end{enumerate}

\subsection{Proof of Lemma \ref{trilinearestimate1}}\label{feb172020prf1}

 Recall the decomposition in  (\ref{jan23eqn16}). As summarized in the following Lemma,  the estimate of the error term ``$\textup{Error}$'' in  (\ref{jan10eqn503}) holds.   

\begin{lemma}\label{jan23errorestimate}
The following estimate holds, 
\[
|\textup{Error} |\lesssim  
 \big(2^{k+d}+2^{2k+2d}\big)  2^{ -4k_{+}}\|a\|_{Y } \|m(\xi)\|_{\mathcal{S}^\infty_k}\big[ \sum_{\tau \in\{1,t \}} \|c(\tau ,x,v)\|_{L^\infty_{x,v}}  E_{\beta;d}^{\alpha}(\tau ) E_{\textup{low}}^{eb}(\tau )    + \int_{1}^{t } (1+|s|)^{-1}
 \]
\be\label{jan10eqn503}
 \times \big( \|c(s,x,v)\|_{L^\infty_{x,v}} + s\|\p_s c(s,x,v)\|_{L^\infty_{x,v}}   \big) E_{\beta;d}^{\alpha}(s) E_{\textup{low}}^{eb}(s) d s\big].
\ee
\end{lemma}
\begin{proof}
Recall (\ref{jan10eqn181}). Since there are many possible destinations of the time derivative ``$\p_t$'', we first classify terms inside the error term ``$\textup{Error}$''. Note that
\be\label{jan10eqn201}
\p_t d(t,x,v) = \p_t \tilde{d}(t,x,v)=(1+|v|^2)^{-1},\quad \p_t\big(|t|-|x+\hat{v}t| \big) =\p_t\big(\frac{\frac{t^2}{1+|v|^2}-2tx\cdot \hat{v}-|x|^2}{|t|+|x+\hat{v}t|}\big)=\sum_{i=1,2,3}c_i(t,x,v),
\ee
where
\[
c_1(t,x,v):= - \frac{|t|-|x+\hat{v}t|}{ |t|+|x+\hat{v}t| }\big(\frac{t}{|t|} +\frac{(x+\hat{v}t)\cdot \hat{v}}{|x+\hat{v}t|}\big) + \frac{-2x\cdot \hat{v}(|t|-|x+\hat{v}t|)}{\frac{t^2}{1+|v|^2}-2tx\cdot \hat{v}-|x|^2 }\big[ \psi_{\geq 20}(x\cdot \hat{v}(1+|v|^2)/ |t|  \big)\mathbf{1}_{[-10,\infty)}(x\cdot \hat{v}) 
\]
\[
+ \psi_{\geq 20}(x\cdot \hat{v}(1+|v|^2)/ |t|  \big)\mathbf{1}_{(-\infty,-10) }(x\cdot \hat{v})\psi_{\geq 2}((1+|t|)/|x||v|)\big],
\]
\[
c_2(t,x,v) = \frac{2t }{(1+|v|^2)\big(|t|+|x+\hat{v}t|\big)} + \frac{-2x\cdot\hat{v}}{|t|+|x+\hat{v}t|} \psi_{< 20}(x\cdot \hat{v}(1+|v|^2)/ |t| \big),
\]
\be\label{april5eqn81}
c_3(t,x,v)= \frac{-2x\cdot\hat{v}}{|t|+|x+\hat{v}t|} \psi_{\geq 20}(x\cdot \hat{v}(1+|v|^2)/|t|\big)\mathbf{1}_{(-\infty,-10)}(x\cdot \hat{v})\psi_{< 2}((1+|t|)/|x||v|).
\ee
From the above detailed formula, we know that the following estimate holds, 
\be\label{april5eqn71}
 |c_1(t,x,v)| \lesssim \frac{1+||t|-|x+\hat{v}t||}{|t|}, \quad |c_2(t,x,v)|\lesssim \frac{1}{1+|v|^2}. 
\ee
Hence, from the equalities in (\ref{jan10eqn201}) and the detailed formula of $C_d(t,x,v)$ in (\ref{2020feb16eqn11}),  we can decompose the error term into three parts as follows, 
\be\label{april5eqn90}
\textup{Error} = \textup{Error}_{ 1}+ \textup{Error}_{ 2}+ \textup{Error}_{ 3},
\ee
where
\[
 \textup{Error}_{ 1}=   \sum_{
\begin{subarray}{c}
 \Lambda^\iota \sim \psi_{\geq 1}(|v|) \widehat{\Omega}^v_j \\
 \textup{or}\,\psi_{\geq 1}(|v|)\Omega_j^x  \\
\end{subarray}}   - 	\int_{1}^{t } \int_{\R^3}\int_{\R^3}  \big(\omega_{\beta}^{\alpha}(s , x, v)\big)^2   g^\alpha_\beta (s ,x,v)  T_{k}^\mu(\tilde{V}_j\cdot\xi m(\xi), h)(s,x+\hat{v}s,v) \big(\sqrt{1+|v|^2}\tilde{d}(s, x,v)\big)^{1-c(\iota)}   \]
\[
\times \psi_{\geq 1}(|v|) \big[     \p_s\big( c(s,x,v)   \psi_{\leq -10}(1-|x+\hat{v}s|/|s|)  \big)a(||s|-|x+\hat{v}s||) \varphi_{d} (||s|-|x+\hat{v}s||) + c_1(s,x,v)   \big( a'(||s|-|x+\hat{v}s||)\]
\[
\times  \varphi_{d} (||s|-|x+\hat{v}s||)+ 2^{-d}a(||s|-|x+\hat{v}s||)  \varphi_{d}'  (||s|-|x+\hat{v}s||) \big)  \big]
 \alpha_i(v) \cdot D_v  
g^\alpha_\kappa(s,x, v) d x d v d s 
\]
\[
+ \sum_{i=1,2,t_1=1,t_2=t} (-1)^{i} \int_{\R^3}\int_{\R^3}  \big(\omega_{\beta}^{\alpha}( t_i, x, v)\big)^2    T_{k}^\mu(\tilde{V}_j\cdot\xi m(\xi), h)(t_i,x+\hat{v}t_i,v)\big(\sqrt{1+|v|^2}\tilde{d}(t_i, x,v)\big)^{1-c(\iota)}c(t_i,x,v)  \psi_{\geq 1}(|v|) \]
\[
\times    a(||t_i|-|x+\hat{v}t_i||)   \psi_{\leq -10}(1-|x+\hat{v} t_i|/| t_i |) \varphi_{d}(||t_i|-|x+\hat{v}t_i||)  g^\alpha_\beta (t_i ,x,v)
    \alpha_i(v) \cdot D_v  
g^\alpha_\kappa(t_i,x, v) d x d v,
\]
\[
 \textup{Error}_{ 2}=   \sum_{
\begin{subarray}{c}
 \Lambda^\iota \sim \psi_{\geq 1}(|v|) \widehat{\Omega}^v_j \\
 \textup{or}\,\psi_{\geq 1}(|v|)\Omega_j^x  \\
\end{subarray}} - 	\int_{1}^{t } \int_{\R^3}\int_{\R^3}  \big(\omega_{\beta}^{\alpha}( s,x, v)\big)^2   g^\alpha_\beta (s ,x,v)  T_{k}^\mu(\tilde{V}_j\cdot\xi m(\xi), h)(s,x+\hat{v}s,v)     \psi_{\leq -10}(1-|x+\hat{v}s|/|s|)  	\]
\[
\times \psi_{\geq 1}(|v|)   c(s,x,v)\big(\sqrt{1+|v|^2}\big)^{1-c(\iota)}\big[
\p_s\big(\tilde{d}(s, x,v)\big)^{1-c(\iota)}  a(||s|-|x+\hat{v}s||) \varphi_{d}(||s|-|x+\hat{v}s||)+c_2(s,x,v) \big(\tilde{d}(s, x,v)\big)^{1-c(\iota)} \]
\[
\times   \big( a'(||s|-|x+\hat{v}s||)\varphi_{d} (||s|-|x+\hat{v}s||)+ 2^{-d}a(||s|-|x+\hat{v}s||)  \varphi_{d}'  (||s|-|x+\hat{v}s||) \big) \big]
  \alpha_i(v) \cdot D_v  
g^\alpha_\kappa(s,x, v) d x d v d s.
\]
\[
 \textup{Error}_{3}=   \sum_{
\begin{subarray}{c}
 \Lambda^\iota \sim \psi_{\geq 1}(|v|) \widehat{\Omega}^v_j \\
 \textup{or}\,\psi_{\geq 1}(|v|)\Omega_j^x  \\
\end{subarray}}  - 	\int_{1}^{t } \int_{\R^3}\int_{\R^3}  \big(\omega_{\beta}^{\alpha}(s, x, v)\big)^2   g^\alpha_\beta (s ,x,v)  T_{k}^\mu(\tilde{V}_j\cdot\xi m(\xi), h)(s,x+\hat{v}s,v)     \psi_{\leq -10}(1-|x+\hat{v}s|/|s|)  	\]
\[
\times   \psi_{\geq 1}(|v|) c(s,x,v) \big(\sqrt{1+|v|^2}\tilde{d}(s, x,v)\big)^{1-c(\iota)} 
 c_3(s,x,v)   \big[ a'(||t|-|x+\hat{v}s||)   \varphi_{d} (||t|-|x+\hat{v}s||)+ 2^{-d}a(||t|-|x+\hat{v}s||)  \]
 \be\label{april5eqn76}
 \times  \varphi_{d}'  (||s|-|x+\hat{v}s||) \big]
  \alpha_i(v) \cdot D_v  
g^\alpha_\kappa(s,x, v) d x d v d s.
\ee
Very importantly, recall that $a'(x)=\varphi_d'(x)=0$ if $|x|\leq 2^{-10}$, we can localize away from zero.

For the first part of error term ``$\textup{Error}_{1}$'', we use the second decomposition of ``$D_v$'' in (\ref{summaryoftwodecomposition}) in Lemma \ref{twodecompositionlemma}.  From the estimate (\ref{april5eqn71}) and  the estimate (\ref{nove501}) in Lemma \ref{lineardecaylemma3}, the following estimate holds, 
\[
|\textup{Error}_{ 1}|\lesssim  \sum_{\tau\in\{1,t \}} \|a\|_{Y } 	\big[\sum_{\rho\in \mathcal{K}, |\rho|=1} 2^{d}\|(1+|v|)^{1-c(\rho)} e_{\rho}(\tau,x,v)  T_{k}^\mu(\tilde{V}_j\cdot\xi m(\xi), h)(\tau,x+\hat{v}\tau,v) \|_{L^\infty_{x,v}} \big]  \|c(\tau  ,x,v)\|_{L^\infty_{x,v}}   E_{\beta;d}^{\alpha}(\tau )
\]
\[
  + \int_{1}^{t } \big[\sum_{\rho\in \mathcal{K}, |\rho|=1} 2^{d}\|(1+|v|)^{1-c(\rho)} e_{\rho}(s,x,v)  T_{k}^\mu(\tilde{V}_j\cdot\xi m(\xi), h)(s,x+\hat{v}s,v) \|_{L^\infty_{x,v}} \big] (1+|s|)^{-1}\|a\|_{Y }E_{\beta;d}^{\alpha}(s ) 
\]
\[
  \times  	 \big( \|c(s,x,v)\|_{L^\infty_{x,v}}  + s\|\p_s c(s,x,v)\|_{L^\infty_{x,v}}   \big) d s
\lesssim \big(2^{k+d}+2^{2k+2d}\big)  2^{ -4k_{+}}\|a\|_{Y } \|m(\xi)\|_{\mathcal{S}^\infty_k}\big[ \sum_{\tau\in\{1,t \}}   \|c(\tau,x,v)\|_{L^\infty_{x,v}}   
\]
\be\label{jan10eqn501}
\times  E_{\beta;d}^{\alpha}(\tau) E_{\textup{low}}^{eb}(\tau) + \int_{1}^{t } (1+|s|)^{-1}  \big( \|c(s,x,v)\|_{L^\infty_{x,v}}+ s\|\p_s c(s,x,v)\|_{L^\infty_{x,v}}   \big)  E_{\textup{low}}^{eb}(s)  E_{\beta;d}^{\alpha}(s) d s \big] .
\ee

For the second part of error term ``$\textup{Error}_{2}$'', we use the first decomposition of $D_v$ in (\ref{summaryoftwodecomposition}) in Lemma \ref{twodecompositionlemma}.  From the estimate of coefficients in (\ref{jan15eqn2}) in Lemma \ref{twodecompositionlemma},  the estimates(\ref{jan10eqn201}) and (\ref{april5eqn71}), and the estimate  (\ref{noveqn500}) in Lemma \ref{twisteddecaylemma3}, the following estimate holds, 
\be\label{jan10eqn502}
 |\textup{Error}_{ 2}|\lesssim   2^{ k+ d-4k_{+}}\|a\|_{Y }  \|m(\xi)\|_{\mathcal{S}^\infty_k}  \int_{1}^{t } (1+|s|)^{-1}  \|c(s,x,v)\|_{L^\infty_{x,v}}E_{\textup{low}}^{eb}(s)  E_{\beta;d}^{\alpha}(s) d s. 
 \ee

Lastly, we estimate $\textup{Error}_3$. Recall (\ref{april5eqn76}). For this case, we use the first decomposition of ``$D_v$'' in (\ref{summaryoftwodecomposition}) in Lemma \ref{twodecompositionlemma}. Recall the detailed formula of ``$d_{\rho}(t,x,v)$'' in (\ref{sepeq947}) and  the detailed formula of $c_3(t,x,v)$ in (\ref{april5eqn81}). From the equality (\ref{march18eqn54}), we have
\be\label{april12eqn81}
\big(\frac{| \tilde{d}(t,x,v)c_3(t,x,v)|}{||t|-|x+\hat{v}t||}   \big) \psi_{\leq -5}(1-|x+\hat{v}t|/|t|)   \lesssim \frac{ t}{t+(1+|v|)\big(|x\cdot v|+|x|\big)}\frac{|x\cdot \hat{v}|}{|t|+|x+\hat{v}t|} \lesssim \frac{1}{1+|v|^2}.
\ee
Recall the definition of $\phi(t,x,v)$ in (\ref{april2eqn1}). The following estimate holds for any fixed $x,v\in \textit{supp}(c_3(t,x,v)\psi_{\geq 1}(|v|))$, 
\be\label{april5eqn82}
\frac{|c_3(t,x,v)|\psi_{\geq 1}(|v|)}{\phi(t,x,v)} \lesssim \frac{|x|}{|t|(1+|v|)}.
\ee
From the above estimates (\ref{april12eqn81}) and (\ref{april5eqn82}), the estimate (\ref{noveqn600}) in Lemma \ref{twisteddecaylemma3}, and  the estimate (\ref{noveqn532}) in Lemma \ref{auxillarylemmajan26},   the following estimate holds 
\[
 |\textup{Error}_{3}|\lesssim \int_{1}^{t }2^{d} \big( \| \frac{1}{1+|v|} T_{k}^\mu(\tilde{V}_j\cdot\xi m(\xi), h)(s,x+\hat{v}s,v)\|_{L^\infty_{x,v}}+ \| \frac{|x|}{s} T_{k}^\mu(\tilde{V}_j\cdot\xi m(\xi), h)(s,x+\hat{v}s,v)\|_{L^\infty_{x,v}}\big) \|a\|_Y \|c(s,x,v)\|_{L^\infty_{x,v}}
\]
\be\label{april5eqn91}
\times   E_{\beta;d}^{\alpha}(s) d s \lesssim   2^{ k+ d-4k_{+}}\|a\|_{Y }  \|m(\xi)\|_{\mathcal{S}^\infty_k}  \int_{1}^{t } (1+|s|)^{-1}  \|c(s,x,v)\|_{L^\infty_{x,v}}E_{\textup{low}}^{eb}(s) E_{\beta;d}^{\alpha}(s) d s,
\ee

To sum up, recall the decomposition (\ref{april5eqn90}),  our desired estimate  (\ref{jan10eqn503}) holds from the estimates (\ref{jan10eqn501}), (\ref{jan10eqn502}) and (\ref{april5eqn91}). 
\end{proof}
\begin{lemma}\label{jan23lemma2}
The following estimate holds, 
\[
|\tilde{T}_1^1(m,a,c,h) + \tilde{T}_2^1(m,a,c,h) |\lesssim   (2^{k+d}+2^{2k+2d})2^{-4k_{+}} \| a\|_{Y } \| m (\xi)\|_{\mathcal{S}^\infty_k}  \int_{1}^{t } (1+|s|)^{-1}   \big(  E_{\textup{low}}^{eb}(s) \big)^2 E_{\beta;d}^{\alpha}(s )
\] 
 \be\label{jan11eqn31}
 \times  \big( \|c(s,x,v)\|_{L^\infty_{x,v}}+\|  D_v c(s,x,v)\|_{L^\infty_{x,v}}\big)   d s. 
\ee 

\end{lemma}

 \begin{proof}
 Recall (\ref{jan11eqn1}), (\ref{jan10eqn162}), and   (\ref{sepeqn61}). As a result of direct computations, we have
\[
[D_{v_m}, D_{v_n}]=0,\quad D_v\cdot K(t,x+\hat{v} t ,v)=0, \quad m, n=1,2,3.
\]
Hence, we have
\[
  K(t, x+\hat{v} t,v)\cdot D_v g_\beta^\alpha(t,x,v)  \alpha_i(v) \cdot D_v  
g^\alpha_\kappa(t ,x, v) 
+ g_{\beta}^\alpha(t,x,v)\alpha_i(v)\cdot D_v\big( K(t,x+\hat{v} t,v)\cdot D_v g_{\kappa}^\alpha(t,x ,v) \big) 
\]
\[
= \sum_{m,n=1,2,3} K_n(t, x+\hat{v} t,v)  D_{v_n} g_\beta^\alpha(t,x,v)     (\alpha_i(v))_m    D_{v_m}  
g^\alpha_\kappa(t ,x, v) +  g_{\beta}^\alpha(t,x,v)(\alpha_i(v))_m   
\]
\[
\times  D_{v_m} K_n(t, x+\hat{v} t,v)  D_{v_n} g_\kappa^\alpha(t,x,v)  
 +  g_{\beta}^\alpha(t,x,v)(\alpha_i(v))_m    K_n(t, x+\hat{v} t,v)  D_{v_n}D_{v_m}  g_\kappa^\alpha(t,x,v)
\]
\[
=  g_{\beta}^\alpha(t,x,v) \big[\big(\alpha_i(v)\cdot \nabla_v \hat{v}\big)\times B(t,x+\hat{v}t)-  \big(K(t,x+\hat{v} t,v)\cdot \nabla_v \alpha_i(v)\big)\big]\cdot D_v g_{\kappa}^\alpha(t,x,v)  
\]
\be\label{jan11ean5}
 + D_{v}\cdot \big[K (t,x+\hat{v} t,v) g^\alpha_\beta(t,x,v) \alpha_i(v)\cdot D_v g_{\kappa}^\alpha(t,x,v) \big] .
\ee

From the above equality, the following equality holds after doing
integration by parts in $x$ and $v$ to 
move the derivatives ``$D_v$'' outside  ``$K (t,x+\hat{v} t, v) g^\alpha_\beta(t,x,v) \alpha_i(v)\cdot D_v g_{\kappa}^\alpha(t,x,v)$'' in  (\ref{jan11ean5}),
\be\label{jan11eqn53}
\tilde{T}_1^1(m,a,c,h) + \tilde{T}_2^1(m,a,c,h):=\sum_{
\begin{subarray}{c}
l=1,\cdots,4,
 j=1,2,3, 
i=1,\cdots,7
\end{subarray}} \sum_{
\begin{subarray}{c}
\iota+\kappa=\beta,\iota,   \kappa\in \mathcal{S}, 
 |\iota|=  1, \Lambda^\iota \sim \psi_{\geq 1}(|v|)\widehat{\Omega}^v_j \textup{or}\,\psi_{\geq 1}(|v|) \Omega_j^x  \\
\end{subarray}}  I_{\iota, \kappa,i,j}^l, 
\ee
where
\[
I_{\iota, \kappa,i,j}^1=\int_{1}^{t  } \int_{\R^3}\int_{\R^3}   \big[ K(s,x+\hat{v}s,v)  g_\beta^\alpha(s,x,v)  \alpha_i(v) \cdot D_v g^\alpha_\kappa(s ,x, v) \big]\cdot D_v\big((\omega_{\beta}^\alpha(s,x,v))^2\big)  \big(\sqrt{1+|v|^2}\tilde{d}(s , x,v)\big)^{1-c(\iota)} 
\]
\be\label{nove16}
\times C_d(s,x,v) T_{k}^\mu(\tilde{V}_j\cdot\xi m(\xi), h)(s,x+\hat{v}s,v)     d x d v d s,
\ee
\[
I_{\iota, \kappa,i,j}^2=  \int_{1}^{t } \int_{\R^3}\int_{\R^3}   \big[(\omega_{\beta}^\alpha(s,x,v))^2   g_\beta^\alpha(s,x,v)  \alpha_i(v) \cdot D_v g^\alpha_\kappa(s ,x, v) \big]  K(s,x+\hat{v} s,v)\cdot  D_v\big[  \big(\sqrt{1+|v|^2}\tilde{d}(s , x,v)\big)^{1-c(\iota)}   
\]
\be\label{jan11eqn21}
\times  C_d(s,x,v)  \big] T_{k}^\mu(\tilde{V}_j\cdot\xi m(\xi), h)(s,x+\hat{v}s,v)     d x d v ds ,
\ee
\[
I_{\iota, \kappa,i,j}^3=   \int_{1}^{t } \int_{\R^3}\int_{\R^3}   \big[(\omega_{\beta}^\alpha(s,x,v))^2   g_\beta^\alpha(s,x,v)  \alpha_i(v) \cdot D_v g^\alpha_\kappa(s ,x, v) \big] \big(\sqrt{1+|v|^2}\tilde{d}(s , x,v)\big)^{1-c(\iota)}  
\]
\be\label{jan11eqn23}
\times  C_d(s,x,v)    K(s  ,x+\hat{v} s,v)\cdot  D_v\big(  T_{k}^\mu(\tilde{V}_j\cdot\xi m(\xi), h)(s,x+\hat{v}s,v) \big)    d x d v d s,
\ee
\[
I_{\iota, \kappa,i,j}^4= \int_{1}^{t } \int_{\R^3}\int_{\R^3} (\omega_{\beta}^\alpha(s,x,v))^2 \big[\big(\alpha_i(v)\cdot \nabla_v \hat{v}\big)\times B(s,x+\hat{v}s)-  \big(K(s,x+\hat{v} s,v)\cdot \nabla_v \alpha_i(v)\big)\big]\cdot D_v g_{\kappa}^\alpha(s,x,v)  g_{\beta}^\alpha(s,x,v)
\]
\be\label{jan11eqn24}
\times \big(\sqrt{1+|v|^2}\tilde{d}(s , x,v)\big)^{1-c(\iota)}      C_d(s,x,v)   T_{k}^\mu(\tilde{V}_j\cdot\xi m(\xi), h)(s,x+\hat{v}s,v)      d x d v d s.
\ee
 
 \noindent $\bullet$\quad The estimate of $I_{\iota, \kappa,i,j}^1$.

Recall (\ref{nove16}). For the term ``$D_v g_\kappa^\alpha(t,x,v)$'' in  ``$I_{\iota, \kappa,i,j}^1$'', we use the second decomposition of $D_v$ in (\ref{summaryoftwodecomposition}).  From the   estimate (\ref{feb8eqn51}) in Lemma \ref{derivativeofweightfunction}, the decay estimate (\ref{noveqn78}) in Lemma \ref{sharpdecaywithderivatives},  and the estimate (\ref{nove501}) in Lemma \ref{lineardecaylemma3}, the following estimate holds from the $L^2_{x,v}-L^2_{x,v}-L^\infty_{x,v}$ type multilinear estimate, 
\[
|I_{\iota, \kappa,i,j}^1| \lesssim \sum_{
\begin{subarray}{c}
\rho\in \mathcal{K},\kappa\in \mathcal{S}, \gamma\in \mathcal{B}, \\
|\rho|=1 
\end{subarray}
}  \int_{1}^{t }  2^{d}  \big\|\frac{D_v\omega_\beta^\alpha(s,x,v)}{\omega_\beta^\alpha(s,x,v)} K(s,x+\hat{v} s,v)\big\|_{L^\infty_x L^\infty_v} \| \omega_{\kappa}^\gamma(s,x, v) g_{\kappa}^\gamma(s, x,v) \varphi_{[d-1,d+1]}  (||s|-|x+\hat{v}s||)   \|_{L^2_x L^2_v }^2  \]
\[
\times  \| a \|_{Y }\|c(s,x,v)\|_{L^\infty_{x,v}}  \|(1+|v|)^{1-c(\rho)} e_{\rho}(t,x,v) \varphi_{d}  (||s|-|x+\hat{v}s||) T_{k}^\mu(\tilde{V}_j\cdot\xi m(\xi), h)(s,x+\hat{v}s ,v)  \|_{L^\infty_{x,v}} ds \]
\be\label{nove73}
\lesssim  \big( 2^{d+k} +2^{2d+2k} \big)2^{-4k_{+}}\| a\|_{Y }\|m(\xi)\|_{\mathcal{S}^\infty_k}  \int_{1}^{t }       (1+|s|)^{-1}\|c(s,x,v)\|_{L^\infty_{x,v}} \big(E_{\textup{low}}^{eb}(s)\big)^2   E_{\beta;d}^{\alpha}(s) d s. 
\ee

 \noindent $\bullet$\quad The estimate of $I_{\iota, \kappa,i,j}^2$.

 Recall (\ref{jan11eqn21}).   From the second equality in  (\ref{dec26eqn1}),  the estimate of coefficients in (\ref{jan26eqn101}), the second part of the estimate (\ref{feb8eqn51}) in Lemma \ref{derivativeofweightfunction}, we know that the following estimate holds, 
\[
\big(\phi(t,x,v)\big)^{1-c(\iota)}\big|  D_v\big[  \big( \tilde{d}(t , x,v)\big)^{1-c(\iota)}  \psi_{\geq 1}(|v|)
  a(||t|-|x+\hat{v}t||)     \varphi_{d}(||t|-|x+\hat{v}t||) \psi_{\leq -10}(1-|x+\hat{v}t|/|t|) c(t,x,v) \big]\big|
  \]
 \be\label{jan11eqn41}
  \lesssim   2^{d} \| a\|_{Y } \big( \|c(t,x,v)\|_{L^\infty_{x,v}}+\|  D_v c(t,x,v)\|_{L^\infty_{x,v}}\big). 
 \ee
For the term ``$D_v g_\kappa^\alpha(t,x,v)$'' in  ``$I_{\iota, \kappa,i,j}^2$'', we use the second decomposition of $D_v$ in (\ref{summaryoftwodecomposition}). From the estimate (\ref{jan11eqn41}), the decay estimate (\ref{noveqn78}) in Lemma \ref{sharpdecaywithderivatives}, and the estimate (\ref{nove501}) in Lemma \ref{lineardecaylemma3}, the following estimate holds from the $L^2_x L^2_v-L^2_x L^2_v-L^\infty_{x} L^\infty_v$ type multilinear estimate, 
\[
|I_{\iota, \kappa,i,j}^2|\lesssim  \sum_{\kappa\in \mathcal{S}, \gamma\in \mathcal{B}, |\kappa|+|\gamma|\leq  N_0}\sum_{u\in\{E,B\}} \int_{1}^{t } (2^{k+d}+2^{2k+2d})2^{-4k_{+}}   \| a\|_{Y } \|\omega_{\kappa}^\gamma(s,x, v) g_{\kappa}^\gamma(s, x,v) \varphi_{[d-1,d+1]} (||s|-|x+\hat{v}s||)  \|_{L^2_x L^2_v }^2 
\]
\[
\times   \| u(s,x+\hat{v} s  )  \|_{L^\infty_{x,v}} \| m (\xi)\|_{\mathcal{S}^\infty_k} E_{\textup{low}}^{eb}(s) \big( \|c(s,x,v)\|_{L^\infty_{x,v}}+\|  D_v c(s,x,v)\|_{L^\infty_{x,v}}\big) d s
\lesssim    (2^{k+d}+2^{2k+2d})2^{-4k_{+}}   
\] 
\be\label{nove51}
  \times  \| a\|_{Y_d} \| m (\xi)\|_{\mathcal{S}^\infty_k}  \int_{1}^{t }  (1+|s|)^{-1}    \big( \|c(s,x,v)\|_{L^\infty_{x,v}}+\|  	D_v c(s,x,v)\|_{L^\infty_{x,v}}\big)  \big(E_{\textup{low}}^{eb}(s)\big)^2   E_{\beta;d}^{\alpha}(s)   d s. 
\ee

 \noindent $\bullet$\quad The estimate of $I_{\iota, \kappa,i,j}^3$.

  Recall   (\ref{jan11eqn23}) and (\ref{noveqn89}). We know  that the following equality holds,  
\be\label{nove4}
D_v\big(  T_{k}^\mu(\tilde{V}_j\cdot\xi m(\xi), h)(t,x+\hat{v}t,v) \big) =  \int_{\R^3} e^{i(x+t\hat{v})\cdot \xi - i t\mu |\xi|} \nabla_v\big[\frac{    -i \tilde{V}_j\cdot\xi m(\xi)  \psi_k(\xi)}{\hat{v}\cdot \xi - \mu |\xi|}\psi_{> 10}\big(    t (|\xi|-\mu\hat{v}\cdot \xi  ) \big] \widehat{h }(t, \xi)    d \xi.
\ee
Note that 
\be\label{nove1}
\nabla_v= \tilde{v}\tilde{v}\cdot \nabla_v + \tilde{V}_i \tilde{V}_i \cdot \nabla_v, \quad \tilde{V}_i\cdot\nabla_v (\hat{v}\cdot \xi)= \frac{1}{\sqrt{1+|v|^2}}\tilde{V}_i \cdot \xi, \quad \tilde{v}\cdot \nabla_v( \hat{v}\cdot \xi)= \frac{\tilde{v}\cdot \xi}{(1+|v|^2)^{3/2}}.
\ee
Recall (\ref{nove4}). From the above equalities (\ref{nove1}) and the decay estimates (\ref{jan25eqn16}) and (\ref{noveqn600}) in Lemma \ref{twisteddecaylemma3}, we have
\be\label{nove8}
\big\|(1+|v|) D_v\big(  T_{k}^\mu(\tilde{V}_j\cdot\xi m(\xi), h)(t,x+\hat{v}t,v) \big) \big\|_{L^\infty_{x,v}} \lesssim 2^{k-4k_{+}} \| m(\xi)\|_{\mathcal{S}^\infty_k} E_{\textup{low}}^{eb}(t).
\ee
   For the term ``$D_v g_\kappa^\alpha(t,x,v)$'' in  ``$I_{\iota, \kappa,i,j}^3$'', we use the first decomposition of $D_v$ in (\ref{summaryoftwodecomposition}).
From the above estimate (\ref{nove8}), the estimate of coefficients in (\ref{jan15eqn2}) in Lemma \ref{summaryoftwodecomposition}, the decay estimate (\ref{noveqn78}) in Lemma \ref{sharpdecaywithderivatives}, and  the $L^2_x L^2_v $--$L^2_x L^2_v $--$L^\infty_{x} L^\infty_v$ type multilinear estimate, the following estimate holds, 
\[
|I_{\iota, \kappa,i,j}^3|\lesssim \sum_{\kappa\in \mathcal{S}, \gamma\in \mathcal{B}, |\kappa|+|\gamma|\leq  N_0} \sum_{u\in\{E,B\}} \int_{1}^{t }  2^{ d}  \|\omega_{\kappa}^\gamma(s,x, v) g_{\kappa}^\gamma(s, x,v) \varphi_{[d-1,d+1]}  (||s|-|x+\hat{v}s||)   \|_{L^2_x L^2_v }^2\| a\|_{Y }    \|c(s,x,v)\|_{L^\infty_{x,v}}
\]
\[
\times  \|u(s,x+\hat{v}s )(1+||s|-|x+\hat{v}s||)\|_{L^\infty_{x,v}}  \big\|(1+|v|) D_v\big(  T_{k}^\mu(\tilde{V}_j\cdot\xi m(\xi), h)(s,x+\hat{v}s,v) \big) \big\|_{L^\infty_{x,v}}  ds 
\]
\be\label{nove61}
\lesssim     \int_{1}^{t }  (1+|s|)^{-1}      2^{d+k-4k_{+}}   \| a\|_{Y } \|m(\xi)\|_{\mathcal{S}^\infty_k}  \|c(s,x,v)\|_{L^\infty_{x,v}} \big(E_{\textup{low}}^{eb}(t)\big)^2  E_{\beta;d}^{\alpha}(s) d s.
\ee

 \noindent $\bullet$\quad The estimate of $I_{\iota, \kappa,i,j}^4$.

 Recall (\ref{jan11eqn24}).  As a result of direct computation, we have  
 \[
\big| (1+|v|) \nabla_v \hat{v}\big|   + \big| (1+|v|) \nabla_v \alpha_i(v)\big|  \lesssim 1.
 \]
   For the term ``$D_v g_\kappa^\alpha(t,x,v)$'' in  ``$I_{\iota, \kappa,i,j}^4$'', we use the first decomposition of ``$D_v$'' in (\ref{summaryoftwodecomposition}). From the $L^2_{x,v}-L^2_{x,v}-L^\infty_{x,v}$ type multilinear estimate, the estimate of coefficients in (\ref{jan15eqn2}) in Lemma \ref{summaryoftwodecomposition},   the decay estimate (\ref{noveqn78}) in Lemma \ref{sharpdecaywithderivatives}, and the estimate (\ref{jan25eqn16}) in Lemma \ref{twisteddecaylemma3},  we have 
   \[
|I_{\iota, \kappa,i,j}^4|\lesssim \sum_{\kappa\in \mathcal{S}, \gamma\in \mathcal{B}, |\kappa|+|\gamma|\leq  N_0} \sum_{u\in\{E,B\}}   \int_{1}^{t }     2^{ d}\|a\|_{Y }  \|\omega_{\kappa}^\gamma(s,x, v) g_{\kappa}^\gamma(s, x,v) \varphi_{[d-1,d+1]}( ||s|-|x+\hat{v}s||) \|_{L^2_x L^2_v }^2\| c(s,x,v)\|_{L^\infty_{x,v}}
      \]
\[
\times  \| u(s,x+\hat{v}s )\big(1+||s|-|x+\hat{v}s||\big) \|_{L^\infty_{x,v}}   \| T_{k}^\mu(\tilde{V}_j\cdot\xi m(\xi), h)(s,x+\hat{v}s, v)    \|_{L^\infty_{x,v}} d s
\]
\be\label{jan11eqn50}
\lesssim    2^{d+k-4k_{+}} \|a\|_{Y}  \|m(\xi)\|_{\mathcal{S}^\infty_k}    \int_{1}^{t } (1+|s|)^{-1}     \| c(s,x,v)\|_{L^\infty_{x,v}}  \big(E_{\textup{low}}^{eb}(s)\big)^2 E_{\beta;d}^{\alpha}(s)d s. 
\ee

To sum up, our desired estimate (\ref{jan11eqn31}) holds   from the decomposition (\ref{jan11eqn53}) and   the estimates (\ref{nove73}), (\ref{nove51}), (\ref{nove61}), and (\ref{jan11eqn50}). 
\end{proof}

\begin{lemma}\label{jan23lemma5}
The following estimate holds for any fixed $i\in\{1,2\}$ and $j\in\{3,4\}$,
\be\label{jan11eqn59}
  |\tilde{T}_i^j(m,a,c,h)|\lesssim   \big(2^{k+d}+2^{2k+2d}\big)2^{-4k_{+}} \|a\|_{Y } \| m(\xi)\|_{\mathcal{S}^\infty_k } \int_{1}^{t }  (1+|s|)^{-1}   \| c(s,x,v)\|_{L^\infty_{x,v}} E_{\textup{low}}^{eb}(s) E_{\beta;d}^{\alpha}(s)  d t.
\ee
\end{lemma}
\begin{proof}
Recall (\ref{jan11eqn60}), (\ref{jan11eqn61}), (\ref{jan10eqn164}), and (\ref{jan10eqn161}). For these terms,   we use the second decomposition of ``$D_v$'' in (\ref{summaryoftwodecomposition}) in Lemma \ref{twodecompositionlemma}. From the $L^2_{x,v}-L^2_{x,v}-L^\infty_{x,v}$ type multilinear estimate, the second part of the estimate (\ref{feb8eqn51}) in Lemma \ref{derivativeofweightfunction}, and the estimate (\ref{nove501}) in Lemma \ref{lineardecaylemma3}, we know that the following estimate holds, 
\[
  |\tilde{T}_i^j(m,a,c,h)| \lesssim \sum_{
\begin{subarray}{c}
\iota,   \kappa, \rho \in \mathcal{S},\iota+\kappa=\beta \\
|\rho|=|\iota|=1    \\
\end{subarray}} \int_{1}^{t }    2^{ d}     \|\omega_{\beta}^\alpha(s,x, v) g_{\beta}^\alpha( s, x,v) \varphi_{[d-1,d+1]}(||s|-|x+\hat{v}s||) \|_{L^2_x L^2_v }    \|a\|_{Y }\| c(s,x,v)\|_{L^\infty_{x,v}}
\]
\[
\times \|(1+|v|)^{1-c(\rho)} e_{\rho}(s,x,v)\varphi_{d}(||s|-|x+\hat{v}s||) T_{k}^\mu(\tilde{V}_j\cdot\xi m(\xi), h)(s,x+\hat{v}s,v)    \|_{L^\infty_{x,v}}   \big[ \|\omega_{\beta}^\alpha(s,x, v)\varphi_{[d-1,d+1]}( ||s|-|x+\hat{v}s||)   \]
\[
\times \big(\textit{h.o.t}_{\beta}^\alpha(s,x,v) -  \textit{bulk}_{\beta}^\alpha(s,x,v)\big)  \|_{L^2_x L^2_v }+\|\omega_{\beta}^\alpha(s,x, v)  \big(\textit{l.o.t}_{\beta}^\alpha(s,x,v)  \big) \varphi_{[d-1,d+1]}(||s|-|x+\hat{v}s||) \|_{L^2_x L^2_v }
\]
\[
+\|\omega_{\rho\circ\kappa}^\alpha(s,x, v) \Lambda^\rho \big(\textit{h.o.t}_{\kappa}^\alpha(s,x,v) -  \textit{bulk}_{\kappa}^\alpha(s,x,v)\big) \varphi_{[d-1,d+1]}(||s|-|x+\hat{v}s||) \|_{L^2_x L^2_v } \]
\[
+  \|\omega_{\rho\circ\kappa}^\alpha(s,x, v) \Lambda^\rho \big(\textit{l.o.t}_{\kappa}^\alpha(s,x,v) \big)\varphi_{[d-1,d+1]}(||s|-|x+\hat{v}s ||) \|_{L^2_x L^2_v } \big] d s
\]
\be
\lesssim       \big(2^{k+d}+2^{2k+2d}\big)2^{-4k_{+}} \|a\|_{Y } \| m(\xi)\|_{\mathcal{S}^\infty_k }\int_{1}^{t }  (1+|s|)^{-1}   \| c(s,x,v)\|_{L^\infty_{x,v}} E_{\textup{low}}^{eb}(s)  E_{\beta;d}^{\alpha}(s) d s.
\ee
Hence finishing the proof of the desired estimate (\ref{jan11eqn59}). 
\end{proof}

\begin{lemma}\label{jan23lemma9}
The following estimate holds, 
\be\label{jan23eqn55}
 | \tilde{T}_2^5(m,a,c,h)|  \lesssim   2^{d+k-4k_{+}}\|a\|_{Y }\| m(\xi)\|_{\mathcal{S}^\infty_k}   \int_{1}^{t } (1+|s|)^{-1}  \| c(s,x,v)\|_{L^\infty_{x,v}} E_{\textup{low}}^{eb}(s)
   E_{\beta;d}^{\alpha}(s)(1+E_{\textup{low}}^{eb}(s)) d s. 
\ee
\end{lemma}
\begin{proof}
Recall  (\ref{jan11eqn91}) and  (\ref{jan16eqn100}). Note that, as a result of direct computations,  the following equality holds, 
\[
\nabla_v \hat{v}\cdot \nabla_x = \frac{\tilde{v}}{(1+|v|^2)^{3/2}} S^x + \frac{\tilde{V}_i}{(1+|v|^2)^{1/2}} \Omega_i^x. 
\]
Recall the first equality in  (\ref{dec26eqn1}) in Lemma \ref{derivativesofcoefficient} and the equality (\ref{sepeqn610}) in Lemma \ref{decompositionofderivatives}. From the estimate of coefficients in (\ref{jan23eqn11}) and (\ref{jan15eqn2}), the second part of the estimate (\ref{feb8eqn51}) in Lemma \ref{derivativeofweightfunction}, the decay estimate  (\ref{noveqn78}) in Lemma \ref{sharpdecaywithderivatives},  the estimate (\ref{noveqn500}) in Lemma \ref{twisteddecaylemma3}, and the $L^2_{x,v}-L^2_{x,v}-L^\infty_{x,v}$ type multilinear estimate,    we have 
\[
| \tilde{T}_2^5(m,a,c,h)| \lesssim 
\sum_{\alpha\in \mathcal{B}, \beta\in \mathcal{S}, |\alpha|+|\beta|\leq N_0}  \int_{1}^{t }   2^{ d}\|a\|_{Y } \|\omega_{\beta}^\alpha(s,x, v) g_{\beta}^\alpha(s, x,v) \varphi_{[d-1,d+1]}( || s|-|x+\hat{v}s||) \|_{L^2_x L^2_v }^2    \| c(s,x,v)\|_{L^\infty_{x,v}}
\]
\[
\times \big[ \sum_{\rho\in \mathcal{B}, |\rho|\leq 4} \sum_{u\in\{E^\rho,B^\rho\}} 1+ \|(1+  ||s|-|x+\hat{v}s|| )^2\nabla_x u  (s,x+\hat{v}s)\|_{L^\infty_{x,v}}  \big]   
     \| \frac{1}{1+|v|} T_{k}^\mu(\tilde{V}_j\cdot\xi m(\xi), h)(s,x+\hat{v}s,v) \|_{L^\infty_{x,v}}   d s
\]
\[
\lesssim    2^{d+k-4k_{+}}\|a\|_{Y }\| m(\xi)\|_{\mathcal{S}^\infty_k}  \int_{1}^{t }  (1+|s|)^{-1} \| c(s,x,v)\|_{L^\infty_{x,v}}    E_{\beta;d}^{\alpha}(s)E_{\textup{low}}^{eb}(s)(1+E_{\textup{low}}^{eb}(s)) d s. 
\]
In the above estimate, we used the fact that we can gain at least $(1+|v|)^{-3}$ from the hierarchy of the weight functions when estimating the error term $ \widetilde{error}_{\kappa}^{\alpha  }(t,x,v)$. 
\end{proof}

\begin{lemma}\label{bulktermbilinearestimate1}
The following estimate holds, 
\[
| \tilde{T}_1^2(m,a,c,h)|  + | \tilde{T}_2^2(m,a,c,h)| \lesssim  \big(2^{k/2+d/2}+2^{2k +2d } \big) 2^{-4k_{+} }  \|m(\xi)\|_{\mathcal{S}^\infty_k}\|a\|_{Y }
\]
\be\label{jan14eqn9}
 \times \big[ \int_{1}^{t }  (1+|s |)^{-1} \|c(s,x,v)\|_{L^\infty_{x,v}} \big(E_{\textup{low}}^{eb}(s)\big)^2   E_{\beta;d}^{\alpha}(s) d s \big]. 
\ee 
\end{lemma}
\begin{proof}
Postponed  to subsection \ref{proofofbulklemma}.  
\end{proof}

\emph{Assuming the validity of  Lemma \ref{bulktermbilinearestimate1} holds}, we finish the proof of the desired estimate (\ref{jan6eqn20}) in Lemma \ref{trilinearestimate1}.

\noindent $\textit{Proof of Lemma \textup{\ref{trilinearestimate1}}}$

Recall the decompositions of $T(m,a,c,h)$ in (\ref{jan23eqn16}), (\ref{jan23eqn17}), and (\ref{jan23eqn19}). Our  desired estimate   (\ref{jan6eqn20}) in Lemma \ref{trilinearestimate1} follows directly from the estimate (\ref{jan10eqn503}) in Lemma \ref{jan23errorestimate}, the estimate (\ref{jan11eqn31}) in Lemma \ref{jan23lemma2}, the estimate  (\ref{jan11eqn59}) in Lemma \ref{jan23lemma5}, the estimate (\ref{jan23eqn55}) in Lemma \ref{jan23lemma9}, and the estimate (\ref{jan14eqn9}) in Lemma 	\ref{bulktermbilinearestimate1}. 

\qed

\subsection{Proof of Lemma \textup{\ref{bulktermbilinearestimate1}}}\label{proofofbulklemma}

 Recall the detailed formula of $\tilde{T}_1^2(m,a,c,h)$ in  (\ref{jan11eqn63}) and  the detailed formula of \textit{bulk} terms in  (\ref{nove120}). Since  new \textit{bulk} terms are introduced because of the integration by parts in time process, there is an issue of  losing another weight of size ``$|v|$'' caused by the new introduced \textit{bulk} terms.

  To get around this issue, intuitively speaking,  we observe that there exists a hidden null structure inside a bilinear form  of the  type ``$\Omega_j^x u_1(t,x+\hat{v}t) \Omega_i^x u_2(t,x+\hat{v}t)$'', $ u_1,u_2\in\{E,B\}$.  To better explain this observation,  we first do dyadic decompositions for   $E(t, x+\hat{v} t ) $ and   $B(t, x+\hat{v} t ) $ inside $\tilde{T}_1^2(m,a,c,h)$. As a result, we have,
\be\label{nov1}
\tilde{T}_1^2(m,a,c,h)=  \sum_{k_1\in \mathbb{Z}} K_{k_1,k}^d,   
\ee
where
\[
K_{k_1,k}^d:= \sum_{\begin{subarray}{l}
 j'=1,2,3\\
 i'=1,\cdots,7\\
\end{subarray}}   \sum_{
\begin{subarray}{c}
  \Lambda^\iota \sim \psi_{\geq 1}(|v|)\widehat{\Omega}^v_j\\  \textup{or}\,\psi_{\geq 1}(|v|) \Omega_j^x  \\
\end{subarray}} \sum_{
\begin{subarray}{c}
\iota'+\kappa'=\beta,\iota,  \kappa\in \mathcal{S},  |\iota'|= 1  \\
  \Lambda^{\iota'} \sim \psi_{\geq 1}(|v|)\widehat{\Omega}^v_{j'}\textup{or}\,\psi_{\geq 1}(|v|)\Omega_{j'}^x  \\
\end{subarray}}   \int_{1}^{t } \int_{\R^3}\int_{\R^3}  \big(\omega_{\beta}^{\alpha}(s, x, v)\big)^2 \alpha_{i'}(v)\cdot  D_v    g^\alpha_{\kappa'}(s,x, v)    \]
\[
 \times   \alpha_{i'}(v)\cdot \Omega_{j'}^x \big(E_{k_1}(s,x+\hat{v}s) + \hat{v}\times B_{k_1}(s,x+\hat{v}s) \big)\alpha_i(v)\cdot  D_v g^\alpha_\kappa(s ,x, v)   \big(\sqrt{1+|v|^2}\tilde{d}(s , x,v)\big)^{2-c(\iota)-c(\iota')} \]
 \be\label{nove141}
 \times   
   C_d(s,x,v)  T_{k}^\mu(\tilde{V}_j\cdot\xi m(\xi), h)(s,x+\hat{v}s, v )  d x d v d s,
\ee

Recall (\ref{jan10eqn163}),  (\ref{jan12eqn1}) and (\ref{nove99}). Similarly, we   do dyadic decomposition for   $E(t, x+\hat{v} t ) $ and   $B(t, x+\hat{v} t ) $ and  have the following decompositions for $\tilde{T}_2^2(m,a,c,h)$,
\be\label{jan14eqn1}
\tilde{T}_2^2(m,a,c,h)=  \sum_{k_1\in \mathbb{Z}} S_{k_1,k}^d, 
\ee
where
\[
S_{k_1,k}^d:=   \sum_{\begin{subarray}{l}
 j'=1,2,3\\
 i'=1,\cdots,7\\
\end{subarray}}   \sum_{
\begin{subarray}{c}
  \Lambda^\iota \sim \psi_{\geq 1}(|v|)\widehat{\Omega}^v_j\\  \textup{or}\,\psi_{\geq 1}(|v|) \Omega_j^x  \\
\end{subarray}} \sum_{
\begin{subarray}{c}
\iota'+\kappa'=\beta,\iota,  \kappa\in \mathcal{S},  |\iota'|= 1  \\
  \Lambda^{\iota'} \sim \psi_{\geq 1}(|v|)\widehat{\Omega}^v_{j'}\textup{or}\,\psi_{\geq 1}(|v|)\Omega_{j'}^x  \\
\end{subarray}}    \int_{1}^{t } \int_{\R^3}\int_{\R^3}  \big(\omega_{\beta}^{\alpha}( s, x, v)\big)^2 \big(\sqrt{1+|v|^2}\tilde{d}(s,x,v) \big)^{2-c(\iota)-c(\iota')}  \]
\[
 \times g_\beta^\alpha(s,x,v)    \alpha_i(v)\cdot \Omega_{j'}^x \big(E(s,x+\hat{v}s) + \hat{v}\times B(s,x+\hat{v}s) \big)    \alpha_i(v)\cdot \big( \alpha_i(v)\cdot D_v D_v g^\alpha_{\kappa'}(s  ,x, v) \big) \]
 \be\label{nov2}
 \times  C_d(s,x,v)	    T_{k}^\mu(\tilde{V}_j\cdot\xi m(\xi), h)(s,x+\hat{v}s,v )   d x d v d s.
\ee
 
We remark that, similar to the dyadic decomposition we did in (\ref{nov112}),  we did the above dyadic decomposition for the new introduced electromagnetic field in (\ref{nov1}) and (\ref{jan14eqn1}),  to get around a technical summability issue with respect to the frequency.

From the detailed formulas in  (\ref{nove141}) and (\ref{nov2}). We know  that there exists a bilinear form of type ``$\Omega_j^x u_{k_1}(t,x+\hat{v}t) T_k^\mu(\tilde{V}_j\cdot \xi m(\xi), h)$'', $u\in\{E,B\}$. Motivated from this type of product, we define  a more  general multilinear operator  as follows. 

\begin{definition}
For any fixed $k, k_1\in \mathbb{Z}$, fixed $j,j'\in\{1,2,3\}$, fixed $a_{\mu}\in\{1/2,i\mu/2,-i\mu/2\}$,  any  fixed $f, g\in\{h_1^\alpha(t), h_2^\alpha(t),\alpha\in \mathcal{B}, |\alpha|\leq 10\} $, where $h_1^\alpha(t)$ and $h_2^\alpha(t)$ are the profiles of the electromagnetic field,  and any given symbol $a_1,a_2\in \mathcal{S}^\infty$, we define a multilinear form $T^{  \nu} (\cdot, \cdot, \cdot,\cdot)$   as follows,  
\[
 T^{\nu}(f, g, a_1,a_2 )(t,x,v)=\sum_{\mu\in\{+,-\}} \int_{\R^3}\int_{\R^3}e^{i(x+t\hat{v})\cdot\xi -it\mu|\xi-\eta|-i t \nu |\eta|} i a_{\mu}\widehat{f}(t, \xi-\eta)    \widehat{g}(t, \eta) \frac{ \tilde{V}_{j'}\cdot(\xi-\eta)\tilde{V}_j\cdot\eta }{|\eta|-\nu \hat{v}\cdot \eta}  \]
 \be\label{jan14eqn11}
 \times a_1(\xi-\eta) a_2(\eta)  \psi_{\geq 10}(t(|\eta|-\nu \hat{v}\cdot \eta)) \psi_{k}(\eta)  \psi_{k_1}(\xi-\eta) d \eta  d\xi. 
\ee 
 
\end{definition}

  In particular,  the multilinear form   $T^{ \nu}(f, g, a_1,a_2 )(t,x,v)$   can be represented as a product of two integrals as follows, 
 \be\label{march17eqn10}
 T^{ \nu}(f, g, a_1,a_2 )(t,x,v)=  \sum_{\mu\in\{+,-\}} a_{\mu}\tilde{\Omega}_{j'}^x\big(\mathcal{F}^{-1}[e^{-it \mu|\xi|}  a_1(\xi)\psi_{k_1}(\xi) \widehat{f}(t,\xi)](t, x+\hat{v}t)\big) T_k^\nu(\tilde{V}_j\cdot\xi a_2(\xi),  g)(t, x+\hat{v}t,v),
 \ee
 where the operator $T^{\nu}_{k}(\cdot, \cdot)$ is defined in (\ref{noveqn89}).

To reveal the hidden null structure inside  the multilinear form $T^{  \nu}(f, g, a_1,a_2 )(t,x,v)$, we decompose $T^{  \nu}_{ }(f, g,a_1,a_2)(t, \xi,v)$ into two parts as follows, 
\be\label{jan14eqn91}
 T^{ \nu}(f, g, a_1,a_2 )(t,x,v)= T^{ \nu;1}(f, g, a_1,a_2 )(t,x,v)+  T^{ \nu;2}(f, g, a_1,a_2 )(t,x,v), 
\ee
where
\[
 T^{ \nu;1}(f, g, a_1,a_2 )(t,x,v)= \sum_{\mu\in\{+,-\}} \int_{\R^3}\int_{\R^3}e^{i(x+t\hat{v})\cdot\xi -it\mu|\xi-\eta|-i t \nu |\eta|}i a_{\mu}\widehat{f}(t, \xi-\eta)  \tilde{V}_{j'}\cdot\big(\frac{\xi-\eta}{|\xi-\eta|} -\mu\nu \frac{\eta}{|\eta|}\big)  \widehat{g}(t, \eta) \]
\be\label{nove153}
 \times  \frac{ |\xi-\eta| \tilde{V}_j\cdot\eta }{|\eta|-\nu \hat{v}\cdot \eta}  a_1(\xi-\eta) a_2(\eta)  \psi_{\geq 10}(t(|\eta|-\nu \hat{v}\cdot \eta)) \psi_{k}(\eta)  \psi_{k_1}(\xi-\eta) d \eta  d\xi, 
\ee 
\[
 T^{ \nu;2}(f, g, a_1,a_2 )(t,x,v)= \sum_{\mu\in\{+,-\}} \int_{\R^3}\int_{\R^3}e^{i(x+t\hat{v})\cdot\xi -it\mu|\xi-\eta|-i t \nu |\eta|} i a_{\mu}\widehat{f}(t, \xi-\eta)   \frac{ \mu\nu \tilde{V}_{j'}\cdot   \eta}{|\eta|}   \widehat{g}(t, \eta)  \frac{ |\xi-\eta| \tilde{V}_j\cdot\eta }{|\eta|-\nu \hat{v}\cdot \eta} \]
\be\label{nove301}
 \times   a_1(\xi-\eta) a_2(\eta)  \psi_{\geq 10}(t(|\eta|-\nu \hat{v}\cdot \eta)) \psi_{k}(\eta)  \psi_{k_1}(\xi-\eta) d \eta  d\xi  
\ee
\[
=  \mathcal{F}^{-1}[e^{-it \nu |\xi|}   \widehat{g}(t, \xi) \frac{ - \nu \tilde{V}_{j'}\cdot\xi \tilde{V}_j\cdot\xi }{|\xi|\big(|\xi|-\nu \hat{v}\cdot \xi\big)}  a_2(\xi)  \psi_{\geq 10}(t(|\xi|-\mu \hat{v}\cdot \xi)) \psi_{k}(\xi)  ](t,x+\hat{v}t,v)
\]
\be\label{march18eqn41}
\times \mathcal{F}^{-1}[a_1(\xi)\psi_{k_1}(\xi)\widehat{u}(t,\xi)](t,x+\hat{v}t),
\ee
where $u\in \{\p_t E^\alpha,\p_t B^\alpha, \d E^\alpha, \d B^\alpha , \alpha \in \mathcal{B}, |\alpha|\leq 10\} $.

We understand the hidden null structure  of the above defined multilinear form in the   sense that the decay rate over time  can be improved. More precisely,   for the first part ``$ T^{ \nu;1}(f, g, a_1,a_2 )(t,x,v)$'', we   can gain ``$1/t$'' by doing integration by parts in ``$\eta$''. Meanwhile,  for the second part, we have one more good derivative $\Omega_{j'}^x$ acts on ``$g$'', which improves the decay rate.

 Note that, as a result of direct computations, the following equality holds, 
 \be\label{nove261}
e^{i(x+t\hat{v})\cdot\xi -it\mu|\xi-\eta|-i t \nu |\eta|}\tilde{V}_{j'}\cdot \big(\frac{\xi-\eta}{|\xi-\eta|} - \mu \nu \frac{\eta}{|\eta|} \big)= \frac{-i \mu}{t} \tilde{V}_{j'} \cdot \nabla_\eta\big(e^{i(x+t\hat{v})\cdot\xi -it\mu|\xi-\eta|-i t \nu |\eta|} \big).
 \ee
 Hence, after doing integration by parts in ``$\eta$'' for $ T^{ \nu;1}(f, g, a_1,a_2 )(t,x,v)$, we have
\be\label{jan14eqn92}
 T^{ \nu;1}(f, g, a_1,a_2 )(t,x,v)=t^{-1}\big( I^1_{ \nu} (f, g, a_1,a_2 )(t, x,v) + I^2_{  \nu}(f, g, a_1,a_2 )(t,x,v)\big),
 \ee
 where
 \[
 I^1_{ \nu} (f, g, a_1,a_2 )(t, x,v)=\sum_{\mu\in\{+,-\}} -\int_{\R^3}\int_{\R^3}e^{i(x+t\hat{v})\cdot\xi -it\mu|\xi-\eta|-i t \nu |\eta|}   \mu a_{\mu}\tilde{V}_{j'}\cdot \nabla_\eta\big(\widehat{f}(t, \xi-\eta) |\xi-\eta|\psi_{k_1}(\xi-\eta)a_1(\xi-\eta)\big)  \]
\be\label{jan14eqn31}
 \times    \widehat{g}(t, \eta) \frac{  \tilde{V}_j\cdot\eta }{|\eta|-\nu \hat{v}\cdot \eta}  a_2(\eta)  \psi_{\geq 10}(t(|\eta|-\mu \hat{v}\cdot \eta)) \psi_{k}(\eta)   d \eta  d\xi. 
\ee 
 \[
I^2_{ \nu}(f, g, a_1,a_2 )(t,x, v)= \sum_{\mu\in\{+,-\}}- \int_{\R^3}\int_{\R^3}e^{i(x+t\hat{v})\cdot\xi -it\mu|\xi-\eta|-i t \nu |\eta|}   \mu a_{\mu} \widehat{f}(t, \xi-\eta) \tilde{V}_{j'}\cdot \nabla_\eta\big(   \widehat{g}(t, \eta) \frac{  \tilde{V}_j\cdot\eta }{|\eta|-\nu \hat{v}\cdot \eta} \]
\[
 \times a_2(\eta)  \psi_{\geq 10}(t(|\eta|-\mu \hat{v}\cdot \eta)) \psi_{k}(\eta) \big)   |\xi-\eta|\psi_{k_1}(\xi-\eta)a_1(\xi-\eta)  d \eta  d\xi=\mathcal{F}^{-1}\big[a_1(\xi)\psi_{k_1}(\xi)\widehat{  u }(t, \xi) \big](t, x+\hat{v}t)    
\]
\be\label{march18eqn42}
\times   \mathcal{F}^{-1}[e^{-it \nu |\xi|}   \tilde{V}_{j'}\cdot \nabla_\xi\big(  \frac{  - \tilde{V}_j\cdot\xi }{|\xi|-\nu \hat{v}\cdot \xi} \widehat{g}(t, \xi)  a_2(\xi)  \psi_{\geq 10}(t(|\xi|-\nu \hat{v}\cdot \xi)) \psi_{k}(\xi) \big) ](t,x+\hat{v}t, v) , 
\ee
where   $u\in \{\p_t E^\alpha,\p_t B^\alpha, \d E^\alpha, \d B^\alpha , \alpha \in \mathcal{B}, |\alpha|\leq 10\} $.

To sum up, from the decompositions (\ref{jan14eqn91}) and (\ref{jan14eqn92}), we have
\be\label{jan23eqn31}
 T^{ \nu}(f, g, a_1,a_2 )(t,x,v)-  t^{-1} I^1_{\nu}(f, g, a_1,a_2 )(t,x,v)= T^{ \nu;2}(f, g, a_1,a_2 )(t,x,v) + t^{-1} I^2_{  \nu}(f, g, a_1,a_2 )(t,x, v).
\ee
 Recall the product formulas of  $T^{ \nu;2}(f, g, a_1,a_2 )(t,x,v)$  and  $I^2_{  \nu}(f, g, a_1,a_2 )(t,x, v) $ in (\ref{march18eqn41}) and (\ref{march18eqn42}). From the decay estimate (\ref{march18eqn30}) in Lemma \ref{sharpdecaywithderivatives},  the estimates (\ref{jan25eqn16}) and (\ref{noveqn600}) in Lemma \ref{twisteddecaylemma3}, we know that the following estimate holds, 
\[
\big|T_1^{ \nu}(f, g, a_1,a_2 )(t,x,v)-  t^{-1} I^1_{\nu}(f, g, a_1,a_2 )(t,x,v) \big|\psi_{\leq -5}(1-|x+\hat{v}t|/|t|) 	 
\]
\be\label{march18eqn51}
\lesssim (1+|t|)^{-2}(1+||t|-|x+\hat{v}t|)^{-1}   2^{ k_1+k} 2^{-4k_{1,+}-4k_{+}} \| a_1(\xi)\|_{\mathcal{S}^\infty_{k_1}}  \| a_2(\xi)\|_{\mathcal{S}^\infty_{k }} \big(E_{\textup{low}}^{eb}(t) \big)^2.
\ee

With the above preparation, we are ready to estimate $K_{k_1,k}^d$ and $S_{k_1,k}^d$. Recall the detailed formulas of $K_{k_1,k}^d$ and $S_{k_1,k}^d$ in (\ref{nove141}) and (\ref{nov2}).  For notational simplicity, we  define the following quantities. 
\begin{definition}
For any fixed $i,i'\in\{1,\cdots,7\}$, $  \nu\in\{+,-\}$ $\alpha\in \mathcal{B}, \beta, \iota, \kappa, \iota',\kappa'. \gamma,\in \mathcal{S}$, s.t., $\iota+\kappa=\iota'+\kappa'=\beta, $ $\iota'+\gamma=\kappa$, $|\iota|=|\iota'|=1$, $\Lambda^\iota\thicksim \psi_{\geq 1}(|v|) \Omega_j^x$ or $\psi_{\geq 1}(|v|) \widehat{\Omega}^v_j $, $\Lambda^{\iota'}\thicksim \psi_{\geq 1}(|v|) \Omega_{j'}^x$ or $\psi_{\geq 1}(|v|) \widehat{\Omega}^v_{j'} $ for some $j, j'\in\{1,2,3\}$, we define four integrals as follows,
\[
\widetilde{K}_{k_1,k}^{d,1}:=     \int_{1}^{t } \int_{\R^3}\int_{\R^3}  \big(\omega_{\beta}^{\alpha}(s, x, v)\big)^2  \alpha_{i'}(v)\cdot  D_v    g^\alpha_{\kappa'}( s,x, v) \alpha_i(v)\cdot  D_v g^\alpha_\kappa(t ,x, v)    \big(\sqrt{1+|v|^2}\tilde{d}(s,x,v) \big)^{2-c(\iota)-c(\iota')}
\] 
\be\label{jan12eqn82} 
\times 	C_d(s,x,v)     \big[   s^{-1}   I^1_{ \nu} (f, g, a_1,a_2 )(s, x,v) \big] d x d v d s , 
\ee
\[
\widetilde{K}_{k_1,k}^{d,2}:=    \int_{1}^{t } \int_{\R^3}\int_{\R^3}  \big(\omega_{\beta}^{\alpha}( s, x, v)\big)^2  \alpha_{i'}(v)\cdot  D_v    g^\alpha_{\kappa'}(s,x, v) \alpha_i(v)\cdot  D_v g^\alpha_\kappa(s ,x, v)   \big(\sqrt{1+|v|^2}\tilde{d}(s,x,v) \big)^{2-c(\iota)-c(\iota')}
\] 
\be\label{jan12eqn81}
\times  C_d(s,x,v)      \big[ T^{ \nu}(f, g, a_1,a_2 )(s,x,v)-  s^{-1} I^1_{\nu}(s,x,v)\big] d x d v d s, 
\ee
\[
\widetilde{S}_{k_1,k}^{d,1}:=    \int_{1}^{t } \int_{\R^3}\int_{\R^3}  \big(\omega_{\beta}^{\alpha}( s, x, v)\big)^2   g^\alpha_{\beta }(s,x, v)\alpha_{i'}\cdot \big( \alpha_i(v)\cdot  D_v D_v g^\alpha_{\gamma}(s ,x, v) \big)      \big(\sqrt{1+|v|^2}\tilde{d}(s,x,v) \big)^{2-c(\iota)-c(\iota')}
\] 
\be\label{jan12eqn84}
\times C_d(s,x,v)       \big[  s^{-1}  I^1_{ \nu} (f, g, a_1,a_2 )(s,x,v)\big] d x d v d s, 
\ee
\[
\widetilde{S}_{k_1,k}^{d,2}:=    \int_{1}^{t } \int_{\R^3}\int_{\R^3}  \big(\omega_{\beta}^{\alpha}( s, x, v)\big)^2      g^\alpha_{\beta }(s,x, v)\alpha_{i'}\cdot \big( \alpha_i(v)\cdot  D_v D_v g^\alpha_{\gamma}(s ,x, v) \big)    \big(\sqrt{1+|v|^2}\tilde{d}(s,x,v) \big)^{2-c(\iota)-c(\iota')}
\] 
\be\label{jan12eqn83} 
\times  C_d(s,x,v)      \big[ T^{ \nu}(f, g, a_1,a_2 )(s,x,v)-  s^{-1}   I^1_{ \nu} (f, g, a_1,a_2 )(s,x,v)\big] d x d v d s. 
\ee
\end{definition}
 
\begin{lemma}\label{march28lemma1}
For the integrals $\widetilde{K}_{k_1,k}^{d,1}$ and $\widetilde{S}_{k_1,k}^{d,1}$ defined in \textup{(\ref{jan12eqn82})} and  \textup{(\ref{jan12eqn84})}, the following estimate holds, 
\[
  |\widetilde{K}_{k_1,k}^{d,1}| + |\widetilde{S}_{k_1,k}^{d,1} |\lesssim    \int_{1}^{t }  (1+|s|)^{-1}  \| c(s,x,v)\|_{L^\infty_{x,v}}  \big(E_{\textup{low}}^{eb}(s) \big)^2 E_{\beta;d}^{\alpha}(s) d  s
\]
\be\label{march18eqn34}
  \times 2^{d+k_1}\big(2^{k+d}+2^{2k+2d} \big)2^{-4k_{+}-4k_{1,+}}\|a\|_{Y } \| a_1(\xi)\|_{\mathcal{S}^\infty_{k_1}}  \| a_2(\xi)\|_{\mathcal{S}^\infty_{k }}.
\ee
\end{lemma}
\begin{proof}

\noindent $\bullet$\quad  The  estimate of  $\widetilde{K}_{k_1,k}^{d,1}$.\qquad  Recall (\ref{jan12eqn82}). For all the vector fields ``$D_v$'' in (\ref{jan12eqn82}),  we use the second decomposition of $D_v$ in (\ref{summaryoftwodecomposition}) in Lemma \ref{twodecompositionlemma}. Note that the following estimate holds for any $\rho\in \mathcal{K}, |\rho|=1$  from the detailed formulas of $e_{\rho}(t,x,v)$   in  (\ref{sepeq932}),
 \be\label{jan13eqn9}
 \|(1+|v|)^{1-c(\rho)}e_{\rho}(t,x,v) \psi_{\leq -5}(1-|x+\hat{v}t|/|t|)  \|_{L^\infty_{x,v}}\lesssim (1+t).
 \ee

Recall (\ref{jan14eqn31}). For any $\rho\in \mathcal{K}, |\rho|=1$, after representing  the mulitlinear form $   I^1_{ \nu} (f, g, a_1,a_2 )(s,x,v)$ as a product  form, we know that  the following estimate holds   from the linear decay estimate (\ref{noveqn555}) in Lemma \ref{twistedlineardecay}  and the estimate (\ref{nove501}) in Lemma \ref{lineardecaylemma3},  
\[
 \|(1+|v|)^{1-c(\rho)   }e_{\rho }(t,x,v) \varphi_{d}(||t|-|x+\hat{v}t||)   I^1_{ \nu}(f, g, a_1,a_2 )(t,x,v) \|_{L^\infty_{x,v}}
\]
\be\label{jan13eqn1}
\lesssim (1+t)^{-1} 2^{k_1}(2^{k}+2^{2k+d})2^{-4k_{+}-4k_{1,+}}\|a_1(\xi)\|_{\mathcal{S}^\infty_{k_1}}\|a_2(\xi)\|_{\mathcal{S}^\infty_{k }}\big( E_{\textup{low}}^{eb}(t)\big)^2.
\ee
Therefore, from the estimate (\ref{jan13eqn9}),  the   estimate (\ref{jan13eqn1}), the second part of the  estimate (\ref{feb8eqn51}) in Lemma \ref{derivativeofweightfunction},  and the $L^2_{x,v}-L^2_{x,v}-L^\infty_{x,v}$ type multilinear estimate,  we know     that  the following estimate holds, 
\[
 |\widetilde{K}_{k_1,k}^{d,1}|     \lesssim \sum_{ \gamma\in \mathcal{S}, |\alpha|+|\gamma|\leq N_0}   2^{d+k_1}\big(2^{k+d}+2^{2k+2d} \big)2^{-4k_{+}-4k_{1,+}} \|a\|_{Y }   \int_{1}^{t }  \|\omega_{\gamma}^\alpha(x, v) g_{\gamma}^\alpha(s, x,v) \varphi_{[d-1,d+1]}  (||s|-|x+\hat{v}s||) \|_{L^2_x L^2_v }^2  
\]
\be\label{jan12eqn101}
\times  (1+|s|)^{-1} \|a_1(\xi)\|_{\mathcal{S}^\infty_{k_1}}\|a_2(\xi)\|_{\mathcal{S}^\infty_{k }}  \| c(t,x,v)\|_{L^\infty_{x,v}} \big( E_{\textup{low}}^{eb}(s)\big)^2  d  s.
\ee

\noindent $\bullet$\quad The estimate of $\widetilde{S}_{k_1,k}^{d,1}$. \qquad  Recall (\ref{jan12eqn84}). As in the estimate of  $\widetilde{K}_{k_1,k}^{d,1}$, we use the second decomposition of ``$D_v$'' in (\ref{summaryoftwodecomposition}) in Lemma \ref{twodecompositionlemma}   for all the vector fields ``$D_v$'' in (\ref{jan12eqn84}). As a result, the following equality holds 
 \be\label{jan12eqn110}
 D_v D_v g_\gamma^\alpha(t,x,v)=\sum_{\rho_1, \rho_2\in \mathcal{S}, |\rho_1|=|\rho_2|=1}   e_{\rho_1}(t,x,v)\Lambda^{\rho_1}\big( e_{\rho_2}(t,x,v)\Lambda^{\rho_2} g_\gamma^\alpha(t,x,v)\big).
 \ee
  Note that  the following estimate holds if  $ \iota+\iota'+\kappa'=\beta, \rho, \rho'\in \mathcal{K}/\{\vec{0}\}$, 
 \be\label{jan13eqn18}
\big| \frac{\omega_{\beta}^\alpha(x,v)}{\omega_{\rho'\circ\rho\circ\kappa'}^\alpha (x,v)} \big|\sim (1+|v|)^{c(\iota)+c(\iota')-c(\rho)-c(\rho')}\big(\phi(t,x,v)\big)^{i(\rho)+i(\rho')-i(\iota)-i(\iota')}. 
\ee
 Recall (\ref{sepeq932}). From the equality (\ref{dec26eqn1}) and the estimate (\ref{jan23eqn11}) in Lemma \ref{derivativesofcoefficient}, the following estimate holds, 
 \be\label{jan13eqn5}
\sum_{\rho_1, \rho_2\in\mathcal{S}, |\rho_1|=|\rho_2|=1} \|(1+|v|)^{1-c(\rho_2)}\Lambda^{\rho_1} e_{\rho_2}(t,x,v) \psi_{\leq -5}(1-|x+\hat{v}t|/|t|) \|_{L^\infty_{x,v}} \lesssim (1+t). 
 \ee
 From the estimates (\ref{jan13eqn9}) and (\ref{jan13eqn5}), the estimate (\ref{jan13eqn1}), and the estimate (\ref{jan13eqn18}), the following estimate holds for fixed $\gamma\in \mathcal{S}$, s.t., $\iota+\iota'+\gamma=\beta$,
 \[
 \|\omega_{\beta}^\alpha(x,v)  \big(\sqrt{1+|v|^2}\tilde{d}(t,x,v) \big)^{2-c(\iota)-c(\iota')}    D_v D_v g_\gamma^\alpha(t,x,v) T_2^{\mu,\nu}(f, g, a_1,a_2 )(t,x,v) \varphi_{d}(||t|-|x+\hat{v}t||) \psi_{\leq -5}(1-|x+\hat{v}t|/|t|) \|_{L^2_{x,v}}
 \]
 \[
 \lesssim \sum_{\kappa\in \mathcal{S}, |\alpha|+|\kappa|\leq N_0}   2^{d+k_1}\big(2^{k+d}+2^{2k+2d} \big)2^{-4k_{+}-4k_{1,+}}  \|a_1(\xi)\|_{\mathcal{S}^\infty_{k_1}}\|a_2(\xi)\|_{\mathcal{S}^\infty_{k }} \big(E_{\textup{low}}^{eb}(t) \big)^2
 \]
\be\label{jan13eqn25}
 \times \| \omega_{\kappa}^\alpha(x,v) g_\kappa^\alpha(t,x,v)\varphi_{[d-2,d+2]} (||t|-|x+\hat{v}t||)\|_{L^2_xL^2_v}. 
 \ee
 Therefore, from the above estimate (\ref{jan13eqn25})    and the  $L^2_{x,v}-L^2_{x,v}-L^\infty_{x,v}$ type multilinear estimate, the following estimate holds, 
\be\label{jan13eqn31}
 |\widetilde{S}_{k_1,k}^{d,1}|     \lesssim     2^{d+k_1}\big(2^{k+d}+2^{2k+2d} \big)2^{-4k_{+}-4k_{1,+}} \int_{1 }^{t }    (1+|s|)^{-1}\|a\|_{Y } \| a_1(\xi)\|_{\mathcal{S}^\infty_{k_1}}  \| a_2(\xi)\|_{\mathcal{S}^\infty_{k }}  \| c(s,x,v)\|_{L^\infty_{x,v}}  \big(E_{\textup{low}}^{eb}(s) \big)^2 E_{\beta;d}^{\alpha}(s) d  s.
\ee

To sum up, our desired estimate (\ref{march18eqn34}) holds  from the estimates (\ref{jan12eqn101}) and (\ref{jan13eqn31}).
\end{proof}

\begin{lemma}\label{march28lemma2}
For the integral  $\widetilde{K}_{k_1,k}^{d,2}$  defined in \textup{(\ref{jan12eqn81})}, the following estimate holds, 
\[
|\widetilde{K}_{k_1,k}^{d,2}| \lesssim \big(2^{k_1/2+k/2+d} + 2^{3(k_1+k)/2+3d} \big)2^{-4k_{1,+}-4k_{+}}  \|a\|_{Y } \| a_1(\xi)\|_{\mathcal{S}^\infty_{k_1}}  \| a_2(\xi)\|_{\mathcal{S}^\infty_{k }} \]
\be\label{march28eqn1}
\times \int_{1}^{t }  (1+|s|)^{-1}   \| c(s,x,v)\|_{L^\infty_{x,v}} \big(E_{\textup{low}}^{eb}(s) \big)^2 E_{\beta;d}^{\alpha}(s ) d s. 
\ee
\end{lemma}
\begin{proof}
  Recall the detailed formula of $\widetilde{K}_{k_1,k}^{d,2}$ in  (\ref{jan12eqn81}). Based on the possible size of ``$|v|$'', ``$x\cdot v$'', and ``$|x|$'', we separate into fours cases by utilizing the following partition of unity, 
\be\label{march27eqn101}
1=\sum_{i=1,\cdots,4}\eta_i(t,x,v), \quad \eta_1(t,x,v):=  \psi_{\leq   (k+ k_1)/4+10}\big(|v| |t|^{ -1/2}\big) ,
\ee
\be\label{march27eqn102}
  \eta_2(t,x,v):= \psi_{>   (k+ k_1)/4+10}\big( |v| |t|^{-1/2}\big) \psi_{ \leq 10}\big(|x|/\big(  |t| /|v|  +|t|^{1/2}  \big) \big) ,
\ee
\be\label{march27eqn103}
   \eta_3(t,x,v):=   \psi_{>   (k+ k_1)/4+10}\big( |v| |t|^{-1/2}\big) \psi_{> 10}\big(|x|/\big(  |t| /|v|  + |t|^{1/2}  \big) \big) \tilde{\eta}(x\cdot \tilde{v}|t|/|x|^2) ,
\ee
\be\label{march27eqn104}
   \eta_4(t,x,v):=  \psi_{>   (k+ k_1)/4+10}\big( |v| |t|^{-1/2}\big)   \psi_{> 10}\big(|x|/\big(  |t| /|v|  + |t|^{1/2}  \big) \big) \big(1- \tilde{\eta}(x\cdot \tilde{v} |t|/|x|^2)\big),
\ee
where $\tilde{\eta}(x):\R\longrightarrow\R$ is a smooth function such that it equals to one inside $(-\infty, 2^{-4}]$ and it is supported inside $(-\infty, 2^{-5}]$. Correspondingly, we define the following corresponding integrals, 
\[
J_i(t)=    \int_{\R^3}\int_{\R^3}  \big(\omega_{\beta}^{\alpha}( x, v)\big)^2  \alpha_{i'}(v)\cdot  D_v    g^\alpha_{\kappa'}(t,x, v) \alpha_i(v)\cdot  D_v g^\alpha_\kappa(t ,x, v)      \big(\sqrt{1+|v|^2}\tilde{d}(t,x,v) \big)^{2-c(\iota)-c(\iota')} 
\] 
\be\label{april14eqn20}
\times   C_d(t,x,v)      \big[ T^{ \nu}(f, g, a_1,a_2 )(t,x,v)-  t^{-1} I^1_{\nu}(f, g, a_1,a_2 )(t,x,v)\big] \eta_i(t,x,v)d x d v  ,  
\ee
where $i\in\{1,2,3,4\}.$
Hence, from    the detailed formula of $\widetilde{K}_{k_1,k}^{d,2}$ in  (\ref{jan12eqn81}) and the partition (\ref{march27eqn101}), we have
\be\label{march27eqn135}
\widetilde{K}_{k_1,k}^{d,2}=\sum_{i=1,\cdots,4}\int_{1}^{t} J_i(s) d s . 
\ee

\noindent $\oplus$\quad The estimate of $J_1$, i.e., the case when $ 1\leq |v|\leq  2^{  (k+ k_1)/4+10} (1+|t|)^{1/2}$.  

 For this case , we use the  first decomposition of $D_v$ in (\ref{summaryoftwodecomposition}) in Lemma \ref{twodecompositionlemma} for all vector fields ``$D_v$'' in   (\ref{april14eqn20}). From the estimate of coefficients in (\ref{jan15eqn2}), the second part of the estimate (\ref{feb8eqn51}) in Lemma \ref{derivativeofweightfunction},  the estimates (\ref{feb19eqn1}) and (\ref{march18eqn51}), and the $L^{2}_{x,v}-L^2_{x,v}-L^\infty_{x,v}$ type multilinear estimate, we know that  the following estimate holds for any fixed time $t\in[t_1,t_2]$,
\[
|J_1(t)|\lesssim 2^{ k_1+k+3d}2^{-4k_{1,+}-4k_{+}} (1+|t|)^{-1} \big(E_{\textup{low}}^{eb}(t) \big)^2 E_{\beta;d}^{\alpha}(t ) \|a\|_{Y } \| a_1(\xi)\|_{\mathcal{S}^\infty_{k_1}}  \| a_2(\xi)\|_{\mathcal{S}^\infty_{k }} \| c(t,x,v)\|_{L^\infty_{x,v}}
 \]
 \[
 \times\big(2^{-d}+ (1+|t|)^{-1}\| |v|^2\psi_{\leq (k+ k_1)/4+10}\big( |v| |t|^{-1/2}\big)\|_{L^\infty_{ v}}\big) \lesssim  (1+|t|)^{-1}\big(2^{3(k_1+k)/2+3d} +2^{  k_1+k +2d}\big) 
 \]
 \be\label{march27eqn141}
 \times 2^{-4k_{1,+}-4k_{+}} \|a\|_{Y } \| a_1(\xi)\|_{\mathcal{S}^\infty_{k_1}}  \| a_2(\xi)\|_{\mathcal{S}^\infty_{k }} \| c(t,x,v)\|_{L^\infty_{x,v}} \big(E_{\textup{low}}^{eb}(t) \big)^2 E_{\beta;d}^{\alpha}(t ).
 \ee

\noindent $\oplus$ \quad The estimate of $J_2$, i.e., the case when $ |v|  >   2^{(k+ k_1)/4+9}  |t|^{1/2}+1 $ and $|x|\leq 2^{10}\big( |t|/|v| +|t|^{1/2}\big) $.

For this case, we use the second decomposition of ``$D_v$'' in (\ref{summaryoftwodecomposition}) in Lemma \ref{twodecompositionlemma} for all the ``$D_v$'' derivatives in (\ref{april14eqn20}). Recall the detailed formula of $e_{\rho}(t,x,v)$ in (\ref{sepeq932}). From the estimate (\ref{march18eqn51}), the second part of the estimate (\ref{feb8eqn51}) in Lemma \ref{derivativeofweightfunction}, and the $L^2_{x,v}-L^2_{x,v}-L^\infty_{x,v}$ type multi-linear estimate, the following estimate holds for any fixed $t\in[t_1,t_2]$,	
\[
|J_2(t)| \lesssim 2^{k_1+k+d} 2^{-4k_{1,+}-4k_{+}}  \big(\|  |v|^{-2} \psi_{>   (k+ k_1)/4+10}\big( |v| |t|^{-1/2}\big)\|_{L^\infty_v}+ (1+|t|)^{-1} 2^{ d} \big)  \]
 \be\label{march27eqn142}
\times  \|a\|_{Y  } \| a_1(\xi)\|_{\mathcal{S}^\infty_{k_1}}  \| a_2(\xi)\|_{\mathcal{S}^\infty_{k }} \| c(t,x,v)\|_{L^\infty_{x,v}}  \big(E_{\textup{low}}^{eb}(t) \big)^2  E_{\beta;d}^{\alpha}(t ).
\ee

\noindent $\oplus$ \quad The estimate of $J_3$, i.e., the case when $|v|\geq  2^{  k_1 /2+9} |t|^{1/2}+1$,  $|x|\geq 2^{10}\big(|t|/|v|+|t|^{1/2}\big)$, and $x\cdot \tilde{v} \leq  -2^{-4} |x|^2/|t|$.

Recall the definition of $\phi(t,x,v)$ in (\ref{april2eqn1}).  We  have  $(x,v)\in \textit{supp}(\phi(t,x,v)-1)$ for the case we are considering. Recall  the equality (\ref{march18eqn54}).  We know that the following estimate holds for the case we are considering,
\be\label{april6eqn100}
|\frac{\tilde{d}(t,x,v)}{||t|-|x+\hat{v}t||}|  \lesssim \frac{t}{t+\big(1+|v|\big)(|x\cdot v|+|x|)}\lesssim \frac{t^2}{(1+|v|)^2|x|^2}, \quad 
\frac{1}{\phi(t,x,v)}\lesssim   \frac{|x|}{|x\cdot v|}  \lesssim  \frac{t}{|v||x|}.
\ee
For this case, we use the  second decomposition of $D_v$ in (\ref{summaryoftwodecomposition}) in Lemma \ref{twodecompositionlemma} for all vector fields ``$D_v$'' in   (\ref{april14eqn20}). From the above estimate (\ref{april6eqn100}), the detailed formulas of $e_{\rho}(t,x,v)$ in (\ref{sepeq932}), the estimate  (\ref{march18eqn51}), and the $L^{2}_{x,v}-L^2_{x,v}-L^\infty_{x,v}$ type multi-linear estimate, we have the following estimate holds for any fixed time $t \in[1,T ]$,
\[
|J_3(t)| \lesssim 2^{k_1+k+d} 2^{-4k_{1,+}-4k_{+}}  \big(\|  |v|^{-2} \psi_{>   (k+ k_1)/4+10}\big( |v| |t|^{-1/2}\big)\|_{L^\infty_v}+ (1+|t|)^{-1} 2^{ d} \big)  \]
 \be\label{march27eqn143}
\times  \|a\|_{Y  } \| a_1(\xi)\|_{\mathcal{S}^\infty_{k_1}}  \| a_2(\xi)\|_{\mathcal{S}^\infty_{k }} \| c(t,x,v)\|_{L^\infty_{x,v}}  \big(E_{\textup{low}}^{eb}(t) \big)^2    E_{\beta;d}^{\alpha}(t ) .
\ee

\noindent $\oplus$ \quad The estimate of $J_4$, i.e., the case when  $|v| \geq  2^{(k_1+k)/2+9}|t|^{1/2}+1$,    $|x|\geq 2^{10}\big(|t|/|v|+|t|^{1/2}\big)$ , and $x\cdot \tilde{v} \geq  -2^{-4} |x|^2/|t|$.

  Recall the first equality in  (\ref{march18eqn54}).  We know that the following estimate holds for the case we are considering,
\be\label{march27eqn131}
\big||t^2|-|x+\hat{v}t|^2\big| \geq |x|^2, \quad \Longrightarrow \quad \frac{1}{|x|^2|t|}\big(\frac{t^2}{|v|^2} + \frac{t|x|}{|v|}+|x|^2 \big)\lesssim \frac{1}{|t|}.
\ee
For this case, we use the second decomposition of ``$D_v$'' in (\ref{summaryoftwodecomposition}) in Lemma \ref{twodecompositionlemma} for all the ``$D_v$'' derivatives in (\ref{april14eqn20}). Recall the detailed formula of $e_{\rho}(t,x,v)$ in (\ref{sepeq932}). From the estimates (\ref{march18eqn51}) and (\ref{march27eqn131}), the second part of the estimate (\ref{feb8eqn51}) in Lemma \ref{derivativeofweightfunction}, and the $L^2_{x,v}-L^2_{x,v}-L^\infty_{x,v}$ type multi-linear estimate, the following estimate holds for any fixed $t\in[1,T]$,	
 \be\label{march27eqn144}
|J_4(t)| \lesssim 2^{k_1+k+2d}   2^{-4k_{1,+}-4k_{+}}   (1+|t|)^{-1}   \|a\|_{Y} \| a_1(\xi)\|_{\mathcal{S}^\infty_{k_1}}  \| a_2(\xi)\|_{\mathcal{S}^\infty_{k }} \| c(t,x,v)\|_{L^\infty_{x,v}} \big(E_{\textup{low}}^{eb}(t) \big)^2   E_{\beta;d}^{\alpha}(t ). 
\ee

To sum up, recall the decomposition in (\ref{march27eqn135}),  our desired estimate (\ref{march28eqn1}) holds from   the estimates (\ref{march27eqn141}), (\ref{march27eqn142}), (\ref{march27eqn143}), and (\ref{march27eqn144}). 
\end{proof}
\begin{lemma}\label{march29lemma3}
For the integral  $\widetilde{S}_{k_1,k}^{d,2}$ defined in    \textup{(\ref{jan12eqn83})}, the following estimate holds, 
\[
 |\widetilde{S}_{k_1,k}^{d,2} |\lesssim  \int_{1}^{t }  (1+|s|)^{-1}   \| c(s,x,v)\|_{L^\infty_{x,v}} \big(E_{\textup{low}}^{eb}(s) \big)^2 E_{\beta;d}^{\alpha}(s ) d s \]
\be\label{march28eqn2}
\times  \big(2^{(k_1+k)/2+d} + 2^{3(k_1+k)/2+3d} \big)2^{-4k_{1,+}-4k_{+}} \|a\|_{Y } \| a_1(\xi)\|_{\mathcal{S}^\infty_{k_1}}  \| a_2(\xi)\|_{\mathcal{S}^\infty_{k }} . 
\ee
\end{lemma}
\begin{proof}
Recall (\ref{jan12eqn83}). From the two decompositions of $D_v$ in (\ref{summaryoftwodecomposition}), we know    that  the following decomposition  holds for the second order derivative ``$D_vD_v$'', 
\be\label{march28eqn41}
D_vD_v= P_1(D_vD_v) + L_{1}(D_vD_v)= P_2(D_vD_v) + L_{2}(D_vD_v),
\ee
where $P_i(D_vD_v)$, $i\in\{1,2 \}$, denotes the principle term of ``$D_vD_v$'', which is a second order derivative and $L_{i}(D_vD_v)$, $i\in\{1,2 \}$, denotes the lower order term of ``$D_vD_v$'', which is a first order derivative. More precisely, we have
\be\label{march28eqn151}
P_1(D_vD_v)= \sum_{\begin{subarray}{c}
\rho_1, \rho_2\in \mathcal{K} \\ 
 |\rho_1|=|\rho_2|=1\\ 
 \end{subarray}} d_{\rho_1}(t,x,v)d_{\rho_2(t,x,v)} \Lambda^{\rho_1\circ\rho_2}, \quad P_2(D_vD_v)=\sum_{\begin{subarray}{c}
\rho_1, \rho_2\in \mathcal{K} \\ 
 |\rho_1|=|\rho_2|=1\\ 
 \end{subarray}}e_{\rho_1}(t,x,v) e_{\rho_2(t,x,v)} \Lambda^{\rho_1\circ\rho_2},
\ee
\be\label{march28eqn152}
L_1(D_vD_v)= \sum_{\begin{subarray}{c}
\rho_1, \rho_2\in \mathcal{K} \\ 
 |\rho_1|=|\rho_2|=1\\ 
 \end{subarray}} d_{\rho_1}(t,x,v)\Lambda^{\rho_1}\big(d_{\rho_2(t,x,v)}\big) \Lambda^{  \rho_2}, \quad L_2(D_vD_v)= \sum_{\begin{subarray}{c}
\rho_1, \rho_2\in \mathcal{K} \\ 
 |\rho_1|=|\rho_2|=1\\ 
 \end{subarray}} e_{\rho_1}(t,x,v)\Lambda^{\rho_1}\big(e_{\rho_2(t,x,v)}\big) \Lambda^{  \rho_2}, 
\ee
From the detailed formulas of $d_{\rho}(t,x,v)$ and $e_{\rho}(t,x,v)$ in (\ref{sepeq947}) and (\ref{sepeq932}), the first equality in (\ref{dec26eqn1}), and the estimate of coefficients in (\ref{jan26eqn101}) and (\ref{jan23eqn11}),  we know that the following estimate holds, 
\be\label{march28eqn31}
\sum_{\rho_1, \rho_2\in \mathcal{K}, |\rho_1|=|\rho_2|=1} |d_{\rho_1}(t,x,v)\Lambda^{\rho_1}\big(d_{\rho_2}(t,x,v) \big)|\lesssim \big(1+|\tilde{d}(t,x,v)| \big)^2(1+|v|)^2, 
\ee
\be\label{march28eqn32}
\sum_{\rho_1, \rho_2\in \mathcal{K}, |\rho_1|=|\rho_2|=1} |e_{\rho_1}(t,x,v)\Lambda^{\rho_1}\big(e_{\rho_2}(t,x,v) \big)|\psi_{\leq 3}(|x|/(1+|t|)\lesssim (1+|t|)^2. 
\ee
Similar to the estimate of $  \widetilde{K}_{k_1,k}^{d,2}  $ in Lemma \ref{march28lemma2},  by using the cutoff functions $\eta_i(t,x,v)$, $i\in\{1,\cdots,4\}$, defined in (\ref{march27eqn101}), (\ref{march27eqn102}), (\ref{march27eqn103}), and (\ref{march27eqn104}), we decompose  $  \widetilde{S}_{k_1,k}^{d,2} $  into four parts as follows, 
\be\label{march28eqn191}
\widetilde{S}_{k_1,k}^{d,2} = \sum_{i=1,\cdots,4} \int_{1}^{t } S_i(s) d s,
\ee
where
\[
S_i(t):=   \int_{\R^3}\int_{\R^3}  \big(\omega_{\beta}^{\alpha}( x, v)\big)^2      g^\alpha_{\beta }(t,x, v)\alpha_{i'}\cdot \big( \alpha_i(v)\cdot  D_v D_v g^\alpha_{\gamma}(t ,x, v) \big)    \big(\sqrt{1+|v|^2}\tilde{d}(t,x,v) \big)^{2-c(\iota)-c(\iota')} 
\] 
\[
\times     C_d(t,x,v) \eta_i(t,x,v)     \big[ T^{ \nu}(f, g, a_1,a_2 )(t,x,v)-  t^{-1} I^1_{\nu}(f, g, a_1,a_2 )(t,x,v)\big] d x d v  .
\]

Similar to the estimate of $J_i(t)$, $i\in\{1,2,3,4\}$, in the proof of Lemma \ref{march28lemma2}, we use the first decomposition of ``$D_v$'' for $S_1(t)$ and use the second decomposition of $D_v$ for $S_i(t)$, $i\in\{2,3,4\}$. More precisely, we separate into two cases as follows.

\noindent $\oplus$\quad The estimate of $S_1$.\quad For this   case we use the first decomposition of $D_v$. Equivalently speaking, we use the first decomposition of $D_vD_v$ in (\ref{march28eqn41}). As a result, the following decomposition holds, 
\be\label{march28eqn192}
S_1(t)=  S_{1,1}(t)+ S_{1,2}(t) , 
\ee
where
\[
S_{1,1}(t)=   \int_{\R^3}\int_{\R^3}  \big(\omega_{\beta}^{\alpha}( x, v)\big)^2      g^\alpha_{\beta }(t,x, v)\alpha_{i'}\cdot \big( \alpha_i(v)\cdot  P_1(D_v D_v) g^\alpha_{\gamma}(t ,x, v) \big)     \big(\sqrt{1+|v|^2}\tilde{d}(t,x,v) \big)^{2-c(\iota)-c(\iota')} 
\] 
\be\label{march28eqn161}
\times C_d(t,x,v) \eta_1(t,x,v)    \big[ T^{ \nu}(f, g, a_1,a_2 )(t,x,v)-  t^{-1} I^1_{\nu}(f, g, a_1,a_2 )(t,x,v)\big] d x d v  ,
\ee
\[
S_{1,2}(t)=    \int_{\R^3}\int_{\R^3}  \big(\omega_{\beta}^{\alpha}( x, v)\big)^2      g^\alpha_{\beta }(t,x, v)\alpha_{i'}\cdot \big( \alpha_i(v)\cdot  L_1(D_v D_v) g^\alpha_{\gamma}(t ,x, v) \big)     \big(\sqrt{1+|v|^2}\tilde{d}(t,x,v) \big)^{2-c(\iota)-c(\iota')} 
\] 
\be\label{march28eqn162}
\times   C_d(t,x,v) \eta_1(t,x,v)      \big[ T^{ \nu}(f, g, a_1,a_2 )(t,x,v)-  t^{-1} I^1_{\nu}(f, g, a_1,a_2 )(t,x,v)\big] d x d v . 
\ee

Recall the detailed formula of $P_1(D_vD_v)$ in (\ref{march28eqn151}) and the detailed formula of $S_{1,1}(t)$ in  (\ref{march28eqn161}). We know  that $S_{1,1}(t)$ and $J_1(t)$ are of the same type. With minor modifications in the estimate of $J_1(t)$ in (\ref{march27eqn141}), the following estimate holds, 
\[
 |S_{1,1}(t)|\lesssim  (1+|t|)^{-1}\big( 2^{ k_1+k+2d} +2^{3(k_1+k)/2+3d}\big)  2^{-4k_{1,+}-4k_{+}}\|a\|_{Y} \| a_1(\xi)\|_{\mathcal{S}^\infty_{k_1}}  \| a_2(\xi)\|_{\mathcal{S}^\infty_{k }} 
\]
 \be\label{march28eqn200}
 \times \| c(t,x,v)\|_{L^\infty_{x,v}} \big(E_{\textup{low}}^{eb}(t) \big)^2 E_{\beta;d}^{\alpha}(t ).
\ee

  Recall the detailed formula of $L_1(D_vD_v)$ in (\ref{march28eqn152}) and  the detailed formula of $S_{1,2}(t)$ in  (\ref{march28eqn162}). Since $L_1(D_vD_v)$ is a lower order derivatives,  we can gain at least $(1+|v|)^{-10}$ from   the hierarchy of the weight functions between different order derivatives.  As a result, from the estimates (\ref{march18eqn51}) and  (\ref{march28eqn31}) and the $L^2_{x,v}-L^2_{x,v}-L^\infty_{x,v}$ type multilinear estimate, we have
\be\label{march28eqn170}
|S_{1,2}(t)| \lesssim (1+|t|)^{-1}2^{ k_1+k  +2d}  2^{-4k_{1,+}-4k_{+}}\|a\|_{Y} \| a_1(\xi)\|_{\mathcal{S}^\infty_{k_1}}  \| a_2(\xi)\|_{\mathcal{S}^\infty_{k }} \| c(t,x,v)\|_{L^\infty_{x,v}} \big(E_{\textup{low}}^{eb}(t) \big)^2 E_{\beta;d}^{\alpha}(t ).
\ee

\noindent $\oplus$\quad The estimate of $S_i$, $i\in\{2,3,4\}$.\quad For these three terms,  we use the second decomposition of $D_vD_v$ in (\ref{march28eqn41}). As a result, the following decomposition holds, 
\be\label{march28eqn190}
S_i(t)= S_{i,1}(t)+S_{i,2}(t) , \quad i\in\{2,3,4\},
\ee
where
\[
S_{i,1}(t)=   \int_{\R^3}\int_{\R^3}  \big(\omega_{\beta}^{\alpha}( x, v)\big)^2      g^\alpha_{\beta }(t,x, v)\alpha_{i'}\cdot \big( \alpha_i(v)\cdot  P_2(D_v D_v) g^\alpha_{\gamma}(t ,x, v) \big)   \big(\sqrt{1+|v|^2}\tilde{d}(t,x,v) \big)^{2-c(\iota)-c(\iota')}
\] 
\be\label{march28eqn181}
\times  C_d(t,x,v) \eta_i(t,x,v)    \big[  T^{ \nu}(f, g, a_1,a_2 )(t,x,v)-  t^{-1} I^1_{\nu}(f, g, a_1,a_2 )(t,x,v)\big] d x d v  ,
\ee
\[
S_{i,2}(t)=    \int_{\R^3}\int_{\R^3}  \big(\omega_{\beta}^{\alpha}( x, v)\big)^2      g^\alpha_{\beta }(t,x, v)\alpha_{i'}\cdot \big( \alpha_i(v)\cdot  L_2(D_v D_v) g^\alpha_{\gamma}(t ,x, v) \big)   \big(\sqrt{1+|v|^2}\tilde{d}(t,x,v) \big)^{2-c(\iota)-c(\iota')} 
\] 
\be\label{march28eqn182}
\times    C_d(t,x,v)  \eta_i(t,x,v)     \big[  T^{ \nu}(f, g, a_1,a_2 )(t,x,v)-  t^{-1} I^1_{\nu}(f, g, a_1,a_2 )(t,x,v)\big] d x d v . 
\ee

Recall the detailed formula of $P_2(D_vD_v)$ in (\ref{march28eqn151}) and the detailed formula of $S_{i,1}(t)$ in  (\ref{march28eqn181}). We know that $S_{i,1}(t)$ and $J_i(t)$, $i\in\{2,3,4\}$, are of the same type. with minor modifications in the estimate of $J_i(t)$, $i\in\{2,3,4\}$, in (\ref{march27eqn142}), (\ref{march27eqn143}), and (\ref{march27eqn144}), the following estimate holds for any $i\in\{2,3,4\}$,
\[
 |S_{i,1}(t)| \lesssim  (1+|t|)^{-1} \big(2^{(k_1+k)/2+d}+ 2^{ k_1+k+2d}  \big)  2^{-4k_{1,+}-4k_{+}}\|a\|_{Y} \| a_1(\xi)\|_{\mathcal{S}^\infty_{k_1}}  \| a_2(\xi)\|_{\mathcal{S}^\infty_{k }} 
\]
 \be\label{march28eqn201}
 \times \| c(t,x,v)\|_{L^\infty_{x,v}} \big(E_{\textup{low}}^{eb}(t) \big)^2 E_{\beta;d}^{\alpha}(t ).
 \ee

Recall (\ref{march28eqn182}) and (\ref{march28eqn152}). Again,  since $L_2(D_vD_v)$ is a lower order derivative, we can gain at least $(1+|v|)^{-10}$ from   the hierarchy of the weight functions between the difference of the orders of derivatives. Note that $ |v| >   2^{(k+ k_1)/4+9} |t|^{1/2} $ inside the support of $\eta_i(t,x,v), i\in\{2,3,4\}$. Therefore, the following estimate holds from the estimates (\ref{march18eqn51}), (\ref{march28eqn32}), and  the $L^2_{x,v}-L^2_{x,v}-L^\infty_{x,v}$ type multi-linear estimate, 
\[
\sum_{i=2,3,4} |S_{i,2}(t)| \lesssim 2^{k_1+k+d}2^{-4k_{+}-4k_{1,+}}\| (1+|v|)^{-2}\psi_{>(k+ k_1)/4+10}( |v| |t|^{-1/2})\|_{L^\infty_v}\|a\|_{Y } \| a_1(\xi)\|_{\mathcal{S}^\infty_{k_1}}  
\]
\[
\times \| a_2(\xi)\|_{\mathcal{S}^\infty_{k }}    \| c(t,x,v)\|_{L^\infty_{x,v}} \big(E_{\textup{low}}^{eb}(t) \big)^2  E_{\beta;d}^{\alpha}(t ) \lesssim (1+|t|)^{-1} 2^{(k_1+k)/2+d}2^{-4k_{+}-4k_{1,+}}\|a\|_{Y_d} \| a_1(\xi)\|_{\mathcal{S}^\infty_{k_1}}  
\]
\be\label{march28eqn381}
\times \| a_2(\xi)\|_{\mathcal{S}^\infty_{k }}    \| c(t,x,v)\|_{L^\infty_{x,v}} \big(E_{\textup{low}}^{eb}(t) \big)^2  E_{\beta;d}^{\alpha}(t ).
\ee

To sum up, recall  the decompositions (\ref{march28eqn191}),  (\ref{march28eqn192}), and (\ref{march28eqn190}),    our desired estimate (\ref{march28eqn2}) holds from    the estimates (\ref{march28eqn200}), (\ref{march28eqn170}), (\ref{march28eqn201}), and (\ref{march28eqn381}). 
\end{proof}

\begin{lemma}\label{auxillarybilinearlemma10}
The following estimate holds, 
\[
|K_{k,k_1}^d | +|S_{k_1,k}^d|   \lesssim        \big( 2^{(k+k_1)/2+d} +2^{3(k+k_1)/2+3d} +2^{ k_1+2k+3d} \big) 2^{-4k_{+}-4k_{1,+}}  \|m(\xi)\|_{\mathcal{S}^\infty_k}\|a\|_{Y}   \]
\be\label{nove761}
\times \big[ \int_{1}^{t } (1+|s|)^{-1}  \|c(s,x,v)\|_{L^\infty_{x,v}}    \big(E_{\textup{low}}^{eb}(s) \big)^2  E_{\beta;d}^{\alpha}(s ) d s \big].
\ee 
  \end{lemma}
\begin{proof}
Recall (\ref{nove141}) and (\ref{nov2}). 
 Note that, for any $u\in\{E,B\}$, the following equality holds from some $i\in \{1,2\}$,
\[
 \Omega_{j'}^x u_{k_1}(t,x+\hat{v}t) T_{k}^\mu(\tilde{V}_j\cdot\xi m(\xi), h)(t,x+\hat{v}t)=      T^{\mu }( h_i(t),h(t),  1, m ).
\]
From the above equality, we know that the desired estimate (\ref{nove761}) follows directly from the   estimate (\ref{march18eqn34}) in Lemma \ref{march28lemma1}, the estimate  (\ref{march28eqn1}) in Lemma \ref{march28lemma2}, and the estimate (\ref{march28eqn2}) in Lemma \ref{march29lemma3}.

\end{proof}
 To get around the summability issue with respect to $k_1$, same as what we did in the decomposition (\ref{jan9eqn31}),  we also use the process of trading spatial derivatives for the decay of the distance with respect to the light cone "$||t|-|x+\hat{v}t||$". As a result, we have 
 \begin{lemma}\label{auxillarybilinearlemma11}
The following estimate holds for any $d\in\mathbb{N}_{+}, d\geq 10,$
\[
|K_{k,k_1}^d | +|S_{k_1,k}^d|   \lesssim     \big[  \big( 2^{(k+k_1)/2+d} +2^{3(k+k_1)/2+3d} +2^{ k_1+2k+3d} \big) \big(2^{-3k_1-3d} +2^{-4k_1-4d}\big) +2^{k-k_1}\big(1+2^{-2k_1-2d}       \]
\be\label{march29eqn6}
+ 2^{k-2k_1-d} + 2^{k+d} \big) \big] 2^{-4k_{+}-4k_{1,+}}  \|m(\xi)\|_{\mathcal{S}^\infty_k} \|a\|_{Y_d} \big[ \int_{1}^{t } (1+|s|)^{-1}  \|c(s,x,v)\|_{L^\infty_{x,v}}    \big(E_{\textup{low}}^{eb}(s) \big)^2  E_{\beta;d}^{\alpha}(s ) d s \big].
\ee 
\end{lemma}
\begin{proof}
 Recall the decomposition (\ref{jan6eqn10}) in Lemma \ref{tradingreg4mod}. For any $u\in\{E,B\}$, we have
\[
\Omega_{j'}^x u_{k_1}(t,x+\hat{v}t) =  L^1_{k_1,j'}[u](t,x+\hat{v}t) +  \widetilde{L_{k_1,j'}}[u](t,x+\hat{v}t,v) + \sum_{i=1,\cdots,5}  {\textup{E}}_{k_1,j'}^i[u](t,x+\hat{v}t,v), 
\]
where $L_{k_1,j'}^1[u](t,x,v)$ and $\widetilde{L_{k_1,j'}}[u](t,x+\hat{v}t,v)$ were defined in \textup{(\ref{nove611})}  and  \textup{(\ref{april3eqn40})}  and $ \textup{E}_{k_1,j'}^i[u](t,x+\hat{v}t,v)$, $i\in\{1,\cdots,5\}$, were defined in (\ref{nove510}) and (\ref{april4eqn2}).

 Correspondingly, we decompose $K_{k,k_1}^d$ and  $S_{k,k_1}^d$ into three parts as follows, 
\be\label{jan13eqn210}
K_{k,k_1}^{d;1}= \sum_{i=1,2,3} K_{k,k_1}^{d;i},\quad   K_{k,k_1}^{d;3}= \sum_{i=1,\cdots,5}   K_{k,k_1;i}^{d;3}, \quad  S_{k,k_1}^{d}= \sum_{i=1,2,3} S_{k,k_1}^{d;i},\quad  S_{k,k_1}^{d;3}= \sum_{i=1,\cdots,5}   S_{k,k_1;i}^{d;3},
\ee
where
\[
K_{k_1,k}^{d;1}:=  \sum_{\begin{subarray}{l}
 j'=1,2,3\\
 i'=1,\cdots,7\\
\end{subarray}}   \sum_{
\begin{subarray}{c}
  \Lambda^\iota \sim \psi_{\geq 1}(|v|)\widehat{\Omega}^v_j\\  \textup{or}\,\psi_{\geq 1}(|v|) \Omega_j^x  \\
\end{subarray}} \sum_{
\begin{subarray}{c}
\iota'+\kappa'=\beta,\iota,  \kappa\in \mathcal{S},  |\iota'|= 1  \\
  \Lambda^{\iota'} \sim \psi_{\geq 1}(|v|)\widehat{\Omega}^v_{j'}\textup{or}\,\psi_{\geq 1}(|v|)\Omega_{j'}^x  \\
\end{subarray}}   \int_{1}^{t  } \int_{\R^3}\int_{\R^3}  \big(\omega_{\beta}^{\alpha}(s, x, v)\big)^2 \alpha_{i'}(v)\cdot  D_v    g^\alpha_{\kappa'}(s,x, v)   \]
\[
 \times   \alpha_{i'}(v)\cdot   \big(L_{k_1,j'}^1[E](s,x+\hat{v}s ) + \hat{v}\times L_{k_1,j'}^1[B](s,x+\hat{v}s  )\big)  \alpha_i(v)\cdot  D_v g^\alpha_\kappa(s ,x, v)
     \]
 \be\label{jan13eqn71}
 \times    C_d(s,x,v)     \big(\sqrt{1+|v|^2}\tilde{d}(s , x,v)\big)^{2-c(\iota)-c(\iota')}      T_{k}^\mu(\tilde{V}_j\cdot\xi m(\xi), h)(s,x+\hat{v}s)  d x d v d s  ,
\ee

\[
K_{k_1,k}^{d;2}:=  \sum_{\begin{subarray}{l}
 j'=1,2,3\\
 i'=1,\cdots,7\\
\end{subarray}}   \sum_{
\begin{subarray}{c}
  \Lambda^\iota \sim \psi_{\geq 1}(|v|)\widehat{\Omega}^v_j\\  \textup{or}\,\psi_{\geq 1}(|v|) \Omega_j^x  \\
\end{subarray}} \sum_{
\begin{subarray}{c}
\iota'+\kappa'=\beta,\iota,  \kappa\in \mathcal{S},  |\iota'|= 1  \\
  \Lambda^{\iota'} \sim \psi_{\geq 1}(|v|)\widehat{\Omega}^v_{j'}\textup{or}\,\psi_{\geq 1}(|v|)\Omega_{j'}^x  \\
\end{subarray}}  \int_{1}^{t  } \int_{\R^3}\int_{\R^3}  \big(\omega_{\beta}^{\alpha}(s, x, v)\big)^2 \alpha_{i'}(v)\cdot  D_v    g^\alpha_{\kappa'}(s,x, v)      \]
\[
 \times   \alpha_{i'}(v)\cdot   \big( \widetilde{L_{k_1,j'}}[E](s,x+\hat{v}s, v) + \hat{v}\times \widetilde{ L_{k_1,j'}}[B](s,x+\hat{v}s, v)\big)  \alpha_i(v)\cdot  D_v g^\alpha_\kappa(s ,x, v)  
     \]
 \be\label{april8eqn21}
 \times C_d(s,x,v)   \big(\sqrt{1+|v|^2}\tilde{d}(s , x,v)\big)^{2-c(\iota)-c(\iota')}   T_{k}^\mu(\tilde{V}_j\cdot\xi m(\xi), h)(s,x+\hat{v}s)  d x d v d s,
\ee 

\[
K_{k_1,k;i}^{d;3}:=   \sum_{\begin{subarray}{l}
 j'=1,2,3\\
 i'=1,\cdots,7\\
\end{subarray}}   \sum_{
\begin{subarray}{c}
  \Lambda^\iota \sim \psi_{\geq 1}(|v|)\widehat{\Omega}^v_j\\  \textup{or}\,\psi_{\geq 1}(|v|) \Omega_j^x  \\
\end{subarray}} \sum_{
\begin{subarray}{c}
\iota'+\kappa'=\beta,\iota,  \kappa\in \mathcal{S},  |\iota'|= 1  \\
  \Lambda^{\iota'} \sim \psi_{\geq 1}(|v|)\widehat{\Omega}^v_{j'}\textup{or}\,\psi_{\geq 1}(|v|)\Omega_{j'}^x  \\
\end{subarray}}  \int_{1}^{t  } \int_{\R^3}\int_{\R^3}  \big(\omega_{\beta}^{\alpha}(s, x, v)\big)^2 \alpha_{i'}(v)\cdot  D_v    g^\alpha_{\kappa'}(s,x, v)  \]
\[
 \times  \alpha_i(v)\cdot  D_v g^\alpha_\kappa(s ,x, v)       \alpha_{i'}(v)\cdot   \big(E_{k_1,j'}^i[E](s,x+\hat{v}s, v) + \hat{v}\times E_{k_1,j'}^i[B](s,x+\hat{v}s, v)\big)  \]
 \be\label{jan13eqn72}
 \times C_d(s,x,v)  \big(\sqrt{1+|v|^2}\tilde{d}(s , x,v)\big)^{2-c(\iota)-c(\iota')}   T_{k}^\mu(\tilde{V}_j\cdot\xi m(\xi), h)(s,x+\hat{v}s)  d x d v d s,
\ee

\[
S_{k_1,k}^{d;1}:=  \sum_{\begin{subarray}{l}
 j'=1,2,3\\
 i'=1,\cdots,7\\
\end{subarray}}   \sum_{
\begin{subarray}{c}
  \Lambda^\iota \sim \psi_{\geq 1}(|v|)\widehat{\Omega}^v_j\\  \textup{or}\,\psi_{\geq 1}(|v|) \Omega_j^x  \\
\end{subarray}} \sum_{
\begin{subarray}{c}
\iota'+\kappa'=\beta,\iota,  \kappa\in \mathcal{S},  |\iota'|= 1  \\
  \Lambda^{\iota'} \sim \psi_{\geq 1}(|v|)\widehat{\Omega}^v_{j'}\textup{or}\,\psi_{\geq 1}(|v|)\Omega_{j'}^x  \\
\end{subarray}}  \int_{1}^{t  } \int_{\R^3}\int_{\R^3}  \big(\omega_{\beta}^{\alpha}(s, x, v)\big)^2 g_\beta^\alpha(s,x,v)   \]
\[
 \times    \alpha_i(v)\cdot  \big(L_{k_1,j'}[E](s, x+\hat{v}s ) + \hat{v}\times L_{k_1,j'}[B](s,x+\hat{v}s )\big)       \alpha_i(v)\cdot \big( \alpha_i(v)\cdot D_v D_v g^\alpha_{\kappa'}(s ,x, v) \big)	   \]
 \be\label{jan13eqn79}
 \times  C_d(s,x,v)    \big(\sqrt{1+|v|^2}\tilde{d}(  s , x,v)\big)^{2-c(\iota)-c(\iota')}  T_{k}^\mu(\tilde{V}_j\cdot\xi m(\xi), h)(s,x+\hat{v}s)   d x d v d s,
\ee

\[
S_{k_1,k}^{d;2}:=  \sum_{\begin{subarray}{l}
 j'=1,2,3\\
 i'=1,\cdots,7\\
\end{subarray}}   \sum_{
\begin{subarray}{c}
  \Lambda^\iota \sim \psi_{\geq 1}(|v|)\widehat{\Omega}^v_j\\  \textup{or}\,\psi_{\geq 1}(|v|) \Omega_j^x  \\
\end{subarray}} \sum_{
\begin{subarray}{c}
\iota'+\kappa'=\beta,\iota,  \kappa\in \mathcal{S},  |\iota'|= 1  \\
  \Lambda^{\iota'} \sim \psi_{\geq 1}(|v|)\widehat{\Omega}^v_{j'}\textup{or}\,\psi_{\geq 1}(|v|)\Omega_{j'}^x  \\
\end{subarray}}   \int_{1}^{t  } \int_{\R^3}\int_{\R^3}  \big(\omega_{\beta}^{\alpha}(s, x, v)\big)^2 g_\beta^\alpha(s,x,v)    \]
\[
 \times    \alpha_i(v)\cdot  \big(\widetilde{L_{k_1,j'}}[E](s, x+\hat{v}s, v) + \hat{v}\times \widetilde{L_{k_1,j'}}[B](s,x+\hat{v}s, v)\big)         \alpha_i(v)\cdot \big( \alpha_i(v)\cdot D_v D_v g^\alpha_{\kappa'}(s ,x, v) \big)	   \]
 \be\label{april8eqn25}
 \times     C_d(s,x,v)  \big(\sqrt{1+|v|^2}\tilde{d}(s , x,v)\big)^{2-c(\iota)-c(\iota')}  T_{k}^\mu(\tilde{V}_j\cdot\xi m(\xi), h)(s,x+\hat{v}s)   d x d v d s,
\ee

\[
S_{k_1,k;i}^{d;3}:=  \sum_{\begin{subarray}{l}
 j'=1,2,3\\
 i'=1,\cdots,7\\
\end{subarray}}   \sum_{
\begin{subarray}{c}
  \Lambda^\iota \sim \psi_{\geq 1}(|v|)\widehat{\Omega}^v_j\\  \textup{or}\,\psi_{\geq 1}(|v|) \Omega_j^x  \\
\end{subarray}} \sum_{
\begin{subarray}{c}
\iota'+\kappa'=\beta,\iota,  \kappa\in \mathcal{S},  |\iota'|= 1  \\
  \Lambda^{\iota'} \sim \psi_{\geq 1}(|v|)\widehat{\Omega}^v_{j'}\textup{or}\,\psi_{\geq 1}(|v|)\Omega_{j'}^x  \\
\end{subarray}}  \int_{1}^{t  } \int_{\R^3}\int_{\R^3}  \big(\omega_{\beta}^{\alpha}(s, x, v)\big)^2 g_\beta^\alpha(  s,x,v)   \]
\[
 \times     \alpha_i(v)\cdot  \big(E^i_{k_1,j'}[E](s,x+\hat{v}s, v) + \hat{v}\times E^i_{k_1,j'}[B](s,x+\hat{v}s , v)\big)   	 \alpha_i(v)\cdot \big( \alpha_i(v)\cdot D_v D_v g^\alpha_{\kappa'}(s ,x, v) \big)   
 \]
 \be\label{jan13eqn86}
 \times     C_d(s,x,v) \big(\sqrt{1+|v|^2}\tilde{d}(s , x,v)\big)^{2-c(\iota)-c(\iota')}  T_{k}^\mu(\tilde{V}_j\cdot\xi m(\xi), h)(s,x+\hat{v}s)   d x d v d s.
\ee

\noindent $\bullet$\quad The estimate of $K_{k,k_1}^{d;i}$ and $S_{k,k_1}^{d;i}$, $i\in\{1,2\}$.

 Recall      (\ref{jan13eqn71}), (\ref{april8eqn21}),    (\ref{jan13eqn79}), and (\ref{april8eqn25}). Moreover, recall again \textup{(\ref{nove611})}  and  \textup{(\ref{april3eqn40})}.  For any $u\in\{E,B\},$we know that the following equality holds  for some $i\in\{1,2\},$
 \[
L_{k_1,j'}^1[u](t,x, v) T_{k}^\mu(\tilde{V}_j\cdot\xi m(\xi), h)(t,x+\hat{v}t)=  \sum_{\alpha\in \mathcal{B}, |\alpha|\leq 3}   (|t|-|x+\hat{v}t|)^{-3} \big[     \tilde{c}_{\alpha}^{0} (t,x+\hat{v}t)  T_1^{\mu  }(  h^\alpha_i(t)  ,h (t),  \tilde{m}_{k_1,\alpha}^0, m ) \]
\[
  +  {i}  \tilde{c}_{\alpha}^{1} (t,x+\hat{v}t)     T_1^{\mu  }(    h^\alpha_i(t)   ,h (t),  |\xi|\tilde{m}_{k_1,\alpha}^1, m )
 -    {\tilde{c}_{\alpha}^{2}}(t,x+\hat{v}t)  T_1^{\mu }(   h^\alpha_i(t)  ,h (t),  |\xi|^2\tilde{m}_{k_1,\alpha}^2, m ) \big],
\]
 \[
\widetilde{L_{k_1,j'}}[u](t,x, v) T_{k}^\mu(\tilde{V}_j\cdot\xi m(\xi), h)(t,x+\hat{v}t)=  \sum_{\alpha\in \mathcal{B}, |\alpha|\leq 3}  - 3    (|t|-|x+\hat{v}t|)^{-4}\frac{i t}{|x+\hat{v}t|} \]
\[
\times  \big[\tilde{c}_{\alpha}^{0} (t,x+\hat{v}t)    T_1^{\mu  }(  h^\alpha_i(t)  ,h (t),  |\xi|^{-1}\tilde{m}_{k_1,\alpha}^0, m )+i   \tilde{c}_{\alpha}^{1} (t,x+\hat{v}t)     T_1^{\mu  }(   h^\alpha_i(t)   ,h (t),  \tilde{m}_{k_1,\alpha}^1, m ) \]
\[
-   {\tilde{c}_{\alpha}^{2}}(t,x+\hat{v}t)  T_1^{\mu }(   h^\alpha_i(t)  ,h (t),  |\xi| \tilde{m}_{k_1,\alpha}^2, m ) \big].
\]

Therefore, from the estimate of coefficients $  \tilde{c}_{\alpha}^{i}(t,x+\hat{v}t)$, $i\in\{1,2,3\}$, in (\ref{april3eqn1}), the estimate of symbols $\tilde{m}_{k_1,\alpha}^i(\xi)$ in (\ref{noveqn141}),  we know  that the following estimate holds from   the   estimate (\ref{march18eqn34}) in Lemma \ref{march28lemma1}, the estimate  (\ref{march28eqn1}) in Lemma \ref{march28lemma2}, and the estimate (\ref{march28eqn2}) in Lemma \ref{march29lemma3}, 
\[
 |K_{k,k_1}^{d;1} | +|S_{k_1,k}^{d;1}|  + |K_{k,k_1}^{d;2} | +|S_{k_1,k}^{d;2}|   \lesssim        \big( 2^{(k+k_1)/2+d} +2^{3(k+k_1)/2+3d} +2^{ k_1+2k+3d} \big) \big(2^{-3k_1-3d} +2^{-4k_1-4d}\big) 2^{-4k_{+}-4k_{1,+}}  \]
\be\label{jan13eqn93}
\times\|m(\xi)\|_{\mathcal{S}^\infty_k}  \|a\|_{Y } \big[ \int_{1}^{t } (1+|  s|)^{-1}  \|c(s,x,v)\|_{L^\infty_{x ,v}}    \big(E_{\textup{low}}^{eb}(s) \big)^2  E_{\beta;d}^{\alpha}(s  ) d s \big].
\ee
In the above estimate, we used the fact that $d\geq 10$ and $\||x|^{-3}\varphi_{[d-2,d+2]}(x) a\|_Y\lesssim 2^{-3d} \|a\|_Y$.

\noindent $\bullet$\quad The estimate of $K_{k,k_1}^{d;3}$ and $S_{k,k_1}^{d;3}$.

Recall (\ref{jan13eqn72}) and (\ref{jan13eqn86}). Moreover, we recall the estimates of error terms  (\ref{nov291}) in Lemma \ref{tradeoffregularity}.  For each fixed $i\in\{1,\cdots,5\}$, we use the second decomposition of  $D_v$ in (\ref{summaryoftwodecomposition}) in Lemma \ref{twodecompositionlemma} and the second decomposition of ``$D_vD_v$'' in (\ref{march28eqn41}). From the estimate (\ref{nov291}) in Lemma \ref{tradeoffregularity} and the estimate (\ref{nove501}) in Lemma \ref{lineardecaylemma3}, the following estimate holds for any fixed $\rho_1, \rho_2\in \mathcal{S},  $ s.t., $|\rho_1|=|\rho_2|=1$,   
\[
  \|(1+|v|)^{2-c(\rho_1)-c(\rho_2)} e_{\rho_1}(t,x,v)e_{\rho_2}(t,x,v) E_{k_1,j'}^i[u](t,x,v)  T_{k}^\mu(\tilde{V}_j\cdot\xi m(\xi), h)(t,x+\hat{v}t)  \|_{L^\infty_{x,v}}
\]
\be\label{jan13eqn121}
\lesssim  (1+t)^{-1} \big( 2^{-4d-4k_1} +2^{-2d-2k_1}\big) 2^{k_1}(2^k+2^{2k+d})2^{-4k_{+}-4k_{1,+}}\| m (\xi)\|_{\mathcal{S}^\infty_k}  \big(  E_{\textup{low}}^{eb}(t)\big)^2 . 
\ee

Moreover, from the estimate (\ref{jan13eqn5}), the estimate (\ref{nov291}) in Lemma \ref{tradeoffregularity} and the estimate (\ref{noveqn500}) in Lemma \ref{twisteddecaylemma3}, the following estimate holds for any fixed $\rho_1, \rho_2\in \mathcal{S},  $ s.t., $|\rho_1|=|\rho_2|=1$,   
\[
 \|(1+|v|)^{1-c(\rho_1)-c(\rho_2)} e_{\rho_1}(t,x,v)\Lambda^{\rho_1}e_{\rho_2}(t,x,v) E_{k_1,j'}^i[u](t,x,v) \varphi_d(||t|-|x+\hat{v}t||) T_{k}^\mu(\tilde{V}_j\cdot\xi m(\xi), h)(t,x+\hat{v}t) \]
\be\label{jan13eqn120}
\times \psi_{\leq -5}(1-|x+\hat{v}t|/|t|)  \|_{L^\infty_{x,v}}  \lesssim  (1+t)^{-1}   \big( 2^{-4d-4k_1} +2^{-2d-2k_1}\big) 2^{k_1}(2^k+2^{2k+d})2^{-4k_{+}-4k_{1,+}}\| m (\xi)\|_{\mathcal{S}^\infty_k}  \big( E_{\textup{low}}^{eb}(t)\big)^2.
\ee

 Recall (\ref{jan13eqn72}) and (\ref{jan13eqn86}).     From the above estimates (\ref{jan13eqn121}) and (\ref{jan13eqn120}), the second part of the estimate (\ref{feb8eqn51}) in Lemma \ref{derivativeofweightfunction}, and  the $L^2_{x,v}-L^2_{x,v}-L^\infty_{x,v}$ type multilinear estimate,   the following estimate holds
 \[
   |K_{k,k_1;i}^{d;3} | +|S_{k,k_1;i}^{d;3} | \lesssim     2^{k-k_1}\big(1+2^{-2k_1-2d}+ 2^{k-2k_1-d} + 2^{k+d} \big) 2^{-4k_{+}-4k_{1,+}} \| m (\xi)\|_{\mathcal{S}^\infty_k} \|a\|_{Y }
 \]
\be\label{jan13eqn103}
 \times   \big[ \int_{1}^{t } (1+|s|)^{-1}  \|c(s,x,v)\|_{L^\infty_{x,v}}    \big(E_{\textup{low}}^{eb}(s) \big)^2  E_{\beta;d}^{\alpha}(s ) d s \big].
\ee

To sum up, recall the decomposition (\ref{jan13eqn210}), our desired estimate  (\ref{march29eqn6}) holds from  the estimates (\ref{jan13eqn93}) and (\ref{jan13eqn103}). 
\end{proof}

\noindent \textit{Proof of Lemma \textup{\ref{bulktermbilinearestimate1}}}. 

Recall the decomposition of $\tilde{T}_{1}^2(m,a,c,h)$ and $\tilde{T}_{2}^2(m,a,c,h)$ in (\ref{nov1}) and (\ref{jan14eqn1}). From the estimate (\ref{nove761})  in Lemma \ref{auxillarybilinearlemma10}, we know that the desired  estimate (\ref{jan14eqn9}) holds directly if $d\leq 10$. If $d\geq 10$, then from  the estimate (\ref{nove761})  in Lemma \ref{auxillarybilinearlemma10} and the  estimate (\ref{march29eqn6}) in Lemma \ref{auxillarybilinearlemma11},   the following estimate holds
\[
|\tilde{T}_{1}^2(m,a,c,h)|+|\tilde{T}_{1}^2(m,a,c,h)| \lesssim  \sum_{k_1\in \mathbb{Z}} |K_{k_1,k}^d|	+ |S_{k_1,k}^d|	\lesssim  \big(2^{k/2+d/2}+2^{2k +2d  } \big) 2^{-4k_{+} }\big[\big(\sum_{k_1\leq -d}  2^{k_1/2+d/2} + 2^{2k_1+2d}  \big)
\] 
\[ 
   + \big[ \sum_{k_1\geq -d} \big(  \big( 2^{  k_1 /2+d/2}  +2^{2k_1 +2d} \big) \big(2^{-3k_1-3d} +2^{-4k_1-4d}\big) +2^{ -k_1-d}\big(1+2^{-2k_1-2d}    \big) \big]    \big)\big] \]
   \[
   \times   \|m(\xi)\|_{\mathcal{S}^\infty_k} \|a\|_{Y }  \big[ \int_{1}^{t } (1+|s|)^{-1}  \|c(s,x,v)\|_{L^\infty_{x,v}}\big( E_{\textup{low}}^{eb}(s) \big)^2  E_{\beta;d}^{\alpha}(s) d s \big]
\] 
\be\label{march29eqn51}
\lesssim  \big(2^{k/2+d/2}+2^{2k +2d } \big) 2^{-4k_{+} }  \|m(\xi)\|_{\mathcal{S}^\infty_k}\|a\|_{Y } \big[  \int_{1}^{t } (1+|s|)^{-1}  \|c(s,x,v)\|_{L^\infty_{x,v}}\big( E_{\textup{low}}^{eb}(s) \big)^2  E_{\beta;d}^{\alpha}(s) d s \big].
\ee
Hence finishing the proof of the desired estimate (\ref{jan14eqn9}) in Lemma \ref{bulktermbilinearestimate1}.

\qed

\subsection{The $L^\infty_x$-type decay estimate for the operator $T_k^\mu(\cdot, \cdot)$ }\label{feb172020prf2}

In this subsection,  by proving several  $L^\infty_x$-type decay estimates for the operator $T_k^\mu(\cdot, \cdot)$ defined  in (\ref{noveqn89}), we finish the proof of Lemma \ref{lineardecaylemma3}, which have been used as black boxes in the previous subsections for the proof of Lemma \ref{trilinearestimate1}.

 \begin{lemma}\label{lineardecaylemma3}
The following estimate holds for any profile $h(t,x)\in\{h^\alpha_i(t,x), i\in\{1,2\},\alpha\in \mathcal{B}, |\alpha|\leq 10\}$ and any $\rho\in \mathcal{K}$, s.t., $|\rho|=1$
\be\label{nove501}
 \|(1+|v|)^{1-c(\rho)} e_{\rho}(t,x,v)  T_{k}^\mu(\tilde{V}_j\cdot\xi m(\xi), h)(t,x+\hat{v}t) \varphi_{d} (||t|-|x+\hat{v}t||)  \|_{L^\infty_{x,v}}  \lesssim   \big(2^k + 2^{2k+d} \big)2^{-4k_{+}}\| m (\xi)\|_{\mathcal{S}^\infty_k} E_{\textup{low}}^{eb}(t). 
\ee 
\end{lemma}
\begin{proof}
 Recall the detailed formulas of the coefficients $e_{\rho}(t,x,v)$ in (\ref{sepeq932})  in Lemma \ref{twodecompositionlemma}.  We know that the desired estimate  (\ref{nove501}) holds directly from the estimates (\ref{noveqn500}) and   (\ref{jan25eqn16})  in Lemma \ref{twisteddecaylemma3} and the estimate  (\ref{noveqn532}) in Lemma \ref{auxillarylemmajan26}. 
\end{proof}

\begin{lemma}\label{twisteddecaylemma3}
For any $i\in\{1,2,3\}$, the following decay estimate holds,
 \[
 \frac{1}{1+|v|}\Big| \int_{\R^3} e^{ix\cdot \xi - i\mu t |\xi|} \frac{m(\xi)\tilde{V}_i\cdot \xi}{|\xi|-\mu \hat{v}\cdot \xi}  \widehat{f}(\xi)\psi_{\geq 10}(t(|\xi|- \mu \hat{v}\cdot \xi))\psi_k(\xi) d \xi  \Big|
\]
\be\label{noveqn500}
  \lesssim  \sum_{\alpha\in \mathcal{B}, |\alpha|\leq 5}  2^{-4k_{+}} \|m(\xi)\|_{\mathcal{S}^\infty_k} \min\{  2^{ k} (1+|t|)^{-1}  \big(  \|f^\alpha\|_{X_0} +\|f^\alpha\|_{X_1}\big), 2^{2k} \|f^\alpha\|_{X_0}\}.
\ee
Moreover, the following estimates also hold for any $n\in\mathbb{N}_{+}$,$i,j\in\{1,2,3\}$,
\[
  \Big| \int_{\R^3} e^{ix\cdot \xi - i\mu t |\xi|} \frac{m(\xi)  }{\big(|\xi|-\mu \hat{v}\cdot \xi\big)^{n+1}}  \widehat{f}(\xi)\psi_{\geq 10}(t(|\xi|- \mu \hat{v}\cdot \xi))\psi_k(\xi) d \xi  \Big| 
\]
\be\label{jan25eqn16}
\lesssim  \sum_{\alpha\in \mathcal{B}, |\alpha|\leq 5}  2^{-4k_{+}} 2^{  k} (1+|t|)^{n}\|m(\xi)\|_{\mathcal{S}^\infty_k}  \big( \|f^\alpha\|_{X_0}+\|f^\alpha\|_{X_1}\big),
\ee
\[
 \frac{1}{1+|v|}\Big| \int_{\R^3} e^{ix\cdot \xi - i\mu t |\xi|} \frac{m(\xi)\tilde{V}_i\cdot \xi}{\big( |\xi|-\mu \hat{v}\cdot \xi\big)^{n+1}}  \widehat{f}(\xi)\psi_{\geq 10}(t(|\xi|- \mu \hat{v}\cdot \xi))\psi_k(\xi) d \xi  \Big|\]
 \[
 +  
  \Big| \int_{\R^3} e^{ix\cdot \xi - i\mu t |\xi|} \frac{m(\xi)(\tilde{V}_j\cdot \xi) (  \tilde{V}_i\cdot \xi)}{|\xi|\big(|\xi|-\mu \hat{v}\cdot \xi\big)^{n+1}}  \widehat{f}(\xi)\psi_k(\xi)\psi_{\geq 10}(t(|\xi|- \mu \hat{v}\cdot \xi)) d \xi  \Big|
\]
\be\label{noveqn600}
 \lesssim  \sum_{\alpha\in \mathcal{B}, |\alpha|\leq 5}  2^{-4k_{+}} (1+|t|)^{n} \|m(\xi)\|_{\mathcal{S}^\infty_k}  \min\{ 2^{ k} (1+|t|)^{-1} \big(  \|f^\alpha\|_{X_0} +\|f^\alpha\|_{X_1}\big), 2^{2k} \|f^\alpha\|_{X_0} \}  .
\ee
\end{lemma}
\begin{proof}
After utilizing the volume of support of $\xi$, we know  that the desired estimates (\ref{noveqn500}), (\ref{jan25eqn16}), and (\ref{noveqn600}) hold easily if $|t|\leq 1$. Hence, from now on, we restrict ourself to the case when $|t|\geq 1$. Note that, for any fixed $t$, s.t., $|t|\geq 1$, there exists a unique $m\in \mathbb{Z}_{+}$, s.t., $t\in[2^{m-1},2^m]$.

 We first prove the desired estimate (\ref{noveqn500}). Note that
\be\label{noveqn503}
|\xi|-\mu\hat{v}\cdot \xi \gtrsim \frac{|\xi|}{1+|v|^2} + |\xi|(1-\cos(\angle(\xi, \mu v )))\gtrsim \frac{|\xi| \angle(\xi, \mu v)}{1+|v|}.
\ee
On one hand, after doing dyadic localization for the angle between $\xi$ and $\mu v$ and using the volume of support of $\xi$ and the above estimate (\ref{noveqn503}), we have
\[
\Big| \int_{\R^3} e^{ix\cdot \xi - i\mu t |\xi|} \frac{m(\xi)\tilde{V}_i\cdot \xi}{|\xi|-\mu \hat{v}\cdot \xi}  \widehat{f}(\xi)\psi_{\geq 10}(t(|\xi|- \mu \hat{v}\cdot \xi))\psi_k(\xi) d \xi  \Big|
\]
 \be\label{jan25eqn41}
\lesssim \sum_{l\in\mathbb{Z}, l\leq 2} 2^{3k+ l} \|m(\xi)\|_{\mathcal{S}^\infty_k} \|\widehat{f}(t, \xi)\psi_k(\xi)\|_{L^\infty_\xi} \lesssim \sum_{\alpha\in \mathcal{B}, |\alpha|\leq 5}    2^{2k-4k_{+}}\|m(\xi)\|_{\mathcal{S}^\infty_k}\|f^\alpha\|_{X_0}.
\ee
On the other hand, for any fixed $x$ and $v$, we define $\xi_0:=\mu x/|x|$ and $\bar{l}:= -m/2-k/2-10$. 
 Note that there exists a unique constant $l_{x,v}$, which depends on $x$ and $v$, such that
\be\label{march29eqn61}
\angle(\xi_0, \mu v)\in (2^{l_{x,v}-1}, 2^{l_{x,v}}].
\ee 
Using this observation, we know that the following partition of unity  holds, 
\[
1= \psi_{\leq \bar{l}}(\angle(\xi, \xi_0)) + \psi_{>  \bar{l}}(\angle(\xi, \xi_0)) \psi_{\leq \bar{l}}(\angle(\xi, \mu v)) + \psi_{>  \bar{l}}(\angle(\xi, \xi_0)) \psi_{> \bar{l}}(\angle(\xi, \mu v)) =  \psi_{\leq \bar{l}}(\angle(\xi, \xi_0))  + \psi_{>  \bar{l}}(\angle(\xi, \xi_0))
 \]
 \be\label{partitionofunity}
 \times\psi_{\leq \bar{l}}(\angle(\xi, \mu v)) + \sum_{l_1> \bar{l}}  \psi_{l_1}(\angle(\xi, \xi_0)) \psi_{\geq l_1-10}(\angle(\xi, \mu v)) + \sum_{\begin{subarray}{c}
 \bar{l}< l_2 < l_1-10\\
    \bar{l}< l_1\leq 2, |l_1-l_{x,v}|\leq 10\\
  \end{subarray}} \psi_{l_1}(\angle(\xi, \xi_0))\psi_{l_2}(\angle(\xi, \mu v)). 
 \ee
Hence, the following decomposition holds, 
\be\label{jan25eqn11}
\int_{\R^3} e^{ix\cdot \xi - i\mu t |\xi|} \frac{1}{1+|v|} \frac{m(\xi)\tilde{V}_i\cdot \xi}{|\xi|-\mu \hat{v}\cdot \xi}  \widehat{f}(\xi)\psi_{\geq 10}(t(|\xi|- \mu \hat{v}\cdot \xi))\psi_k(\xi)  d \xi = \sum_{  \bar{l}\leq l \leq 2}I_{l  }  +  \sum_{\begin{subarray}{c}
 \bar{l}< l_2 < l_1-10\\
 \bar{l}< l_1\leq 2, |l_1-l_{x,v}|\leq 10\\
  \end{subarray}}   I_{l_1,l_2},
\ee
where 
\[
 I_{\bar{l}}= \int_{\R^3} e^{ix\cdot \xi - i\mu t |\xi|} \frac{1}{1+|v|} \frac{m(\xi)\tilde{V}_i\cdot \xi}{|\xi|-\mu \hat{v}\cdot \xi}  \widehat{f}(\xi)\psi_{\geq 10}(t(|\xi|- \mu \hat{v}\cdot \xi))\psi_k(\xi)  
 \big[\psi_{\leq \bar{l}}(\angle(\xi, \xi_0)) + \psi_{> \bar{l}}(\angle(\xi, \xi_0)) \psi_{\leq \bar{l}}(\angle(\xi, \mu v)) \big] d \xi,  \]
  \[
 I_l= \int_{\R^3} e^{ix\cdot \xi - i\mu t |\xi|} \frac{1}{1+|v|} \frac{m(\xi)\tilde{V}_i\cdot \xi}{|\xi|-\mu \hat{v}\cdot \xi}  \widehat{f}(\xi)\psi_{\geq 10}(t(|\xi|- \mu \hat{v}\cdot \xi))\psi_k(\xi)     \psi_{l }(\angle(\xi, \xi_0)) \psi_{\geq l -10}(\angle(\xi, \mu v))d \xi,\quad \textup{if}\, l > \bar{l},
\]
\[
 I_{l_1,l_2}= \int_{\R^3} e^{ix\cdot \xi - i\mu t |\xi|} \frac{1}{1+|v|} \frac{m(\xi)\tilde{V}_i\cdot \xi}{|\xi|-\mu \hat{v}\cdot \xi}  \widehat{f}(\xi)\psi_{\geq 10}(t(|\xi|- \mu \hat{v}\cdot \xi))\psi_k(\xi)    \psi_{l_1}(\angle(\xi, \xi_0))\psi_{l_2}(\angle(\xi, \mu v))  d \xi.  
\]
 From the volume of support of $\xi$ and the estimate (\ref{noveqn503}), the following estimate holds for $I_{\bar{l}}$,
\be\label{noveqn524}
|I_{\bar{l}}|\lesssim 2^{3k+2\bar{l}} \|m(\xi)\|_{\mathcal{S}^\infty_k}  \|\widehat{f}(\xi)\psi_k(\xi)\|_{L^\infty_\xi}\lesssim   \sum_{\alpha\in \mathcal{B}, |\alpha|\leq 5}   2^{-m + k-4k_{+}}\|m(\xi)\|_{\mathcal{S}^\infty_k}  \|f^\alpha\|_{X_0}.
\ee
  For $I_{l_1,l_2}$ and $I_l$, $l> \bar{l}$, we do integration by parts in ``$\xi$''.  As a result, we have
\be\label{jan25eqn6}
I_l = I_l^1 + I_l^2,\quad I_{l_1,l_2}= I_{l_1,l_2}^1+ I_{l_1,l_2}^2,
\ee
where
\[
I_l^1= \int e^{i x \cdot \xi -\mu it|\xi|} i \frac{  \frac{x}{t} - \mu\frac{\xi}{|\xi|} }{t\big|\frac{x}{t} - \mu\frac{\xi}{|\xi|}\big|^2} \cdot \nabla_\xi \widehat{f}(\xi)  \frac{1}{1+|v|} \frac{m(\xi)\tilde{V}_i\cdot \xi}{|\xi|-\mu \hat{v}\cdot \xi}  \psi_{\geq 10}(t(|\xi|- \mu \hat{v}\cdot \xi))\psi_k(\xi)   \psi_{l  }(\angle(\xi, \xi_0)) \psi_{\geq l  -10}(\angle(\xi, \mu v))   d \xi,
\]
\[
I_l^2= \int e^{i x \cdot \xi -\mu it|\xi|} i \widehat{f}(\xi) \nabla_\xi \cdot  \big[\frac{\frac{x}{t} - \mu\frac{\xi}{|\xi|}}{t\big|\frac{x}{t} - \mu\frac{\xi}{|\xi|}\big|^2}    \frac{1}{1+|v|} \frac{m(\xi)\tilde{V}_i\cdot \xi}{|\xi|-\mu \hat{v}\cdot \xi}  \psi_{\geq 10}(t(|\xi|- \mu \hat{v}\cdot \xi))\psi_k(\xi)   \psi_{l  }(\angle(\xi, \xi_0)) \psi_{\geq l  -10}(\angle(\xi, \mu v))  \big] d \xi,
\]
\[
I_{l_1,l_2}^1= \int e^{i x \cdot \xi -\mu it|\xi|} i \frac{  \frac{x}{t} - \mu\frac{\xi}{|\xi|} }{t\big|\frac{x}{t} - \mu\frac{\xi}{|\xi|}\big|^2} \cdot \nabla_\xi \widehat{f}(\xi)  \frac{1}{1+|v|} \frac{m(\xi)\tilde{V}_i\cdot \xi}{|\xi|-\mu \hat{v}\cdot \xi}  \psi_{\geq 10}(t(|\xi|- \mu \hat{v}\cdot \xi))\psi_k(\xi)   \psi_{l_1}(\angle(\xi, \xi_0))\psi_{l_2}(\angle(\xi, \mu v))   d \xi,
\]
\[
I_{l_1,l_2}^2= \int e^{i x \cdot \xi -\mu it|\xi|} i \widehat{f}(\xi) \nabla_\xi \cdot  \big[\frac{\frac{x}{t} - \mu\frac{\xi}{|\xi|}}{t\big|\frac{x}{t} - \mu\frac{\xi}{|\xi|}\big|^2}    \frac{1}{1+|v|} \frac{m(\xi)\tilde{V}_i\cdot \xi}{|\xi|-\mu \hat{v}\cdot \xi}  \psi_{\geq 10}(t(|\xi|- \mu \hat{v}\cdot \xi))\psi_k(\xi)   \psi_{l_1}(\angle(\xi, \xi_0))\psi_{l_2}(\angle(\xi, \mu v))  \big] d \xi,
\]
For $I_{l}^2$  and $I_{l_1,l_2}^2$, we  do integration by parts in $\xi$ one more time. As a result, we have
\[
I_l^2= \int e^{i x \cdot \xi -\mu it|\xi|} - \nabla_\xi\cdot \Big[\frac{\frac{x}{t} - \mu\frac{\xi}{|\xi|}}{t\big|\frac{x}{t} - \mu\frac{\xi}{|\xi|}\big|^2} \widehat{f}(\xi) \nabla_\xi \cdot  \big[\frac{\frac{x}{t} - \mu\frac{\xi}{|\xi|}}{t\big|\frac{x}{t} - \mu\frac{\xi}{|\xi|}\big|^2}    \frac{1}{1+|v|} \frac{m(\xi)\tilde{V}_i\cdot \xi}{|\xi|-\mu \hat{v}\cdot \xi}   \psi_{\geq 10}(t(|\xi|- \mu \hat{v}\cdot \xi))\]
\[
\times \psi_k(\xi)  \psi_{l  }(\angle(\xi, \xi_0))   \psi_{\geq l  -10}(\angle(\xi, \mu v))  \big] \Big] 	d \xi,
\]
\[
I_{l_1,l_2}^2= \int e^{i x \cdot \xi -\mu it|\xi|} - \nabla_\xi\cdot \Big[\frac{\frac{x}{t} - \mu\frac{\xi}{|\xi|}}{t\big|\frac{x}{t} - \mu\frac{\xi}{|\xi|}\big|^2} \widehat{f}(\xi) \nabla_\xi \cdot  \big[\frac{\frac{x}{t} - \mu\frac{\xi}{|\xi|}}{t\big|\frac{x}{t} - \mu\frac{\xi}{|\xi|}\big|^2}    \frac{1}{1+|v|} \frac{m(\xi)\tilde{V}_i\cdot \xi}{|\xi|-\mu \hat{v}\cdot \xi}  \psi_{\geq 10}(t(|\xi|- \mu \hat{v}\cdot \xi))  \]
\[
\times \psi_k(\xi) \psi_{l_1}(\angle(\xi, \xi_0))\psi_{l_2}(\angle(\xi, \mu v))  \big] \Big]d \xi.
\]
 Therefore, from the estimate (\ref{noveqn503}) and the volume of support of ``$\xi$'', the following estimate holds 
\be\label{jan25eqn1}
|I_{l}^1| \lesssim 2^{-m-l+3k+2l}  \|m(\xi)\|_{\mathcal{S}^\infty_k} \|\nabla_\xi \widehat{f}(t, \xi) \psi_k(\xi)\|_{L^\infty_\xi},
\ee
\be\label{jan25eqn2} 
|I_{l_1, l_2}^1|\lesssim 2^{-m-l_1+3k+2l_2}\|m(\xi)\|_{\mathcal{S}^\infty_k}   \|\nabla_\xi \widehat{f}(t, \xi) \psi_k(\xi)\|_{L^\infty_\xi}, 
\ee
\be\label{jan25eqn3}
|I_{l}^2| \lesssim 2^{-2m-2l +3k+2l}\|m(\xi)\|_{\mathcal{S}^\infty_k} \big(2^{-2k-2l} \|\widehat{f}(t,\xi)\psi_k(\xi)\|_{L^\infty_\xi} +2^{-k-l} \|\nabla_\xi \widehat{f}(t, \xi) \psi_k(\xi)\|_{L^\infty_\xi} \big) ,
\ee

\be\label{jan25eqn4}
|I_{l_1, l_2}^2| \lesssim 2^{-2m-2l_1 +3k+2l_2}\|m(\xi)\|_{\mathcal{S}^\infty_k} \big(2^{-2k-2l_2} \|\widehat{f}(t,\xi)\psi_k(\xi)\|_{L^\infty_\xi} +2^{-k-l_2} \|\nabla_\xi \widehat{f}(t, \xi) \psi_k(\xi)\|_{L^\infty_\xi} \big) ,
\ee
Recall that $\bar{l}:=-m/2-k/2-10$. From  the estimates (\ref{jan25eqn1}), (\ref{jan25eqn2}), (\ref{jan25eqn3}), and (\ref{jan25eqn4}), we have 
\be\label{noveqn523}
\sum_{  \bar{l} < l \leq 2} |I_{l  }^1| + |I_{l  }^2|  +  \sum_{\begin{subarray}{c}
 \bar{l}< l_2 < l_1-10\\
 \bar{l}< l_1\leq 2, |l_1-l_{x,v}|\leq 10\\
  \end{subarray}}  | I_{l_1,l_2}^1| + | I_{l_1,l_2}^2| \lesssim   \sum_{\alpha\in \mathcal{B}, |\alpha|\leq 5}  2^{-m+k-4k_{+}} \|m(\xi)\|_{\mathcal{S}^\infty_k}  \big(\|f^\alpha\|_{X_0} +\| f^\alpha\|_{X_1} \big).
\ee

To sum up, recall the decompositions (\ref{jan25eqn11}) and (\ref{jan25eqn6}), our desired estimate  (\ref{noveqn500}) holds from the estimates (\ref{jan25eqn41}),  (\ref{noveqn524}), and (\ref{noveqn523}). With minor modifications,  all other desired estimates (\ref{jan25eqn16})  and (\ref{noveqn600}) hold very similarly, we omit details here.

\end{proof}

\begin{lemma}\label{auxillarylemmajan26}
For any $i, j\in \{1,2,3\}$, the following estimate holds for any fixed $x\in \R^3$,
\[
 |x|\Big|   \int_{\R^3} e^{i(x+t\hat{v})\cdot \xi - i\mu t |\xi|} \frac{m(\xi) \tilde{V}_j \cdot \xi}{|\xi|-\mu \hat{v}\cdot \xi}   \widehat{f}(\xi)\psi_{\geq 10}(t(|\xi|- \mu \hat{v}\cdot \xi)) \psi_k(\xi) d \xi  \Big|
\]
\be\label{noveqn532}
\lesssim\sum_{\alpha\in \mathcal{B}, |\alpha|\leq 5}   2^{ k-4k_{+}} \|m(\xi)\|_{\mathcal{S}^\infty_k}  \big( \|f^\alpha\|_{X_0}+ \|f^\alpha\|_{X_1}\big).
\ee 
\end{lemma}
\begin{proof}
 Recall the first equality in  (\ref{octeqn456}). The following decomposition holds,
 \be\label{april6eqn1}
|x|= \frac{x}{|x|}\cdot\big(\tilde{v} x\cdot \tilde{v}+ \sum_{i=1,2,3} \tilde{V}_i x\cdot \tilde{V}_i \big).
 \ee
 Therefore, to prove the desired estimate (\ref{noveqn532}), it would be sufficient to control both the radial part and the rotational parts. Note that
    the following decomposition holds for any $i,j\in\{1,2,3\},$
\be\label{jan25eqn22}
(x\cdot \tilde{V}_i) \int_{\R^3} e^{i(x+t\hat{v})\cdot \xi - i\mu t |\xi|}  \frac{m(\xi) \tilde{V}_j \cdot \xi}{|\xi|-\mu \hat{v}\cdot \xi}   \widehat{f}(\xi)\psi_{\geq 10}(t(|\xi|- \mu \hat{v}\cdot \xi)) \psi_k(\xi) d \xi = J_1+J_2,
\ee
\be\label{april6eqn2}
(x\cdot \tilde{v} ) \int_{\R^3}  e^{i(x+t\hat{v})\cdot \xi - i\mu t |\xi|}  \frac{m(\xi) \tilde{V}_j \cdot \xi}{|\xi|-\mu \hat{v}\cdot \xi}   \widehat{f}(\xi)\psi_{\geq 10}(t(|\xi|- \mu \hat{v}\cdot \xi)) \psi_k(\xi) d \xi = K_1+K_2,
\ee
where
\be\label{april6eqn10}
J_1=  \int_{\R^3}  e^{i(x+t\hat{v})\cdot \xi - i\mu t |\xi|}  \big(  x -\frac{ \mu t\xi}{|\xi|} \big)\cdot\tilde{V}_i \frac{m(\xi) \tilde{V}_j \cdot \xi}{|\xi|-\mu \hat{v}\cdot \xi}   \widehat{f}(\xi)\psi_{\geq 10}(t(|\xi|- \mu \hat{v}\cdot \xi)) \psi_k(\xi) d \xi,
\ee 
\be\label{april6eqn11}
J_2=  \int_{\R^3} e^{i(x+t\hat{v})\cdot \xi - i\mu t |\xi|}   \frac{ \mu t\xi}{|\xi|}   \cdot  \tilde{V}_i \frac{m(\xi) \tilde{V}_j \cdot \xi}{|\xi|-\mu \hat{v}\cdot \xi}   \widehat{f}(\xi)\psi_{\geq 10}(t(|\xi|- \mu \hat{v}\cdot \xi)) \psi_k(\xi) d \xi,
\ee
\be\label{april6eqn12}
K_1=  \int_{\R^3}  e^{i(x+t\hat{v})\cdot \xi - i\mu t |\xi|}  \big(  x+t\hat{v} -\frac{ \mu t\xi}{|\xi|} \big)\cdot   \tilde{v}  \frac{m(\xi) \tilde{V}_j \cdot \xi}{|\xi|-\mu \hat{v}\cdot \xi}   \widehat{f}(\xi)\psi_{\geq 10}(t(|\xi|- \mu \hat{v}\cdot \xi)) \psi_k(\xi) d \xi,
\ee
\be\label{april6eqn13}
K_2=  \int_{\R^3}  e^{i(x+t\hat{v})\cdot \xi - i\mu t |\xi|}  -t\big( \hat{v} -\frac{ \mu \xi}{|\xi|}  \big) \cdot\tilde{v}   \frac{m(\xi) \tilde{V}_j \cdot \xi}{|\xi|-\mu \hat{v}\cdot \xi}   \widehat{f}(\xi)\psi_{\geq 10}(t(|\xi|- \mu \hat{v}\cdot \xi)) \psi_k(\xi) d \xi. 
\ee 
Note that
\[
 e^{i(x+t\hat{v})\cdot \xi - i\mu t |\xi|} \big(  x -\frac{ \mu t\xi}{|\xi|} \big)\cdot \tilde{V}_i= -i  \tilde{V}_i\cdot\nabla_\xi \big(e^{i(x+t\hat{v})\cdot \xi - i\mu t |\xi|}\big),
 e^{i(x+t\hat{v})\cdot \xi - i\mu t |\xi|} \big(  x+t\hat{v} -\frac{ \mu t\xi}{|\xi|} \big)\cdot   \tilde{v}= -i \tilde{v}\cdot \nabla_\xi \big(e^{i(x+t\hat{v})\cdot \xi - i\mu t |\xi|}\big).
\]
Therefore,  we do integration by parts in $\xi$ in the $\tilde{V}_i$ direction for $J_1$ and   do integration by parts in $\xi$ in the $\tilde{v}$ direction for $K_1$. As a result, we have
\be\label{jan25eqn24}
J_1= \tilde{J}_1+\tilde{J}_2,\quad K_1=\tilde{K}_1+\tilde{K}_2,
\ee
where
\be\label{april6eqn16}
\tilde{J}_1:= \int_{\R^3} e^{ix\cdot \xi - i\mu t |\xi|}  i \tilde{V}_i\cdot \nabla_\xi\big[ \frac{m(\xi) \tilde{V}_j \cdot \xi}{|\xi|-\mu \hat{v}\cdot \xi}   \psi_{\geq 10}(t(|\xi|- \mu \hat{v}\cdot \xi)) \psi_k(\xi)\big]\widehat{f}(\xi) d \xi,  
\ee
\be\label{april6eqn17}
\tilde{J}_2:= \int_{\R^3} e^{ix\cdot \xi - i\mu t |\xi|}  i \tilde{V}_i\cdot \nabla_\xi \widehat{f}(\xi)  \frac{m(\xi) \tilde{V}_j \cdot \xi}{|\xi|-\mu \hat{v}\cdot \xi}   \psi_{\geq 10}(t(|\xi|- \mu \hat{v}\cdot \xi)) \psi_k(\xi)  d \xi. 
\ee
\be\label{april6eqn18}
\tilde{K}_1:= \int_{\R^3} e^{ix\cdot \xi - i\mu t |\xi|}  i \tilde{v} \cdot \nabla_\xi\big[ \frac{m(\xi) \tilde{V}_j \cdot \xi}{|\xi|-\mu \hat{v}\cdot \xi}   \psi_{\geq 10}(t(|\xi|- \mu \hat{v}\cdot \xi)) \psi_k(\xi)\big]\widehat{f}(\xi) d \xi, 
\ee
\be\label{april6eqn19}
\tilde{K}_2:= \int_{\R^3} e^{ix\cdot \xi - i\mu t |\xi|}  i \tilde{v} \cdot \nabla_\xi \widehat{f}(\xi)  \frac{m(\xi) \tilde{V}_j \cdot \xi}{|\xi|-\mu \hat{v}\cdot \xi}   \psi_{\geq 10}(t(|\xi|- \mu \hat{v}\cdot \xi)) \psi_k(\xi)  d \xi. 
\ee
Recall (\ref{april6eqn11}) and (\ref{april6eqn16}). 
  From the estimates (\ref{jan25eqn16}) and (\ref{noveqn600}) in Lemma \ref{twisteddecaylemma3},    we have the following estimate, 
\be\label{noveqn551}
|J_2|+|\tilde{J}_1|\lesssim  \sum_{\alpha\in \mathcal{B}, |\alpha|\leq 5}   2^{ k-4k_{+}}  \|m(\xi)\|_{\mathcal{S}^\infty_k} \big( \|f^\alpha\|_{X_0}+ \|f^\alpha\|_{X_1}\big).
\ee
Now, we proceed to estimate $K_2$ and $\tilde{K}_1$. Recall (\ref{april6eqn13}) and (\ref{april6eqn18}). Note that
\[
\big|\big( \hat{v} -\frac{ \mu \xi}{|\xi|}  \big) \cdot\tilde{v} \big| \lesssim  \frac{|\xi|-\mu \hat{v}\cdot \xi }{|\xi|}, \quad \big|\tilde{v}\cdot \nabla_\xi\big(|\xi|-\mu \hat{v}\cdot \xi \big)\big|=\big|\frac{\tilde{v}\cdot \xi}{|\xi|}- \mu \tilde{v}\cdot \hat{v}\big|\lesssim \frac{|\xi|-\mu \hat{v}\cdot \xi }{|\xi|}.
\]
Therefore, from the above estimate and   the estimate  (\ref{noveqn600}) in Lemma \ref{twisteddecaylemma3} and the estimate (\ref{noveqn555}) in Lemma \ref{twistedlineardecay}, we have
\be\label{april6eqn20}
|K_2|+|\tilde{K}_1|\lesssim  \sum_{\alpha\in \mathcal{B}, |\alpha|\leq 5}   2^{ k-4k_{+}}  \|m(\xi)\|_{\mathcal{S}^\infty_k} \big( \|f^\alpha\|_{X_0}+ \|f^\alpha\|_{X_1}\big).
\ee
Lastly, from the volume of support of $\xi$, the following estimate holds for $\tilde{J}_2$ and $\tilde{K}_2$, 
\be\label{jan26eqn21}
|\tilde{J}_2| +|\tilde{K}_2|\lesssim \sum_{l\in \mathbb{Z}, l\leq 2} 2^{3k + l}  \|m(\xi)\|_{\mathcal{S}^\infty_k}\|\nabla_\xi\widehat{f}(t, \xi)\psi_k(\xi)\|_{L^\infty_\xi }\lesssim \sum_{\alpha\in \mathcal{B}, |\alpha|\leq 5} 2^{k-4k_{+}} \|m(\xi)\|_{\mathcal{S}^\infty_k}  \big( \|f^\alpha\|_{X_0}+ \|f^\alpha\|_{X_1}\big).
\ee

To sum up, recall  the decompositions (\ref{april6eqn1}),  (\ref{jan25eqn22}), (\ref{april6eqn2}), and (\ref{jan25eqn24}),    our desired estimate   (\ref{noveqn532}) holds from the estimates (\ref{noveqn551}), (\ref{april6eqn20}), and (\ref{jan26eqn21}),.
\end{proof}

\end{document}